\documentclass{amsart}

\usepackage{amssymb}
\usepackage{amsmath}
\usepackage{amsthm}
\usepackage{amsfonts}

\usepackage{a4wide}
\usepackage{longtable}
\usepackage{multirow}
\usepackage{setspace}
\usepackage{stackrel}

\newtheorem{theorem}{Theorem}[section]
\newtheorem{proposition}{Proposition}[section]
\newtheorem{corollary}{Corollary}[section]
\newtheorem{lemma}{Lemma}[section]
\newtheorem{conjecture}{Conjecture}[section]

\theoremstyle{definition}
\newtheorem{definition}{Definition}[section]
\newtheorem{remark}{Remark}[section]
\newtheorem{example}{Example}[section]

\numberwithin{equation}{section}

\begin{document}

\title[Co-periodic forests of finite \(p\)-groups]
{Co-periodicity isomorphisms \\ between forests of finite \(p\)-groups}

\author{Daniel C. Mayer}
\address{Naglergasse 53 \\ 8010 Graz \\ Austria}
\email{algebraic.number.theory@algebra.at}
\urladdr{http://www.algebra.at}

\thanks{Research supported by the Austrian Science Fund (FWF): P 26008-N25}


\date{December 21, 2017}

\begin{abstract}
Based on a general theory of descendant trees of finite \(p\)-groups
and the virtual periodicity isomorphisms between the branches of a coclass subtree,
the behavior of algebraic invariants of the tree vertices
and their automorphism groups under these isomorphisms
is described with simple transformation laws.
For the tree of finite \(3\)-groups with elementary bicyclic commutator quotient,
the information content of each coclass subtree with metabelian mainline is shown to be finite.
As a striking novelty in this paper,
evidence is provided of co-periodicity isomorphisms between coclass forests
which reduce the information content of the entire metabelian skeleton
and a significant part of non-metabelian vertices
to a finite amount of data.
\end{abstract}

\maketitle

%\newpage

%--------------------------------------------------------------------------------

\section{Introduction}
\label{s:Intro}
\noindent
Denote by \(\mathcal{T}\) the rooted tree of all finite \(3\)-groups \(G\)
with elementary bicyclic commutator quotient \(G/G^\prime\simeq C_3\times C_3\),
and let \(\mathcal{T}_\ast\) be the infinite pruned subtree of \(\mathcal{T}\),
where all descendants of capable non-metabelian vertices are eliminated.
The main intention of this paper is to prove that the information content of the tree \(\mathcal{T}_\ast\) 
can be reduced to a finite set of representatives with the aid of two kinds of periodicity.
\begin{itemize}
\item
Firstly,
the well-known \textit{virtual periodicity} isomorphisms \(\mathcal{B}(n+\ell)\simeq\mathcal{B}(n)\)
between the finite depth-pruned branches
\(\mathcal{B}(n)\), \(n\ge n_\ast\), of a coclass subtree \(\mathcal{T}^r\subset\mathcal{T}_\ast\)
are refined to \textit{strict periodicity} isomorphisms between complete branches
which reduce the information content of the infinite coclass subtree to the finite union
of pre-period \((\mathcal{B}(n))_{n_\ast\le n<p_\ast}\)
and first primitive period \((\mathcal{B}(n))_{p_\ast\le n<p_\ast+\ell}\).
The virtual periodicity was proved by du Sautoy
\cite{dS}
and independently by Eick and Leedham-Green
\cite{EkLg}
for groups of any prime power order.
The strict periodicity for \(p=3\) and type \(C_3\times C_3\) is proved in the present paper.
\item
Secondly,
evidence is provided of \textit{co-periodicity} isomorphisms \(\mathcal{F}(r+2)\simeq\mathcal{F}(r)\)
between the infinite coclass forests
\(\mathcal{F}(r)\), \(r\ge 1\),
which reduce the information content of the pruned tree \(\mathcal{T}_\ast\) to the union
of pre-period \((\mathcal{F}(r))_{1\le r\le 4}\) and first primitive period \((\mathcal{F}(r))_{5\le r\le 6}\),
consisting of the leading six coclass forests only.
The discovery of this co-periodicity is the progressive innovation in the present paper.
\end{itemize}
Together with the coclass theorems of Leedham-Green
\cite{Lg}
and Shalev
\cite{Sv},
which imply that each coclass forest \(\mathcal{F}(r)\) consists of
a finite sporadic part \(\mathcal{F}_0(r)\) and a finite number of coclass trees \(\mathcal{T}^r_j\), \(1\le j\le t\),
each having a finite information content due to the strict periodicity,
this shows that the pruned infinite subtree \(\mathcal{T}_\ast\) of the tree \(\mathcal{T}\)
is described by finitely many representatives only.

We begin with a general theory of descendant trees of finite \(p\)-groups with arbitrary prime \(p\) in \S\
\ref{s:TreesForests}
and we explain the conceptual foundations of
the \textit{virtual periodicity} isomorphisms between the finite branches of coclass subtrees
\cite{dS,EkLg,dSSg}
and the recently discovered \textit{co-periodicity} isomorphisms between infinite coclass forests in \S\
\ref{s:IsomorphicDigraphs}.
The behavior of algebraic invariants of the tree vertices and their automorphism groups
is described with simple \textit{transformation laws} in \S\
\ref{s:StructuredDigraphs}.
The graph theoretic preliminaries are supplemented by connections
between depth, width, \textit{information content} and numbers of immediate descendants in \S\
\ref{s:GraphStruc},
identifiers of groups in \S\
\ref{s:SmallGroups},
and precise definitions of mainlines and sporadic parts in \S\
\ref{s:MainlineSporadic}.
The main theorems are presented in \S\
\ref{s:MainTheorems}.

Then we focus on the tree \(\mathcal{T}\) of finite \(3\)-groups \(G\) with abelianization \(G/G^\prime\simeq C_3\times C_3\).
The flow of our investigations is guided by \S\
\ref{s:MainTrunk}
concerning the remarkable infinite \textit{main trunk} \((P_{2r-1})_{r\ge 2}\) of certain metabelian vertices in \(\mathcal{T}\)
which gives rise to the top vertices of all coclass forests \(\mathcal{F}(r)\), \(r\ge 2\), by periodic bifurcations
and constitutes the germ of the newly discovered co-periodicity \(\mathcal{F}(r+2)\simeq\mathcal{F}(r)\) of length two.
To start with a beautiful highlight, we immediately celebrate
the simple structure of the first primitive period \((\mathcal{F}(r))_{5\le r\le 6}\) in \S\S\
\ref{s:PeriodicSporadic4}
and
\ref{s:PeriodicSporadic5}
and defer the somewhat arduous task of describing
the exceptional pre-period \((\mathcal{F}(r))_{1\le r\le 4}\) to the concluding \S\S\
\ref{s:Coclass1}
and
\ref{s:Conclusion}.

Finally, we point out that our theory,
together with the investigations of Eick
\cite{Ek},
provides an independent verification and confirmation
of all results about the metabelian skeleton \(\mathcal{M}\) of the tree \(\mathcal{T}\)
in the dissertation of Nebelung
\cite{Ne},
since \(\mathcal{M}\) is a subtree of the pruned tree \(\mathcal{T}_\ast\).
The present paper shows the co-periodicity of the sporadic parts \(\mathcal{F}_0(r)\)
and coclass trees \(\mathcal{T}_j^r\), \(1\le j\le t\), of the coclass forests \(\mathcal{F}(r)\), and
\cite[\S\ 5.2, pp. 114--116]{Ek}
establishes the connection between the coclass trees \(\mathcal{T}_j^r\) and
infinite metabelian pro-\(3\) groups of coclass \(r\).

%\newpage

%--------------------------------------------------------------------------------

\section{Descendant trees and coclass forests}
\label{s:TreesForests}
\noindent
Let \(p\) be a prime number.
In the mathematical theory of finite groups of order a power of \(p\),
so-called \(p\)-\textit{groups},
the introduction of the \textit{parent-child relation} by Leedham-Green and Newman
\cite[pp. 194--195]{LgNm}
has simplified the classification of such groups considerably.
The relation is defined in terms of the \textit{lower central series} \((\gamma_i{G})_{i\ge 1}\) of a \(p\)-group \(G\), where
\begin{equation}
\label{eqn:LowerCentral}
\gamma_1{G}:=G \quad \text{ and } \quad \gamma_i{G}:=\lbrack\gamma_{i-1}{G},G\rbrack \quad \text{ for } i\ge 2,
\end{equation}
in particular, \(\gamma_2{G}=G^\prime\) is the commutator subgroup of \(G\).
Since the series becomes stationary,
\begin{equation}
\label{eqn:Nilpotency}
\gamma_1{G}>\gamma_2{G}>\ldots>\gamma_c{G}>\gamma_{c+1}{G}, \quad \text{ and } \quad \gamma_i{G}=1 \quad \text{ for } i\ge c+1,
\end{equation}
a non-trivial \(p\)-group \(G>1\) is \textit{nilpotent} of \textit{class} \(\mathrm{cl}(G)=c\ge 1\).

%--------------------------------------------------------------------------------

\begin{definition}
\label{dfn:ParentChild}
If \(G\) is non-abelian,
then the class-\((c-1)\) quotient
\begin{equation}
\label{eqn:Parent}
\pi{G}:=G/\gamma_c{G} \quad \text{ with } \quad c\ge 2
\end{equation}
is called the \textit{parent} of \(G\),
and \(G\) is a \textit{child} (or \textit{immediate descendant}) of \(\pi{G}\).
\end{definition}

%--------------------------------------------------------------------------------

Parent and child share a common class-\(1\) quotient (or derived quotient or \textit{abelianization}), since
\begin{equation}
\label{eqn:Class1Quot}
\gamma_2(\pi{G})=\gamma_2{G}/\gamma_c{G} \quad \text{ and } \quad \pi{G}/\gamma_2(\pi{G})=(G/\gamma_c{G})/(\gamma_2{G}/\gamma_c{G})\simeq G/\gamma_2{G},
\end{equation}
according to the isomorphism theorem.
The lower central series of \(\pi{G}\) is shorter by one term:
\begin{equation}
\label{eqn:ParentNilpotency}
\gamma_1(\pi{G})=\gamma_1{G}/\gamma_c{G}>\gamma_2(\pi{G})=\gamma_2{G}/\gamma_c{G}>\ldots>\gamma_{c-1}(\pi{G})=\gamma_{c-1}{G}/\gamma_c{G}>\gamma_c(\pi{G})=1,
\end{equation}
and thus \(\mathrm{cl}(\pi{G})=c-1=\mathrm{cl}(G)-1\).

%--------------------------------------------------------------------------------

\begin{definition}
\label{dfn:DescTree}
For an assigned finite \(p\)-group \(R>1\),
the \textit{descendant tree} \(\mathcal{T}(R)\) with root \(R\)
is defined as the digraph \((V,E)\)
whose set of vertices \(V\) consists of all isomorphism classes of \(p\)-groups \(G\)
with \(G/\gamma_i{G}\simeq R\), for some \(2\le i\le\mathrm{cl}(G)+1\),
and whose set of directed edges \(E\) consists of all child-parent pairs
\begin{equation}
\label{eqn:DirectedEdge}
(G\to\pi{G}):=(G,\pi{G})\in V\times V.
\end{equation}
The mapping \(\pi:\,V\setminus\lbrace R\rbrace\to V\) is called the \textit{parent operator}.
\end{definition}

%--------------------------------------------------------------------------------

If the root \(R\) is \textit{abelian}, then all vertices of the tree \(\mathcal{T}(R)\)
share the common abelianization \(G/G^\prime\simeq R\).
Since a nilpotent group with cyclic abelianization is abelian,
the descendant tree \(\mathcal{T}(R)\) of a cyclic root \(R>1\) consists of the single isolated vertex \(R\).
The classification of \(p\)-groups by their abelianization is refined further,
if directed edges are restricted to starting vertices \(G\)
with cyclic last non-trivial lower central \(\gamma_c{G}\) of order \(p\).
Then the descendant tree of \(R\) splits into a countably infinite disjoint union
\begin{equation}
\label{eqn:CoclassSubGraphs}
\mathcal{T}(R)=\dot{\cup}_{r=e-1}^\infty\,\mathcal{G}(p,r,R)
\end{equation}
of directed subgraphs,
where \(\mathrm{ord}(R)=p^e\) and the vertices of the component \(\mathcal{G}(p,r,R)\) with fixed \(r\ge e-1\)
share the same \textit{coclass}, \(\mathrm{cc}=r\), as a common invariant,
since the logarithmic order \(\mathrm{lo}:=(\log_p\circ\,\mathrm{ord})\) and the nilpotency class \(\mathrm{cl}\)
of the parent \(\pi{G}\) and child \(G\) satisfy the rule
\begin{equation}
\label{eqn:FixedCoclass}
\begin{aligned}
\mathrm{lo}(\pi{G})=\mathrm{lo}(G/\gamma_c{G})=\mathrm{lo}(G)-\mathrm{lo}(\gamma_c{G})=\mathrm{lo}(G)-1, \quad \mathrm{cl}(\pi{G})=\mathrm{cl}(G/\gamma_c{G})=\mathrm{cl}(G)-1, \\
\text{whence } \quad \mathrm{cc}(\pi{G})=\mathrm{lo}(\pi{G})-\mathrm{cl}(\pi{G})=\mathrm{lo}(G)-1-(\mathrm{cl}(G)-1)=\mathrm{lo}(G)-\mathrm{cl}(G)=\mathrm{cc}(G).
\end{aligned}
\end{equation}

%--------------------------------------------------------------------------------

\begin{definition}
\label{dfn:CcGraphs}
Thus, the components \(\mathcal{G}(p,r,R)\) with \(r\ge e-1\) are called the \textit{coclass subgraphs}
of the descendant tree \(\mathcal{T}(R)\).
\end{definition}

%--------------------------------------------------------------------------------

According to the \textit{coclass theorems} by Leedham-Green
\cite{Lg}
and Shalev
\cite{Sv},
a coclass graph \(\mathcal{G}(p,r,R)\) is the disjoint union of a finite \textit{sporadic} part \(\mathcal{G}_0(p,r,R)\)
and finitely many \textit{coclass trees} \(\mathcal{T}_j^r\) (with infinite mainlines),
that is, a \textit{forest} for which there exist integers \(\tilde{s},\tilde{t}\ge 0\) such that
\begin{equation}
\label{eqn:CoclassTrees}
\mathcal{G}(p,r,R)=\mathcal{G}_0(p,r,R)\dot{\cup}\left(\dot{\bigcup}_{j=1}^{\tilde{t}}\,\mathcal{T}_j^r\right) \quad \text{ with } \quad \#\mathcal{G}_0(p,r,R)=\tilde{s}.
\end{equation}

%--------------------------------------------------------------------------------

\begin{definition}
\label{dfn:Forests}
In the present paper, the focus will lie on finite \(p\)-groups with fixed prime \(p=3\)
arising as descendants of the fixed elementary bicyclic \(3\)-group \(R:=C_3\times C_3\) of order \(3^e\) with \(e=2\),
where \(C_n\) denotes the cyclic group of order \(n\).
This assumption permits a simplified notation by omitting the explicit mention of \(p\) and \(R\).
Further, we shall slightly reduce the complexity of the forests \(\mathcal{G}(r):=\mathcal{G}(3,r,C_3\times C_3)\), \(r\ge e-1=1\),
by \textit{eliminating the descendants of capable} (i.e., non-terminal) \textit{non-metabelian} vertices.
This pruned light-weight version of \(\mathcal{G}(r)\) will be denoted by \(\mathcal{F}(r)\), called the \textit{coclass-\(r\) forest},
and Formula
\eqref{eqn:CoclassTrees}
becomes
\begin{equation}
\label{eqn:Forests}
\mathcal{F}(r)=\mathcal{F}_0(r)\dot{\cup}\left(\dot{\bigcup}_{j=1}^t\,\mathcal{T}_j^r\right) \quad \text{ with } \quad \#\mathcal{F}_0(r)=s
\end{equation}
and possibly different integers \(s\ne\tilde{s}\) and \(1\le t\le\tilde{t}\).
\end{definition}

%--------------------------------------------------------------------------------

\begin{remark}
\label{rmk:Pruning}
In \S\S\
\ref{s:PeriodicSporadic4}
and
\ref{s:PeriodicSporadic5}
it will turn out that the coclass trees \(\mathcal{T}_j^r\) with metabelian mainlines
do not contain any capable non-metabelian vertices.
So the pruning process from \(\mathcal{G}(r)\) to \(\mathcal{F}(r)\)
concerns the sporadic part \(\mathcal{F}_0(r)\),
and reduces the number \(t\le\tilde{t}\) of coclass trees
by eliminating those with non-metabelian mainlines entirely,
but does not affect the coclass trees with metabelian mainlines,
which remain complete in spite of pruning.
\end{remark}

%\newpage

%--------------------------------------------------------------------------------

\section{Isomorphic digraphs and trees}
\label{s:IsomorphicDigraphs}
\noindent
In general, we denote a graph \(\mathcal{G}\) as a pair \(\mathcal{G}=(V,E)\)
with set of vertices \(V\) and set of edges \(E\).

\begin{definition}
\label{dfn:GraphIso}
Let \(\mathcal{G}=(V,E)\) and \(\tilde{\mathcal{G}}=(\tilde{V},\tilde{E})\) be two digraphs
with directed edges in \(E\subset V\times V\), respectively \(\tilde{E}\subset\tilde{V}\times\tilde{V}\).
If there exists a bijection \(\psi:\,V\to\tilde{V}\) such that
\begin{equation}
\label{eqn:GraphIso}
\lbrack\quad (v,w)\in E\ \Longleftrightarrow\ (\psi(v),\psi(w))\in\tilde{E} \quad\rbrack, \text{ for all } (v,w)\in V\times V,
\end{equation}
then \(\mathcal{G}\) and \(\tilde{\mathcal{G}}\) are called \textit{isomorphic digraphs},
and \(\psi\) is an \textit{isomorphism of digraphs}.
\end{definition}

\noindent
When \(\mathcal{G}=(V,E)\) is a finite digraph with vertex cardinality \(\#(V)=n\in\mathbb{N}\),
we can identify \(V\) with the set \(\lbrace 1,\ldots,n\rbrace\).
Then the set of directed edges \(E\subset V\times V\) is characterized uniquely
by the characteristic function \(\chi_E\) of \(E\) in \(V\times V=\lbrace 1,\ldots,n\rbrace\times\lbrace 1,\ldots,n\rbrace\),
which is called the \(n\times n\) \textit{adjacency matrix} \(A=(a_{i,j})_{1\le i,j\le n}\) of \(\mathcal{G}\).
Its entries are defined, for all \(1\le i,j\le n\), by \\
\begin{equation}
\label{eqn:Adjacency}
a_{i,j}=\chi_E(i,j)=
\begin{cases}
1 & \text{ if } (i,j)\in E, \\
0 & \text{ otherwise}.
\end{cases}
\end{equation}

%--------------------------------------------------------------------------------

\begin{proposition}
\label{prp:GraphIso}
Let \(\mathcal{G}=(V,E)\) and \(\tilde{\mathcal{G}}=(\tilde{V},\tilde{E})\) be two finite digraphs with \(n\) vertices.
Then \(\mathcal{G}\) and \(\tilde{\mathcal{G}}\) are isomorphic if and only if
there exists a bijection \(\psi:\,V\to\tilde{V}\) such that
the entries of the adjacency matrices
\(a_{i,j}=\chi_E(i,j)=\chi_{\tilde{E}}(\psi(i),\psi(j))=\tilde{a}_{\psi(i),\psi(j)}\)
coincide for all \(1\le i,j\le n\).
\end{proposition}

\begin{proof}
The bijection \(\psi:\,V\to\tilde{V}\) satisfies the condition in Formula
\eqref{eqn:GraphIso}
if and only if
\(a_{i,j}=\chi_E(i,j)=1\ \Longleftrightarrow\ (i,j)\in E\ \Longleftrightarrow\ (\psi(i),\psi(j))\in\tilde{E}\ \Longleftrightarrow\ \tilde{a}_{\psi(i),\psi(j)}=\chi_{\tilde{E}}(\psi(i),\psi(j))=1\).
\end{proof}

%--------------------------------------------------------------------------------

\noindent
The \textit{in-} resp. \textit{out-degree} of a vertex \(v\in V\) in a finite digraph
can be expressed in terms of the \(v\)th column- resp. row-sum of the adjacency matrix:
\begin{equation}
\label{eqn:GraphInOut}
\mathrm{in}(v)=\sum_{w\in V}\,a_{w,v} \quad \text{ and } \quad \mathrm{out}(v)=\sum_{w\in V}\,a_{v,w}
\end{equation}
In particular, if \(\mathcal{G}=\mathcal{T}(R)\) is a finite directed \textit{in-tree} with \textit{root} \(R\), then
each row of the adjacency matrix \(A\) corresponding to a vertex \(v\ne R\) contains a unique \(1\) and
\begin{equation}
\label{eqn:TreeInOut}
\mathrm{in}(v)=\sum_{w\in V}\,a_{w,v}
\begin{cases}
   =0 & \text{ if } v \text{ is terminal}, \\
\ge 1 & \text{ if } v \text{ is capable},
\end{cases}
\quad \text{ and } \quad
\mathrm{out}(v)=\sum_{w\in V}\,a_{v,w}=
\begin{cases}
0 & \text{ if } v=R, \\
1 & \text{ else}.
\end{cases}
\end{equation}

%\newpage

%--------------------------------------------------------------------------------

\begin{proposition}
\label{prp:TreeIso}
Let \(\mathcal{T}(R)=(V,E)\) and \(\tilde{\mathcal{T}}(\tilde{R})=(\tilde{V},\tilde{E})\) be
two rooted directed in-trees, and
denote by \(\pi\) and \(\tilde{\pi}\) their parent operators.
Then a bijection \(\psi:\,V\to\tilde{V}\) with \(\psi(R)=\tilde{R}\)
is an isomorphism of rooted directed in-trees if and only if
\(\psi(\pi(v))=\tilde{\pi}(\psi(v))\) for all \(v\in V\setminus\lbrace R\rbrace\), that is, \(\psi\circ\pi=\tilde{\pi}\circ\psi\)
(briefly: \(\psi\) commutes with the parent operator), as shown in Figure
\ref{fig:IsomTrees}.
\end{proposition}

%\newpage

%--------------------------------------------------------------------------------

\begin{proof}
Recall that each row of the adjacency matrix \(A\) of the tree \(\mathcal{T}(R)\)
corresponding to a vertex \(v\in V\), \(v\ne R\), contains a unique \(1\).
This fact can be used to define the parent operator \(\pi\) of \(\mathcal{T}(R)\) by
\(\pi(v)=w\) \(\Longleftrightarrow\) \(a_{v,w}=1\).
Consequently, if \(\psi\) has the claimed property to commute with the parent operator, then
\((v,w)\in E\) \(\Longleftrightarrow\) \(a_{v,w}=1\) \(\Longleftrightarrow\) \(\pi(v)=w\)
\(\Longleftrightarrow\) \(\psi(\pi(v))=\psi(w)\) \(\Longleftrightarrow\) \(\tilde{\pi}(\psi(v))=\psi(w)\)
\(\Longleftrightarrow\) \(\tilde{a}_{\psi(v),\psi(w)}=1\) \(\Longleftrightarrow\) \((\psi(v),\psi(w))\in\tilde{E}\).
For infinite trees, the steps concerning adjacency matrices must be omitted.
The proof of the converse statement is similar.
\end{proof}

%\newpage

%--------------------------------------------------------------------------------

\begin{figure}[ht]
\caption{Isomorphism \(\psi\) of in-trees \((\mathcal{T},\pi)\) and \((\tilde{\mathcal{T}},\tilde{\pi})\)}
\label{fig:IsomTrees}

% Isomorphism of in-Trees

\setlength{\unitlength}{0.8cm}
\begin{picture}(7,5)(-4,-6.5)

\put(-3.2,-2){\makebox(0,0)[rc]{\(V\ni\pi(v)=w\)}}
\put(-3.0,-2){\circle*{0.2}}
\put(-2.6,-4){\makebox(0,0)[rc]{\(\pi\)}}
%\put(-2.5,-4){\makebox(0,0)[cc]{\(\uparrow\)}}
\put(-2.0,-6){\vector(-1,4){0.95}}
\put(-2.0,-6){\circle*{0.2}}
\put(-2.2,-6){\makebox(0,0)[rc]{\(V\ni v\)}}

\put(-0.5,-1.9){\makebox(0,0)[cb]{\(\psi\)}}
%\put(-0.5,-2){\makebox(0,0)[cc]{\(\longrightarrow\)}}
\put(-2.5,-2){\vector(1,0){4}}

\put(0,-4){\makebox(0,0)[cc]{\(\backslash\backslash\backslash\)}}

\put(-1.5,-6){\vector(1,0){4}}
%\put(0.5,-6){\makebox(0,0)[cc]{\(\longrightarrow\)}}
\put(0.5,-6.1){\makebox(0,0)[ct]{\(\psi\)}}

\put(2.2,-2){\makebox(0,0)[lc]{\(\tilde{w}=\tilde{\pi}(\tilde{v})=\psi(w)\in\tilde{V}\)}}
\put(2.0,-2){\circle*{0.2}}
\put(2.6,-4){\makebox(0,0)[lc]{\(\tilde{\pi}\)}}
%\put(2.5,-4){\makebox(0,0)[cc]{\(\uparrow\)}}
\put(3.0,-6){\vector(-1,4){0.95}}
\put(3.0,-6){\circle*{0.2}}
\put(3.2,-6){\makebox(0,0)[lc]{\(\tilde{v}=\psi(v)\in\tilde{V}\)}}

\end{picture}

\end{figure}

%\newpage

%--------------------------------------------------------------------------------

\section{Algebraically structured digraphs}
\label{s:StructuredDigraphs}

\subsection{General invariants and their transformation laws}
\label{ss:InvariantsTransforms}
\noindent
Since the vertices of all trees and branches in this paper
are realized by isomorphism classes of finite \(p\)-groups,
the abstract intrinsic graph theoretic structure of the trees and branches
can be extended by additional concrete structures
defined with the aid of algebraic invariants of \(p\)-groups.

Not all algebraic structures are strict invariants under graph isomorphisms.
Some of them change in a well defined way, described by a mapping \(\phi\),
the \textit{transformation law},
when a graph isomorphism is applied.
This behaviour is made precise in the following definitions.

%--------------------------------------------------------------------------------

\begin{definition}
\label{dfn:GraphStruc}
Let \(\mathcal{G}=(V,E)\) be a graph.
Suppose that \(X\ne\emptyset\) is a set,
and each vertex \(v\in V\) is associated with some kind of information \(\mathcal{S}(v)\in X\).
Then \((\mathcal{G},\mathcal{S})\) is called a \textit{structured graph} with respect to the mapping \(\mathcal{S}:\,V\to X\), \(v\mapsto\mathcal{S}(v)\).

If \((\mathcal{G},\mathcal{S})\) and \((\tilde{\mathcal{G}},\tilde{\mathcal{S}})\) are two structured digraphs with respect to mappings
\(\mathcal{S}:\,V\to X\), \(v\mapsto\mathcal{S}(v)\), and \(\tilde{\mathcal{S}}:\,\tilde{V}\to\tilde{X}\), \(v\mapsto\tilde{\mathcal{S}}(v)\),
and \(\phi:\,X\to\tilde{X}\) is a mapping,
then an isomorphism of digraphs \(\psi:\,V\to\tilde{V}\) is called
a \(\phi\)-\textit{isomorphism of structured digraphs} \((\mathcal{G},\mathcal{S})\) and \((\tilde{\mathcal{G}},\tilde{\mathcal{S}})\),
if \(\tilde{\mathcal{S}}(\psi(v))=\phi(\mathcal{S}(v))\) for all \(v\in V\), that is, \(\tilde{\mathcal{S}}\circ\psi=\phi\circ\mathcal{S}\),
as visualized in Figure
\ref{fig:PhiIsomStrucTrees}.

In particular, if the sets \(\tilde{X}=X\) coincide and \(\phi=1_X\) is the identity mapping of the set \(X\),
then \(\psi\) is called a \textit{strict isomorphism of structured digraphs},
and it satisfies the relation \(\tilde{\mathcal{S}}\circ\psi=\mathcal{S}\).
\end{definition}

%\newpage

%--------------------------------------------------------------------------------

\begin{figure}[ht]
\caption{\(\phi\)-Isomorphism \(\psi\) of structured digraphs \((\mathcal{G},\mathcal{S})\) and \((\tilde{\mathcal{G}},\tilde{\mathcal{S}})\)}
\label{fig:PhiIsomStrucTrees}

% phi-Isomorphism of Structured in-Trees

\setlength{\unitlength}{0.8cm}
\begin{picture}(7,5)(-4,-6.5)

\put(-2.5,-2){\makebox(0,0)[cc]{\(V\)}}
\put(-2.5,-2.5){\vector(0,-1){3}}
%\put(-2.5,-4){\makebox(0,0)[cc]{\(\downarrow\)}}
\put(-2.6,-4){\makebox(0,0)[rc]{\(\mathcal{S}\)}}
\put(-2.5,-6){\makebox(0,0)[cc]{\(X\)}}

\put(0,-1.9){\makebox(0,0)[cb]{\(\psi\)}}
%\put(0,-2){\makebox(0,0)[cc]{\(\longrightarrow\)}}
\put(-2.0,-2){\vector(1,0){4}}

\put(0,-4){\makebox(0,0)[cc]{\(\slash\slash\slash\)}}

\put(-2.0,-6){\vector(1,0){4}}
%\put(0,-6){\makebox(0,0)[cc]{\(\longrightarrow\)}}
\put(0,-6.1){\makebox(0,0)[ct]{\(\phi\)}}

\put(2.5,-2){\makebox(0,0)[cc]{\(\tilde{V}\)}}
\put(2.5,-2.5){\vector(0,-1){3}}
%\put(2.5,-4){\makebox(0,0)[cc]{\(\downarrow\)}}
\put(2.6,-4){\makebox(0,0)[lc]{\(\tilde{\mathcal{S}}\)}}
\put(2.5,-6){\makebox(0,0)[cc]{\(\tilde{X}\)}}

\end{picture}

\end{figure}

%\newpage

%--------------------------------------------------------------------------------

\begin{definition}
\label{dfn:IsoInvariants}
Let \(\mathcal{G}=(V,E)\) and \(\tilde{\mathcal{G}}=(\tilde{V},\tilde{E})\) be two structured digraphs
with structure mappings \(\mathcal{S}:\,V\to X\) and \(\tilde{\mathcal{S}}:\,\tilde{V}\to\tilde{X}\),
and let \(\psi:\,V\to\tilde{V}\) be a \(\phi\)-isomorphism of the two structured digraphs
with respect to a mapping \(\phi:\, X\to\tilde{X}\),
that is, \(\tilde{\mathcal{S}}\circ\psi=\phi\circ\mathcal{S}\).
Then \(\mathcal{S}\) is called a \(\phi\)-\textit{invariant} under \(\psi\)
(or \textit{invariant under the isomorphism} \(\psi\) \textit{and transformation law} \(\phi\)).
In particular,
if \(\tilde{X}=X\), \(\phi=1_X\) and \(\tilde{\mathcal{S}}\circ\psi=\mathcal{S}\),
then \(\mathcal{S}\) is called a \textit{strict invariant} under \(\psi\).
\end{definition}

%\newpage

%--------------------------------------------------------------------------------

\subsection{Algebraic invariants considered in this paper}
\label{ss:AlgebraicInvariants}
\noindent
With respect to applications in other mathematical theories,
in particular, algebraic number theory and class field theory,
certain properties of the automorphism group \(\mathrm{Aut}(G)\)
of a finite \(3\)-group \(G\) are crucial.
The general frame of these aspects is the following.

\begin{definition}
\label{dfn:Automorphisms}
Let \(p\) be an odd prime number
and let \(G\) be a pro-\(p\) group.
We call \(G\) a \textit{group with GI-action} or a \(\sigma\)-\textit{group},
if there exists a \textit{generator inverting} automorphism \(\sigma\in\mathrm{Aut}(G)\)
such that \(\sigma(x)\equiv x^{-1}\mod{G^\prime}\), for all \(x\in G\),
or equivalently \(\sigma(x)=x^{-1}\), for all \(x\in\mathrm{H}_1(G,\mathbb{F}_p)\).
If additionally \(\sigma(x)=x^{-1}\), for all \(x\in\mathrm{H}_2(G,\mathbb{F}_p)\),
then \(G\) is called a \textit{group with RI-action}
or group with \textit{relator inverting} automorphism.
If \(\mathrm{Aut}(G)\) contains a bicyclic subgroup \(C_2\times C_2\),
then we call \(G\) a \textit{group with} \(V_4\)-\textit{action}.
It is convenient to define the \textit{action flag} of \(G\) by
\begin{equation}
\label{eqn:ActionFlag}
\sigma(G):=
\begin{cases}
2 & \text{ if } G \text{ possesses GI-action and } V_4\text{-action, starred } (2^\ast) \text{ for RI-action}, \\
1 & \text{ if } G \text{ possesses GI-action but no } V_4\text{-action, starred } (1^\ast) \text{ for RI-action}, \\
0 & \text{ if } G \text{ has no GI-action and no } V_4\text{-action}.
\end{cases}
\end{equation}
\end{definition}

%--------------------------------------------------------------------------------

\begin{remark}
\label{rmk:Automorphisms}
Suppose that \(G\) is a finite \(p\)-group with odd prime \(p\).
We point out that \(2\) divides the order \(\#\mathrm{Aut}(G)\),
if \(G\) is a group with GI-action, but the converse claim may be false.
If \(G\) is a group with \(V_4\)-action, then \(4\) divides \(\#\mathrm{Aut}(G)\),
but we emphasize that the converse statement, even in the case that \(8\) divides \(\#\mathrm{Aut}(G)\),
may be false, when \(\mathrm{Aut}(G)\) contains a cyclic group \(C_4\)
or a (generalized) quaternion group \(Q(2^e)\) of order \(2^e\) with \(e\ge 3\).
\end{remark}

%--------------------------------------------------------------------------------

For a brief description of abelian quotient invariants in logarithmic form, 
we need the concept of nearly homocyclic \(p\)-groups.
With an arbitrary prime \(p\ge 2\)
these groups appear in
\cite[p. 68, Thm. 3.4]{Bl2}
and they are treated systematically in
\cite[\S\ 2.4]{Ne}.
For our purpose, it suffices to consider the special case \(p=3\).

\begin{definition}
\label{dfn:NrlHomAbl}
By the \textit{nearly homocyclic abelian \(3\)-group}
\(A(3,n)\) of order \(3^n\),
for an integer \(n\ge 2\),
we understand the abelian group with logarithmic type invariants
\((q+r,q)\),
where \(n=2q+r\) with integers \(q\ge 1\) and \(0\le r<2\),
by Euclidean division with remainder.
Additionally, including two degenerate cases, we define that
\(A(3,1)\) denotes the cyclic group \(C_3\) of order \(3\),
and \(A(3,0)\) denotes the trivial group \(1\).
\end{definition}

%--------------------------------------------------------------------------------

The following invariants \(\mathcal{S}:\,V\to X\) of finite \(3\)-groups \(v\in V\) with abelianization \(v/v^\prime\simeq C_3\times C_3\)
will be of particular interest in the whole paper:

\begin{itemize}
\item
The \textit{logarithmic order} \(\mathrm{lo}:\,V\to\mathbb{N}_0\), \(v\mapsto n:=\log_p(\mathrm{ord}(v))\),
\item
the \textit{nilpotency class} \(\mathrm{cl}:\,V\to\mathbb{N}_0\), connected with the index of nilpotency \(m\)
by the relation \(c:=\mathrm{cl}(v)=m-1\), where the lower central series stops with \(\gamma_{m-1}{v}>\gamma_m{v}=1\),
\item
the \textit{coclass} \(\mathrm{cc}:\,V\to\mathbb{N}_0\), defined by \(r:=\mathrm{cc}(v)=n-c=\mathrm{lo}(v)-\mathrm{cl}(v)\),
\item
the \textit{order of the automorphism group} \(\#\mathrm{Aut}:\,V\to\mathbb{N}\), \(v\mapsto\#\mathrm{Aut}(v)\),
\item
the \textit{action flag} \(\sigma:\,V\to\mathbb{N}_0\), defined by Formula
\eqref{eqn:ActionFlag},
\item
the \textit{transfer kernel type} (TKT) \(\varkappa:\,V\to X\), \(v\mapsto (\ker(T_i))_{1\le i\le 4}\),
where \(T_i:v/v^\prime\to u_i/u_i^\prime\) denote the transfer homomorphisms from \(v\) to the maximal subgroups \(u_i<v\), for \(1\le i\le 4\),
\item
the \textit{transfer target type} (TTT) \(\tau:\,V\to X\), \(v\mapsto (u_i/u_i^\prime)_{1\le i\le 4}\),
viewed as abelian quotient invariants,
where \(u_1,\ldots u_4\) denote the maximal subgroups of \(v\)
\cite[Dfn. 5.3, p. 83]{Ma9},
\item
the \textit{abelian quotient invariants} of the first TTT component \(\tau(1):\,V\to X\), \(v\mapsto u_1/u_1^\prime\simeq A(3,c-k)\),
where \(k\) denotes the defect of commutativity of \(v\)
\cite[\S\ 2, p. 469]{Ma2},
\item
the \textit{abelian quotient invariants} of the commutator subgroup \(\tau_2:\,V\to X\), \(v\mapsto v^\prime/v^{\prime\prime}\simeq A(3,c-1)\times A(3,r-1)\)
(or \(A(3,c-2)\times A(3,r)\) in irregular cases)
\cite[Satz 4.2.4, p. 131]{Ne},
\item
the \textit{relation rank} \(\mu:\,V\to\mathbb{N}_0\), \(v\mapsto\dim_{\mathbb{F}_p}\mathrm{H}_2(v,\mathbb{F}_p)\),
which coincides with the rank of the \(p\)-multiplicator of \(v\)
\cite[Thm. 2.4]{Ob},
\item
the \textit{nuclear rank} \(\nu:\,V\to\mathbb{N}_0\), i.e.
the rank of the nucleus of \(v\in V\)
\cite[Thm. 2.4]{Ob}.
For a coclass tree, the nuclear rank is given by
\(v\mapsto\begin{cases} 1 & \text{ if \(v\) is capable and coclass settled}, \\ 0 & \text{ if \(v\) is terminal}.
\end{cases}\)
\end{itemize}

%--------------------------------------------------------------------------------

\begin{remark}
\label{rmk:TKT}
\noindent
Abelian quotient invariants are given in logarithmic notation.
The transfer kernel type \(\varkappa(v)\) is simplified
by a family of non-negative integers, in the following way:
for \(1\le i\le 4\),
\begin{equation}
\label{eqn:TKT}
\varkappa(v)_i:=
\begin{cases}
j & \text{ if } \ker(T_i)=u_j/v^\prime \text{ for some } j\in\lbrace 1,\ldots,4\rbrace \text{ (partial kernel)}, \\
0 & \text{ if } \ker(T_i)=v/v^\prime \text{ (total kernel)}.
\end{cases}
\end{equation}
\end{remark}

%\newpage

%--------------------------------------------------------------------------------

\section{The graph theoretic structure of a tree}
\label{s:GraphStruc}

\subsection{Cardinality of branches and layers, depth and width of a tree}
\label{ss:BranchesDepthWidth}
\noindent
The graph theoretic structure of a coclass tree \(\mathcal{T}\)
with unique infinite mainline and finite branches,
consisting of isomorphism classes of finite \(p\)-groups,
is described by the following concepts.

\begin{definition}
\label{dfn:TreeWidth}
Let \(\mathcal{T}\) be a coclass tree.
Suppose that the tree root \(R\) is of logarithmic order \(\mathrm{lo}(R)=n_\ast\),
and denote by \(m_e\) the unique mainline vertex with \(\mathrm{lo}(m_e)=e\ge n_\ast\).
In particular, \(R=m_{n_\ast}\).

For \(e\ge n_\ast\), the difference set
\(\mathcal{B}(e):=\mathcal{T}(m_e)\setminus\mathcal{T}(m_{e+1})\)
is called the \(e\)th \textit{branch} of \(\mathcal{T}\).

\noindent
Let \(\mathcal{B}\subset\mathcal{T}\) be one of the branches of \(\mathcal{T}\).
For any integer \(n\ge n_\ast\), we let
\begin{equation}
\label{eqn:TreeLayer}
\mathrm{Lyr}_n{\mathcal{T}}:=\lbrace v\in\mathcal{T}\mid\mathrm{lo}(v)=n\rbrace,
\quad \text{ respectively } \quad
\mathrm{Lyr}_n{\mathcal{B}}:=\lbrace v\in\mathcal{B}\mid\mathrm{lo}(v)=n\rbrace,
\end{equation}
denote the \(n\)th \textit{layer} of \(\mathcal{T}\), respectively \(\mathcal{B}\).

The \textit{width} of the tree is the maximal cardinality of its layers,
\begin{equation}
\label{eqn:TreeWidth}
\mathrm{wd}(\mathcal{T}):=\sup\lbrace\#\mathrm{Lyr}_n{\mathcal{T}}\mid n\ge n_\ast\rbrace.
\end{equation}
\end{definition}

%--------------------------------------------------------------------------------

\noindent
Each vertex \(v\) of the branch \(\mathcal{B}(e)\)
is connected with the mainline by a unique finite path of directed edges
from \(v\) to the branch root \(m_e\),
formed by the iterated parents \(\pi^i{v}\) of \(v\),
\begin{equation}
\label{eqn:BranchPath}
v=\pi^0{v}\rightarrow\pi^1{v}\rightarrow\pi^2{v}\rightarrow\ldots\rightarrow\pi^d{v}=m_e.
\end{equation}
The length \(d\ge 0\) of this path is called the \textit{depth} \(\mathrm{dp}(v)\) of \(v\).

The \textit{depth} of a branch \(\mathcal{B}\subset\mathcal{T}\)
is the maximal depth of its vertices,
\begin{equation}
\label{eqn:BranchDepth}
\mathrm{dp}(\mathcal{B}):=\max\lbrace\mathrm{dp}(v)\mid v\in\mathcal{B}\rbrace.
\end{equation}

\begin{definition}
\label{dfn:TreeDepth}
Let \(\mathcal{T}\) be a coclass tree.
The \textit{depth} of the tree is the maximal depth of its branches,
\begin{equation}
\label{eqn:TreeDepth}
\mathrm{dp}(\mathcal{T}):=\sup\lbrace\mathrm{dp}(\mathcal{B}(n))\mid n\ge n_\ast\rbrace.
\end{equation}
\end{definition}

Throughout this paper, we assume that both,
the depth \(\mathrm{dp}(\mathcal{T})\)
and the width \(\mathrm{wd}(\mathcal{T})\)
of the tree, are \textit{bounded}.
This assumption is satisfied by all trees of finite \(3\)-groups
under investigation in the sequel.
However, we point out that
that tree \(\mathcal{T}^1(C_5\times C_5)\) of finite \(5\)-groups with coclass one
has unbounded depth,
and the tree \(\mathcal{T}^1(C_7\times C_7)\) of finite \(7\)-groups with coclass one
even has unbounded width and depth.
(Compare \cite[\S\ 5.1, pp. 113--114]{Ek}.)

%--------------------------------------------------------------------------------

\begin{lemma}
\label{lem:BranchCardinality}
Let \(e\ge n_\ast\), \(\mathcal{B}:=\mathcal{B}(e)\) and \(d:=\mathrm{dp}(\mathcal{B})\).
Then
\begin{equation}
\label{eqn:BranchCardinality}
\mathcal{B}=\dot{\bigcup}_{n=e}^{e+d}\,\mathrm{Lyr}_n{\mathcal{B}} \quad \text{ and } \quad \#\mathcal{B}=\sum_{n=e}^{e+d}\,\#\mathrm{Lyr}_n{\mathcal{B}}.
\end{equation}
\end{lemma}

\begin{proof}
Since \(m_e\) is the root of the branch \(\mathcal{B}(e)\), we have
\(\mathrm{Lyr}_e{\mathcal{B}}=\lbrace m_e\rbrace\),
but \((\forall\,n<e)\) \(\mathrm{Lyr}_n{\mathcal{B}}=\emptyset\).
Since \(d=\mathrm{dp}(\mathcal{B})=\max\lbrace\mathrm{dp}(v)\mid v\in\mathcal{B}\rbrace\),
there exists a vertex \(t\in\mathcal{B}\), necessarily terminal if \(d>0\), such that \(\mathrm{dp}(t)=d\).
The iterated parents \(\pi^i{t}\) of \(t\) form the unique finite path from \(t\) to the branch root \(m_e\)
(see Figure
\ref{fig:SchematicTree}),
\[t=\pi^0{t}\rightarrow\pi^1{t}\rightarrow\pi^2{t}\rightarrow\ldots\rightarrow\pi^d{t}=m_e,\]
and we have
\((\forall\,e\le n\le e+d)\) \(\mathrm{Lyr}_n{\mathcal{B}}\supseteq\lbrace\pi^{e+d-n}{t}\rbrace\ne\emptyset\)
but \((\forall\,n>e+d)\) \(\mathrm{Lyr}_n{\mathcal{B}}=\emptyset\).
\end{proof}

%--------------------------------------------------------------------------------

\begin{lemma}
\label{lem:LayerCardinality}
Let \(n\ge n_\ast\) and \(d:=\mathrm{dp}(\mathcal{T})\). Then
\begin{equation}
\label{eqn:LayerCardinality}
\#\mathrm{Lyr}_n{\mathcal{T}}=\sum_{i=M}^{n}\,\#\mathrm{Lyr}_n{\mathcal{B}(i)}, \quad \text{ where } M:=\max(n_\ast,n-d).
\end{equation}
\end{lemma}

\begin{proof}
Since \(\mathrm{dp}(\mathcal{T})=\sup\lbrace\mathrm{dp}(\mathcal{B}(n))\mid n\ge n_\ast\rbrace\),
we have \(\mathrm{dp}(\mathcal{B}(i))\le\mathrm{dp}(\mathcal{T})=d\) for each \(i\ge n_\ast\).
A branch \(\mathcal{B}(i)\) with \(i>n\) cannot contribute to \(\mathrm{Lyr}_n{\mathcal{T}}\).
On the other hand, if \(n-d>n_\ast\), then
a branch \(\mathcal{B}(i)\) with \(i<M=n-d\) cannot contribute to \(\mathrm{Lyr}_n{\mathcal{T}}\) either,
since \(\mathcal{B}(i)=\dot{\bigcup}_{j=i}^{i+\mathrm{dp}(\mathcal{B}(i))}\,\mathrm{Lyr}_j{\mathcal{B}(i)}\),
according to Lemma
\ref{lem:BranchCardinality},
and we obtain \(i+\mathrm{dp}(\mathcal{B}(i))\le i+d<n\)
(see Figure
\ref{fig:SchematicTree}).
Consequently,
\begin{equation}
\label{eqn:LayerDecomposition}
\mathrm{Lyr}_n{\mathcal{T}}=\dot{\bigcup}_{i=M}^{n}\,\mathrm{Lyr}_n{\mathcal{B}(i)}.
\qedhere
\end{equation}
\end{proof}

%\newpage

%--------------------------------------------------------------------------------

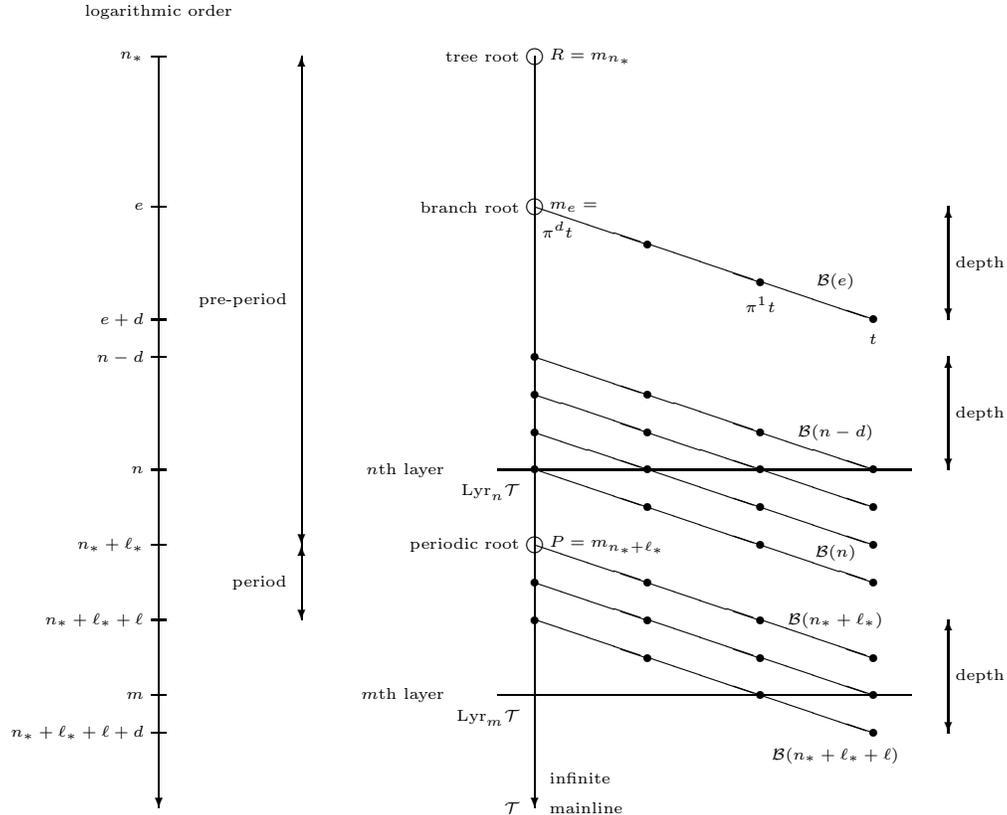
\begin{figure}[ht]
\caption{Schematic coclass tree \(\mathcal{T}\) with ultimately periodic branches and layers}
\label{fig:SchematicTree}

% Schematic Coclass Tree with
% Branches, Layers, Pre-Period and Period

{\tiny

\setlength{\unitlength}{1cm}
\begin{picture}(13,11)(-7,-10)

% scale of logarithmic orders
\put(-5,0.5){\makebox(0,0)[cb]{logarithmic order}}

\put(-5,0){\line(0,-1){8}}
\put(-5.1,0){\line(1,0){0.2}}
\put(-5.1,-2){\line(1,0){0.2}}
\put(-5.1,-3.5){\line(1,0){0.2}}
\put(-5.1,-4){\line(1,0){0.2}}
\put(-5.1,-5.5){\line(1,0){0.2}}
\put(-5.1,-6.5){\line(1,0){0.2}}
\put(-5.1,-7.5){\line(1,0){0.2}}
\put(-5.1,-8.5){\line(1,0){0.2}}
\put(-5.1,-9){\line(1,0){0.2}}

\put(-5.2,0){\makebox(0,0)[rc]{\(n_\ast\)}}
\put(-5.2,-2){\makebox(0,0)[rc]{\(e\)}}
\put(-5.2,-3.5){\makebox(0,0)[rc]{\(e+d\)}}
\put(-5.2,-4){\makebox(0,0)[rc]{\(n-d\)}}
\put(-5.2,-5.5){\makebox(0,0)[rc]{\(n\)}}
\put(-5.2,-6.5){\makebox(0,0)[rc]{\(n_\ast+\ell_\ast\)}}
\put(-5.2,-7.5){\makebox(0,0)[rc]{\(n_\ast+\ell_\ast+\ell\)}}
\put(-5.2,-8.5){\makebox(0,0)[rc]{\(m\)}}
\put(-5.2,-9){\makebox(0,0)[rc]{\(n_\ast+\ell_\ast+\ell+d\)}}

\put(-5,-8){\vector(0,-1){2}}

% depth of branches
\put(5.5,-2.75){\vector(0,1){0.75}}
\put(5.6,-2.75){\makebox(0,0)[lc]{depth}}
\put(5.5,-2.75){\vector(0,-1){0.75}}

\put(5.5,-4.75){\vector(0,1){0.75}}
\put(5.6,-4.75){\makebox(0,0)[lc]{depth}}
\put(5.5,-4.75){\vector(0,-1){0.75}}

\put(5.5,-8.25){\vector(0,1){0.75}}
\put(5.6,-8.25){\makebox(0,0)[lc]{depth}}
\put(5.5,-8.25){\vector(0,-1){0.75}}

% periodicity of branches
\put(-3.1,-3.25){\vector(0,1){3.25}}
\put(-3.3,-3.25){\makebox(0,0)[rc]{pre-period}}
\put(-3.1,-3.25){\vector(0,-1){3.25}}

\put(-3.1,-7){\vector(0,1){0.5}}
\put(-3.3,-7){\makebox(0,0)[rc]{period}}
\put(-3.1,-7){\vector(0,-1){0.5}}

% infinite mainline
\put(0,0){\line(0,-1){8}}
\put(0,-8){\vector(0,-1){2}}
\put(0.2,-9.6){\makebox(0,0)[lc]{infinite}}
\put(0.2,-10){\makebox(0,0)[lc]{mainline}}
\put(-0.2,-10){\makebox(0,0)[rc]{\(\mathcal{T}\)}}

% tree root
\put(-0.2,0){\makebox(0,0)[rc]{tree root}}
\put(0,0){\circle{0.2}}
\put(0.2,0){\makebox(0,0)[lc]{\(R=m_{n_\ast}\)}}

% branch root
\put(-0.2,-2){\makebox(0,0)[rc]{branch root}}
\put(0,-2){\circle{0.2}}
\put(0.2,-2){\makebox(0,0)[lc]{\(m_e=\)}}

% periodic root
\put(-0.2,-6.5){\makebox(0,0)[rc]{periodic root}}
\put(0,-6.5){\circle{0.2}}
\put(0.2,-6.5){\makebox(0,0)[lc]{\(P=m_{n_\ast+\ell_\ast}\)}}

% branch
\multiput(0,-2)(1.5,-0.5){3}{\line(3,-1){1.5}}
\multiput(1.5,-2.5)(1.5,-0.5){3}{\circle*{0.1}}
\put(4,-3){\makebox(0,0)[cc]{\(\mathcal{B}(e)\)}}
\put(4.5,-3.7){\makebox(0,0)[ct]{\(t\)}}
\put(3,-3.2){\makebox(0,0)[ct]{\(\pi^1{t}\)}}
\put(0.1,-2.2){\makebox(0,0)[lt]{\(\pi^d{t}\)}}

% layer branches
\multiput(0,-4)(0,-0.5){4}{\circle*{0.1}}

\multiput(0,-4)(1.5,-0.5){3}{\line(3,-1){1.5}}
\multiput(1.5,-4.5)(1.5,-0.5){3}{\circle*{0.1}}
\put(4,-5){\makebox(0,0)[cc]{\(\mathcal{B}(n-d)\)}}
\multiput(0,-4.5)(1.5,-0.5){3}{\line(3,-1){1.5}}
\multiput(1.5,-5)(1.5,-0.5){3}{\circle*{0.1}}
\multiput(0,-5)(1.5,-0.5){3}{\line(3,-1){1.5}}
\multiput(1.5,-5.5)(1.5,-0.5){3}{\circle*{0.1}}
\multiput(0,-5.5)(1.5,-0.5){3}{\line(3,-1){1.5}}
\multiput(1.5,-6)(1.5,-0.5){3}{\circle*{0.1}}
\put(4,-6.6){\makebox(0,0)[cc]{\(\mathcal{B}(n)\)}}

\put(-1.2,-5.5){\makebox(0,0)[rc]{\(n\)th layer}}
\put(-0.2,-5.8){\makebox(0,0)[rc]{\(\mathrm{Lyr}_n{\mathcal{T}}\)}}

% period branches
\multiput(0,-7)(0,-0.5){2}{\circle*{0.1}}

\multiput(0,-6.5)(1.5,-0.5){3}{\line(3,-1){1.5}}
\multiput(1.5,-7)(1.5,-0.5){3}{\circle*{0.1}}
\put(4,-7.5){\makebox(0,0)[cc]{\(\mathcal{B}(n_\ast+\ell_\ast)\)}}
\multiput(0,-7)(1.5,-0.5){3}{\line(3,-1){1.5}}
\multiput(1.5,-7.5)(1.5,-0.5){3}{\circle*{0.1}}
\multiput(0,-7.5)(1.5,-0.5){3}{\line(3,-1){1.5}}
\multiput(1.5,-8)(1.5,-0.5){3}{\circle*{0.1}}
\put(4,-9.3){\makebox(0,0)[cc]{\(\mathcal{B}(n_\ast+\ell_\ast+\ell)\)}}

\put(-1.2,-8.5){\makebox(0,0)[rc]{\(m\)th layer}}
\put(-0.2,-8.8){\makebox(0,0)[rc]{\(\mathrm{Lyr}_m{\mathcal{T}}\)}}

% layers
\multiput(-0.5,-5.5)(0,-3){2}{\line(1,0){5.5}}

\end{picture}

}

\end{figure}

%--------------------------------------------------------------------------------

\noindent
Since the implementation of the \(p\)-group generation algorithm
\cite{Nm2,Ob,HEO}
in the computational algebra system MAGMA
\cite{BCP,BCFS,MAGMA}
is able to give the number of all, respectively only the capable,
immediate descendants (children) of an assigned finite \(p\)-group,
we express the cardinalities of the branches of a coclass tree,
which were given in a preliminary form in Lemma
\ref{lem:BranchCardinality},
in terms of these numbers \(N_1\), respectively \(C_1\).

\begin{theorem}
\label{thm:DescendantNumbers}
\noindent
Let \(\mathcal{T}=(V,E)\) be a coclass tree
with tree root \(R\) of logarithmic order \(\mathrm{lo}(R)=n_\ast\),
pre-period of length \(\ell_\ast\ge 0\), and
period of primitive length \(\ell\ge 1\).
For each vertex \(v\in V\),
denote by \(N_1(v)\) the number of all children (of step size \(s=1\))
and by \(C_1(v)\) the number of capable children of \(v\).
When \(m_e\) is the vertex with \(\mathrm{lo}(m_e)=e\ge n_\ast\)
on the mainline of \(\mathcal{T}\),
let \(v_i=v_i(m_e)\) with \(1\le i\le C_1(m_e)\) be the capable children of \(m_e\),
in particular, let \(v_1=m_{e+1}\) be the next mainline vertex.
Finally, let \(v_{i,j}=v_{i,j}(m_e)\) with \(1\le j\le C_1(v_i)\) denote
the capable children of \(v_i\), for each \(2\le i\le C_1(m_e)\). 
\begin{enumerate}
\item
If the tree is of depth \(\mathrm{dp}(\mathcal{T})=1\), then
\begin{equation}
\label{eqn:BranchDepth1}
\#\mathcal{B}(e)=N_1(m_e).
\end{equation}
\item
If the tree is of depth \(\mathrm{dp}(\mathcal{T})=2\), then
\begin{equation}
\label{eqn:BranchDepth2}
\#\mathcal{B}(e)=N_1(m_e)+\sum_{i=2}^{C_1(m_e)}\,N_1(v_i).
\end{equation}
\item
If the tree is of depth \(\mathrm{dp}(\mathcal{T})=3\), then
\begin{equation}
\label{eqn:BranchDepth3}
\#\mathcal{B}(e)=N_1(m_e)+\sum_{i=2}^{C_1(m_e)}\,\left(N_1(v_i)+\sum_{j=1}^{C_1(v_i)}\,N_1(v_{i,j})\right).
\end{equation}
\end{enumerate}
\end{theorem}

\begin{proof}
Put \(\mathcal{B}:=\mathcal{B}(e)\). Generally, we have
\(\mathcal{B}=\mathrm{Lyr}_e{\mathcal{B}}\,\dot{\cup}\,\ldots\,\dot{\cup}\,\mathrm{Lyr}_{e+d}{\mathcal{B}}\)
with \(d:=\mathrm{dp}(\mathcal{B})\),
according to Lemma
\ref{lem:BranchCardinality}.

If \(\mathrm{dp}(\mathcal{T})=1\), then \(d\le 1\) and
\(\mathcal{B}=\mathrm{Lyr}_e{\mathcal{B}}\,\dot{\cup}\,\mathrm{Lyr}_{e+1}{\mathcal{B}}\).
We have \(\mathrm{Lyr}_e{\mathcal{B}}=\lbrace m_e\rbrace\)
and \(\#\mathrm{Lyr}_{e+1}{\mathcal{B}}=N_1(m_e)-1\),
since the next mainline vertex \(m_{e+1}\) is one of the \(N_1(m_e)\) children of \(m_e\)
but does not belong to \(\mathcal{B}\).
Thus, we obtain \(\#\mathcal{B}=\#\mathrm{Lyr}_e{\mathcal{B}}+\#\mathrm{Lyr}_{e+1}{\mathcal{B}}=1+N_1(m_e)-1=N_1(m_e)\).

If \(\mathrm{dp}(\mathcal{T})=2\), then \(d\le 2\) and
\(\mathcal{B}=\mathrm{Lyr}_e{\mathcal{B}}\,\dot{\cup}\,\mathrm{Lyr}_{e+1}{\mathcal{B}}\,\dot{\cup}\,\mathrm{Lyr}_{e+2}{\mathcal{B}}\),
where \(\#\mathrm{Lyr}_e{\mathcal{B}}=1\) and \(\#\mathrm{Lyr}_{e+1}{\mathcal{B}}=N_1(m_e)-1\) as before, and
\(\#\mathrm{Lyr}_{e+2}{\mathcal{B}}=\sum_{i=2}^{C_1(m_e)}\,N_1(v_i)\). Therefore,
\(\#\mathcal{B}=\#\mathrm{Lyr}_e{\mathcal{B}}+\ldots+\#\mathrm{Lyr}_{e+2}{\mathcal{B}}=N_1(m_e)+\sum_{i=2}^{C_1(m_e)}\,N_1(v_i)\).

If \(\mathrm{dp}(\mathcal{T})=3\), then \(d\le 3\) and
\(\mathcal{B}=\mathrm{Lyr}_e{\mathcal{B}}\,\dot{\cup}\,\ldots\,\dot{\cup}\,\mathrm{Lyr}_{e+3}{\mathcal{B}}\),
where \(\#\mathrm{Lyr}_e{\mathcal{B}}=1\), \(\#\mathrm{Lyr}_{e+1}{\mathcal{B}}=N_1(m_e)-1\),
\(\#\mathrm{Lyr}_{e+2}{\mathcal{B}}=\sum_{i=2}^{C_1(m_e)}\,N_1(v_i)\) as before, and
\(\#\mathrm{Lyr}_{e+3}{\mathcal{B}}=\sum_{i=2}^{C_1(m_e)}\,\sum_{j=1}^{C_1(v_i)}\,N_1(v_{i,j})\).
Thus,
\(\#\mathcal{B}=\#\mathrm{Lyr}_e{\mathcal{B}}+\ldots+\#\mathrm{Lyr}_{e+3}{\mathcal{B}}=N_1(m_e)+\sum_{i=2}^{C_1(m_e)}\,\left(N_1(v_i)+\sum_{j=1}^{C_1(v_i)}\,N_1(v_{i,j})\right)\).
\end{proof}

%--------------------------------------------------------------------------------

\begin{remark}
\label{rmk:DescendantNumbers}
In Theorem
\ref{thm:DescendantNumbers},
item (1) is included in item (2),
since \(\mathrm{dp}(\mathcal{T})=1\) implies \(C_1(m_e)=1\), and
item (2) is included in item (3),
since \(\mathrm{dp}(\mathcal{T})=2\) implies \(C_1(v_i)=0\),
for all \(2\le i\le C_1(m_e)\).
\end{remark}

%--------------------------------------------------------------------------------

\begin{corollary}
\label{cor:TreeWidth}
Under the same assumptions as in Theorem
\ref{thm:DescendantNumbers},
the width of the coclass tree \(\mathcal{T}\),
in dependence on the depth \(d:=\mathrm{dp}(\mathcal{T})\) and the periodicity \((\ell_\ast,\ell)\),
is generally given by
\begin{equation}
\label{eqn:TreeWidthd}
\mathrm{wd}(\mathcal{T})=\max\lbrace\#\mathrm{Lyr}_n{\mathcal{T}}\mid n_\ast<n<n_\ast+\ell_\ast+\ell+d\rbrace.
\end{equation}
For assigned small values of the depth \(d\le 3\),
the width can be expressed in terms of descendant numbers in the following manner:
\begin{enumerate}
\item
If the tree is of depth \(d=1\), then
\begin{equation}
\label{eqn:TreeWidth1}
\mathrm{wd}(\mathcal{T})=\max\lbrace N_1(m_{n-1})\mid n_\ast+1\le n\le n_\ast+\ell_\ast+\ell\rbrace.
\end{equation}
\item
If the tree is of depth \(d=2\), then
\(\mathrm{wd}(\mathcal{T})\) is the maximum among the number \(N_1(m_{n_\ast})\) and all expressions
\begin{equation}
\label{eqn:TreeWidth2}
N_1(m_{n-1})+\sum_{i=2}^{C_1(m_{n-2})}\,N_1(v_i(m_{n-2})),
\end{equation}
where \(n\) runs from \(n_\ast+2\) to \(n_\ast+\ell_\ast+\ell+1\).
\item
If the tree is of depth \(d=3\), then
\(\mathrm{wd}(\mathcal{T})\) is the maximum among the numbers \(N_1(m_{n_\ast})\),
\(N_1(m_{n_\ast+1})+\sum_{i=2}^{C_1(m_{n_\ast})}\,N_1(v_i(m_{n_\ast}))\), and all expressions
\begin{equation}
\label{eqn:TreeWidth3}
N_1(m_{n-1})+\sum_{i=2}^{C_1(m_{n-2})}\,N_1(v_i(m_{n-2}))+\sum_{i=2}^{C_1(m_{n-3})}\,\sum_{j=1}^{C_1(v_i(m_{n-3}))}\,N_1(v_{i,j}(m_{n-3})),
\end{equation}
where \(n\) runs from \(n_\ast+3\) to \(n_\ast+\ell_\ast+\ell+2\).
\end{enumerate}
\end{corollary}

\begin{proof}
According to
\cite{dS,EkLg,dSSg},
the periodicity of the branches of a coclass tree \(\mathcal{T}\)
with root \(R:=m_{n_\ast}\) and
bounded depth \(d:=\mathrm{dp}(\mathcal{T})\) and width \(\mathrm{wd}(\mathcal{T})\)
can be expressed by means of isomorphisms between branches,
starting from the \textit{periodic root} \(P:=m_{n_\ast+\ell_\ast}\):
\begin{equation}
\label{eqn:Periodicity}
(\forall\,k\ge n_\ast+\ell_\ast)\ \mathcal{B}(k+\ell)\simeq\mathcal{B}(k),
\end{equation}
where \(\ell_\ast\ge 0\) denotes the length of the pre-period and \(\ell\ge 1\) is the primitive period length.
With Lemma
\ref{lem:BranchCardinality},
an immediate consequence is the periodicity of branch layer cardinalities:
\[
(\forall\,k\ge n_\ast+\ell_\ast)\ (\forall\,k\le n\le k+d)\ \#\mathrm{Lyr}_{n+\ell}{\mathcal{B}(k+\ell)}=\#\mathrm{Lyr}_n{\mathcal{B}(k)}.
\]
According to Lemma
\ref{lem:LayerCardinality},
we have \(\#\mathrm{Lyr}_n{\mathcal{T}}=\sum_{k=\max(n_\ast,n-d)}^{n}\,\#\mathrm{Lyr}_n{\mathcal{B}(k)}\), and thus
\[
\#\mathrm{Lyr}_n{\mathcal{T}}=
\begin{cases}
\#\mathrm{Lyr}_n{\mathcal{B}(n_\ast)} & \text{ if } n=n_\ast, \\
\#\mathrm{Lyr}_n{\mathcal{B}(n_\ast)}+\#\mathrm{Lyr}_n{\mathcal{B}(n_\ast+1)} & \text{ if } n=n_\ast+1, \\
 & \ldots \\
\#\mathrm{Lyr}_n{\mathcal{B}(n_\ast)}+\ldots+\#\mathrm{Lyr}_n{\mathcal{B}(n_\ast+d)} & \text{ if } n=n_\ast+d, \\
\#\mathrm{Lyr}_n{\mathcal{B}(n_\ast+x-d)}+\ldots+\#\mathrm{Lyr}_n{\mathcal{B}(n_\ast+x)} & \text{ if } n=n_\ast+x,\ x\ge d.
\end{cases}
\]
For finding the maximal layer cardinality,
the root term \(\#\mathrm{Lyr}_{n_\ast}{\mathcal{T}}=\#\mathrm{Lyr}_{n_\ast}{\mathcal{B}(n_\ast)}=1\)
can be omitted, since each layer contains a mainline vertex.
Beginning with \(n=n_\ast+d\),
the expression for the tree layer cardinality \(\#\mathrm{Lyr}_n{\mathcal{T}}\) is a sum of \(d+1\) terms
and we must find the logarithmic order \(n=n_\ast+x\) where periodicity of all terms sets in.
This leads to the inequality \(n_\ast+x-d\ge n_\ast+\ell_\ast+\ell\)
with solution \(x\ge\ell_\ast+\ell+d\).
Consequently, \(m=n_\ast+\ell_\ast+\ell+d-1\) is the biggest logarithmic order
for which a new value of the tree layer cardinality \(\#\mathrm{Lyr}_m{\mathcal{T}}\) may occur
(see Figure
\ref{fig:SchematicTree}).
At the logarithmic order \(m+1\), periodic repetitions of the values of tree layer cardinalities begin.

\noindent
In the special case of \(d\le 3\), Theorem
\ref{thm:DescendantNumbers}
yields an expression in terms of descendant numbers:
\begin{equation*}
\begin{aligned}
\#\mathrm{Lyr}_n{\mathcal{T}}
= & \#\mathrm{Lyr}_n{\mathcal{B}(n-d)}+\ldots+\#\mathrm{Lyr}_n{\mathcal{B}(n)} \\
= & N_1(m_{n-1})+\sum_{i=2}^{C_1(m_{n-2})}\,N_1(v_i(m_{n-2}))+\sum_{i=2}^{C_1(m_{n-3})}\,\sum_{j=1}^{C_1(v_i(m_{n-3}))}\,N_1(v_{i,j}(m_{n-3})).
\end{aligned}
\qedhere
\end{equation*}
\end{proof}

%--------------------------------------------------------------------------------

\noindent
The following concept provides
a quantitative measure for the \textit{finite} information content
of an \textit{infinite} tree with periodic branches.

\begin{definition}
\label{dfn:InfoCont}
By the \textit{information content} of a coclass tree \(\mathcal{T}\)
we understand the sum of the cardinalities of all branches
belonging to the pre-period and to the primitive period of \(\mathcal{T}\),
\begin{equation}
\label{eqn:InfoCont}
\mathrm{IC}(\mathcal{T}):=\left(\sum_{n=n_\ast}^{p_\ast-1}\,\#\mathcal{B}(n)\right)+\left(\sum_{n=p_\ast}^{p_\ast+\ell-1}\,\#\mathcal{B}(n)\right),
\end{equation}
where \(p_\ast=n_\ast+\ell_\ast\) denotes the logarithmic order of the periodic root \(P\) of \(\mathcal{T}\)
(see Figure
\ref{fig:SchematicTree}).
\end{definition}

%\newpage

%--------------------------------------------------------------------------------

\section{Identifiers of the SmallGroups Library}
\label{s:SmallGroups}
\noindent
Independently of being metabelian or non-metabelian,
a finite \(3\)-group \(G\) of order up to \(3^8=6561\) will be characterized by
its absolute identifier \(G\simeq\langle\lvert G\rvert,i\rangle\),
according to the SmallGroups Database
\cite{BEO1,BEO}.
Starting with order \(3^9=19683\), a group \(G\) is characterized by
the absolute identifier of the parent \(\pi(G)\simeq\langle\lvert\pi(G)\rvert,i\rangle\)
in the SmallGroups Database
\cite{BEO}
together with a relative identifier \(-\#s;j\) generated by the ANUPQ package
\cite{GNO}
of MAGMA
\cite{MAGMA}.
Here, \(s\) denotes the step size of the directed edge \(G\to\pi(G)\).
Occasionally, certain groups of order \(3^6=729\) and coclass \(2\)
are identified by single capital letters \(\mathrm{A},\ldots,\mathrm{X}\) similarly as in
\cite{As1,As2,AHL}.

%\newpage

%--------------------------------------------------------------------------------

\section{Mainlines of coclass trees and sporadic parts of coclass forests}
\label{s:MainlineSporadic}
\noindent
If we define a \textit{mainline} as a maximal path of infinitely many directed edges of step size \(s=1\),
then there arises the ambiguity that a vertex could be root of several coclass trees.
The metabelian \(3\)-group \(\langle 243,3\rangle=P_5\),
%in Figure
%\ref{fig:SporCc2},
for instance, would be the end vertex of more then one mainline, namely
on the one hand of the metabelian mainline
\[P_5\leftarrow\langle 729,40\rangle=B=R_1^2\leftarrow\langle 2187,247\rangle\leftarrow\langle 6561,1988\rangle\leftarrow\ldots\]
and on the other hand of non-metabelian mainlines,
one which ends with
\[P_5\leftarrow\langle 729,35\rangle=I\leftarrow\langle 2187,235\rangle\leftarrow\langle 6561,1979\rangle\leftarrow\ldots \text{ and}\]
three which end with
\(P_5\leftarrow\langle 729,34\rangle=H\leftarrow\langle 2187,228\rangle\leftarrow\langle 6561,i\rangle\leftarrow\ldots\),
\(i\in\lbrace 1916,1920,1928\rbrace\).

Therefore, an additional condition is required in the precise definition of a mainline.

%--------------------------------------------------------------------------------

\begin{definition}
\label{dfn:mainline}
A \textit{mainline} is a maximal path of infinitely many equally oriented edges of step size \(s=1\),
in none of whose vertices other infinite paths of step size \(s=1\) are ending.

The end vertex of a mainline is called the \textit{root of a coclass tree}.
\end{definition}

\noindent
Definition
\ref{dfn:mainline}
can be expressed equivalently in terms of infinite pro-\(p\) groups
\cite[\S\ 3.1, p. 107]{Ek}.

%--------------------------------------------------------------------------------

\begin{example}
\label{exm:mainline}
The metabelian \(3\)-group \(\langle 729,40\rangle=B\) is root of the coclass-\(2\) tree \(\mathcal{T}^2{B}\)
with \textit{metabelian} mainline.
%in Figure
%\ref{fig:SporCc2}.
%This tree will be investigated in \S\
%\ref{s:Coclass1And2}.

The metabelian \(3\)-group \(\langle 729,35\rangle=I\) is root of the coclass-\(2\) tree \(\mathcal{T}^2{I}\)
with \textit{non-metabelian} mainline.
According to our pruning convention that descendants of capable non-metabelian vertices
do not belong to the coclass forests \(\mathcal{F}(r)\),
this tree is not an object of examination in the present paper.

Finally, the metabelian \(3\)-group \(\langle 729,34\rangle=H\) is \textit{not} root of a coclass tree.
\end{example}

\noindent
Based on the precise definition of a mainline and a root of a coclass tree,
we are now in the position to give an exact specification of the sporadic part of a coclass forest.

%--------------------------------------------------------------------------------

\begin{definition}
\label{dfn:sporadic}
The \textit{sporadic part} of the coclass forest \(\mathcal{F}(r)\) with \(r\ge 1\)
is the complement of the union of the (finitely many) coclass trees in the forest,
\begin{equation}
\label{eqn:sporadic}
\mathcal{F}_0(r)=\mathcal{F}(r)\setminus\left(\dot{\bigcup}_{i=1}^t\,\mathcal{T}^r_i\right) \text{ with } t\ge 0.
\end{equation}
\end{definition}

There is no necessity, to restrict the concepts of a mainline, a coclass tree and its root further
by stipulating the \textit{coclass stability} of the root.
It is therefore admissible that directed edges of step size \(s\ge 2\)
end in vertices (mainline or of depth \(\mathrm{dp}\ge 1\)) of a coclass tree,
due to the phenomenon of \textit{multifurcation}.

%--------------------------------------------------------------------------------

\begin{example}
\label{exm:multifurcation}
In the second mainline vertex \(m_6=\langle 729,49\rangle=Q\)
of the coclass-\(2\) tree \(\mathcal{T}^2{R_2^2}\) with root \(R_2^2=\langle 243,6\rangle\),
%which is visualized in Figure
%\ref{fig:SporCc2}
%and more detailed in Figure
%\ref{fig:TKTc18TreeCc2},
a \textit{bifurcation} occurs, due to the nuclear rank \(\nu(Q)=2\).
In fact, the directed edge of step size \(s=2\)
which ends in the vertex \(m_6=Q\)
is the final edge of an infinite path with alterating step sizes \(s=2\) and \(s=1\),
due to \textit{periodic bifurcations}.
However, this non-metabelian path is not the topic of investigations in the present paper.
For detailed information on these matters see
\cite{Ma6}
and
\cite{Ma16}.

The same is true for the second mainline vertex \(\langle 729,54\rangle=U\)
of the coclass-\(2\) tree \(\mathcal{T}^2{R_3^2}\) with root \(R_3^2=\langle 243,8\rangle\).
%in Figure
%\ref{fig:SporCc2}.

The unnecessary requirement of coclass stability
would eliminate the pre-periods of the trees \(\mathcal{T}^2{R_2}\) and \(\mathcal{T}^2{R_3}\)
and enforce \textit{purely} periodic subtrees with periodic \textit{coclass-settled} roots, namely
\(\mathcal{T}^2\langle 2187,285\rangle\subset\mathcal{T}^2{R_2}\) and
\(\mathcal{T}^2\langle 2187,303\rangle\subset\mathcal{T}^2{R_3}\).
\end{example}

%\newpage

%--------------------------------------------------------------------------------

\section{Two main theorems on periodicity and co-periodicity isomorphisms}
\label{s:MainTheorems}
\noindent
An important technique in the theory of descendant trees
is to reduce the structure of an infinite tree
to a periodically repeating finite pattern.
In particular, it is well known
\cite{dS,EkLg,dSSg}
that an infinite coclass tree \(\mathcal{T}^r\) of finite \(p\)-groups with fixed coclass \(r\ge 1\)
is the disjoint union of its branches \(\mathcal{T}^r=\dot{\bigcup}_{n=n_\ast}^\infty\,\mathcal{B}(n)\),
which can be partitioned into a single finite pre-period \(\dot{\bigcup}_{n=n_\ast}^{p_\ast-1}\,\mathcal{B}(n)\) of length \(\ell_\ast=p_\ast-n_\ast\ge 0\)
and infinitely many copies of a finite primitive period \(\dot{\bigcup}_{n=p_\ast}^{p_\ast+\ell-1}\,\mathcal{B}(n)\) of length \(\ell\ge 1\),
where the integer \(p_\ast\ge n_\ast\) characterizes the position of the \textit{periodic root} on the mainline,
provided the tree is suitably \textit{depth-pruned}.

%--------------------------------------------------------------------------------

The following \textit{first main result} of this paper
establishes the details of the primitive period of branches of five coclass-\(4\) trees
\(\mathcal{T}^4{R_i^4}=\dot{\bigcup}_{n=n_\ast}^\infty\,\mathcal{B}(n)\), with \(n_\ast=9\),
respectively of three coclass-\(5\) trees
\(\mathcal{T}^5{R_j^5}=\dot{\bigcup}_{n=n_\ast}^\infty\,\mathcal{B}(n)\), with \(n_\ast=11\),
of finite \(3\)-groups with mainline vertices having a single total transfer kernel
and roots \(R_i^4:=\langle 2187,64\rangle-\#2;n(i)\) with \((n(i))_{2\le i\le 6}=(39,44,54,57,59)\),
respectively \(R_j^5:=\langle 2187,64\rangle-\#2;33-\#2;n(j)\) with \((n(j))_{2\le j\le 4}=(29,37,39)\),
written in the notation of
\cite{BEO1,BEO,GNO}.
In fact, we prove more than the \textit{virtual} periodicity for arbitrary finite \(p\)-groups in
\cite{dS,EkLg,dSSg},
since all trees of the particular finite \(3\)-groups in our investigation have \textit{bounded depth}
and therefore reveal \textit{strict} periodicity.

\begin{theorem}
\label{thm:FirstPeriod}
\textbf{(Main Theorem on Strict Periodicity Isomorphisms of Branches.)} \\
For each integer \(n\ge n_\ast\),
there exists a bijective mapping \(\psi:\,\mathcal{B}(n)\to\mathcal{B}(n+2)\)
which is a strict isomorphism of finite structured in-trees for the \textbf{strict invariants}
in-degree \(\mathrm{in}()\),
out-degree \(\mathrm{out}()\),
coclass \(r=\mathrm{cc}()\),
relation rank \(\mu()\),
nuclear rank \(\nu()\),
action flag \(\sigma()\),
and transfer kernel type \(\varkappa()\).
Moreover, \(\psi\) is a \(\phi\)-isomorphism of finite structured in-trees
for the following \(\phi\)-\textbf{invariants} with their transformation laws \(\phi\):
\begin{itemize}
\item
logarithmic order \(n=\mathrm{lo}()\) with \(\phi(n)=n+2\),
\item
nilpotency class \(m-1=c=\mathrm{cl}()\) with \(\phi(c)=c+2\),
\item
order of the automorphism group \(a=\#\mathrm{Aut}()\) with \(\phi(a)=a\cdot 3^4\),
\item
first component of the transfer target type \(\tau(1)()\) with \(\phi(A(3,c-k))=A(3,2+c-k)\),
and
\item
commutator subgroup \(\tau_2()\) with \(\phi(A(3,c-1)\times A(3,r-1))=A(3,c+1)\times A(3,r-1)\),
respectively \(\phi(A(3,c-2)\times A(3,r))=A(3,c)\times A(3,r)\).
\end{itemize}
Consequently, the branches of each tree \(\mathcal{T}^4{R_i^4}\) \((2\le i\le 6)\),
respectively \(\mathcal{T}^5{R_j^5}\) \((2\le j\le 4)\),
are purely periodic with primitive length at most \(\ell=2\).
\end{theorem}

%--------------------------------------------------------------------------------

\begin{proof}
The \(\phi\)-isomorphisms between the finite branches of a tree describe the \textit{first periodicity} and
reduce an infinite tree to its finite primitive period, provided the periodicity is pure.
This will be proved for even coclass \(r\ge 4\) in Theorem
\ref{thm:TKTd19Tree1Cc4}
for \(i=39\), in Thm.
\ref{thm:TKTd19Tree2Cc4}
for \(i=44\), in Thm.
\ref{thm:TKTd23TreeCc4}
for \(i=54\), in Thm.
\ref{thm:TKTd25Tree1Cc4}
for \(i=57\), and in Thm.
\ref{thm:TKTd25Tree2Cc4}
for \(i=59\).
For odd coclass \(r\ge 5\), it will be proved in Theorem
\ref{thm:TKTd19TreeCc5}
for \(j=29\), in Thm.
\ref{thm:TKTd23TreeCc5}
for \(j=37\), and in Thm.
\ref{thm:TKTd25TreeCc5}
for \(j=39\).

Invariants connected with the nilpotency class are not strict and satisfy the following transformation laws:
the shift \(\phi(i)=i+2\) for \(\mathrm{lo}()\) and \(\mathrm{cl}()\),
and the corresponding transformations
\(\phi(A(3,c-k))=A(3,2+c-k)\) for \(\tau(1)()\), and
\(\phi(A(3,c-1)\times A(3,r-1))=A(3,2+c-1)\times A(3,r-1)\) for \(\tau_2()\),
with fixed coclass \(r\).
For \(\#\mathrm{Aut}()\), the transformation law is described by the homothety \(\phi(i)=i\cdot 3^4\).
\end{proof}

Theorems
\ref{thm:TKTd19Tree1Cc4},
\ref{thm:TKTd19Tree2Cc4},
\ref{thm:TKTd23TreeCc4},
\ref{thm:TKTd25Tree1Cc4},
\ref{thm:TKTd25Tree2Cc4}
and
\ref{thm:TKTd19TreeCc5},
\ref{thm:TKTd23TreeCc5},
\ref{thm:TKTd25TreeCc5}
will give detailed descriptions of the structure of these trees,
in particular they will establish a
quantitative measure for the finite information content of each tree.

%--------------------------------------------------------------------------------

\begin{remark}
\label{rmk:FirstPeriod}
According to Theorem
\ref{thm:FirstPeriod},
the diagrams of coclass-\(r\) trees \((r\ge 4)\) whose mainline vertices \(V\) possess
a single total kernel \(\ker(T_1)=V\) among the transfers \(T_i:\,V\to U_i/U_i^\prime\)
to the four maximal subgroups \(U_i<V\) \((1\le i\le 4)\)
reveal several \textbf{surprising features}:
firstly, the branches are \textbf{purely periodic} of primitive length at most \(2\) without pre-period,
secondly, the branches are of \textbf{uniform depth \(2\) only}, and
finally, none of the vertices gives rise to descendants of coclass bigger than \(r\).
So the trees are \textbf{entirely regular and coclass-stable},
in contrast to the trees with \(3\)-groups \(G\) of coclass \(r=\mathrm{cc}(G)\in\lbrace 1,2,3\rbrace\) as vertices.
\end{remark}

%\newpage

%--------------------------------------------------------------------------------

Unfortunately it is much less well known that
the entire metabelian skeleton \(\mathcal{M}(R)\) of the descendant tree \(\mathcal{T}(R)\)
of the elementary bicyclic \(3\)-group \(R:=\langle 9,2\rangle\simeq C_3\times C_3\)
is the disjoint union of its coclass subgraphs \(\mathcal{M}(R)=\dot{\bigcup}_{r=1}^\infty\,\mathcal{M}^r\),
where each component \(\mathcal{M}^r=\mathcal{M}_0^r\dot{\cup}(\dot{\bigcup}\,\tilde{\mathcal{T}}_i^r)\) consists of
a finite sporadic part \(\mathcal{M}_0^r\) and finitely many metabelian coclass trees \(\tilde{\mathcal{T}}_i^r\),
and there is a periodicity \(\mathcal{M}^r\simeq\mathcal{M}^{r+2}\) for each \(r\ge 3\).
This was proved by Nebelung
\cite{Ne}
and confirmed by Eick
\cite[Cnj. 14, p. 115]{Ek}.

%--------------------------------------------------------------------------------

The following \textit{second main result} of this paper
extends the periodicity from the metabelian skeleton to the entire descendant tree,
including all the non-metabelian vertices,
provided the mainline vertices are still metabelian.
Here, we include coclass trees of finite \(3\)-groups
with mainline vertices having two total transfer kernels
and roots \(R_1^4:=\langle 2187,64\rangle-\#2;n(1)\) with \(n(1)=33\),
respectively \(R_1^5:=\langle 2187,64\rangle-\#2;33-\#2;n(1)\) with \(n(1)=25\).

\begin{theorem}
\label{thm:SecondPeriod}
\textbf{(Main Theorem on Co-Periodicity Isomorphisms of Coclass Trees.)} \\
Let the integer \(u:=19\) be an upper bound.
For each integer \(4\le r\le u\), and
for each of the six roots \(R_i^r\), \(1\le i\le 6\), with even coclass \(r\ge 4\),
respectively the four roots \(R_j^r\), \(1\le j\le 4\), with odd coclass \(r\ge 5\),
there exists a bijective mapping \(\psi:\,\mathcal{T}^r{R_i^r}\to\mathcal{T}^{r+2}{R_i^{r+2}}\),
respectively \(\psi:\,\mathcal{T}^r{R_j^r}\to\mathcal{T}^{r+2}{R_j^{r+2}}\),
which is a strict isomorphism of infinite structured in-trees for the \textbf{strict invariants}
in-degree \(\mathrm{in}()\),
out-degree \(\mathrm{out}()\),
relation rank \(\mu()\),
nuclear rank \(\nu()\),
action flag \(\sigma()\),
and transfer kernel type \(\varkappa()\).
Moreover, \(\psi\) is a \(\phi\)-isomorphism of infinite structured in-trees
for the following \(\phi\)-\textbf{invariants} with their transformation laws \(\phi\):
\begin{itemize}
\item
logarithmic order \(n=\mathrm{lo}()\) with \(\phi(n)=n+4\),
\item
nilpotency class \(m-1=c=\mathrm{cl}()\) with \(\phi(c)=c+2\),
\item
coclass \(r=\mathrm{cc}()\) with \(\phi(r)=r+2\),
\item
order of the automorphism group \(a=\#\mathrm{Aut}()\) with \(\phi(a)=a\cdot 3^8\),
\item
first component of the transfer target type \(\tau(1)()\) with \(\phi(A(3,c-k))=A(3,2+c-k)\), and
\item
commutator subgroup \(\tau_2()\) with \(\phi(A(3,c-1)\times A(3,r-1))=A(3,c+1)\times A(3,r+1)\),
respectively \(\phi(A(3,c-2)\times A(3,r))=A(3,c)\times A(3,r+2)\).
\end{itemize}
\end{theorem}

\begin{proof}
The statement for the metabelian skeletons \(\tilde{\mathcal{T}}^r{R_i^r}\) of the coclass trees \(\mathcal{T}^r{R_i^r}\)
is one of the main results of Nebelung's thesis
\cite{Ne}.
With the aid of Theorem
\ref{thm:FirstPeriod},
the periodicity of the entire coclass trees \(\mathcal{T}^r{R_i^r}\) with \(4\le r\le 21\) and fixed subscript \(i\)
has been verified by computing the metabelian and non-metabelian vertices
of the first four branches \(\mathcal{B}^r(n)\) with \(2r+1\le n\le 2r+4\) of the trees \(\mathcal{T}^r{R_i^r}\).
The computations were executed by running our own program scripts for the Computer Algebra System MAGMA
\cite{MAGMA},
which contains an implementation of the \(p\)-group generation algorithm by Newman
\cite{Nm2,Nm}
and O'Brien
\cite{Ob,HEO},
the SmallGroups Database
\cite{BEO1,BEO},
and the ANUPQ package
\cite{GNO}.
It turned out that,
firstly, \(\mathcal{B}^r(2r+1)\simeq\mathcal{B}^r(2r+3)\) and \(\mathcal{B}^r(2r+2)\simeq\mathcal{B}^r(2r+4)\), for each \(4\le r\le 21\),
and secondly, \(\mathcal{B}^r(2r+1)\simeq\mathcal{B}^{r+2}(2(r+2)+1)\) and \(\mathcal{B}^r(2r+2)\simeq\mathcal{B}^{r+2}(2(r+2)+2)\), for each \(4\le r\le 19\).

The established \(\phi\)-isomorphisms between the infinite coclass trees
\(\mathcal{T}^r{R_i^r}\) and \(\mathcal{T}^{r+2}{R_i^{r+2}}\), for \(4\le r\le 19\),
describe the germ of the \textit{second periodicity} expressed in Conjecture
\ref{cnj:SecondPeriod}.
Invariants connected with the nilpotency class or coclass are not strict and are subject to the following mappings:
the shifts \(\phi(c)=c+2\) for \(c=\mathrm{cl}()\), \(\phi(r)=r+2\) for \(r=\mathrm{cc}()\), and \(\phi(n)=n+4\) for \(n=\mathrm{lo}()\),
and the corresponding transformations
\(\phi(A(3,c-k))=A(3,2+c-k)\) for \(\tau(1)()\), and
\(\phi(A(3,c-1)\times A(3,r-1))=A(3,c+1)\times A(3,r+1)\) for \(\tau_2()\).
For \(a=\#\mathrm{Aut}()\), the transformation law is described by the homothety \(\phi(a)=a\cdot 3^8\).
\end{proof}

%\newpage

%--------------------------------------------------------------------------------

Thus, the confidence in the validity of the following conjecture is supported extensively by sound numerical data.

\begin{conjecture}
\label{cnj:SecondPeriod}
\textbf{(Co-Periodicity Isomorphisms of All Coclass-\(r\) Trees for \(r\ge 4\).)} \\
Theorem
\ref{thm:SecondPeriod}
remains true
when the upper bound \(u=19\) is replaced by any upper bound \(u>19\).

Consequently, all coclass trees \(\mathcal{T}^r{R_i^r}\) with \(r\ge 4\) and fixed subscript \(i\)
are co-periodic in the variable coclass parameter \(r\) with primitive length \(\ell=2\).
The eight coclass trees \(\mathcal{T}^r{R_i^r}\) with \(r\in\lbrace 1,2,3\rbrace\),
and \(i=1\) for \(r=1\), \(i\in\lbrace 1,2,3\rbrace\) for \(r=2\), \(i\in\lbrace 1,2,3,4\rbrace\) for \(r=3\),
can be viewed as the pre-period of the co-periodicity.
(Compare
\cite[Cnj. 14, p. 115]{Ek}.)
\end{conjecture}

%\newpage

%--------------------------------------------------------------------------------

\section{Parametrized polycyclic power-commutator presentations}
\label{s:Presentations}
\noindent
The general graph theoretic and algebraic foundations of the coclass forests \(\mathcal{F}(r)\) with \(r\ge 1\)
have been developed completely in the preceding sections
\ref{s:TreesForests}
--
\ref{s:MainlineSporadic}.
Now we can turn to the main goal of the present paper, that is,
the proof of the main theorems in section
\ref{s:MainTheorems}
by the systematic investigation of finite \(3\)-groups \(G\) with commutator quotient \(G/G^\prime\simeq R:=C_3\times C_3\),
represented by vertices of the descendant tree \(\mathcal{T}(R)\),
with the single restriction that the parent \(\pi(G)\) of \(G\) is metabelian.
To this end, we first need parametrized presentations for all metabelian vertices of \(\mathcal{T}(R)\).

%--------------------------------------------------------------------------------

\subsection{\(3\)-groups of coclass \(r=1\)}
\label{ss:PresentationsMaxCls}
\noindent
The identification of \(3\)-groups \(G\) with coclass \(\mathrm{cc}(G)=1\),
which are metabelian without exceptions
\cite{Bl1},
will be achieved with the aid of
parametrized polycyclic power-commutator presentations, as given by Blackburn
\cite{Bl2}:
\begin{equation}
\label{eqn:PresentationCc1}
\begin{aligned}
G_a^n(z,w) := \langle x,y,s_2,\ldots,s_{n-1}\mid s_2=\lbrack y,x\rbrack,\ (\forall_{i=3}^n)\ s_i=\lbrack s_{i-1},x\rbrack,\ s_n=1,\ \lbrack y,s_2\rbrack=s_{n-1}^a, \\
(\forall_{i=3}^{n-1})\ \lbrack y,s_i\rbrack=1,\ x^3=s_{n-1}^w,\ y^3s_2^3s_3=s_{n-1}^z,\ (\forall_{i=2}^{n-3})\ s_i^3s_{i+1}^3s_{i+2}=1,\ s_{n-2}^3=s_{n-1}^3=1\ \rangle,
\end{aligned}
\end{equation}
where \(a\in\lbrace 0,1\rbrace\) and \(w,z\in\lbrace -1,0,1\rbrace\) are bounded parameters,
and the \textit{index of nilpotency} \(m=\mathrm{cl}(G)+1=\mathrm{cl}(G)+\mathrm{cc}(G)=\log_3(\mathrm{ord}(G))=\mathrm{lo}(G)=n\) is an unbounded parameter.

%--------------------------------------------------------------------------------

\subsection{\(3\)-groups of coclass \(r\ge 2\)}
\label{ss:PresentationsNonMaxCls}
\noindent
Metabelian \(3\)-groups with coclass \(\mathrm{cc}(G)\ge 2\) will be identified with the aid of
parametrized polycyclic power-commutator presentations, given by Nebelung
\cite{Ne}:
\(\stackbin[\rho]{m,n}{G}\begin{pmatrix} \alpha & \beta \\ \gamma & \delta \end{pmatrix}:=G_\rho^{m,n}(\alpha,\beta,\gamma,\delta):=\)
\begin{equation}
\label{eqn:Presentation}
\begin{aligned}
& \langle\ x,y,s_2,s_3,t_3,\sigma_3,\ldots,\sigma_{m-1},\tau_3,\ldots,\tau_{e+1}\ \mid
\ s_2=\lbrack y,x\rbrack,\ s_3=\lbrack s_2,x\rbrack,\ t_3=\lbrack s_2,y\rbrack, \\
& \sigma_3=y^3,\ (\forall_{i=4}^m)\ \sigma_i=\lbrack\sigma_{i-1},x\rbrack,\ \sigma_m=1,
\ \tau_3=x^3,\ (\forall_{i=4}^{e+2})\ \tau_i=\lbrack\tau_{i-1},y\rbrack,\ \tau_{e+2}=1, \\
& s_2^3=\sigma_4\sigma_{m-1}^{-\rho\beta}\tau_4^{-1},
\ s_3\sigma_3\sigma_4=\sigma_{m-2}^{\rho\beta}\sigma_{m-1}^\gamma\tau_e^\delta,
\ t_3^{-1}\tau_3\tau_4=\sigma_{m-2}^{\rho\delta}\sigma_{m-1}^\alpha\tau_e^\beta,
\ \tau_{e+1}=\sigma_{m-1}^{-\rho}, \\
& \lbrack s_3,y\rbrack=\sigma_{m-1}^{-\rho\delta},\ (\forall_{i=3}^{m-1})\ \lbrack\sigma_i,y\rbrack=1,
\ (\forall_{i=3}^{m-3})\ \sigma_i^3\sigma_{i+1}^3\sigma_{i+2}=1,\ \sigma_{m-2}^3=\sigma_{m-1}^3=1, \\
& \lbrack t_3,x\rbrack=\sigma_{m-1}^{-\rho\delta},\ (\forall_{i=3}^{e+1})\ \lbrack\tau_i,x\rbrack=1,
\ (\forall_{i=3}^{e-1})\ \tau_i^3\tau_{i+1}^3\tau_{i+2}=1,\ \tau_{e}^3=\tau_{e+1}^3=1\ \rangle,
\end{aligned}
\end{equation}
where \(\alpha,\beta,\gamma,\delta,\rho\in\lbrace -1,0,1\rbrace\) are bounded parameters,
and the \textit{index of nilpotency} \(m=\mathrm{cl}(G)+1\),
the \textit{logarithmic order} \(n=\mathrm{lo}(G)=\log_3(\mathrm{ord}(G))=\mathrm{cl}(G)+\mathrm{cc}(G)=m+e-2\),
and the CF-\textit{invariant} \(e=\mathrm{cc}(G)+1\)
are unbounded parameters.

%\newpage

%--------------------------------------------------------------------------------

\section{The backbone of the tree \(\mathcal{T}(R)\): the infinite main trunk}
\label{s:MainTrunk}
\noindent
The flow of our investigations is guided by the present section
concerning the remarkable infinite \textit{main trunk} \((P_{2r-1})_{r\ge 2}\) of certain metabelian vertices in \(\mathcal{T}\)
which gives rise to the top vertices of all coclass forests \((\mathcal{F}(r))_{r\ge 2}\) by periodic bifurcations
and constitutes the germ of the newly discovered co-periodicity \(\mathcal{F}(r+2)\simeq\mathcal{F}(r)\) of length two.
Since the minimal possible values of the nilpotency class and logarithmic order
of a finite metabelian \(3\)-group with coclass \(\mathrm{cc}(G)=r\ge 2\), belonging to the forest \(\mathcal{F}(r)\),
are given by \(c=\mathrm{cl}(G)=r+1\) and \(\mathrm{lo}(G)=\mathrm{cl}(G)+\mathrm{cc}(G)=2r+1\),
it follows that \(G\) must be an immediate descendant of step size \(s=2\) of its parent \(\pi(G)=G/\gamma_c{G}\).
The crucial fact is that this parent is precisely
the vertex \(\pi(G)=P_{2r-1}\) with \(\mathrm{lo}(\pi(G))=2r-1\) of the main trunk.
In the following, we rather use the coclass \(j\) of the parent than \(r\) of the children.

%\newpage

%--------------------------------------------------------------------------------

\begin{figure}[hb]
\caption{Metabelian mainline skeleton of the descendant tree \(\mathcal{T}(C_3\times C_3)\)}
\label{fig:MainTrunk}

% Nebelung's Main Trunk
% Metabelian Mainline Skeleton of the
% Descendant Tree T(C3*C3)

{\tiny

\setlength{\unitlength}{0.9cm}
\begin{picture}(17,17)(0,-16)

% scale of logarithmic orders
\put(0,0.8){\makebox(0,0)[cb]{Logarithmic}}
\put(0,0.5){\makebox(0,0)[cb]{Order}}

\put(0,0){\line(0,-1){14}}
\multiput(-0.1,0)(0,-1){14}{\line(1,0){0.2}}

%\put(-0.2,0){\makebox(0,0)[rc]{\(9\)}}
%\put(0.2,0){\makebox(0,0)[lc]{\(3^2\)}}
\put(-0.2,0){\makebox(0,0)[rc]{\(2\)}}
%\put(-0.2,-1){\makebox(0,0)[rc]{\(27\)}}
%\put(0.2,-1){\makebox(0,0)[lc]{\(3^3\)}}
\put(-0.2,-1){\makebox(0,0)[rc]{\(3\)}}
%\put(-0.2,-2){\makebox(0,0)[rc]{\(81\)}}
%\put(0.2,-2){\makebox(0,0)[lc]{\(3^4\)}}
\put(-0.2,-2){\makebox(0,0)[rc]{\(4\)}}
%\put(-0.2,-3){\makebox(0,0)[rc]{\(243\)}}
%\put(0.2,-3){\makebox(0,0)[lc]{\(3^5\)}}
\put(-0.2,-3){\makebox(0,0)[rc]{\(5\)}}
%\put(-0.2,-4){\makebox(0,0)[rc]{\(729\)}}
%\put(0.2,-4){\makebox(0,0)[lc]{\(3^6\)}}
\put(-0.2,-4){\makebox(0,0)[rc]{\(6\)}}
%\put(-0.2,-5){\makebox(0,0)[rc]{\(2\,187\)}}
%\put(0.2,-5){\makebox(0,0)[lc]{\(3^7\)}}
\put(-0.2,-5){\makebox(0,0)[rc]{\(7\)}}
%\put(-0.2,-6){\makebox(0,0)[rc]{\(6\,561\)}}
%\put(0.2,-6){\makebox(0,0)[lc]{\(3^8\)}}
\put(-0.2,-6){\makebox(0,0)[rc]{\(8\)}}
%\put(-0.2,-7){\makebox(0,0)[rc]{\(19\,683\)}}
%\put(0.2,-7){\makebox(0,0)[lc]{\(3^9\)}}
\put(-0.2,-7){\makebox(0,0)[rc]{\(9\)}}
%\put(-0.2,-8){\makebox(0,0)[rc]{\(59\,049\)}}
%\put(0.2,-8){\makebox(0,0)[lc]{\(3^{10}\)}}
\put(-0.2,-8){\makebox(0,0)[rc]{\(10\)}}
%\put(-0.2,-9){\makebox(0,0)[rc]{\(177\,147\)}}
%\put(0.2,-9){\makebox(0,0)[lc]{\(3^{11}\)}}
\put(-0.2,-9){\makebox(0,0)[rc]{\(11\)}}
%\put(-0.2,-10){\makebox(0,0)[rc]{\(531\,441\)}}
%\put(0.2,-10){\makebox(0,0)[lc]{\(3^{12}\)}}
\put(-0.2,-10){\makebox(0,0)[rc]{\(12\)}}
%\put(-0.2,-11){\makebox(0,0)[rc]{\(1\,594\,323\)}}
%\put(0.2,-11){\makebox(0,0)[lc]{\(3^{13}\)}}
\put(-0.2,-11){\makebox(0,0)[rc]{\(13\)}}
%\put(-0.2,-12){\makebox(0,0)[rc]{\(4\,782\,969\)}}
%\put(0.2,-12){\makebox(0,0)[lc]{\(3^{14}\)}}
\put(-0.2,-12){\makebox(0,0)[rc]{\(14\)}}
%\put(-0.2,-13){\makebox(0,0)[rc]{\(14\,348\,907\)}}
%\put(0.2,-13){\makebox(0,0)[lc]{\(3^{15}\)}}
\put(-0.2,-13){\makebox(0,0)[rc]{\(15\)}}

\put(0,-14){\vector(0,-1){1}}

%*******************************************************

% abelian tree root
\put(1.8,0.2){\makebox(0,0)[rc]{root}}
\put(2.2,0.2){\makebox(0,0)[lc]{\(\langle 9,2\rangle=C_3\times C_3\)}}
\put(1.95,-0.05){\framebox(0.1,0.1){}}

\put(2,0){\line(0,-1){1}}

%*******************************************************

% periodic bifurcations
\multiput(2.3,-1.2)(2.0,-2.0){6}{\makebox(0,0)[lc]{bifurcation}}
\multiput(2,-1)(2.0,-2.0){6}{\circle*{0.1}}

% mainlines
\multiput(2,-1)(2.0,-2.0){6}{\line(0,-1){3}}
\multiput(2,-4)(2.0,-2.0){6}{\vector(0,-1){1}}

% exception (pre-period)
\put(1.8,-5.2){\makebox(0,0)[rc]{TKT}}
\put(2.2,-5.2){\makebox(0,0)[lc]{a.1}}
% standard (period)
\multiput(4.2,-7.2)(2.0,-2.0){5}{\makebox(0,0)[lc]{b.10}}

% further mainlines
\multiput(2,-1)(2.0,-2.0){5}{\line(1,-2){1}}
\multiput(3,-3)(2.0,-2.0){5}{\circle{0.1}}
\multiput(3,-3)(2.0,-2.0){5}{\line(0,-1){3}}
\multiput(3,-6)(2.0,-2.0){5}{\vector(0,-1){1}}

\multiput(7.15,-7)(4.0,-4.0){2}{\circle{0.1}}
\multiput(7.15,-7)(4.0,-4.0){2}{\line(0,-1){3}}
\multiput(7.15,-10)(4.0,-4.0){2}{\vector(0,-1){1}}

\multiput(5.3,-5)(2.0,-2.0){4}{\circle{0.1}}
\multiput(5.3,-5)(2.0,-2.0){4}{\line(0,-1){3}}
\multiput(5.3,-8)(2.0,-2.0){4}{\vector(0,-1){1}}

\multiput(7.45,-7)(4.0,-4.0){2}{\circle{0.1}}
\multiput(7.45,-7)(4.0,-4.0){2}{\line(0,-1){3}}
\multiput(7.45,-10)(4.0,-4.0){2}{\vector(0,-1){1}}

\multiput(5.6,-5)(2.0,-2.0){4}{\circle{0.1}}
\multiput(5.6,-5)(2.0,-2.0){4}{\line(0,-1){3}}
\multiput(5.6,-8)(2.0,-2.0){4}{\vector(0,-1){1}}

% exception (pre-period)
\put(3.2,-7.2){\makebox(0,0)[rc]{c.18}}
\put(3.5,-3){\circle{0.1}}
\put(3.5,-3){\line(0,-1){3}}
\put(3.5,-6){\vector(0,-1){1}}
\put(3.7,-7.5){\makebox(0,0)[rc]{c.21}}
% standard
\multiput(5.2,-9.2)(4.0,-4.0){2}{\makebox(0,0)[rc]{d.19}}
\multiput(7.2,-11.2)(4.0,-4.0){2}{\makebox(0,0)[rc]{2 d.19}}
\multiput(5.5,-9.5)(2.0,-2.0){4}{\makebox(0,0)[rc]{d.23}}
\multiput(5.8,-9.8)(4.0,-4.0){2}{\makebox(0,0)[rc]{d.25}}
\multiput(7.8,-11.8)(4.0,-4.0){2}{\makebox(0,0)[rc]{2 d.25}}

% coclass trees
\put(2.0,-2.0){\makebox(0,0)[cc]{\(\mathrm{cc}=1\)}}
\put(3.5,-4.0){\makebox(0,0)[cc]{\(\mathrm{cc}=2\)}}
\put(5.5,-6.0){\makebox(0,0)[cc]{\(\mathrm{cc}=3\)}}
\put(7.5,-8.0){\makebox(0,0)[cc]{\(\mathrm{cc}=4\)}}
\put(9.5,-10.0){\makebox(0,0)[cc]{\(\mathrm{cc}=5\)}}
\put(11.5,-12.0){\makebox(0,0)[cc]{\(\mathrm{cc}=6\)}}

% main trunk edges
\multiput(2,-1)(2.0,-2.0){6}{\line(1,-1){2}}

% main trunk
\put(14.2,-12.8){\makebox(0,0)[lc]{\textbf{main trunk}}}
\put(14.3,-13.2){\makebox(0,0)[lc]{TKT b.10}}
\put(14,-13){\vector(1,-1){2}}

%*******************************************************

% metabelian immediate descendant numbers
\put(3.7,-1.2){\makebox(0,0)[lc]{\((4/1;7/5)\)}}
\put(5.7,-3.2){\makebox(0,0)[lc]{\((10/6;15/15)\)}}
\multiput(7.7,-5.2)(4.0,-4.0){2}{\makebox(0,0)[lc]{\((12/1;27/14)\)}}
\multiput(9.7,-7.2)(4.0,-4.0){2}{\makebox(0,0)[lc]{\((10/1;15/8)\)}}

%*******************************************************

% main trunk vertices
\put(2.2,-0.8){\makebox(0,0)[lc]{\(P_3=\langle 27,3\rangle=G_0^3(0,0)\)}}
\put(4.2,-2.8){\makebox(0,0)[lc]{\(P_5=\langle 243,3\rangle=G_0^{4,5}(0,0,0,0)\)}}
\put(6.2,-4.8){\makebox(0,0)[lc]{\(P_7=\langle 2187,64\rangle=G_0^{5,7}(0,0,0,0)\)}}
\put(8.2,-6.8){\makebox(0,0)[lc]{\(P_9=P_7-\#2;33=G_0^{6,9}(0,0,0,0)\)}}
\put(10.2,-8.8){\makebox(0,0)[lc]{\(P_{11}=P_9-\#2;25=G_0^{7,11}(0,0,0,0)\)}}
\put(12.2,-10.8){\makebox(0,0)[lc]{\(P_{13}=P_{11}-\#2;37=G_0^{8,13}(0,0,0,0)\)}}

%*******************************************************
\end{picture}

}

\end{figure}
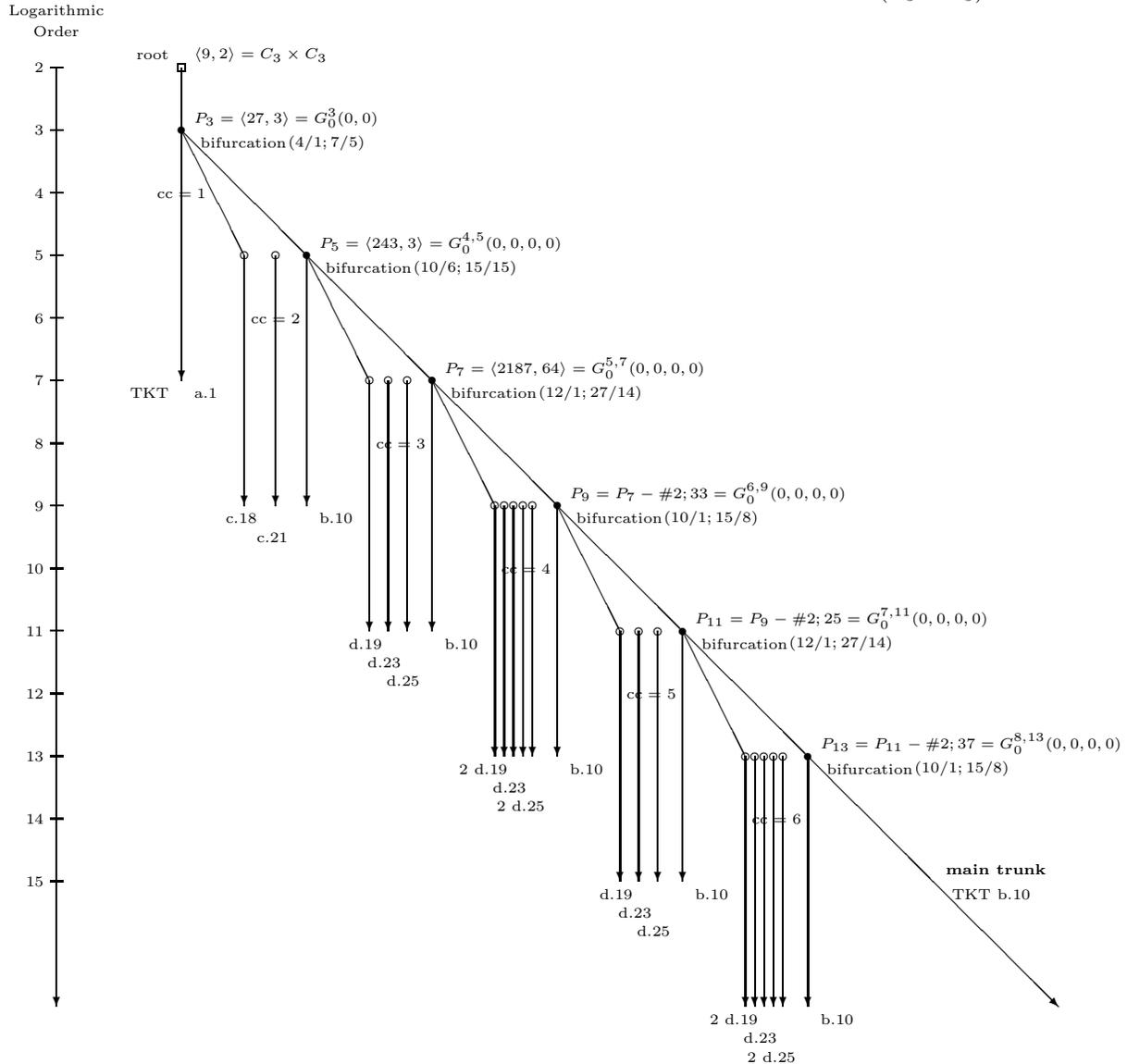

%--------------------------------------------------------------------------------

\begin{theorem}
\label{thm:MainTrunk}
\textbf{(The main trunk.)}
\begin{enumerate}
\item
In the descendant tree \(\mathcal{T}( R )\) of the abelian root \(R:=C_3\times C_3=\langle 9,2\rangle\),
there exists a unique infinite path of (reverse) directed edges \(\left(P_{2j+1}\leftarrow P_{2j+3}\right)_{j\ge 1}\)
such that, for each fixed coclass \(r=j+1\ge 2\),
every metabelian \(3\)-group \(G\) with \(G/G^\prime\simeq (3,3)\) and \(\mathrm{cc}(G)=r\)
is a proper descendant of \(P_{2j+1}\).
\item
The trailing vertex \(P_3\) is exactly the extra special Blackburn group \(G_0^3(0,0)=\langle 27,3\rangle\)
with exceptional transfer kernel type (TKT) \(\mathrm{a}.1\), \(\varkappa=(0000)\).
\item
All the other vertices \(P_{2j+1}\) with \(j\ge 2\)
share the common TKT \(\mathrm{b}.10\), \(\varkappa=(0043)\),
possess nilpotency class \(c=j+1\), coclass \(r=j\), logarithmic order \(n=c+r=2j+1\),
abelian commutator subgroup of type \(\mathrm{A}(3,c-1)\times\mathrm{A}(3,r-1)\),
%\(D:=\mathrm{A}(3,c-1)\times\mathrm{A}(3,r-1)\),
and transfer target type \(\tau=\left\lbrack\mathrm{A}(3,c),\mathrm{A}(3,r+1),1^3,1^3\right\rbrack\),
where \(r+1=c\).
%and iterated IPAD of second order
%\(\tau^{(2)}=\left\lbrack 1^2;
%(\mathrm{A}(3,c);D,\mathrm{B}(3,c-1)\times C(3))^2,
%(1^3;D,(1^3)^{12})^2\right\rbrack\),
%where
%\[
%\mathrm{B}(3,c-1):=
%\begin{cases}
%C(3^t)\times C(3^{t-1}) & \text{ if } c=2t \text{ is even,} \\
%C(3^t)\times C(3^{t-2}) & \text{ if } c=2t-1 \text{ is odd.} \\
%\end{cases}
%\]
\item
For \(j\ge 4\), periodicity of length \(2\) sets in,
\(P_{2j+1}\) has nuclear rank \(\nu=2\), relation rank \(\mu=6\), and
%\(p\)-multiplicator rank \(\mu=6\), and
immediate descendant numbers (including non-metabelian groups)
\[
(N_1/C_1,N_2/C_2)=
\begin{cases}
(21/1,151/21) & \text{ if } j\ge 4 \text{ is even,} \\
(30/1,295/37) & \text{ if } j\ge 5 \text{ is odd.} \\
\end{cases}
\]
Restricted to metabelian groups, the immediate descendant numbers are
\[
(\tilde{N}_1/\tilde{C}_1,\tilde{N}_2/\tilde{C}_2)=
\begin{cases}
(10/1,15/8) & \text{ if } j\ge 4 \text{ is even,} \\
(12/1,27/14) & \text{ if } j\ge 3 \text{ is odd.} \\
\end{cases}
\]
All immediate descendants are \(\sigma\)-groups, if \(j\ge 1\) is odd,
but only \((3/3,1/1)\), if \(j=2\), and
\((3/1,1/1)\), if \(j\ge 4\) is even.
\end{enumerate}
\end{theorem}

\begin{proof}
See the dissertation of Nebelung
\cite[p. 192]{Ne}.
\end{proof}

%\newpage

%--------------------------------------------------------------------------------

\begin{remark}
\label{rmk:MainTrunk}
Although the number of metabelian children of step sizes \(1\le s\le 2\) of the vertex \(P_7\) with \(j=3\)
fit into the periodic pattern \((\tilde{N}_1/\tilde{C}_1,\tilde{N}_2/\tilde{C}_2)=(12/1,27/14)\),
the number of all children of step sizes \(1\le s\le 2\) of \(P_7\) is bigger than usual with
\((N_1/C_1,N_2/C_2)=(33/2,453/84)\) instead of \((30/1,295/37)\).
Therefore, periodicity starts with \(j=4\) and not with \(j=3\).
\end{remark}

%--------------------------------------------------------------------------------

\begin{corollary}
\label{cor:MainTrunk}
\textbf{(All coclass trees with metabelian mainlines.)}\\
\noindent
The coclass trees of \(3\)-groups \(G\) with \(G/G^\prime\simeq (3,3)\),
whose mainlines consist of metabelian vertices,
possess the following remarkable periodicity of length \(2\),
drawn impressively in Figure
\(\ref{fig:MainTrunk}\).
\begin{enumerate}
\item
For even \(j\ge 2\), the vertex \(P_{2j+1}\) with subscript \(2j+1\ge 5\) of the main trunk
has exactly \(4\) immediate descendants of step size \(s=2\)
giving rise to coclass trees \(\mathcal{T}^{j+1}\subset\mathcal{F}(j+1)\)
whose mainline vertices are metabelian \(3\)-groups \(G\) with odd \(\mathrm{cc}(G)=j+1\) and fixed TKT,
either \(\mathrm{d}.19\), \(\varkappa=(0343)\),
or \(\mathrm{d}.23\), \(\varkappa=(0243)\),
or \(\mathrm{d}.25\), \(\varkappa=(0143)\),
or \(\mathrm{b}.10\), \(\varkappa=(0043)\),
the latter with root \(P_{2j+3}\).
\item
For odd \(j\ge 3\), the vertex \(P_{2j+1}\) with subscript \(2j+1\ge 7\) of the main trunk
has exactly \(6\) immediate descendants of step size \(s=2\)
giving rise to coclass trees \(\mathcal{T}^{j+1}\subset\mathcal{F}(j+1)\)
whose mainline vertices are metabelian \(3\)-groups \(G\) with even \(\mathrm{cc}(G)=j+1\) and fixed TKT,
either \(\mathrm{d}.19\), \(\varkappa=(0343)\), twice,
or \(\mathrm{d}.23\), \(\varkappa=(0243)\),
or \(\mathrm{d}.25\), \(\varkappa=(0143)\), twice,
or \(\mathrm{b}.10\), \(\varkappa=(0043)\),
the latter with root \(P_{2j+3}\).
\item
The unique pre-periodic exception is the vertex \(P_3\) of the main trunk,
which has exactly \(3\) immediate descendants of step size \(s=2\)
giving rise to coclass trees \(\mathcal{T}^{2}\subset\mathcal{F}(2)\)
whose mainline vertices are metabelian \(3\)-groups \(G\) with even \(\mathrm{cc}(G)=2\) and fixed TKT,
either \(\mathrm{c}.18\), \(\varkappa=(0313)\),
or \(\mathrm{c}.21\), \(\varkappa=(0231)\),
or \(\mathrm{b}.10\), \(\varkappa=(0043)\),
the latter with root \(P_{5}\).
\end{enumerate}
\end{corollary}

\begin{proof}
See the dissertation of Nebelung
\cite[\S\ 5.2, pp. 181--195]{Ne}.
\end{proof}

%\newpage

%--------------------------------------------------------------------------------

\section{Sporadic and periodic \(3\)-groups \(G\) of even coclass \(\mathrm{cc}(G)\ge 4\)}
\label{s:PeriodicSporadic4}
\noindent
Although formulated for the particular coclass \(r=4\),
all results for periodic groups and
most of the results for sporadic groups
in this section are valid for any even coclass \(r\ge 4\).
The only exception is the bigger (and thus pre-periodic) sporadic part \(\mathcal{F}_0(4)\) of the coclass forest \(\mathcal{F}(4)\),
described in Proposition
\ref{prp:SporCc4},
whereas the (co-periodic) standard case, the sporadic part \(\mathcal{F}_0(6)\) of the coclass forest \(\mathcal{F}(6)\),
is presented in Proposition
\ref{prp:SporCc6}.

%\newpage

%--------------------------------------------------------------------------------

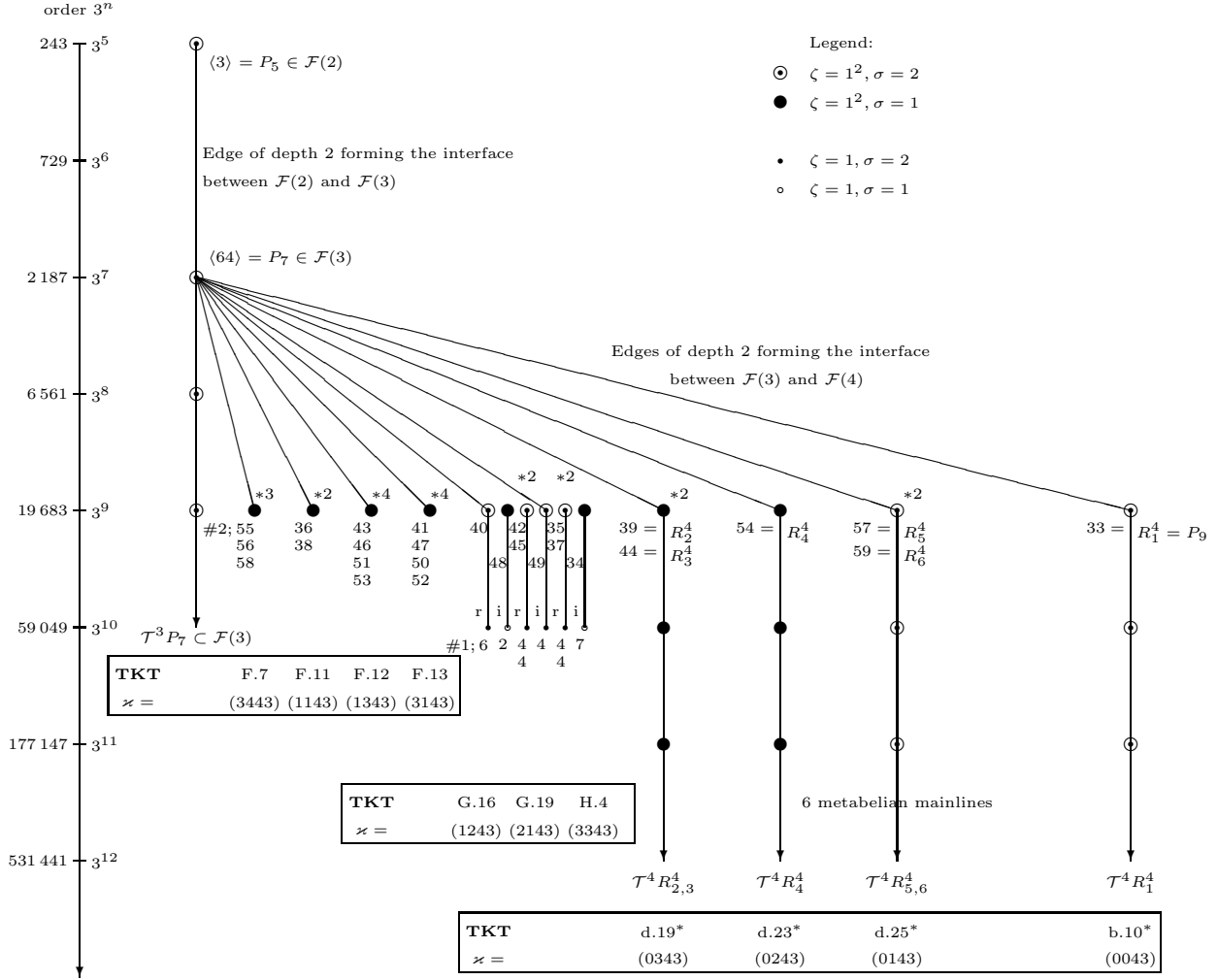
\begin{figure}[ht]
\caption{Metabelian interface between the coclass forests \(\mathcal{F}(3)\) and \(\mathcal{F}(4)\)}
\label{fig:SporCc4}

% Metabelian Interface Vertices of the Coclass Forest F(4)

{\tiny

\begin{center}

\setlength{\unitlength}{0.8cm}
\begin{picture}(19,17)(-3,-14)

% scale of orders
\put(-2,2.5){\makebox(0,0)[cb]{order \(3^n\)}}

\put(-2,2){\line(0,-1){14}}
\multiput(-2.1,2)(0,-2){8}{\line(1,0){0.2}}

\put(-2.2,2){\makebox(0,0)[rc]{\(243\)}}
\put(-1.8,2){\makebox(0,0)[lc]{\(3^5\)}}
\put(-2.2,0){\makebox(0,0)[rc]{\(729\)}}
\put(-1.8,0){\makebox(0,0)[lc]{\(3^6\)}}
\put(-2.2,-2){\makebox(0,0)[rc]{\(2\,187\)}}
\put(-1.8,-2){\makebox(0,0)[lc]{\(3^7\)}}
\put(-2.2,-4){\makebox(0,0)[rc]{\(6\,561\)}}
\put(-1.8,-4){\makebox(0,0)[lc]{\(3^8\)}}
\put(-2.2,-6){\makebox(0,0)[rc]{\(19\,683\)}}
\put(-1.8,-6){\makebox(0,0)[lc]{\(3^9\)}}
\put(-2.2,-8){\makebox(0,0)[rc]{\(59\,049\)}}
\put(-1.8,-8){\makebox(0,0)[lc]{\(3^{10}\)}}
\put(-2.2,-10){\makebox(0,0)[rc]{\(177\,147\)}}
\put(-1.8,-10){\makebox(0,0)[lc]{\(3^{11}\)}}
\put(-2.2,-12){\makebox(0,0)[rc]{\(531\,441\)}}
\put(-1.8,-12){\makebox(0,0)[lc]{\(3^{12}\)}}

\put(-2,-12){\vector(0,-1){2}}

% legend
\put(10.5,2){\makebox(0,0)[lc]{Legend:}}
\put(10,1.5){\circle{0.2}}
\put(10,1.5){\circle*{0.1}}
\put(10,1){\circle*{0.2}}
%\put(10,0.5){\circle{0.2}}
\put(10,0){\circle*{0.1}}
\put(10,-0.5){\circle{0.1}}
\put(10.5,1.5){\makebox(0,0)[lc]{\(\zeta=1^2,\sigma=2\)}}
\put(10.5,1){\makebox(0,0)[lc]{\(\zeta=1^2,\sigma=1\)}}
%\put(10.5,0.5){\makebox(0,0)[lc]{\(\zeta=1^2,\sigma=0\)}}
\put(10.5,0){\makebox(0,0)[lc]{\(\zeta=1,\sigma=2\)}}
\put(10.5,-0.5){\makebox(0,0)[lc]{\(\zeta=1,\sigma=1\)}}

% vertex of F(2)
\put(0,2){\circle{0.2}}
\put(0,2){\circle*{0.1}}

% directed edge of depth 2 from F(2) to F(3)
\put(0,2){\line(0,-1){4}}
\put(0.1,0){\makebox(0,0)[lb]{Edge of depth \(2\) forming the interface}}
\put(0.1,-0.5){\makebox(0,0)[lb]{between \(\mathcal{F}(2)\) and \(\mathcal{F}(3)\)}}

% vertex of F(3) main line
\multiput(0,-2)(0,-2){3}{\circle{0.2}}
\multiput(0,-2)(0,-2){3}{\circle*{0.1}}

% single F(3) main line
\put(0,-2){\vector(0,-1){6}}
\put(0,-8.0){\makebox(0,0)[ct]{\(\mathcal{T}^3P_7\subset\mathcal{F}(3)\)}}

% 10(+4) directed edges of depth 2 from F(3) to F(4)
\put(0,-2){\line(1,-4){1}}
\put(0,-2){\line(1,-2){2}}
\put(0,-2){\line(3,-4){3}}
\put(0,-2){\line(1,-1){4}}
\put(0,-2){\line(5,-4){5}}
\put(0,-2){\line(3,-2){6}}
\put(0,-2){\line(2,-1){8}}
\put(0,-2){\line(5,-2){10}}
\put(0,-2){\line(3,-1){12}}
\put(0,-2){\line(4,-1){16}}
\put(7.1,-3.4){\makebox(0,0)[lb]{Edges of depth \(2\) forming the interface}}
\put(8.1,-3.9){\makebox(0,0)[lb]{between \(\mathcal{F}(3)\) and \(\mathcal{F}(4)\)}}

% 7 top vertices of F(4)
\multiput(1,-6)(1,0){4}{\circle*{0.2}}

% together with children
% regular
\multiput(5,-6)(0.66,0){3}{\circle{0.2}}
\multiput(5,-6)(0.66,0){3}{\circle*{0.1}}

% irregular
\put(5.33,-6){\circle*{0.2}}
\put(5.99,-6){\circle{0.2}}
\put(5.99,-6){\circle*{0.1}}
\put(6.65,-6){\circle*{0.2}}

% together with further mainline vertices
\multiput(8,-6)(2,0){2}{\circle*{0.2}}
\multiput(8,-8)(2,0){2}{\circle*{0.2}}
\multiput(8,-10)(2,0){2}{\circle*{0.2}}

\multiput(12,-6)(0,-2){3}{\circle{0.2}}
\multiput(12,-6)(0,-2){3}{\circle*{0.1}}

\multiput(16,-6)(0,-2){3}{\circle{0.2}}
\multiput(16,-6)(0,-2){3}{\circle*{0.1}}

% four(+two) F(4) mainlines
\put(8,-6){\vector(0,-1){6}}
\put(10,-6){\vector(0,-1){6}}
\put(12,-6){\vector(0,-1){6}}
\put(16,-6){\vector(0,-1){6}}

\put(12,-11){\makebox(0,0)[cc]{\(6\) metabelian mainlines}}

\put(8,-12.2){\makebox(0,0)[ct]{\(\mathcal{T}^4R_{2,3}^4\)}}
\put(10,-12.2){\makebox(0,0)[ct]{\(\mathcal{T}^4R_4^4\)}}
\put(12,-12.2){\makebox(0,0)[ct]{\(\mathcal{T}^4R_{5,6}^4\)}}
\put(16,-12.2){\makebox(0,0)[ct]{\(\mathcal{T}^4R_1^4\)}}

% directed edges to depth 1
\multiput(5,-6)(0.33,0){6}{\line(0,-1){2}}
% sporadic vertices of depth 1 of F(4)
% regular
\multiput(5,-8)(0.66,0){3}{\circle*{0.1}}
% irregular
\put(5.33,-8){\circle{0.1}}
\put(5.99,-8){\circle*{0.1}}
\put(6.65,-8){\circle{0.1}}

% isomorphism classes and multiplicity counters
\put(0.2,1.8){\makebox(0,0)[lt]{\(\langle 3\rangle=P_5\in\mathcal{F}(2)\)}}
\put(0.2,-1.8){\makebox(0,0)[lb]{\(\langle 64\rangle=P_7\in\mathcal{F}(3)\)}}

\put(1,-5.8){\makebox(0,0)[lb]{\(\ast 3\)}}
\put(1,-6.2){\makebox(0,0)[rt]{\(\#2;55\)}}
\put(1,-6.5){\makebox(0,0)[rt]{\(56\)}}
\put(1,-6.8){\makebox(0,0)[rt]{\(58\)}}
\put(2,-5.8){\makebox(0,0)[lb]{\(\ast 2\)}}
\put(2,-6.2){\makebox(0,0)[rt]{\(36\)}}
\put(2,-6.5){\makebox(0,0)[rt]{\(38\)}}
\put(3,-5.8){\makebox(0,0)[lb]{\(\ast 4\)}}
\put(3,-6.2){\makebox(0,0)[rt]{\(43\)}}
\put(3,-6.5){\makebox(0,0)[rt]{\(46\)}}
\put(3,-6.8){\makebox(0,0)[rt]{\(51\)}}
\put(3,-7.1){\makebox(0,0)[rt]{\(53\)}}
\put(4,-5.8){\makebox(0,0)[lb]{\(\ast 4\)}}
\put(4,-6.2){\makebox(0,0)[rt]{\(41\)}}
\put(4,-6.5){\makebox(0,0)[rt]{\(47\)}}
\put(4,-6.8){\makebox(0,0)[rt]{\(50\)}}
\put(4,-7.1){\makebox(0,0)[rt]{\(52\)}}

% regular
\put(5.66,-5.5){\makebox(0,0)[cb]{\(\ast 2\)}}
\put(6.32,-5.5){\makebox(0,0)[cb]{\(\ast 2\)}}
\multiput(4.9,-7.8)(0.66,0){3}{\makebox(0,0)[rb]{r}}
\put(5.00,-6.2){\makebox(0,0)[rt]{\(40\)}}
\put(5.00,-8.2){\makebox(0,0)[rt]{\(\#1;6\)}}
\put(5.66,-6.2){\makebox(0,0)[rt]{\(42\)}}
\put(5.66,-6.5){\makebox(0,0)[rt]{\(45\)}}
\put(5.66,-8.2){\makebox(0,0)[rt]{\(4\)}}
\put(5.66,-8.5){\makebox(0,0)[rt]{\(4\)}}
\put(6.32,-6.2){\makebox(0,0)[rt]{\(35\)}}
\put(6.32,-6.5){\makebox(0,0)[rt]{\(37\)}}
\put(6.32,-8.2){\makebox(0,0)[rt]{\(4\)}}
\put(6.32,-8.5){\makebox(0,0)[rt]{\(4\)}}

% irregular
\multiput(5.23,-7.8)(0.66,0){3}{\makebox(0,0)[rb]{i}}
\put(5.33,-6.8){\makebox(0,0)[rt]{\(48\)}}
\put(5.33,-8.2){\makebox(0,0)[rt]{\(2\)}}
\put(5.99,-6.8){\makebox(0,0)[rt]{\(49\)}}
\put(5.99,-8.2){\makebox(0,0)[rt]{\(4\)}}
\put(6.65,-6.8){\makebox(0,0)[rt]{\(34\)}}
\put(6.65,-8.2){\makebox(0,0)[rt]{\(7\)}}

\put(8.1,-5.8){\makebox(0,0)[lb]{\(\ast 2\)}}
\put(7.9,-6.2){\makebox(0,0)[rt]{\(39=\)}}
\put(8.1,-6.2){\makebox(0,0)[lt]{\(R_2^4\)}}
\put(7.9,-6.6){\makebox(0,0)[rt]{\(44=\)}}
\put(8.1,-6.6){\makebox(0,0)[lt]{\(R_3^4\)}}
\put(9.9,-6.2){\makebox(0,0)[rt]{\(54=\)}}
\put(10.1,-6.2){\makebox(0,0)[lt]{\(R_4^4\)}}
\put(12.1,-5.8){\makebox(0,0)[lb]{\(\ast 2\)}}
\put(11.9,-6.2){\makebox(0,0)[rt]{\(57=\)}}
\put(12.1,-6.2){\makebox(0,0)[lt]{\(R_5^4\)}}
\put(11.9,-6.6){\makebox(0,0)[rt]{\(59=\)}}
\put(12.1,-6.6){\makebox(0,0)[lt]{\(R_6^4\)}}
\put(15.9,-6.2){\makebox(0,0)[rt]{\(33=\)}}
\put(16.1,-6.2){\makebox(0,0)[lt]{\(R_1^4=P_9\)}}

% transfer kernel types
\put(-1,-8.8){\makebox(0,0)[cc]{\textbf{TKT}}}
\put(1,-8.8){\makebox(0,0)[cc]{F.7}}
\put(2,-8.8){\makebox(0,0)[cc]{F.11}}
\put(3,-8.8){\makebox(0,0)[cc]{F.12}}
\put(4,-8.8){\makebox(0,0)[cc]{F.13}}
\put(-1,-9.3){\makebox(0,0)[cc]{\(\varkappa=\)}}
\put(1,-9.3){\makebox(0,0)[cc]{\((3443)\)}}
\put(2,-9.3){\makebox(0,0)[cc]{\((1143)\)}}
\put(3,-9.3){\makebox(0,0)[cc]{\((1343)\)}}
\put(4,-9.3){\makebox(0,0)[cc]{\((3143)\)}}
\put(-1.5,-9.5){\framebox(6,1){}}

\put(3,-11){\makebox(0,0)[cc]{\textbf{TKT}}}
\put(4.8,-11){\makebox(0,0)[cc]{G.16}}
\put(5.8,-11){\makebox(0,0)[cc]{G.19}}
\put(6.8,-11){\makebox(0,0)[cc]{H.4}}
\put(3,-11.5){\makebox(0,0)[cc]{\(\varkappa=\)}}
\put(4.8,-11.5){\makebox(0,0)[cc]{\((1243)\)}}
\put(5.8,-11.5){\makebox(0,0)[cc]{\((2143)\)}}
\put(6.8,-11.5){\makebox(0,0)[cc]{\((3343)\)}}
\put(2.5,-11.7){\framebox(5,1){}}

\put(5,-13.2){\makebox(0,0)[cc]{\textbf{TKT}}}
\put(8,-13.2){\makebox(0,0)[cc]{d.19\({}^\ast\)}}
\put(10,-13.2){\makebox(0,0)[cc]{d.23\({}^\ast\)}}
\put(12,-13.2){\makebox(0,0)[cc]{d.25\({}^\ast\)}}
\put(16,-13.2){\makebox(0,0)[cc]{b.10\({}^\ast\)}}
\put(5,-13.7){\makebox(0,0)[cc]{\(\varkappa=\)}}
\put(8,-13.7){\makebox(0,0)[cc]{\((0343)\)}}
\put(10,-13.7){\makebox(0,0)[cc]{\((0243)\)}}
\put(12,-13.7){\makebox(0,0)[cc]{\((0143)\)}}
\put(16,-13.7){\makebox(0,0)[cc]{\((0043)\)}}
\put(4.5,-13.9){\framebox(12,1){}}

\end{picture}

\end{center}

}

\end{figure}

%\newpage

%--------------------------------------------------------------------------------

\noindent
Figure
\ref{fig:SporCc4}
sketches an outline of the \textit{metabelian skeleton} of the coclass forest \(\mathcal{F}(4)\)
in its top region.
The vertices \(P_5=\langle 243,3\rangle\in\mathcal{F}(2)\) and
\(P_7=\langle 2187,64\rangle\in\mathcal{F}(3)\),
with the crucial bifucation from \(\mathcal{F}(3)\) to \(\mathcal{F}(4)\),
belong to the infinite main trunk (\S\
\ref{s:MainTrunk}).

%\newpage

%--------------------------------------------------------------------------------

\begin{proposition}
\label{prp:SporCc6}
\textbf{(Co-periodic standard case.)} \\
The sporadic part \(\mathcal{F}_0(6)\) of the coclass-\(6\) forest \(\mathcal{F}(6)\) consists of
\begin{itemize}
\item
\(13\) \((3+2+4+4)\) isolated metabelian vertices of order \(3^{13}\) with types \(\mathrm{F}.7\), \(\mathrm{F}.11\), \(\mathrm{F}.12\), \(\mathrm{F}.13\),
\item
\(8\) \((2+3+3)\) metabelian roots of finite trees with types \(\mathrm{G}.16\), \(\mathrm{G}.19\), \(\mathrm{H}.4\),
together with a metabelian child having a GI-action, which is unique for each root, and
\(22\) metabelian and \(38\) non-metabelian children without GI-action, all with depth \(\mathrm{dp}=1\) and \(\mathrm{lo}=14\),
\item
\(66\) \((32+16+18)\) isolated vertices with \(\mathrm{dl}=3\) and types \(\mathrm{d}.19\), \(\mathrm{d}.23\), \(\mathrm{d}.25\),
\item
\(179\) isolated vertices with \(\mathrm{dl}=3\) and type \(\mathrm{b}.10\),
\item
\(23\) capable vertices with \(\mathrm{dl}=3\) and type \(\mathrm{b}.10\), \\
whose children do not belong to \(\mathcal{F}_0(6)\), by definition.
\end{itemize}
The action flag of all metabelian top vertices with depth \(\mathrm{dp}=0\) is \(\sigma\ge 1\).
The value \(\sigma=2\) only occurs for all vertices with type \(\mathrm{b}.10\), \(\mathrm{d}.25\), \(\mathrm{G}.19\),
and certain vertices with type \(\mathrm{G}.16\), \(\mathrm{H}.4\),
but never for type \(\mathrm{d}.19\), \(\mathrm{d}.23\), \(\mathrm{F}.7\), \(\mathrm{F}.11\), \(\mathrm{F}.12\), \(\mathrm{F}.13\).
Exactly the isolated vertices with depth \(\mathrm{dp}=0\) have an RI-action. 

Together with the \(6\) metabelian roots \(R_i^6\), \(1\le i\le 6\),  of coclass-\(6\) trees,
the \(13+8+66+179+23=21+268\) top vertices of depth \(\mathrm{dp}=0\) of \(\mathcal{F}_0(6)\)
are exactly the \(N_2=295\) children of step size \(s=2\) of the main trunk vertex \(P_{11}=\langle 2187,64\rangle-\#2;33-\#2;25\),
and the \(6+8+23\) capable vertices among them correspond to the invariant \(C_2=37\) of \(P_{11}\).
\end{proposition}

%--------------------------------------------------------------------------------

\begin{proposition}
\label{prp:SporCc4}
\textbf{(Pre-periodic exception.)} \\
The constitution of the sporadic part \(\mathcal{F}_0(4)\) of the coclass-\(4\) forest \(\mathcal{F}(4)\)
with respect to the \(21\) metabelian top vertices and their \(68\) children (here with order \(3^{9}\), resp. \(\mathrm{lo}=10\))
is the same as described for \(\mathcal{F}_0(6)\) in Proposition
\ref{prp:SporCc6},
but the number of non-metabelian top vertices of depth \(\mathrm{dp}=0\) is bigger, namely
\begin{itemize}
\item
\(88\) \((40+22+26)\) isolated vertices with \(\mathrm{dl}=3\) and types \(\mathrm{d}.19\), \(\mathrm{d}.23\), \(\mathrm{d}.25\),
\item
\(12\) \((8+2+2)\) capable vertices with \(\mathrm{dl}=3\) and types \(\mathrm{d}.19\), \(\mathrm{d}.23\), \(\mathrm{d}.25\), \\
whose children do not belong to \(\mathcal{F}_0(4)\), by definition,
\item
\(268\) isolated vertices with \(\mathrm{dl}=3\) and type \(\mathrm{b}.10\),
\item
\(58\) capable vertices with \(\mathrm{dl}=3\) and type \(\mathrm{b}.10\), \\
whose children do not belong to \(\mathcal{F}_0(4)\), by definition.
\end{itemize}
The distribution of the action flags \(\sigma\) is the same as in Proposition
\ref{prp:SporCc6},
but the total census of top vertices is considerably bigger: 

Together with the \(6\) metabelian roots \(R_i^4\), \(1\le i\le 6\), of coclass-\(4\) trees,
the \(13+8+88+12+268+58=21+426\) top vertices of depth \(\mathrm{dp}=0\) of \(\mathcal{F}_0(4)\)
are exactly the \(N_2=453\) children of step size \(s=2\) of the main trunk vertex \(P_7=\langle 2187,64\rangle\),
and the \(6+8+12+58\) capable vertices among them correspond to the invariant \(C_2=84\) of \(P_7\).
\end{proposition}

%\newpage

%--------------------------------------------------------------------------------

\renewcommand{\arraystretch}{1.2}

\begin{table}[ht]
%\caption{Invariants of the fork \(P_7=\langle 2187,64\rangle\) and the sporadic part \(\mathcal{F}_0(4)\) of \(\mathcal{F}(4)\)}
\caption{Data for sporadic \(3\)-groups \(G\) with \(9\le n=\mathrm{lo}(G)\le 10\) in the forest \(\mathcal{F}(4)\)}
\label{tbl:SporCc4}
\begin{center}
\begin{tabular}{|c|c|l||c|c|c||c|c||c|c|l|c||c|c|}
\hline
 \(\#\) & \(m,n\)  & \(\rho;\alpha,\beta,\gamma,\delta\) & dp & dl & \(\zeta\) & \(\mu\) & \(\nu\) & \(\tau(1)\) & \(\tau_2\) & Type                   & \(\varkappa\) & \(\sigma\) & \(\#\mathrm{Aut}\)  \\
\hline
  \(1\) & \(5,7\)  & \(0;0,0,0,0\) (\(P_7\))             & 0  & 2  & \(1^2\)   & \(6\)   & \(4\)   & \(2^2\)     & \(21^3\)   & \(\mathrm{b}.10^\ast\) & \((0043)\)    & \(2^\ast\) & \(2^3\cdot 3^{10}\) \\
\hline
  \(1\) & \(6,9\)  & \(0;0,0,0,0\) (\(P_9\))             & 0  & 2  & \(1^2\)   & \(6\)   & \(2\)   & \(32\)      & \(2^31\)   & \(\mathrm{b}.10^\ast\) & \((0043)\)    & \(2\)      & \(2^3\cdot 3^{14}\) \\
  \(2\) & \(6,9\)  & \(0;0,\pm 1,0,1\)                   & 0  & 2  & \(1^2\)   & \(5\)   & \(1\)   & \(32\)      & \(2^31\)   & \(\mathrm{d}.19^\ast\) & \((0343)\)    & \(1\)      & \(2\cdot 3^{14}\)   \\
  \(1\) & \(6,9\)  & \(0;0,0,0,1\)                       & 0  & 2  & \(1^2\)   & \(5\)   & \(1\)   & \(32\)      & \(2^31\)   & \(\mathrm{d}.23^\ast\) & \((0243)\)    & \(1\)      & \(2\cdot 3^{14}\)   \\
  \(2\) & \(6,9\)  & \(0;0,\pm 1,0,0\)                   & 0  & 2  & \(1^2\)   & \(5\)   & \(1\)   & \(32\)      & \(2^31\)   & \(\mathrm{d}.25^\ast\) & \((0143)\)    & \(2\)      & \(2^2\cdot 3^{14}\) \\
  \(2\) & \(6,9\)  & \(0;-1,\pm(1,1),1\)                 & 0  & 2  & \(1^2\)   & \(4\)   & \(0\)   & \(32\)      & \(2^31\)   & \(\mathrm{F}.7\)       & \((3443)\)    & \(1^\ast\) & \(2^2\cdot 3^{14}\) \\
  \(1\) & \(6,9\)  & \(0;1,1,-1,1\)                      & 0  & 2  & \(1^2\)   & \(4\)   & \(0\)   & \(32\)      & \(2^31\)   & \(\mathrm{F}.7\)       & \((3443)\)    & \(1^\ast\) & \(2\cdot 3^{14}\)   \\
  \(2\) & \(6,9\)  & \(0;1,\pm 1,0,0\)                   & 0  & 2  & \(1^2\)   & \(4\)   & \(0\)   & \(32\)      & \(2^31\)   & \(\mathrm{F}.11\)      & \((1143)\)    & \(1^\ast\) & \(2\cdot 3^{14}\)   \\
  \(4\) & \(6,9\)  & \(0;\pm 1,0,\pm 1,1\)               & 0  & 2  & \(1^2\)   & \(4\)   & \(0\)   & \(32\)      & \(2^31\)   & \(\mathrm{F}.12\)      & \((1343)\)    & \(1^\ast\) & \(2\cdot 3^{14}\)   \\
  \(4\) & \(6,9\)  & \(0;1,\pm 1,\pm 1,0\)               & 0  & 2  & \(1^2\)   & \(4\)   & \(0\)   & \(32\)      & \(2^31\)   & \(\mathrm{F}.13\)      & \((3143)\)    & \(1^\ast\) & \(2\cdot 3^{14}\)   \\
  \(1\) & \(6,9\)  & \(0;1,0,0,1\)                       & 0  & 2  & \(1^2\)   & \(5\)   & \(1\)   & \(32\)      & \(2^31\)   & \(\mathrm{G}.16\)      & \((1243)\)    & \(2\)      & \(2^2\cdot 3^{14}\) \\
  \(1\) & \(6,9\)  & \(0;-1,0,0,1\)                      & 0  & 2  & \(1^2\)   & \(5\)   & \(1\)   & \(32\)      & \(2^31\)   & \(\mathrm{G}.16\)      & \((1243)\)    & \(1\)      & \(2^2\cdot 3^{14}\) \\
  \(2\) & \(6,9\)  & \(0;0,\pm(1,1),0\)                  & 0  & 2  & \(1^2\)   & \(5\)   & \(1\)   & \(32\)      & \(2^31\)   & \(\mathrm{G}.19\)      & \((2143)\)    & \(2\)      & \(2^3\cdot 3^{14}\) \\
  \(1\) & \(6,9\)  & \(0;0,1,-1,0\)                      & 0  & 2  & \(1^2\)   & \(5\)   & \(1\)   & \(32\)      & \(2^31\)   & \(\mathrm{G}.19\)      & \((2143)\)    & \(2\)      & \(2^2\cdot 3^{14}\) \\
  \(2\) & \(6,9\)  & \(0;1,\pm(1,1),1\)                  & 0  & 2  & \(1^2\)   & \(5\)   & \(1\)   & \(32\)      & \(2^31\)   & \(\mathrm{H}.4\)       & \((3343)\)    & \(2\)      & \(2^2\cdot 3^{14}\) \\
  \(1\) & \(6,9\)  & \(0;-1,1,-1,1\)                     & 0  & 2  & \(1^2\)   & \(5\)   & \(1\)   & \(32\)      & \(2^31\)   & \(\mathrm{H}.4\)       & \((3343)\)    & \(1\)      & \(2\cdot 3^{14}\)   \\
\hline
  \(1\) & \(7,10\) & \(1;1,0,-1,1\)                      & 1  & 2  & \(1\)     & \(4\)   & \(0\)   & \(32\)      & \(32^21\)  & \(\mathrm{G}.16\)r     & \((1243)\)    & \(2^\ast\) & \(2^2\cdot 3^{16}\) \\
  \(1\) & \(7,10\) & \(-1;-1,0,1,1\)                     & 1  & 2  & \(1\)     & \(4\)   & \(0\)   & \(32\)      & \(2^4\)    & \(\mathrm{G}.16\)i     & \((1243)\)    & \(1^\ast\) & \(2^2\cdot 3^{16}\) \\
  \(2\) & \(7,10\) & \(1;0,\pm 1,0,0\)                   & 1  & 2  & \(1\)     & \(4\)   & \(0\)   & \(32\)      & \(32^21\)  & \(\mathrm{G}.19\)r     & \((2143)\)    & \(2^\ast\) & \(2^3\cdot 3^{16}\) \\
  \(1\) & \(7,10\) & \(-1;0,1,0,0\)                      & 1  & 2  & \(1\)     & \(4\)   & \(0\)   & \(32\)      & \(2^4\)    & \(\mathrm{G}.19\)i     & \((2143)\)    & \(2^\ast\) & \(2^2\cdot 3^{16}\) \\
  \(2\) & \(7,10\) & \(1;1,\pm 1,-1,1\)                  & 1  & 2  & \(1\)     & \(4\)   & \(0\)   & \(32\)      & \(32^21\)  & \(\mathrm{H}.4\)r      & \((3343)\)    & \(2^\ast\) & \(2^2\cdot 3^{16}\) \\
  \(1\) & \(7,10\) & \(-1;-1,1,1,1\)                     & 1  & 2  & \(1\)     & \(4\)   & \(0\)   & \(32\)      & \(2^4\)    & \(\mathrm{H}.4\)i      & \((3343)\)    & \(1^\ast\) & \(2\cdot 3^{16}\)   \\
\hline
 \(12\) & \(7,10\) &                                     & 1  & 2  & \(1\)     & \(4\)   & \(0\)   & \(32\)      &            &  G or H                &               & \(0\)      & \(2\cdot 3^{16}\)   \\
 \(10\) & \(7,10\) &                                     & 1  & 2  & \(1\)     & \(4\)   & \(0\)   & \(32\)      &            &  G or H                &               & \(0\)      & \(3^{16}\)          \\
\hline
  \(8\) & \(7,10\) &                                     & 1  & 3  & \(1\)     & \(4\)   & \(0\)   & \(32\)      & \(2^31\)   &  G or H                &               & \(0\)      & \(2\cdot 3^{15}\)   \\
 \(20\) & \(7,10\) &                                     & 1  & 3  & \(1\)     & \(4\)   & \(0\)   & \(32\)      & \(2^31\)   &  G or H                &               & \(0\)      & \(3^{15}\)          \\
  \(4\) & \(7,10\) &                                     & 1  & 3  & \(1\)     & \(4\)   & \(0\)   & \(32\)      & \(2^31\)   & \(\mathrm{G}.19\)      & \((2143)\)    & \(0\)      & \(2\cdot 3^{14}\)   \\
  \(6\) & \(7,10\) &                                     & 1  & 3  & \(1\)     & \(4\)   & \(0\)   & \(32\)      & \(2^31\)   &  G or H                &               & \(0\)      & \(3^{14}\)          \\
\hline
\end{tabular}
\end{center}
\end{table}

%\newpage

%--------------------------------------------------------------------------------

\begin{theorem}
\label{thm:SporCc4}
The coclass-\(r\) forest \(\mathcal{F}(r)\) with any even \(r\ge 4\) is the disjoint union of
its finite sporadic part \(\mathcal{F}_0(r)\) with total information content
\begin{equation}
\label{eqn:SporadicIC4}
s=\#\mathcal{F}_0(r)=
\begin{cases}
515 & \text{ if } r=4, \\
357 & \text{ if } r\ge 6, \text{ even},
\end{cases}
\end{equation}
and \(t=6\) infinite coclass-\(r\) trees \(\mathcal{T}^r(R_i^r)\)
with roots \(R_i^r:=P_{2r-1}-\#2;n_i\), where
\[
(n_i)_{1\le i\le 6}=
\begin{cases}
(33,39,44,54,57,59)\text{ for } r=4, \\
(37,43,48,58,61,63)\text{ for } r=6.
\end{cases}
\]
The algebraic invariants for groups with positive action flag \(\sigma\ge 1\),
and in cumulative form for \(\sigma=0\), are given for \(r=4\) in Table
\ref{tbl:SporCc4},
where the parent vertex \(P_{2r-1}=P_7\) on the main trunk is also included,
but the \(426\) non-metabelian top vertices of depth \(\mathrm{dp}=0\) are excluded.
\end{theorem}

%--------------------------------------------------------------------------------

\begin{proof}
(of Propositions
\ref{prp:SporCc6},
\ref{prp:SporCc4},
and Theorem
\ref{thm:SporCc4})
We have computed the sporadic parts \(\mathcal{F}_0(r)\) of coclass forests \(\mathcal{F}(r)\)
with even \(r\ge 4\) up to \(r\le 20\) by means of MAGMA
\cite{MAGMA}.
Except for the differences pointed out in the Propositions
\ref{prp:SporCc6}
and
\ref{prp:SporCc4},
they all share a common graph theoretic structure with \(\mathcal{F}_0(4)\).
The forest \(\mathcal{F}(4)\) contains \(6\) roots of infinite coclass trees with metabelian mainlines
(a unique root \(R_1^4\) of type \(\mathrm{b}\) and five roots \(R_2^4,\ldots,R_6^4\) of type \(\mathrm{d}\)), namely \\
\(R_1^4=P_9=G^{6,9}_0(0,0,0,0)\simeq P_7-\#2;33\), \quad \(R_2^4=G^{6,9}_0(0,-1,0,1)\simeq P_7-\#2;39\), \\
\(R_3^4=G^{6,9}_0(0,1,0,1)\simeq P_7-\#2;44\), \quad \(R_4^4=G^{6,9}_0(0,0,0,1)\simeq P_7-\#2;54\), \\
\(R_5^4=G^{6,9}_0(0,1,0,0)\simeq P_7-\#2;57\), \quad \(R_6^4=G^{6,9}_0(0,-1,0,0)\simeq P_7-\#2;59\), \\
which give rise to the periodic part of \(\mathcal{F}(4)\),
and \(51\) sporadic metabelian groups of type \(\mathrm{F}\), \(\mathrm{G}\) or \(\mathrm{H}\).
Among the groups of the sporadic part \(\mathcal{F}_0(4)\), there are
\(13\) isolated metabelian vertices with type \(\mathrm{F}\),
and \(8\) metabelian roots of finite trees with type \(\mathrm{G}\) or \(\mathrm{H}\) and tree depth \(1\),
each with a \textit{unique metabelian} child having \(\sigma\ge 1\).
The other \(60=8+6+5+5+8+8+8+12\) children with \(\sigma=0\),
of which \(22=3+2+2+2+3+3+3+4\) are metabelian
and \(38=5+4+3+3+5+5+5+8\) have derived length \(3\),
are \textit{omitted} in the forest diagram, Figure
\ref{fig:SporCc4}.
Additionally, \(\mathcal{F}_0(4)\), respectively \(\mathcal{F}_0(6)\),
contains \(426\), respectively \(268\), non-metabelian top vertices,
which gives a total information content \(s=\#\mathcal{F}_0(r)\)
of \(515=426+21+68\), respectively \(357=268+21+68\), representatives.
The difference is an excess of \(158=515-357\) vertices in \(\mathcal{F}_0(6)\).

The metabelian skeleton of both, \(\mathcal{F}_0(4)\) and \(\mathcal{F}_0(6)\),
consists of \(51=13+(8+8+22)\) vertices.
The results for metabelian groups are in accordance
with the fourth tree diagram \(e\ge 5\), \(e\equiv 1\pmod{2}\), in
\cite[fourth double page between pp. 191--192]{Ne}.
The metabelian groups in Table
\ref{tbl:SporCc4}
correpond to the representatives of isomorphism classes in
\cite[pp. 36--38 and 42--45]{Ne2}.
\end{proof}

%\newpage

%--------------------------------------------------------------------------------

\subsection{The unique mainline of type \(\mathrm{b}.10^\ast\) for even coclass \(r\ge 4\)}
\label{ss:b10Cc4}

\begin{proposition}
\label{prp:TKTb10TreeCc4}
\textbf{(Periodicity and descendant numbers.)} \\
The branches \(\mathcal{B}(i)\), \(i\ge n_\ast=9\),
of the coclass-\(4\) tree \(\mathcal{T}^4(\langle 2187,64\rangle-\#2;33)\)
with mainline vertices of transfer kernel type \(\mathrm{b}.10^\ast\), \(\varkappa\sim (0043)\),
are periodic with pre-period length \(\ell_\ast=1\) and with primitive period length \(\ell=2\),
that is, \(\mathcal{B}(i+2)\simeq\mathcal{B}(i)\) are isomorphic as digraphs,
for all \(i\ge p_\ast=n_\ast+\ell_\ast=10\).

The graph theoretic structure of the tree is determined uniquely by the numbers
\(N_1\) of immediate descendants and \(C_1\) of capable immediate descendants
of the mainline vertices \(m_n\) with logarithmic order \(n=\mathrm{lo}(m_n)\ge n_\ast=9\): \\
\((N_1,C_1)=(21,1)\) for the root \(m_9\) with \(n=9\), \\
\((N_1,C_1)=(30,1)\) for all mainline vertices \(m_n\) with even logarithmic order \(n\ge 10\), \\
\((N_1,C_1)=(27,1)\) for all mainline vertices \(m_n\) with odd logarithmic order \(n\ge 11\).
\end{proposition}

%--------------------------------------------------------------------------------

\begin{proof}
(of Proposition
\ref{prp:TKTb10TreeCc4})
The statements concerning the numbers \(N_1(m_n)\) of immediate descendants
of the mainline vertices \(m_n\) with \(n\ge n_\ast=9\)
have been obtained by direct computation with MAGMA
\cite{MAGMA},
where the \(p\)-group generation algorithm by Newman and O'Brien
\cite{Nm2,Ob,HEO}
is implemented.
In detail, we proved that there are \\
\(4\), resp. \(6\), metabelian vertices with bicyclic centre \(\zeta=1^2\), resp. cyclic centre \(\zeta=1\), and  \\
\(5\), resp. \(6\), non-metabelian vertices with \(\zeta=1^2\), resp. \(\zeta=1\),  \\
together \(21\) vertices (\(10\) of them metabelian) in the pre-periodic branch \(\mathcal{B}(9)\), \\
and the primitive period \((\mathcal{B}(10),\mathcal{B}(11))\) of length \(\ell=2\) consists of \\
\(6\), resp. \(6\), metabelian vertices with \(\zeta=1^2\), resp. \(\zeta=1\), and \\
\(9\), resp. \(9\), non-metabelian vertices with \(\zeta=1^2\), resp. \(\zeta=1\), \\
together \(30\) vertices (\(12\) of them metabelian) in branch \(\mathcal{B}(10)\), and \\
\(4\), resp. \(8\), metabelian vertices with \(\zeta=1^2\), resp. \(\zeta=1\), and \\
\(5\), resp. \(10\), non-metabelian vertices with \(\zeta=1^2\), resp. \(\zeta=1\), \\
together \(27\) vertices (\(12\) of them metabelian) in branch \(\mathcal{B}(11)\).

The results concerning the metabelian skeleton confirm the corresponding statements in the dissertation of Nebelung
\cite[Thm. 5.1.16, pp. 178--179, and the fourth Figure, \(e\ge 5\), \(e\equiv 1\pmod{2}\), on the double page between pp. 191--192]{Ne}.
The tree \(\mathcal{T}^4{P_9}\) corresponds to
the infinite metabelian pro-\(3\) group \(S_{3,1}\) in
\cite[Cnj. 15 (b), p. 116]{Ek}.
Although every branch contains \(12\) metabelian vertices,
the primitive period length is \(\ell=2\) rather than \(\ell=1\),
even for the metabelian skeleton,
since the constitution \(12=6+6\) of branch \(\mathcal{B}(10)\)
is different from \(12=4+8\) for branch \(\mathcal{B}(11)\),
as proved above.

The claim of the virtual periodicity of branches
has been proved generally for any coclass tree by du Sautoy
\cite{dS},
and independently by Eick and Leedham-Green
\cite{EkLg}.
Here, the strict periodicity was confirmed by computation up to branch \(\mathcal{B}(31)\)
and undoubtedly sets in at \(p_\ast=10\).
\end{proof}

%\newpage

%--------------------------------------------------------------------------------

\begin{figure}[hb]
\caption{The unique coclass-\(4\) tree \(\mathcal{T}^4{P_9}\) with mainline of type \(\mathrm{b}.10^\ast\)}
\label{fig:TKTb10TreeCc4}

\input{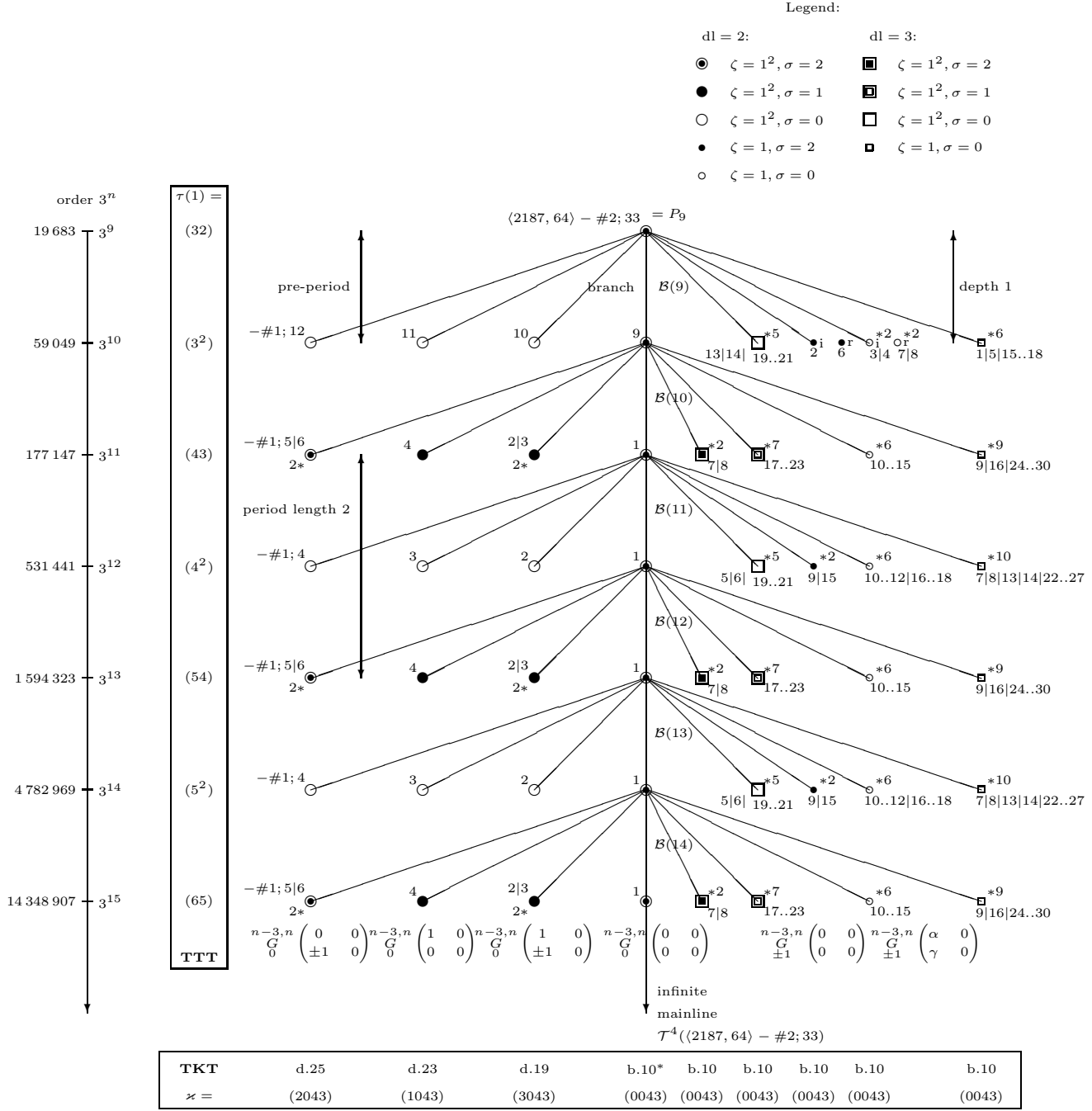}

\end{figure}

%--------------------------------------------------------------------------------

\begin{theorem}
\label{thm:TKTb10TreeCc4}
\textbf{(Graph theoretic and algebraic invariants.)} \\
The coclass-\(4\) tree \(\mathcal{T}:=\mathcal{T}^4{P_9}\) of finite \(3\)-groups \(G\) with coclass \(\mathrm{cc}(G)=4\)
which arises from the metabelian root \(P_9:=\langle 2187,64\rangle-\#2;33\)
has the following graph theoretic properties.
\begin{enumerate}
\item
The pre-period \((\mathcal{B}(9))\) of length \(\ell_\ast=1\) is irregular.
\item
The cardinality of the irregular branch is \(\#\mathcal{B}(9)=21\).
\item
The branches \(\mathcal{B}(i)\), \(i\ge p_\ast=10\), are periodic with
primitive period \((\mathcal{B}(10),\mathcal{B}(11))\) of length \(\ell=2\).
\item
The cardinalities of the regular branches are
\(\#\mathcal{B}(10)=30\) and \(\#\mathcal{B}(11)=27\).
\item
Depth, width, and information content of the tree are given by
\begin{equation}
\label{eqn:TKTb10TreeCc4}
\mathrm{dp}(\mathcal{T}^4{P_9})=1, \quad \mathrm{wd}(\mathcal{T}^4{P_9})=30, \quad \text{ and } \quad \mathrm{IC}(\mathcal{T}^4{P_9})=78.
\end{equation}
\end{enumerate}
The algebraic invariants of the groups represented by vertices forming
the pre-period \((\mathcal{B}(9))\) and the primitive period \((\mathcal{B}(10),\mathcal{B}(11))\) of the tree
are given in Table
\ref{tbl:TKTb10TreeCc4}.
The leading six branches \(\mathcal{B}(9),\ldots,\mathcal{B}(14)\) are drawn in Figure
\ref{fig:TKTb10TreeCc4}.
\end{theorem}

%\newpage

%--------------------------------------------------------------------------------

\renewcommand{\arraystretch}{1.2}

\begin{table}[ht]
%\caption{Root, pre-period \((\mathcal{B}(9))\), and primitive period \((\mathcal{B}(10),\mathcal{B}(11))\) of \(\mathcal{T}^4{P_9}\)}
\caption{Data for \(3\)-groups \(G\) with \(9\le n=\mathrm{lo}(G)\le 12\) of the coclass tree \(\mathcal{T}^4{P_9}\)}
\label{tbl:TKTb10TreeCc4}
\begin{center}
\begin{tabular}{|c|c|l||c|c|c||c|c||c|c|l|c||c|c|}
\hline
 \(\#\) & \(m,n\) & \(\rho;\alpha,\beta,\gamma,\delta\) & dp & dl & \(\zeta\) & \(\mu\) & \(\nu\) & \(\tau(1)\) & \(\tau_2\) & Type                   & \(\varkappa\) & \(\sigma\) & \(\#\mathrm{Aut}\)  \\
\hline
  \(1\) & 6,9     & \(0;0,0,0,0\)                       & 0  & 2  & \(1^2\)   & \(6\)   & \(2\)   & \(32\)      & \(2^31\)   & \(\mathrm{b}.10^\ast\) & \((0043)\)    & \(2\)      & \(2^3\cdot 3^{14}\) \\
\hline
  \(1\) & 7,10    & \(0;0,0,0,0\)                       & 0  & 2  & \(1^2\)   & \(6\)   & \(1\)   & \(3^2\)     & \(32^21\)  & \(\mathrm{b}.10^\ast\) & \((0043)\)    & \(2\)      & \(2^2\cdot 3^{16}\) \\
  \(1\) & 7,10    & \(0;1,0,1,0\)                       & 1  & 2  & \(1^2\)   & \(5\)   & \(0\)   & \(3^2\)     & \(32^21\)  & \(\mathrm{d}.19\)      & \((3043)\)    & \(0\)      & \(3^{16}\)          \\
  \(1\) & 7,10    & \(0;1,0,0,0\)                       & 1  & 2  & \(1^2\)   & \(5\)   & \(0\)   & \(3^2\)     & \(32^21\)  & \(\mathrm{d}.23\)      & \((1043)\)    & \(0\)      & \(2\cdot 3^{16}\)   \\
  \(1\) & 7,10    & \(0;0,0,1,0\)                       & 1  & 2  & \(1^2\)   & \(5\)   & \(0\)   & \(3^2\)     & \(32^21\)  & \(\mathrm{d}.25\)      & \((2043)\)    & \(0\)      & \(2\cdot 3^{16}\)   \\
  \(1\) & 7,10    & \(1;0,0,0,0\)                       & 1  & 2  & \(1\)     & \(5\)   & \(0\)   & \(32\)      & \(32^21\)  & \(\mathrm{b}.10\)r     & \((0043)\)    & \(2\)      & \(2^3\cdot 3^{16}\) \\
  \(2\) & 7,10    & \(1;\alpha,0,\gamma,0\)             & 1  & 2  & \(1\)     & \(5\)   & \(0\)   & \(32\)      & \(32^21\)  & \(\mathrm{b}.10\)r     & \((0043)\)    & \(0\)      & \(2\cdot 3^{16}\)   \\
  \(1\) & 7,10    & \(-1;0,0,0,0\)                      & 1  & 2  & \(1\)     & \(5\)   & \(0\)   & \(32\)      & \(2^4\)    & \(\mathrm{b}.10\)i     & \((0043)\)    & \(2\)      & \(2^3\cdot 3^{16}\) \\
  \(2\) & 7,10    & \(-1;\alpha,0,\gamma,0\)            & 1  & 2  & \(1\)     & \(5\)   & \(0\)   & \(32\)      & \(2^4\)    & \(\mathrm{b}.10\)i     & \((0043)\)    & \(0\)      & \(2\cdot 3^{16}\)   \\
  \(3\) & 7,10    &                                     & 1  & 3  & \(1^2\)   & \(5\)   & \(0\)   & \(32\)      & \(2^31\)   & \(\mathrm{b}.10\)      & \((0043)\)    & \(0\)      & \(3^{15}\)          \\
  \(2\) & 7,10    &                                     & 1  & 3  & \(1^2\)   & \(5\)   & \(0\)   & \(32\)      & \(2^31\)   & \(\mathrm{b}.10\)      & \((0043)\)    & \(0\)      & \(2\cdot 3^{14}\)   \\
  \(2\) & 7,10    &                                     & 1  & 3  & \(1\)     & \(5\)   & \(0\)   & \(32\)      & \(2^31\)   & \(\mathrm{b}.10\)      & \((0043)\)    & \(0\)      & \(2\cdot 3^{15}\)   \\
  \(2\) & 7,10    &                                     & 1  & 3  & \(1\)     & \(5\)   & \(0\)   & \(32\)      & \(2^31\)   & \(\mathrm{b}.10\)      & \((0043)\)    & \(0\)      & \(3^{15}\)          \\
  \(2\) & 7,10    &                                     & 1  & 3  & \(1\)     & \(5\)   & \(0\)   & \(32\)      & \(2^31\)   & \(\mathrm{b}.10\)      & \((0043)\)    & \(0\)      & \(2\cdot 3^{14}\)   \\
\hline
  \(1\) & 8,11    & \(0;0,0,0,0\)                       & 0  & 2  & \(1^2\)   & \(6\)   & \(1\)   & \(43\)      & \(3^221\)  & \(\mathrm{b}.10^\ast\) & \((0043)\)    & \(2\)      & \(2^2\cdot 3^{18}\) \\
  \(2\) & 8,11    & \(0;1,0,\pm 1,0\)                   & 1  & 2  & \(1^2\)   & \(5\)   & \(0\)   & \(43\)      & \(3^221\)  & \(\mathrm{d}.19\)      & \((3043)\)    & \(1\)      & \(2\cdot 3^{18}\)   \\
  \(1\) & 8,11    & \(0;1,0,0,0\)                       & 1  & 2  & \(1^2\)   & \(5\)   & \(0\)   & \(43\)      & \(3^221\)  & \(\mathrm{d}.23\)      & \((1043)\)    & \(1\)      & \(2\cdot 3^{18}\)   \\
  \(2\) & 8,11    & \(0;0,0,\pm 1,0\)                   & 1  & 2  & \(1^2\)   & \(5\)   & \(0\)   & \(43\)      & \(3^221\)  & \(\mathrm{d}.25\)      & \((2043)\)    & \(2\)      & \(2^2\cdot 3^{18}\) \\
  \(3\) & 8,11    & \(\pm 1;\alpha,0,\gamma,0\)         & 1  & 2  & \(1\)     & \(5\)   & \(0\)   & \(3^2\)     & \(3^221\)  & \(\mathrm{b}.10\)      & \((0043)\)    & \(0\)      & \(2\cdot 3^{18}\)   \\
  \(3\) & 8,11    & \(\pm 1;\alpha,0,\gamma,0\)         & 1  & 2  & \(1\)     & \(5\)   & \(0\)   & \(3^2\)     & \(3^221\)  & \(\mathrm{b}.10\)      & \((0043)\)    & \(0\)      & \(3^{18}\)          \\
  \(6\) & 8,11    &                                     & 1  & 3  & \(1^2\)   & \(5\)   & \(0\)   & \(3^2\)     & \(32^21\)  & \(\mathrm{b}.10\)      & \((0043)\)    & \(1\)      & \(2\cdot 3^{17}\)   \\
  \(2\) & 8,11    &                                     & 1  & 3  & \(1^2\)   & \(5\)   & \(0\)   & \(3^2\)     & \(32^21\)  & \(\mathrm{b}.10\)      & \((0043)\)    & \(2\)      & \(2^2\cdot 3^{16}\) \\
  \(1\) & 8,11    &                                     & 1  & 3  & \(1^2\)   & \(5\)   & \(0\)   & \(3^2\)     & \(32^21\)  & \(\mathrm{b}.10\)      & \((0043)\)    & \(1\)      & \(2\cdot 3^{16}\)   \\
  \(6\) & 8,11    &                                     & 1  & 3  & \(1\)     & \(5\)   & \(0\)   & \(3^2\)     & \(32^21\)  & \(\mathrm{b}.10\)      & \((0043)\)    & \(0\)      & \(3^{17}\)          \\
  \(2\) & 8,11    &                                     & 1  & 3  & \(1\)     & \(5\)   & \(0\)   & \(3^2\)     & \(32^21\)  & \(\mathrm{b}.10\)      & \((0043)\)    & \(0\)      & \(2\cdot 3^{16}\)   \\
  \(1\) & 8,11    &                                     & 1  & 3  & \(1\)     & \(5\)   & \(0\)   & \(3^2\)     & \(32^21\)  & \(\mathrm{b}.10\)      & \((0043)\)    & \(0\)      & \(3^{16}\)          \\
\hline
  \(1\) & 9,12    & \(0;0,0,0,0\)                       & 0  & 2  & \(1^2\)   & \(6\)   & \(1\)   & \(4^2\)     & \(4321\)   & \(\mathrm{b}.10^\ast\) & \((0043)\)    & \(2\)      & \(2^2\cdot 3^{20}\) \\
  \(1\) & 9,12    & \(0;1,0,1,0\)                       & 1  & 2  & \(1^2\)   & \(5\)   & \(0\)   & \(4^2\)     & \(4321\)   & \(\mathrm{d}.19\)      & \((3043)\)    & \(0\)      & \(3^{20}\)          \\
  \(1\) & 9,12    & \(0;1,0,0,0\)                       & 1  & 2  & \(1^2\)   & \(5\)   & \(0\)   & \(4^2\)     & \(4321\)   & \(\mathrm{d}.23\)      & \((1043)\)    & \(0\)      & \(2\cdot 3^{20}\)   \\
  \(1\) & 9,12    & \(0;0,0,1,0\)                       & 1  & 2  & \(1^2\)   & \(5\)   & \(0\)   & \(4^2\)     & \(4321\)   & \(\mathrm{d}.25\)      & \((2043)\)    & \(0\)      & \(2\cdot 3^{20}\)   \\
  \(2\) & 9,12    & \(\pm 1;0,0,0,0\)                   & 1  & 2  & \(1\)     & \(5\)   & \(0\)   & \(43\)      & \(4321\)   & \(\mathrm{b}.10\)      & \((0043)\)    & \(2\)      & \(2^2\cdot 3^{20}\) \\
  \(4\) & 9,12    & \(\pm 1;\alpha,0,\gamma,0\)         & 1  & 2  & \(1\)     & \(5\)   & \(0\)   & \(43\)      & \(4321\)   & \(\mathrm{b}.10\)      & \((0043)\)    & \(0\)      & \(2\cdot 3^{20}\)   \\
  \(2\) & 9,12    & \(\pm 1;1,0,1,0\)                   & 1  & 2  & \(1\)     & \(5\)   & \(0\)   & \(43\)      & \(4321\)   & \(\mathrm{b}.10\)      & \((0043)\)    & \(0\)      & \(3^{20}\)          \\
  \(3\) & 9,12    &                                     & 1  & 3  & \(1^2\)   & \(5\)   & \(0\)   & \(43\)      & \(3^221\)  & \(\mathrm{b}.10\)      & \((0043)\)    & \(0\)      & \(3^{19}\)          \\
  \(2\) & 9,12    &                                     & 1  & 3  & \(1^2\)   & \(5\)   & \(0\)   & \(43\)      & \(3^221\)  & \(\mathrm{b}.10\)      & \((0043)\)    & \(0\)      & \(2\cdot 3^{18}\)   \\
  \(6\) & 9,12    &                                     & 1  & 3  & \(1\)     & \(5\)   & \(0\)   & \(43\)      & \(3^221\)  & \(\mathrm{b}.10\)      & \((0043)\)    & \(0\)      & \(3^{19}\)          \\
  \(4\) & 9,12    &                                     & 1  & 3  & \(1\)     & \(5\)   & \(0\)   & \(43\)      & \(3^221\)  & \(\mathrm{b}.10\)      & \((0043)\)    & \(0\)      & \(2\cdot 3^{18}\)   \\
\hline
\end{tabular}
\end{center}
\end{table}

%\newpage

%--------------------------------------------------------------------------------

\begin{remark}
\label{rmk:TKTb10TreeCc4}
The algebraic information in Table
\ref{tbl:TKTb10TreeCc4}
is visualized in Figure
\ref{fig:TKTb10TreeCc4}.
By periodic continuation, the figure shows more branches than the table
but less details concerning the exact order \(\#\mathrm{Aut}\) of the automorphism group.
\end{remark}

%\newpage

%--------------------------------------------------------------------------------

\begin{proof}
(of Theorem
\ref{thm:TKTb10TreeCc4})
According to Proposition
\ref{prp:TKTb10TreeCc4},
the logarithmic order of the tree root, respectively of the periodic root, is
\(n_\ast=9\), respectively \(p_\ast=n_\ast+\ell_\ast=10\).

Since \(C_1(m_n)=1\) for all mainline vertices \(m_n\) with \(n\ge n_\ast\),
according to Proposition
\ref{prp:TKTb10TreeCc4},
the unique capable child of \(m_n\) is \(m_{n+1}\),
and each branch has depth \(\mathrm{dp}(\mathcal{B}(n))=1\), for \(n\ge n_\ast\).
Consequently, the tree is also of depth \(\mathrm{dp}(\mathcal{T})=1\).

With the aid of Formula
\eqref{eqn:BranchDepth1}
in Theorem
\ref{thm:DescendantNumbers},
the claims (2) and (4) are consequences of Proposition
\ref{prp:TKTb10TreeCc4}:\\
\(\#\mathcal{B}(9)=N_1(m_9)=21\),
\(\#\mathcal{B}(10)=N_1(m_{10})=30\), and
\(\#\mathcal{B}(11)=N_1(m_{11})=27\).

According to Formula
\eqref{eqn:TreeWidth1}
in Corollary
\ref{cor:TreeWidth},
where \(n\) runs from \(n_\ast+1=10\) to \(n_\ast+\ell_\ast+\ell+0=12\),
the tree width is the maximum \(\mathrm{wd}(\mathcal{T})=30\)
of the expressions \(N_1(m_9)=21\), \(N_1(m_{10})=30\), and \(N_1(m_{11})=27\).

The information content of the tree is given by Formula
\eqref{eqn:InfoCont}
in the Definition
\ref{dfn:InfoCont}: \\
\(\mathrm{IC}(\mathcal{T})=\#\mathcal{B}(9)+(\#\mathcal{B}(10)+\#\mathcal{B}(11))=21+(30+27)=78\).

The algebraic invariants in Table
\ref{tbl:TKTb10TreeCc4},
that is,
depth \(\mathrm{dp}\),
derived length \(\mathrm{dl}\),
abelian type invariants of the centre \(\zeta\),
relation rank \(\mu\),
nuclear rank \(\nu\),
abelian quotient invariants \(\tau(1)\) of the first maximal subgroup,
respectively \(\tau_2\) of the commutator subgroup,
transfer kernel type \(\varkappa\),
action flag \(\sigma\), and
the factorized order \(\#\mathrm{Aut}\) of the automorphism group
have been computed by means of program scripts written for MAGMA
\cite{MAGMA}.

Each group is characterized by the parameters
of the normalized representative \(G_\rho^{m,n}(\alpha,\beta,\gamma,\delta)\) of its isomorphism class,
according to Formula
\eqref{eqn:Presentation},
and by its identifier \(\langle 3^n,i\rangle\) in the SmallGroups Database
\cite{BEO}.

The column with header \(\#\) contains
the number of groups with identical invariants (except the presentation),
for each row.
\end{proof}

%--------------------------------------------------------------------------------

\begin{corollary}
\label{cor:TKTb10TreeCc4}
\textbf{(Actions and relation ranks.)}
The algebraic invariants of the vertices of the structured coclass-\(4\) tree \(\mathcal{T}^4{P_9}\) are listed in Table
\ref{tbl:TKTb10TreeCc4}. In particular:
\begin{enumerate}
\item
The groups with \(V_4\)-action are
all mainline vertices \(\stackbin[0]{n-3,n}{G}\begin{pmatrix}0&0\\ 0&0\end{pmatrix}\), \(n\ge 9\),
the two terminal vertices \(\stackbin[0]{n-3,n}{G}\begin{pmatrix}0&0\\ \pm 1&0\end{pmatrix}\) with odd \(n\ge 11\),
the two terminal vertices \(\stackbin[\pm 1]{n-3,n}{G}\begin{pmatrix}0&0\\ 0&0\end{pmatrix}\) with even \(n\ge 10\),
and two terminal non-metabelian vertices with odd \(n\ge 11\).
\item
With respect to the kernel types, 
all mainline groups of type \(\mathrm{b}.10^\ast\), \(\varkappa=(0043)\),
the two leaves of type \(\mathrm{d}.25\), \(\varkappa\sim (2043)\), with every odd logarithmic order,
two distinguished metabelian leaves of type \(\mathrm{b}.10\) with every even logarithmic order,
and two distinguished non-metabelian leaves of type \(\mathrm{b}.10\) with every odd logarithmic order
possess a \(V_4\)-action.
\item
The relation rank is given by
\(\mu=6\) for the mainline vertices \(\stackbin[0]{n-3,n}{G}\begin{pmatrix}0&0\\ 0&0\end{pmatrix}\), \(n\ge 9\), and
\(\mu=5\) otherwise.
There do not occur any RI-actions.
\end{enumerate}
\end{corollary}

%--------------------------------------------------------------------------------

\begin{proof}
(of Corollary
\ref{cor:TKTb10TreeCc4})
The existence of an RI-action on \(G\) has been checked by means of
an algorithm involving the \(p\)-covering group of \(G\), written for MAGMA
\cite{MAGMA}.
The other claims follow immediately from Table
\ref{tbl:TKTb10TreeCc4},
continued indefinitely with the aid of the periodicity in Prop.
\ref{prp:TKTb10TreeCc4}. 
\end{proof}

%\newpage

%--------------------------------------------------------------------------------

\renewcommand{\arraystretch}{1.2}

\begin{table}[ht]
\caption{Data for \(3\)-groups \(G\) with \(9\le n=\mathrm{lo}(G)\le 12\) of the coclass tree \(\mathcal{T}^4{R_2^4}\)}
\label{tbl:TKTd19TreeCc4}
\begin{center}
\begin{tabular}{|c|c|l||c|c|c||c|c||c|c|l|c||c|c|}
\hline
 \(\#\) & \(m,n\)  & \(\rho;\alpha,\beta,\gamma,\delta\) & dp & dl & \(\zeta\) & \(\mu\) & \(\nu\) & \(\tau(1)\) & \(\tau_2\) & Type              & \(\varkappa\)    & \(\sigma\) & \(\#\mathrm{Aut}\) \\
\hline
  \(1\) & \(6,9\)  & \(0;0,-1,0,1\)                      & 0  & 2  & \(1^2\)   & \(5\)   & \(1\)   & \(32\)      & \(2^31\)   & \(\mathrm{d}.19^\ast\) & \((0343)\)  & \(1\)      & \(2\cdot 3^{14}\)  \\
\hline
  \(1\) & \(7,10\) & \(0;0,-1,0,1\)                      & 0  & 2  & \(1^2\)   & \(5\)   & \(1\)   & \(3^2\)     & \(32^21\)  & \(\mathrm{d}.19^\ast\) & \((0343)\)  & \(1^\ast\) & \(2\cdot 3^{16}\)  \\
  \(1\) & \(7,10\) & \(0;1,-1,1,1\)                      & 1  & 2  & \(1^2\)   & \(4\)   & \(0\)   & \(3^2\)     & \(32^21\)  & \(\mathrm{F}.7\)  & \((4343)\)       & \(0\)      & \(3^{16}\)         \\
  \(1\) & \(7,10\) & \(0;1,-1,0,1\)                      & 1  & 2  & \(1^2\)   & \(4\)   & \(0\)   & \(3^2\)     & \(32^21\)  & \(\mathrm{F}.12\) & \((1343)\)       & \(0\)      & \(3^{16}\)         \\
  \(1\) & \(7,10\) & \(0;0,-1,1,1\)                      & 1  & 2  & \(1^2\)   & \(4\)   & \(0\)   & \(3^2\)     & \(32^21\)  & \(\mathrm{F}.13\) & \((2343)\)       & \(0\)      & \(3^{16}\)         \\
  \(1\) & \(7,10\) & \(0;1,-1,-1,1\)                     & 1  & 2  & \(1^2\)   & \(5\)   & \(1\)   & \(3^2\)     & \(32^21\)  & \(\mathrm{H}.4\)  & \((3343)\)       & \(0\)      & \(3^{16}\)         \\
  \(6\) & \(7,10\) &                                     & 1  & 3  & \(1^2\)   & \(4\)   & \(0\)   & \(32\)      & \(2^31\)   & \(\mathrm{d}.19\) & \((0343)\)       & \(0\)      & \(3^{15}\)         \\
  \(2\) & \(7,10\) &                                     & 1  & 3  & \(1^2\)   & \(4\)   & \(0\)   & \(32\)      & \(2^31\)   & \(\mathrm{d}.19\) & \((0343)\)       & \(0\)      & \(3^{14}\)         \\
  \(9\) & \(8,11\) & \(\pm 1;\alpha,-1,\gamma,1\)        & 2  & 2  & \(1\)     & \(4\)   & \(0\)   & \(3^2\)     & \(3^221\)  & \(\mathrm{H}.4\)  & \((3343)\)       & \(0\)      & \(3^{18}\)         \\
 \(12\) & \(8,11\) &                                     & 2  & 3  & \(1\)     & \(4\)   & \(0\)   & \(3^2\)     & \(32^21\)  & \(\mathrm{H}.4\)  & \((3343)\)       & \(0\)      & \(3^{17}\)         \\
  \(4\) & \(8,11\) &                                     & 2  & 3  & \(1\)     & \(4\)   & \(0\)   & \(3^2\)     & \(32^21\)  & \(\mathrm{H}.4\)  & \((3343)\)       & \(0\)      & \(3^{16}\)         \\
\hline
  \(1\) & \(8,11\) & \(0;0,-1,0,1\)                      & 0  & 2  & \(1^2\)   & \(5\)   & \(1\)   & \(43\)      & \(3^221\)  & \(\mathrm{d}.19^\ast\) & \((0343)\)  & \(1\)      & \(2\cdot 3^{18}\)  \\
  \(2\) & \(8,11\) & \(0;\pm 1,-1,\pm 1,1\)              & 1  & 2  & \(1^2\)   & \(4\)   & \(0\)   & \(43\)      & \(3^221\)  & \(\mathrm{F}.7\)  & \((4343)\)       & \(1^\ast\) & \(2\cdot 3^{18}\)  \\
  \(2\) & \(8,11\) & \(0;\pm 1,-1,0,1\)                  & 1  & 2  & \(1^2\)   & \(4\)   & \(0\)   & \(43\)      & \(3^221\)  & \(\mathrm{F}.12\) & \((1343)\)       & \(1^\ast\) & \(2\cdot 3^{18}\)  \\
  \(2\) & \(8,11\) & \(0;0,-1,\pm 1,1\)                  & 1  & 2  & \(1^2\)   & \(4\)   & \(0\)   & \(43\)      & \(3^221\)  & \(\mathrm{F}.13\) & \((2343)\)       & \(1^\ast\) & \(2\cdot 3^{18}\)  \\
  \(2\) & \(8,11\) & \(0;\pm 1,-1,\mp 1,1\)              & 1  & 2  & \(1^2\)   & \(5\)   & \(1\)   & \(43\)      & \(3^221\)  & \(\mathrm{H}.4\)  & \((3343)\)       & \(1\)      & \(2\cdot 3^{18}\)  \\
 \(12\) & \(8,11\) &                                     & 1  & 3  & \(1^2\)   & \(4\)   & \(0\)   & \(3^2\)     & \(32^21\)  & \(\mathrm{d}.19\) & \((0343)\)       & \(1^\ast\) & \(2\cdot 3^{17}\)  \\
  \(4\) & \(8,11\) &                                     & 1  & 3  & \(1^2\)   & \(4\)   & \(0\)   & \(3^2\)     & \(32^21\)  & \(\mathrm{d}.19\) & \((0343)\)       & \(1^\ast\) & \(2\cdot 3^{16}\)  \\
  \(2\) & \(9,12\) & \(\pm 1;0,-1,0,1\)                  & 2  & 2  & \(1\)     & \(4\)   & \(0\)   & \(43\)      & \(4321\)   & \(\mathrm{H}.4\)  & \((3343)\)       & \(1^\ast\) & \(2\cdot 3^{20}\)  \\
  \(8\) & \(9,12\) & \(\pm 1;\alpha,-1,\gamma,1\)        & 2  & 2  & \(1\)     & \(4\)   & \(0\)   & \(43\)      & \(4321\)   & \(\mathrm{H}.4\)  & \((3343)\)       & \(0\)      & \(3^{20}\)         \\
 \(12\) & \(9,12\) &                                     & 2  & 3  & \(1\)     & \(4\)   & \(0\)   & \(43\)      & \(3^221\)  & \(\mathrm{H}.4\)  & \((3343)\)       & \(0\)      & \(3^{19}\)         \\
  \(4\) & \(9,12\) &                                     & 2  & 3  & \(1\)     & \(4\)   & \(0\)   & \(43\)      & \(3^221\)  & \(\mathrm{H}.4\)  & \((3343)\)       & \(0\)      & \(3^{18}\)         \\
\hline
\end{tabular}
\end{center}
\end{table}

%\newpage

%--------------------------------------------------------------------------------

\subsection{Two mainlines of type \(\mathrm{d}.19^\ast\) for even coclass \(r\ge 4\)}
\label{ss:d19Cc4}

\begin{proposition}
\label{prp:TKTd19Tree1Cc4}
\textbf{(Periodicity and descendant numbers.)} \\
The branches \(\mathcal{B}(i)\), \(i\ge n_\ast=9\),
of the first coclass-\(4\) tree \(\mathcal{T}^4(P_7-\#2;39)\)
with mainline vertices of transfer kernel type \(\mathrm{d}.19^\ast\), \(\varkappa\sim (0343)\),
are purely periodic with primitive length \(\ell=2\) and without pre-period, \(\ell_\ast=0\), that is,
\(\mathcal{B}(i+2)\simeq\mathcal{B}(i)\) are isomorphic as digraphs, for all \(i\ge p_\ast=n_\ast+\ell_\ast=9\).

The graph theoretic structure of the tree is determined uniquely by the numbers
\(N_1\) of immediate descendants and \(C_1\) of capable immediate descendants
of the mainline vertices \(m_n\) with logarithmic order \(n=\mathrm{lo}(m_n)\ge n_\ast=9\)
and of capable vertices \(v\) with depth \(1\) and \(\mathrm{lo}(v)\ge n_\ast+1=10\): \\
\((N_1,C_1)=(13,2)\) for mainline vertices \(m_n\) with odd logarithmic order \(n\ge 9\), \\
\((N_1,C_1)=(25,3)\) for mainline vertices \(m_n\) with even logarithmic order \(n\ge 10\), \\
\((N_1,C_1)=(25,0)\) for the capable vertex \(v\) of depth \(1\) and even logarithmic order \(\mathrm{lo}(v)\ge 10\), \\
\((N_1,C_1)=(13,0)\) for two capable vertices \(v\) of depth \(1\) and odd logarithmic order \(\mathrm{lo}(v)\ge 11\).
\end{proposition}

%--------------------------------------------------------------------------------

\begin{proof}
(of Proposition
\ref{prp:TKTd19Tree1Cc4})
The statements concerning the numbers \(N_1(m_n)\) of immediate descendants
of the mainline vertices \(m_n\) with \(n\ge n_\ast=9\),
and \(N_1(v)\) of vertices with depth \(\mathrm{dp}(v)=1\)
and logarithmic order \(\mathrm{lo}(v)\ge n_\ast+1=10\),
have been obtained by direct computation with the \(p\)-group generation algorithm
\cite{Nm2,Ob,HEO}
in MAGMA
\cite{MAGMA}.
In detail, we proved that there is no pre-period, \(\ell_\ast=0\), 
and the primitive period \((\mathcal{B}(9),\mathcal{B}(10))\) of length \(\ell=2\) consists of \\
\(5\), resp. \(9\), metabelian vertices with \(\zeta=1^2\), resp. \(\zeta=1\), and \\
\(8\), resp. \(16\), non-metabelian vertices with \(\zeta=1^2\), resp. \(\zeta=1\), \\
(\(13=5+8\) children of \(m_9\), and \(25=9+16\) children of \(v_2(m_9)\) with depth \(1\)) \\
together \(38\) vertices (\(14\) of them metabelian) in branch \(\mathcal{B}(9)\), and \\
\(9\), resp. \(10\), metabelian vertices with \(\zeta=1^2\), resp. \(\zeta=1\), and \\
\(16\), resp. \(16\), non-metabelian vertices with \(\zeta=1^2\), resp. \(\zeta=1\), \\
(\(25=9+16\) children of \(m_{10}\), and \(26=2\cdot (5+8)\) children of \(v_{2,3}(m_{10})\), both with depth \(1\)) \\
together \(51\) vertices (\(19\) of them metabelian) in branch \(\mathcal{B}(10)\).

The results concerning the metabelian skeleton confirm the corresponding statements in the dissertation of Nebelung
\cite[Thm. 5.1.16, pp. 178--179, and the fourth Figure, \(e\ge 5\), \(e\equiv 1\pmod{2}\), on the double page between pp. 191--192]{Ne}.
The tree \(\mathcal{T}^4{R_2^4}\) corresponds to
the infinite metabelian pro-\(3\) group \(S_{3,4}\) in
\cite[Cnj. 15 (b), p. 116]{Ek}.

The claim of the virtual periodicity of branches
has been proved generally for any coclass tree in
\cite{dS}
and
\cite{EkLg}.
Here, the strict periodicity was confirmed by computation up to branch \(\mathcal{B}(30)\)
and undoubtedly sets in at \(p_\ast=9\).
\end{proof}

%\newpage

%--------------------------------------------------------------------------------

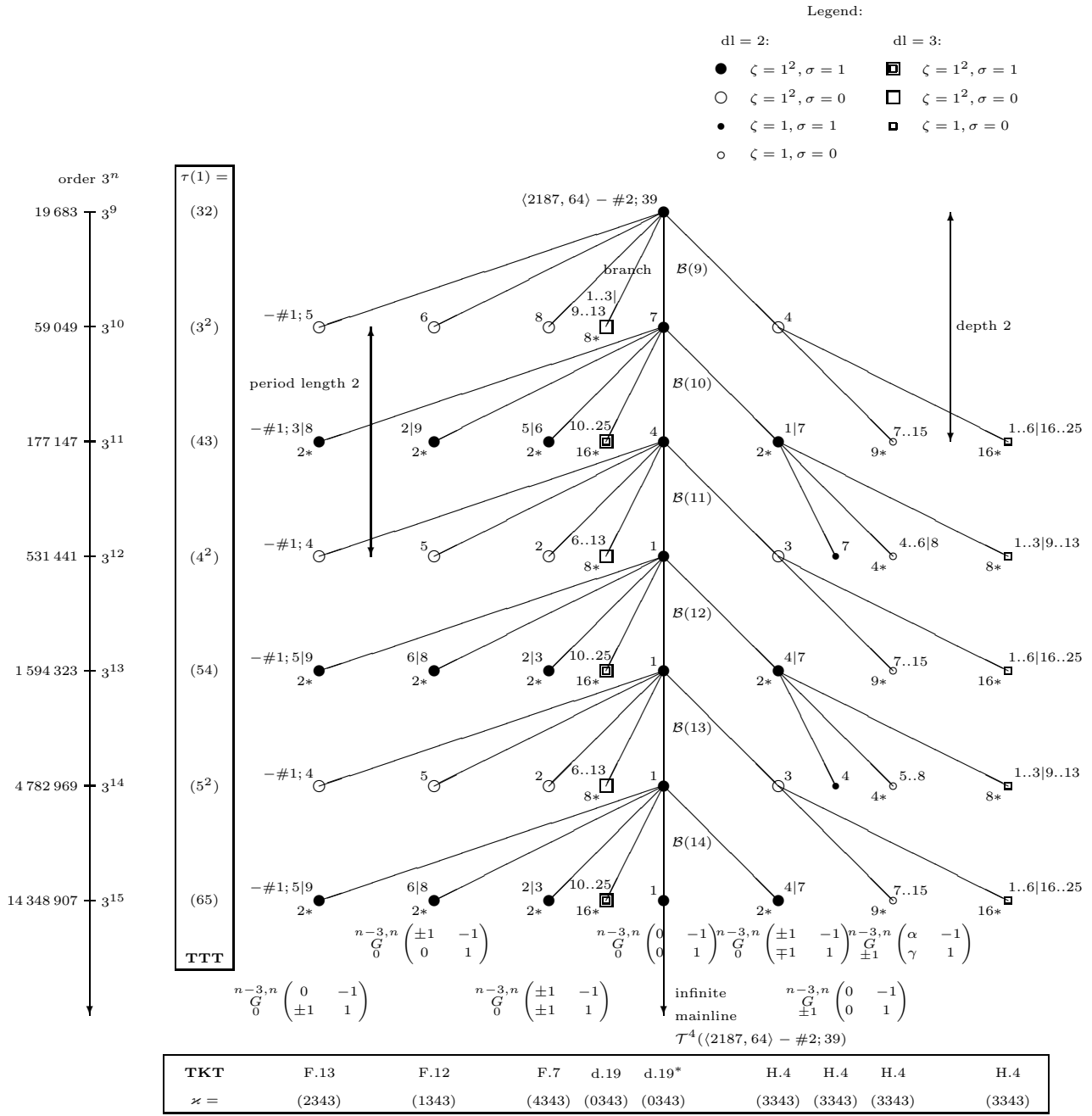
\begin{figure}[hb]
\caption{The first coclass-\(4\) tree \(\mathcal{T}^4(P_7-\#2;39)\) with mainline of type \(\mathrm{d}.19^\ast\)}
\label{fig:TKTd19Tree1Cc4}

% First Coclass-4 Tree with Mainline of Type d.19 in the Coclass Forest F(4)
% Structured Algebraically by Derived Length, Central Rank, and GI-Action

{\tiny

\setlength{\unitlength}{0.9cm}
\begin{picture}(18,20)(-10,-16)

% scale of orders
\put(-10,0.5){\makebox(0,0)[cb]{order \(3^n\)}}

\put(-10,0){\line(0,-1){12}}
\multiput(-10.1,0)(0,-2){7}{\line(1,0){0.2}}

\put(-10.2,0){\makebox(0,0)[rc]{\(19\,683\)}}
\put(-9.8,0){\makebox(0,0)[lc]{\(3^9\)}}
\put(-10.2,-2){\makebox(0,0)[rc]{\(59\,049\)}}
\put(-9.8,-2){\makebox(0,0)[lc]{\(3^{10}\)}}
\put(-10.2,-4){\makebox(0,0)[rc]{\(177\,147\)}}
\put(-9.8,-4){\makebox(0,0)[lc]{\(3^{11}\)}}
\put(-10.2,-6){\makebox(0,0)[rc]{\(531\,441\)}}
\put(-9.8,-6){\makebox(0,0)[lc]{\(3^{12}\)}}
\put(-10.2,-8){\makebox(0,0)[rc]{\(1\,594\,323\)}}
\put(-9.8,-8){\makebox(0,0)[lc]{\(3^{13}\)}}
\put(-10.2,-10){\makebox(0,0)[rc]{\(4\,782\,969\)}}
\put(-9.8,-10){\makebox(0,0)[lc]{\(3^{14}\)}}
\put(-10.2,-12){\makebox(0,0)[rc]{\(14\,348\,907\)}}
\put(-9.8,-12){\makebox(0,0)[lc]{\(3^{15}\)}}

\put(-10,-12){\vector(0,-1){2}}

% polarized transfer target
\put(-8,0.5){\makebox(0,0)[cb]{\(\tau(1)=\)}}

\put(-8,0){\makebox(0,0)[cc]{\((32)\)}}
\put(-8,-2){\makebox(0,0)[cc]{\((3^2)\)}}
\put(-8,-4){\makebox(0,0)[cc]{\((43)\)}}
\put(-8,-6){\makebox(0,0)[cc]{\((4^2)\)}}
\put(-8,-8){\makebox(0,0)[cc]{\((54)\)}}
\put(-8,-10){\makebox(0,0)[cc]{\((5^2)\)}}
\put(-8,-12){\makebox(0,0)[cc]{\((65)\)}}

\put(-8,-13){\makebox(0,0)[cc]{\textbf{TTT}}}
\put(-8.5,-13.2){\framebox(1,14){}}

% legend
\put(2.5,3.5){\makebox(0,0)[lc]{Legend:}}

\put(1,3){\makebox(0,0)[lc]{\(\mathrm{dl}=2\):}}
%\put(1,3){\circle{0.2}}
%\put(1,3){\circle*{0.1}}
\put(1,2.5){\circle*{0.2}}
\put(1,2){\circle{0.2}}
\put(1,1.5){\circle*{0.1}}
\put(1,1){\circle{0.1}}
%\put(1.5,3){\makebox(0,0)[lc]{\(\zeta=1^2,\sigma=2\)}}
\put(1.5,2.5){\makebox(0,0)[lc]{\(\zeta=1^2,\sigma=1\)}}
\put(1.5,2){\makebox(0,0)[lc]{\(\zeta=1^2,\sigma=0\)}}
\put(1.5,1.5){\makebox(0,0)[lc]{\(\zeta=1,\sigma=1\)}}
\put(1.5,1){\makebox(0,0)[lc]{\(\zeta=1,\sigma=0\)}}

\put(4,3){\makebox(0,0)[lc]{\(\mathrm{dl}=3\):}}
%\put(3.9,2.9){\framebox(0.2,0.2){\vrule height 3pt width 3pt}}
\put(3.9,2.4){\framebox(0.2,0.2){}}
\put(3.95,2.45){\framebox(0.1,0.1){}}
\put(3.9,1.9){\framebox(0.2,0.2){}}
\put(3.95,1.45){\framebox(0.1,0.1){}}
%\put(4.5,3){\makebox(0,0)[lc]{\(\zeta=1^2,\sigma=2\)}}
\put(4.5,2.5){\makebox(0,0)[lc]{\(\zeta=1^2,\sigma=1\)}}
\put(4.5,2){\makebox(0,0)[lc]{\(\zeta=1^2,\sigma=0\)}}
\put(4.5,1.5){\makebox(0,0)[lc]{\(\zeta=1,\sigma=0\)}}

% depth of branches
\put(5,-2){\vector(0,1){2}}
\put(5.1,-2){\makebox(0,0)[lc]{depth \(2\)}}
\put(5,-2){\vector(0,-1){2}}

% periodicity of branches
\put(-5.1,-4){\vector(0,1){2}}
\put(-5.3,-3){\makebox(0,0)[rc]{period length \(2\)}}
\put(-5.1,-4){\vector(0,-1){2}}

% metabelian GI-vertices with central rank 2
\multiput(0,0)(0,-2){7}{\circle*{0.2}}
\multiput(2,-4)(0,-4){3}{\circle*{0.2}}
\multiput(-2,-4)(0,-4){3}{\circle*{0.2}}
\multiput(-4,-4)(0,-4){3}{\circle*{0.2}}
\multiput(-6,-4)(0,-4){3}{\circle*{0.2}}
% metabelian non-GI vertices with central rank 2
\multiput(-2,-2)(0,-4){3}{\circle{0.2}}
\multiput(-4,-2)(0,-4){3}{\circle{0.2}}
\multiput(-6,-2)(0,-4){3}{\circle{0.2}}
\multiput(2,-2)(0,-4){3}{\circle{0.2}}

% non-metabelian GI-vertices
\multiput(-1.1,-4.1)(0,-4){3}{\framebox(0.2,0.2){}}
\multiput(-1.05,-4.05)(0,-4){3}{\framebox(0.1,0.1){}}
% non-metabelian non-GI vertices
\multiput(-1.1,-2.1)(0,-4){3}{\framebox(0.2,0.2){}}

% metabelian GI-vertices with central rank 1
\multiput(3,-6)(0,-4){2}{\circle*{0.1}}
% metabelian non-GI vertices with central rank 1
\multiput(4,-4)(0,-2){5}{\circle{0.1}}

% non-metabelian non-GI vertices
\multiput(5.95,-4.05)(0,-2){5}{\framebox(0.1,0.1){}}

% directed edges to vertices of depth 1
\multiput(0,0)(0,-2){6}{\line(0,-1){2}}
\multiput(0,0)(0,-2){6}{\line(-1,-2){1}}
\multiput(0,0)(0,-2){6}{\line(-1,-1){2}}
\multiput(0,0)(0,-2){6}{\line(-2,-1){4}}
\multiput(0,0)(0,-2){6}{\line(-3,-1){6}}
\multiput(0,0)(0,-2){6}{\line(1,-1){2}}

% directed edges to vertices of depth 2
\multiput(2,-4)(0,-4){2}{\line(1,-2){1}}
\multiput(2,-2)(0,-2){5}{\line(1,-1){2}}
\multiput(2,-2)(0,-2){5}{\line(2,-1){4}}

% multiplicity counters
\multiput(-1.1,-2.1)(0,-4){3}{\makebox(0,0)[rt]{\(8\ast\)}}
\multiput(-6.1,-4.1)(0,-4){3}{\makebox(0,0)[rt]{\(2\ast\)}}
\multiput(-4.1,-4.1)(0,-4){3}{\makebox(0,0)[rt]{\(2\ast\)}}
\multiput(-2.1,-4.1)(0,-4){3}{\makebox(0,0)[rt]{\(2\ast\)}}
\multiput(-1.1,-4.1)(0,-4){3}{\makebox(0,0)[rt]{\(16\ast\)}}
\multiput(1.9,-4.1)(0,-4){3}{\makebox(0,0)[rt]{\(2\ast\)}}
\multiput(3.9,-4.1)(0,-4){3}{\makebox(0,0)[rt]{\(9\ast\)}}
\multiput(3.9,-6.1)(0,-4){2}{\makebox(0,0)[rt]{\(4\ast\)}}
\multiput(5.9,-4.1)(0,-4){3}{\makebox(0,0)[rt]{\(16\ast\)}}
\multiput(5.9,-6.1)(0,-4){2}{\makebox(0,0)[rt]{\(8\ast\)}}

% infinite mainline
\put(0,-12){\vector(0,-1){2}}
\put(0.2,-13.6){\makebox(0,0)[lc]{infinite}}
\put(0.2,-14){\makebox(0,0)[lc]{mainline}}
\put(0.2,-14.4){\makebox(0,0)[lc]{\(\mathcal{T}^4(\langle 2187,64\rangle-\#2;39)\)}}

% root
\put(-0.1,0.1){\makebox(0,0)[rb]{\(\langle 2187,64\rangle-\#2;39\)}}

% 13 vertices of branch B(9)
\put(-0.2,-1){\makebox(0,0)[rc]{branch}}
\put(0.5,-1){\makebox(0,0)[cc]{\(\mathcal{B}(9)\)}}
\put(-6.1,-1.9){\makebox(0,0)[rb]{\(-\#1;5\)}}
\put(-4.1,-1.9){\makebox(0,0)[rb]{\(6\)}}
\put(-2.1,-1.9){\makebox(0,0)[rb]{\(8\)}}
\put(-0.8,-1.6){\makebox(0,0)[rb]{\(1..3\vert\)}}
\put(-1.0,-1.8){\makebox(0,0)[rb]{\(9..13\)}}
\put(-0.1,-1.9){\makebox(0,0)[rb]{\(7\)}}
\put(2.1,-1.9){\makebox(0,0)[lb]{\(4\)}}

% 25 vertices of branch B(10)
\put(0.5,-3){\makebox(0,0)[cc]{\(\mathcal{B}(10)\)}}
\put(-6.1,-3.9){\makebox(0,0)[rb]{\(-\#1;3\vert 8\)}}
\put(-4.2,-3.9){\makebox(0,0)[rb]{\(2\vert 9\)}}
\put(-2.1,-3.9){\makebox(0,0)[rb]{\(5\vert 6\)}}
\put(-0.9,-3.8){\makebox(0,0)[rb]{\(10..25\)}}
\put(-0.1,-3.9){\makebox(0,0)[rb]{\(4\)}}
\put(2.1,-3.9){\makebox(0,0)[lb]{\(1\vert 7\)}}

% 25 vertices of depth 2 of branch B(9)
\put(4,-3.9){\makebox(0,0)[lb]{\(7..15\)}}
\put(6,-3.9){\makebox(0,0)[lb]{\(1..6\vert 16..25\)}}

% 13 vertices of branch B(11)
\put(0.5,-5){\makebox(0,0)[cc]{\(\mathcal{B}(11)\)}}
\put(-6.1,-5.9){\makebox(0,0)[rb]{\(-\#1;4\)}}
\put(-4.1,-5.9){\makebox(0,0)[rb]{\(5\)}}
\put(-2.1,-5.9){\makebox(0,0)[rb]{\(2\)}}
\put(-1.0,-5.8){\makebox(0,0)[rb]{\(6..13\)}}
\put(-0.1,-5.9){\makebox(0,0)[rb]{\(1\)}}
\put(2.1,-5.9){\makebox(0,0)[lb]{\(3\)}}

% 13 vertices of depth 2 of branch B(10)
\put(3.1,-5.9){\makebox(0,0)[lb]{\(7\)}}
\put(4.1,-5.9){\makebox(0,0)[lb]{\(4..6\vert 8\)}}
\put(6.1,-5.9){\makebox(0,0)[lb]{\(1..3\vert 9..13\)}}

% 25 vertices of branch B(12)
\put(0.5,-7){\makebox(0,0)[cc]{\(\mathcal{B}(12)\)}}
\put(-6.1,-7.9){\makebox(0,0)[rb]{\(-\#1;5\vert 9\)}}
\put(-4.1,-7.9){\makebox(0,0)[rb]{\(6\vert 8\)}}
\put(-2.1,-7.9){\makebox(0,0)[rb]{\(2\vert 3\)}}
\put(-0.9,-7.8){\makebox(0,0)[rb]{\(10..25\)}}
\put(-0.1,-7.9){\makebox(0,0)[rb]{\(1\)}}
\put(2.1,-7.9){\makebox(0,0)[lb]{\(4\vert 7\)}}

% 25 vertices of depth 2 of branch B(11)
\put(4,-7.9){\makebox(0,0)[lb]{\(7..15\)}}
\put(6,-7.9){\makebox(0,0)[lb]{\(1..6\vert 16..25\)}}

% 13 vertices of branch B(13)
\put(0.5,-9){\makebox(0,0)[cc]{\(\mathcal{B}(13)\)}}
\put(-6.1,-9.9){\makebox(0,0)[rb]{\(-\#1;4\)}}
\put(-4.1,-9.9){\makebox(0,0)[rb]{\(5\)}}
\put(-2.1,-9.9){\makebox(0,0)[rb]{\(2\)}}
\put(-1.0,-9.8){\makebox(0,0)[rb]{\(6..13\)}}
\put(-0.1,-9.9){\makebox(0,0)[rb]{\(1\)}}
\put(2.1,-9.9){\makebox(0,0)[lb]{\(3\)}}

% 13 vertices of depth 2 of branch B(12)
\put(3.1,-9.9){\makebox(0,0)[lb]{\(4\)}}
\put(4.1,-9.9){\makebox(0,0)[lb]{\(5..8\)}}
\put(6.1,-9.9){\makebox(0,0)[lb]{\(1..3\vert 9..13\)}}

% 25 vertices of branch B(14)
\put(0.5,-11){\makebox(0,0)[cc]{\(\mathcal{B}(14)\)}}
\put(-6.1,-11.9){\makebox(0,0)[rb]{\(-\#1;5\vert 9\)}}
\put(-4.1,-11.9){\makebox(0,0)[rb]{\(6\vert 8\)}}
\put(-2.1,-11.9){\makebox(0,0)[rb]{\(2\vert 3\)}}
\put(-0.9,-11.8){\makebox(0,0)[rb]{\(10..25\)}}
\put(-0.1,-11.9){\makebox(0,0)[rb]{\(1\)}}
\put(2.1,-11.9){\makebox(0,0)[lb]{\(4\vert 7\)}}

% 25 vertices of depth 2 of branch B(13)
\put(4,-11.9){\makebox(0,0)[lb]{\(7..15\)}}
\put(6,-11.9){\makebox(0,0)[lb]{\(1..6\vert 16..25\)}}

% isomorphism classes of coclass families
\put(-0.1,-12.4){\makebox(0,0)[ct]{\(\stackbin[0]{n-3,n}{G}\begin{pmatrix}0&-1\\ 0&1\end{pmatrix}\)}}
\put(-2.1,-13.4){\makebox(0,0)[ct]{\(\stackbin[0]{n-3,n}{G}\begin{pmatrix}\pm 1&-1\\ \pm 1&1\end{pmatrix}\)}}
\put(-4.2,-12.4){\makebox(0,0)[ct]{\(\stackbin[0]{n-3,n}{G}\begin{pmatrix}\pm 1&-1\\ 0&1\end{pmatrix}\)}}
\put(-6.3,-13.4){\makebox(0,0)[ct]{\(\stackbin[0]{n-3,n}{G}\begin{pmatrix}0&-1\\ \pm 1&1\end{pmatrix}\)}}
\put(2.1,-12.4){\makebox(0,0)[ct]{\(\stackbin[0]{n-3,n}{G}\begin{pmatrix}\pm 1&-1\\ \mp 1&1\end{pmatrix}\)}}
\put(3.2,-13.4){\makebox(0,0)[ct]{\(\stackbin[\pm 1]{n-3,n}{G}\begin{pmatrix}0&-1\\ 0&1\end{pmatrix}\)}}
\put(4.3,-12.4){\makebox(0,0)[ct]{\(\stackbin[\pm 1]{n-3,n}{G}\begin{pmatrix}\alpha&-1\\ \gamma&1\end{pmatrix}\)}}

% TKT of coclass families
\put(-8,-15){\makebox(0,0)[cc]{\textbf{TKT}}}
\put(0,-15){\makebox(0,0)[cc]{d.19\({}^\ast\)}}
\put(-1,-15){\makebox(0,0)[cc]{d.19}}
\put(-2,-15){\makebox(0,0)[cc]{F.7}}
\put(-4,-15){\makebox(0,0)[cc]{F.12}}
\put(-6,-15){\makebox(0,0)[cc]{F.13}}
\put(2,-15){\makebox(0,0)[cc]{H.4}}
\put(3,-15){\makebox(0,0)[cc]{H.4}}
\put(4,-15){\makebox(0,0)[cc]{H.4}}
\put(6,-15){\makebox(0,0)[cc]{H.4}}
\put(-8,-15.5){\makebox(0,0)[cc]{\(\varkappa=\)}}
\put(0,-15.5){\makebox(0,0)[cc]{\((0343)\)}}
\put(-1,-15.5){\makebox(0,0)[cc]{\((0343)\)}}
\put(-2,-15.5){\makebox(0,0)[cc]{\((4343)\)}}
\put(-4,-15.5){\makebox(0,0)[cc]{\((1343)\)}}
\put(-6,-15.5){\makebox(0,0)[cc]{\((2343)\)}}
\put(2,-15.5){\makebox(0,0)[cc]{\((3343)\)}}
\put(3,-15.5){\makebox(0,0)[cc]{\((3343)\)}}
\put(4,-15.5){\makebox(0,0)[cc]{\((3343)\)}}
\put(6,-15.5){\makebox(0,0)[cc]{\((3343)\)}}
\put(-8.7,-15.7){\framebox(15.4,1){}}

\end{picture}

}

\end{figure}

%\newpage

%--------------------------------------------------------------------------------

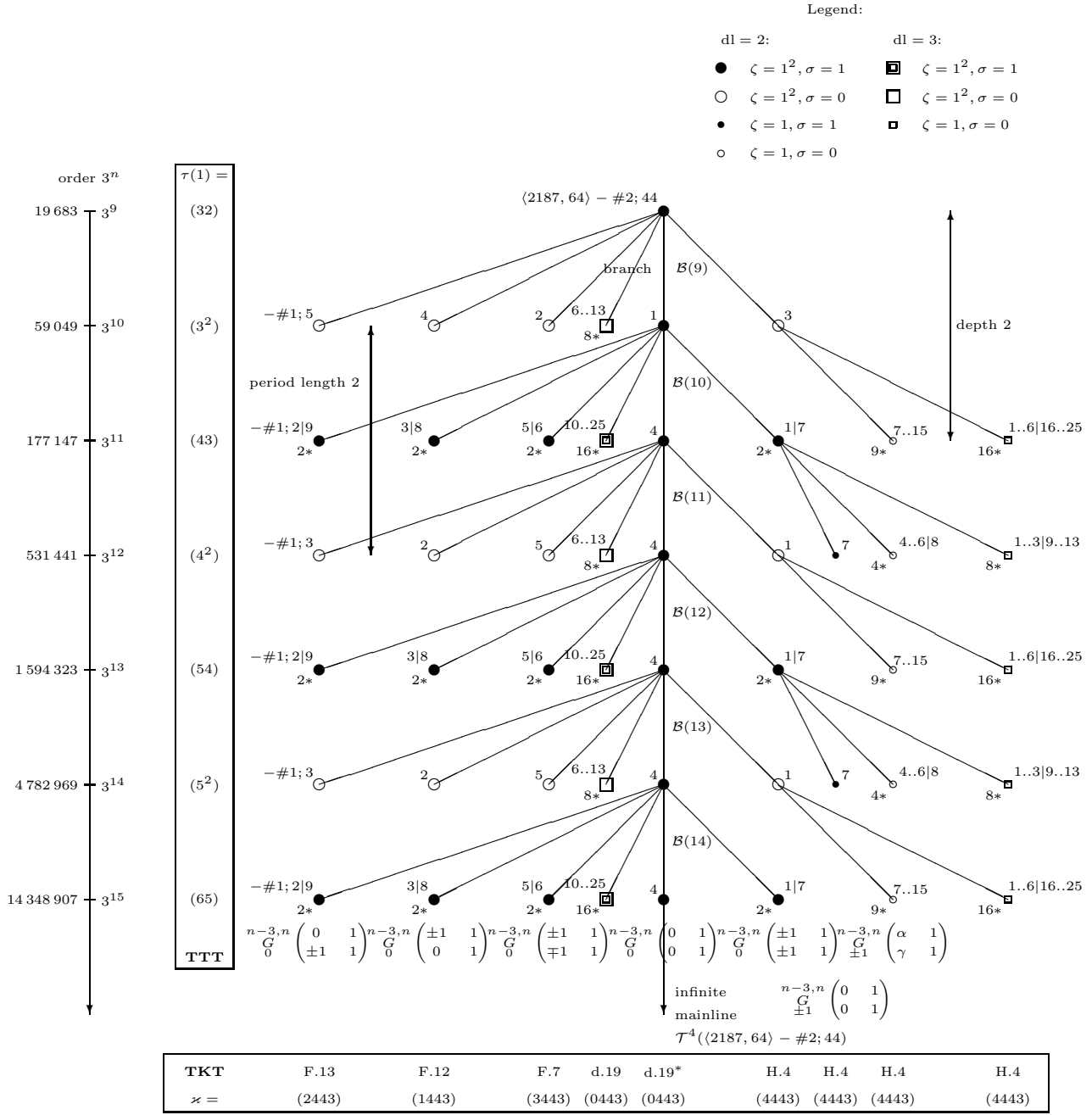
\begin{figure}[ht]
\caption{The second coclass-\(4\) tree \(\mathcal{T}^4(P_7-\#2;44)\) with mainline of type \(\mathrm{d}.19^\ast\)}
\label{fig:TKTd19Tree2Cc4}

% Second Coclass-4 Tree with Mainline of Type d.19 in the Coclass Forest F(4)
% Structured Algebraically by Derived Length, Central Rank, and GI-Action

{\tiny

\setlength{\unitlength}{0.9cm}
\begin{picture}(18,20)(-10,-16)

% scale of orders
\put(-10,0.5){\makebox(0,0)[cb]{order \(3^n\)}}

\put(-10,0){\line(0,-1){12}}
\multiput(-10.1,0)(0,-2){7}{\line(1,0){0.2}}

\put(-10.2,0){\makebox(0,0)[rc]{\(19\,683\)}}
\put(-9.8,0){\makebox(0,0)[lc]{\(3^9\)}}
\put(-10.2,-2){\makebox(0,0)[rc]{\(59\,049\)}}
\put(-9.8,-2){\makebox(0,0)[lc]{\(3^{10}\)}}
\put(-10.2,-4){\makebox(0,0)[rc]{\(177\,147\)}}
\put(-9.8,-4){\makebox(0,0)[lc]{\(3^{11}\)}}
\put(-10.2,-6){\makebox(0,0)[rc]{\(531\,441\)}}
\put(-9.8,-6){\makebox(0,0)[lc]{\(3^{12}\)}}
\put(-10.2,-8){\makebox(0,0)[rc]{\(1\,594\,323\)}}
\put(-9.8,-8){\makebox(0,0)[lc]{\(3^{13}\)}}
\put(-10.2,-10){\makebox(0,0)[rc]{\(4\,782\,969\)}}
\put(-9.8,-10){\makebox(0,0)[lc]{\(3^{14}\)}}
\put(-10.2,-12){\makebox(0,0)[rc]{\(14\,348\,907\)}}
\put(-9.8,-12){\makebox(0,0)[lc]{\(3^{15}\)}}

\put(-10,-12){\vector(0,-1){2}}

% polarized transfer target
\put(-8,0.5){\makebox(0,0)[cb]{\(\tau(1)=\)}}

\put(-8,0){\makebox(0,0)[cc]{\((32)\)}}
\put(-8,-2){\makebox(0,0)[cc]{\((3^2)\)}}
\put(-8,-4){\makebox(0,0)[cc]{\((43)\)}}
\put(-8,-6){\makebox(0,0)[cc]{\((4^2)\)}}
\put(-8,-8){\makebox(0,0)[cc]{\((54)\)}}
\put(-8,-10){\makebox(0,0)[cc]{\((5^2)\)}}
\put(-8,-12){\makebox(0,0)[cc]{\((65)\)}}

\put(-8,-13){\makebox(0,0)[cc]{\textbf{TTT}}}
\put(-8.5,-13.2){\framebox(1,14){}}

% legend
\put(2.5,3.5){\makebox(0,0)[lc]{Legend:}}

\put(1,3){\makebox(0,0)[lc]{\(\mathrm{dl}=2\):}}
%\put(1,3){\circle{0.2}}
%\put(1,3){\circle*{0.1}}
\put(1,2.5){\circle*{0.2}}
\put(1,2){\circle{0.2}}
\put(1,1.5){\circle*{0.1}}
\put(1,1){\circle{0.1}}
%\put(1.5,3){\makebox(0,0)[lc]{\(\zeta=1^2,\sigma=2\)}}
\put(1.5,2.5){\makebox(0,0)[lc]{\(\zeta=1^2,\sigma=1\)}}
\put(1.5,2){\makebox(0,0)[lc]{\(\zeta=1^2,\sigma=0\)}}
\put(1.5,1.5){\makebox(0,0)[lc]{\(\zeta=1,\sigma=1\)}}
\put(1.5,1){\makebox(0,0)[lc]{\(\zeta=1,\sigma=0\)}}

\put(4,3){\makebox(0,0)[lc]{\(\mathrm{dl}=3\):}}
%\put(3.9,2.9){\framebox(0.2,0.2){\vrule height 3pt width 3pt}}
\put(3.9,2.4){\framebox(0.2,0.2){}}
\put(3.95,2.45){\framebox(0.1,0.1){}}
\put(3.9,1.9){\framebox(0.2,0.2){}}
\put(3.95,1.45){\framebox(0.1,0.1){}}
%\put(4.5,3){\makebox(0,0)[lc]{\(\zeta=1^2,\sigma=2\)}}
\put(4.5,2.5){\makebox(0,0)[lc]{\(\zeta=1^2,\sigma=1\)}}
\put(4.5,2){\makebox(0,0)[lc]{\(\zeta=1^2,\sigma=0\)}}
\put(4.5,1.5){\makebox(0,0)[lc]{\(\zeta=1,\sigma=0\)}}

% depth of branches
\put(5,-2){\vector(0,1){2}}
\put(5.1,-2){\makebox(0,0)[lc]{depth \(2\)}}
\put(5,-2){\vector(0,-1){2}}

% periodicity of branches
\put(-5.1,-4){\vector(0,1){2}}
\put(-5.3,-3){\makebox(0,0)[rc]{period length \(2\)}}
\put(-5.1,-4){\vector(0,-1){2}}

% metabelian GI-vertices with central rank 2
\multiput(0,0)(0,-2){7}{\circle*{0.2}}
\multiput(2,-4)(0,-4){3}{\circle*{0.2}}
\multiput(-2,-4)(0,-4){3}{\circle*{0.2}}
\multiput(-4,-4)(0,-4){3}{\circle*{0.2}}
\multiput(-6,-4)(0,-4){3}{\circle*{0.2}}
% metabelian non-GI vertices with central rank 2
\multiput(-2,-2)(0,-4){3}{\circle{0.2}}
\multiput(-4,-2)(0,-4){3}{\circle{0.2}}
\multiput(-6,-2)(0,-4){3}{\circle{0.2}}
\multiput(2,-2)(0,-4){3}{\circle{0.2}}

% non-metabelian GI-vertices
\multiput(-1.1,-4.1)(0,-4){3}{\framebox(0.2,0.2){}}
\multiput(-1.05,-4.05)(0,-4){3}{\framebox(0.1,0.1){}}
% non-metabelian non-GI vertices
\multiput(-1.1,-2.1)(0,-4){3}{\framebox(0.2,0.2){}}

% metabelian GI-vertices with central rank 1
\multiput(3,-6)(0,-4){2}{\circle*{0.1}}
% metabelian non-GI vertices with central rank 1
\multiput(4,-4)(0,-2){5}{\circle{0.1}}

% non-metabelian non-GI vertices
\multiput(5.95,-4.05)(0,-2){5}{\framebox(0.1,0.1){}}

% directed edges to vertices of depth 1
\multiput(0,0)(0,-2){6}{\line(0,-1){2}}
\multiput(0,0)(0,-2){6}{\line(-1,-2){1}}
\multiput(0,0)(0,-2){6}{\line(-1,-1){2}}
\multiput(0,0)(0,-2){6}{\line(-2,-1){4}}
\multiput(0,0)(0,-2){6}{\line(-3,-1){6}}
\multiput(0,0)(0,-2){6}{\line(1,-1){2}}

% directed edges to vertices of depth 2
\multiput(2,-4)(0,-4){2}{\line(1,-2){1}}
\multiput(2,-2)(0,-2){5}{\line(1,-1){2}}
\multiput(2,-2)(0,-2){5}{\line(2,-1){4}}

% multiplicity counters
\multiput(-1.1,-2.1)(0,-4){3}{\makebox(0,0)[rt]{\(8\ast\)}}
\multiput(-6.1,-4.1)(0,-4){3}{\makebox(0,0)[rt]{\(2\ast\)}}
\multiput(-4.1,-4.1)(0,-4){3}{\makebox(0,0)[rt]{\(2\ast\)}}
\multiput(-2.1,-4.1)(0,-4){3}{\makebox(0,0)[rt]{\(2\ast\)}}
\multiput(-1.1,-4.1)(0,-4){3}{\makebox(0,0)[rt]{\(16\ast\)}}
\multiput(1.9,-4.1)(0,-4){3}{\makebox(0,0)[rt]{\(2\ast\)}}
\multiput(3.9,-4.1)(0,-4){3}{\makebox(0,0)[rt]{\(9\ast\)}}
\multiput(3.9,-6.1)(0,-4){2}{\makebox(0,0)[rt]{\(4\ast\)}}
\multiput(5.9,-4.1)(0,-4){3}{\makebox(0,0)[rt]{\(16\ast\)}}
\multiput(5.9,-6.1)(0,-4){2}{\makebox(0,0)[rt]{\(8\ast\)}}

% infinite mainline
\put(0,-12){\vector(0,-1){2}}
\put(0.2,-13.6){\makebox(0,0)[lc]{infinite}}
\put(0.2,-14){\makebox(0,0)[lc]{mainline}}
\put(0.2,-14.4){\makebox(0,0)[lc]{\(\mathcal{T}^4(\langle 2187,64\rangle-\#2;44)\)}}

% root
\put(-0.1,0.1){\makebox(0,0)[rb]{\(\langle 2187,64\rangle-\#2;44\)}}

% 13 vertices of branch B(9)
\put(-0.2,-1){\makebox(0,0)[rc]{branch}}
\put(0.5,-1){\makebox(0,0)[cc]{\(\mathcal{B}(9)\)}}
\put(-6.1,-1.9){\makebox(0,0)[rb]{\(-\#1;5\)}}
\put(-4.1,-1.9){\makebox(0,0)[rb]{\(4\)}}
\put(-2.1,-1.9){\makebox(0,0)[rb]{\(2\)}}
\put(-1.0,-1.8){\makebox(0,0)[rb]{\(6..13\)}}
\put(-0.1,-1.9){\makebox(0,0)[rb]{\(1\)}}
\put(2.1,-1.9){\makebox(0,0)[lb]{\(3\)}}

% 25 vertices of branch B(10)
\put(0.5,-3){\makebox(0,0)[cc]{\(\mathcal{B}(10)\)}}
\put(-6.1,-3.9){\makebox(0,0)[rb]{\(-\#1;2\vert 9\)}}
\put(-4.2,-3.9){\makebox(0,0)[rb]{\(3\vert 8\)}}
\put(-2.1,-3.9){\makebox(0,0)[rb]{\(5\vert 6\)}}
\put(-1.0,-3.8){\makebox(0,0)[rb]{\(10..25\)}}
\put(-0.1,-3.9){\makebox(0,0)[rb]{\(4\)}}
\put(2.1,-3.9){\makebox(0,0)[lb]{\(1\vert 7\)}}

% 25 vertices of depth 2 of branch B(9)
\put(4,-3.9){\makebox(0,0)[lb]{\(7..15\)}}
\put(6,-3.9){\makebox(0,0)[lb]{\(1..6\vert 16..25\)}}

% 13 vertices of branch B(11)
\put(0.5,-5){\makebox(0,0)[cc]{\(\mathcal{B}(11)\)}}
\put(-6.1,-5.9){\makebox(0,0)[rb]{\(-\#1;3\)}}
\put(-4.1,-5.9){\makebox(0,0)[rb]{\(2\)}}
\put(-2.1,-5.9){\makebox(0,0)[rb]{\(5\)}}
\put(-1.0,-5.8){\makebox(0,0)[rb]{\(6..13\)}}
\put(-0.1,-5.9){\makebox(0,0)[rb]{\(4\)}}
\put(2.1,-5.9){\makebox(0,0)[lb]{\(1\)}}

% 13 vertices of depth 2 of branch B(10)
\put(3.1,-5.9){\makebox(0,0)[lb]{\(7\)}}
\put(4.1,-5.9){\makebox(0,0)[lb]{\(4..6\vert 8\)}}
\put(6.1,-5.9){\makebox(0,0)[lb]{\(1..3\vert 9..13\)}}

% 25 vertices of branch B(12)
\put(0.5,-7){\makebox(0,0)[cc]{\(\mathcal{B}(12)\)}}
\put(-6.1,-7.9){\makebox(0,0)[rb]{\(-\#1;2\vert 9\)}}
\put(-4.1,-7.9){\makebox(0,0)[rb]{\(3\vert 8\)}}
\put(-2.1,-7.9){\makebox(0,0)[rb]{\(5\vert 6\)}}
\put(-1.0,-7.8){\makebox(0,0)[rb]{\(10..25\)}}
\put(-0.1,-7.9){\makebox(0,0)[rb]{\(4\)}}
\put(2.1,-7.9){\makebox(0,0)[lb]{\(1\vert 7\)}}

% 25 vertices of depth 2 of branch B(11)
\put(4,-7.9){\makebox(0,0)[lb]{\(7..15\)}}
\put(6,-7.9){\makebox(0,0)[lb]{\(1..6\vert 16..25\)}}

% 13 vertices of branch B(13)
\put(0.5,-9){\makebox(0,0)[cc]{\(\mathcal{B}(13)\)}}
\put(-6.1,-9.9){\makebox(0,0)[rb]{\(-\#1;3\)}}
\put(-4.1,-9.9){\makebox(0,0)[rb]{\(2\)}}
\put(-2.1,-9.9){\makebox(0,0)[rb]{\(5\)}}
\put(-1.0,-9.8){\makebox(0,0)[rb]{\(6..13\)}}
\put(-0.1,-9.9){\makebox(0,0)[rb]{\(4\)}}
\put(2.1,-9.9){\makebox(0,0)[lb]{\(1\)}}

% 13 vertices of depth 2 of branch B(12)
\put(3.1,-9.9){\makebox(0,0)[lb]{\(7\)}}
\put(4.1,-9.9){\makebox(0,0)[lb]{\(4..6\vert 8\)}}
\put(6.1,-9.9){\makebox(0,0)[lb]{\(1..3\vert 9..13\)}}

% 25 vertices of branch B(14)
\put(0.5,-11){\makebox(0,0)[cc]{\(\mathcal{B}(14)\)}}
\put(-6.1,-11.9){\makebox(0,0)[rb]{\(-\#1;2\vert 9\)}}
\put(-4.1,-11.9){\makebox(0,0)[rb]{\(3\vert 8\)}}
\put(-2.1,-11.9){\makebox(0,0)[rb]{\(5\vert 6\)}}
\put(-1.0,-11.8){\makebox(0,0)[rb]{\(10..25\)}}
\put(-0.1,-11.9){\makebox(0,0)[rb]{\(4\)}}
\put(2.1,-11.9){\makebox(0,0)[lb]{\(1\vert 7\)}}

% 25 vertices of depth 2 of branch B(13)
\put(4,-11.9){\makebox(0,0)[lb]{\(7..15\)}}
\put(6,-11.9){\makebox(0,0)[lb]{\(1..6\vert 16..25\)}}

% isomorphism classes of coclass families
\put(0,-12.4){\makebox(0,0)[ct]{\(\stackbin[0]{n-3,n}{G}\begin{pmatrix}0&1\\ 0&1\end{pmatrix}\)}}
\put(-2,-12.4){\makebox(0,0)[ct]{\(\stackbin[0]{n-3,n}{G}\begin{pmatrix}\pm 1&1\\ \mp 1&1\end{pmatrix}\)}}
\put(-4.1,-12.4){\makebox(0,0)[ct]{\(\stackbin[0]{n-3,n}{G}\begin{pmatrix}\pm 1&1\\ 0&1\end{pmatrix}\)}}
\put(-6.2,-12.4){\makebox(0,0)[ct]{\(\stackbin[0]{n-3,n}{G}\begin{pmatrix}0&1\\ \pm 1&1\end{pmatrix}\)}}
\put(2,-12.4){\makebox(0,0)[ct]{\(\stackbin[0]{n-3,n}{G}\begin{pmatrix}\pm 1&1\\ \pm 1&1\end{pmatrix}\)}}
\put(3,-13.4){\makebox(0,0)[ct]{\(\stackbin[\pm 1]{n-3,n}{G}\begin{pmatrix}0&1\\ 0&1\end{pmatrix}\)}}
\put(4,-12.4){\makebox(0,0)[ct]{\(\stackbin[\pm 1]{n-3,n}{G}\begin{pmatrix}\alpha&1\\ \gamma&1\end{pmatrix}\)}}

% TKT of coclass families
\put(-8,-15){\makebox(0,0)[cc]{\textbf{TKT}}}
\put(0,-15){\makebox(0,0)[cc]{d.19\({}^\ast\)}}
\put(-1,-15){\makebox(0,0)[cc]{d.19}}
\put(-2,-15){\makebox(0,0)[cc]{F.7}}
\put(-4,-15){\makebox(0,0)[cc]{F.12}}
\put(-6,-15){\makebox(0,0)[cc]{F.13}}
\put(2,-15){\makebox(0,0)[cc]{H.4}}
\put(3,-15){\makebox(0,0)[cc]{H.4}}
\put(4,-15){\makebox(0,0)[cc]{H.4}}
\put(6,-15){\makebox(0,0)[cc]{H.4}}
\put(-8,-15.5){\makebox(0,0)[cc]{\(\varkappa=\)}}
\put(0,-15.5){\makebox(0,0)[cc]{\((0443)\)}}
\put(-1,-15.5){\makebox(0,0)[cc]{\((0443)\)}}
\put(-2,-15.5){\makebox(0,0)[cc]{\((3443)\)}}
\put(-4,-15.5){\makebox(0,0)[cc]{\((1443)\)}}
\put(-6,-15.5){\makebox(0,0)[cc]{\((2443)\)}}
\put(2,-15.5){\makebox(0,0)[cc]{\((4443)\)}}
\put(3,-15.5){\makebox(0,0)[cc]{\((4443)\)}}
\put(4,-15.5){\makebox(0,0)[cc]{\((4443)\)}}
\put(6,-15.5){\makebox(0,0)[cc]{\((4443)\)}}
\put(-8.7,-15.7){\framebox(15.4,1){}}

\end{picture}

}

\end{figure}

%\newpage

%--------------------------------------------------------------------------------

\begin{theorem}
\label{thm:TKTd19Tree1Cc4}
\textbf{(Graph theoretic and algebraic invariants.)} \\
The coclass-\(4\) tree \(\mathcal{T}:=\mathcal{T}^4{R_2^4}\) of \(3\)-groups \(G\) with coclass \(\mathrm{cc}(G)=4\)
which arises from the metabelian root \(R_2^4:=\langle 2187,64\rangle-\#2;39\)
has the following abstract graph theoretic properties.
\begin{enumerate}
\item
The branches \(\mathcal{B}(i)\), \(i\ge n_\ast=9\), are purely periodic with
primitive period \((\mathcal{B}(9),\mathcal{B}(10))\) of length \(\ell=2\).
\item
The cardinalities of the periodic branches are
\(\#\mathcal{B}(9)=38\) and \(\#\mathcal{B}(10)=51\).
\item
Depth, width, and information content of the tree are given by
\begin{equation}
\label{eqn:TKTd19Tree1Cc4}
\mathrm{dp}(\mathcal{T}^4{R_2^4})=2, \quad \mathrm{wd}(\mathcal{T}^4{R_2^4})=50, \quad \text{ and } \quad \mathrm{IC}(\mathcal{T}^4{R_2^4})=89.
\end{equation}
\end{enumerate}
The algebraic invariants of the vertices forming the primitive period \((\mathcal{B}(9),\mathcal{B}(10))\) of the tree are given in Table
\ref{tbl:TKTd19TreeCc4}.
The six leading branches \(\mathcal{B}(9),\ldots,\mathcal{B}(14)\) are drawn in Figure
\ref{fig:TKTd19Tree1Cc4}.
\end{theorem}

%--------------------------------------------------------------------------------

\begin{proof}
(of Theorem
\ref{thm:TKTd19Tree1Cc4})
Since every mainline vertex \(m_n\) of the tree \(\mathcal{T}\)
has several capable children, \(C_1(m_n)\ge 2\),
but every capable vertex \(v\) of depth \(1\)
has only terminal children, \(C_1(v)=0\),
according to Proposition
\ref{prp:TKTd19Tree1Cc4},
the depth of the tree is \(\mathrm{dp}(\mathcal{T})=2\).
In this case,
the cardinality of a branch \(\mathcal{B}\) is the sum of
the number \(N_1(m_n)\) of immediate descendants of the branch root \(m_n\)
and the numbers \(N_1(v_i)\) of terminal children of capable vertices \(v_i\) of depth \(1\)
with \(2\le i\le C_1(m_n)\) (excluding the next mainline vertex \(v_1=m_{n+1}\)),
according to Formula
\eqref{eqn:BranchDepth2},
that is,
\[\#\mathcal{B}=N_1(m_n)+\sum_{i=2}^{C_1(m_n)}\,N_1(v_i).\]
Applied to the primitive period, this yields
\(\#\mathcal{B}(9)=13+25=38\), \(\#\mathcal{B}(10)=25+13+13=51\). 
According to Formula
\eqref{eqn:TreeWidth2},
the width of the tree is the maximum of all sums of the shape
\[\#\mathrm{Lyr}_n{\mathcal{T}}=\#\lbrace v\in\mathcal{T}\mid\mathrm{lo}(v)=n\rbrace=N_1(m_{n-1})+\sum_{i=1}^{C_1(m_{n-2})}\,N_1(v_i(m_{n-2})),\]
taken over all branch roots \(m_n\) with logarithmic orders \(n_\ast+2\le n\le n_\ast+\ell_\ast+\ell+1\).
Applied to \(n_\ast=9\), \(\ell_\ast=0\), and \(\ell=2\), this yields
\(\mathrm{wd}(\mathcal{T})=\max(25+25,13+13+13)=\max(50,39)=50\). \\
Finally, we have \(\mathrm{IC}(\mathcal{T})=\#\mathcal{B}(9)+\#\mathcal{B}(10)=38+51=89\).
\end{proof}

%\newpage

%--------------------------------------------------------------------------------

\begin{corollary}
\label{cor:TKTd19Tree1Cc4}
\textbf{(Actions and relation ranks.)}
The algebraic invariants of the vertices of the structured coclass-\(4\) tree \(\mathcal{T}^4{R_2^4}\) are listed in Table
\ref{tbl:TKTd19TreeCc4}. In particular:
\begin{enumerate}
\item
There are no groups with \(V_4\)-action.
\item
Two distinguished terminal metabelian vertices of depth \(2\) with even class and type \(\mathrm{H}.4\),
all terminal vertices of depth \(1\) with odd class,
and the mainline vertices with even class,
possess an RI-action.
\item
The relation rank is given by
\(\mu=5\) for the mainline vertices \(\stackbin[0]{n-3,n}{G}\begin{pmatrix}0&-1\\ 0&1\end{pmatrix}\) with \(n\ge 9\),
and the capable vertices \(\stackbin[0]{n-3,n}{G}\begin{pmatrix}\pm 1&-1\\ \mp 1&1\end{pmatrix}\) of depth \(1\) with \(n\ge 10\),
and
\(\mu=4\) otherwise.
\end{enumerate}
\end{corollary}

%--------------------------------------------------------------------------------

\begin{proof}
(of Corollary
\ref{cor:TKTd19Tree1Cc4})
The existence of an RI-action on \(G\) has been checked by means of
an algorithm involving the \(p\)-covering group of \(G\), written for MAGMA
\cite{MAGMA}.
The other claims follow immediately from Table
\ref{tbl:TKTd19TreeCc4},
continued indefinitely with the aid of the periodicity in Prop.
\ref{prp:TKTd19Tree1Cc4}. 
\end{proof}

%\newpage

%--------------------------------------------------------------------------------

\begin{theorem}
\label{thm:TKTd19Tree2Cc4}
\textbf{(Strict isomorphism of the two trees.)} \\
Viewed as an algebraically structured infinite digraph,
the second coclass-\(4\) tree \(\mathcal{T}^4(\langle 2187,64\rangle-\#2;44)\) with mainline of type \(\mathrm{d}.19^\ast\)
in Figure
\ref{fig:TKTd19Tree2Cc4}
is \textbf{strictly} isomorphic to
the first coclass-\(4\) tree \(\mathcal{T}^4(\langle 2187,64\rangle-\#2;39)\) with mainline of type \(\mathrm{d}.19^\ast\)
in Figure
\ref{fig:TKTd19Tree1Cc4}.
Only the presentations of corresponding vertices are different,
but they share common algebraic invariants.
\end{theorem}

%--------------------------------------------------------------------------------

\begin{proof}
(Proof of Theorem
\ref{thm:TKTd19Tree1Cc4}
and Theorem
\ref{thm:TKTd19Tree2Cc4})
The claims have been verified with the aid of MAGMA
\cite{MAGMA}
for all vertices \(v\) with logarithmic orders \(9\le\mathrm{lo}(v)\le 31\).
Pure periodicity of branches with primitive length \(2\)
sets in from the very beginning with \(\mathcal{B}(9)\simeq\mathcal{B}(11)\).
There is no pre-period.
Thus, the claims for all vertices \(v\) with logarithmic orders \(\mathrm{lo}(v)\ge 32\)
are a consequence of the virtual periodicity theorems by du Sautoy in
\cite[Thm. 1.11, p. 68, and Thm. 8.3, p. 103]{dS}
and by Eick and Leedham-Green in
\cite[Thm. 6, p. 277, Thm. 9, p. 278, and Thm. 29, p. 287]{EkLg},
without the need of pruning the depth, which is bounded uniformly by \(2\).
\end{proof}

%\newpage

%--------------------------------------------------------------------------------

\subsection{The unique mainline of type \(\mathrm{d}.23^\ast\) for even coclass \(r\ge 4\)}
\label{ss:d23Cc4}

\begin{proposition}
\label{prp:TKTd19and23TreeCc4}
\textbf{(A special nearly strict isomorphism.)} \\
Viewed as an algebraically structured infinite digraph,
the unique coclass-\(4\) tree \(\mathcal{T}^4(\langle 2187,64\rangle-\#2;54)\) with mainline of type \(\mathrm{d}.23^\ast\)
in Figure
\ref{fig:TKTd23TreeCc4}
is \textbf{almost strictly} isomorphic to
the first coclass-\(4\) tree \(\mathcal{T}^4(\langle 2187,64\rangle-\#2;39)\) with mainline of type \(\mathrm{d}.19^\ast\)
in Figure
\ref{fig:TKTd19Tree1Cc4},
and thus also to
the second coclass-\(4\) tree \(\mathcal{T}^4(\langle 2187,64\rangle-\#2;44)\) with mainline of type \(\mathrm{d}.19^\ast\)
in Figure
\ref{fig:TKTd19Tree2Cc4}.
Only the presentations of corresponding vertices are different,
but they share common algebraic invariants,
with the \textbf{transfer kernel types as single exception}:
the nearly strict isomorphism of directed trees maps
\(\mathrm{d}.23^\ast\) \(\mapsto\) \(\mathrm{d}.19^\ast\),
\(\mathrm{G}.16\) \(\mapsto\) \(\mathrm{H}.4\),
\(\mathrm{F}.11\) \(\mapsto\) \(\mathrm{F}.7\), and
\(\mathrm{F}.12\) either remains fixed or
\(\mathrm{F}.12\) \(\mapsto\) \(\mathrm{F}.13\).
\end{proposition}

\begin{proof}
This follows immediately from comparing Table 
\ref{tbl:TKTd23TreeCc4}
with Table
\ref{tbl:TKTd19TreeCc4},
and using periodicity.
\end{proof}

%\newpage

%--------------------------------------------------------------------------------

\renewcommand{\arraystretch}{1.2}

\begin{table}[ht]
\caption{Data for \(3\)-groups \(G\) with \(9\le n=\mathrm{lo}(G)\le 12\) of the coclass tree \(\mathcal{T}^4{R_4^4}\)}
\label{tbl:TKTd23TreeCc4}
\begin{center}
\begin{tabular}{|c|c|l||c|c|c||c|c||c|c|l|c||c|c|}
\hline
 \(\#\) & \(m,n\)  & \(\rho;\alpha,\beta,\gamma,\delta\) & dp & dl & \(\zeta\) & \(\mu\) & \(\nu\) & \(\tau(1)\) & \(\tau_2\) & Type                   & \(\varkappa\) & \(\sigma\) & \(\#\mathrm{Aut}\) \\
\hline
  \(1\) & \(6,9\)  & \(0;0,0,0,1\)                       & 0  & 2  & \(1^2\)   & \(5\)   & \(1\)   & \(32\)      & \(2^31\)   & \(\mathrm{d}.23^\ast\) & \((0243)\)    & \(1\)      & \(2\cdot 3^{14}\)  \\
\hline
  \(1\) & \(7,10\) & \(0;0,0,0,1\)                       & 0  & 2  & \(1^2\)   & \(5\)   & \(1\)   & \(3^2\)     & \(32^21\)  & \(\mathrm{d}.23^\ast\) & \((0243)\)    & \(1^\ast\) & \(2\cdot 3^{16}\)  \\
  \(1\) & \(7,10\) & \(0;1,0,1,1\)                       & 1  & 2  & \(1^2\)   & \(4\)   & \(0\)   & \(3^2\)     & \(32^21\)  & \(\mathrm{F}.11\)      & \((2243)\)    & \(0\)      & \(3^{16}\)         \\
  \(1\) & \(7,10\) & \(0;1,0,0,1\)                       & 1  & 2  & \(1^2\)   & \(4\)   & \(0\)   & \(3^2\)     & \(32^21\)  & \(\mathrm{F}.12\)      & \((3243)\)    & \(0\)      & \(3^{16}\)         \\
  \(1\) & \(7,10\) & \(0;0,0,1,1\)                       & 1  & 2  & \(1^2\)   & \(4\)   & \(0\)   & \(3^2\)     & \(32^21\)  & \(\mathrm{F}.12\)      & \((4243)\)    & \(0\)      & \(3^{16}\)         \\
  \(1\) & \(7,10\) & \(0;1,0,-1,1\)                      & 1  & 2  & \(1^2\)   & \(5\)   & \(1\)   & \(3^2\)     & \(32^21\)  & \(\mathrm{G}.16\)      & \((1243)\)    & \(0\)      & \(3^{16}\)         \\
  \(6\) & \(7,10\) &                                     & 1  & 3  & \(1^2\)   & \(4\)   & \(0\)   & \(32\)      & \(2^31\)   & \(\mathrm{d}.23\)      & \((0243)\)    & \(0\)      & \(3^{15}\)         \\
  \(2\) & \(7,10\) &                                     & 1  & 3  & \(1^2\)   & \(4\)   & \(0\)   & \(32\)      & \(2^31\)   & \(\mathrm{d}.23\)      & \((0243)\)    & \(0\)      & \(3^{14}\)         \\
  \(9\) & \(8,11\) & \(\pm 1;\alpha,0,\gamma,1\)         & 2  & 2  & \(1\)     & \(4\)   & \(0\)   & \(3^2\)     & \(3^221\)  & \(\mathrm{G}.16\)      & \((1243)\)    & \(0\)      & \(3^{18}\)         \\
 \(12\) & \(8,11\) &                                     & 2  & 3  & \(1\)     & \(4\)   & \(0\)   & \(3^2\)     & \(32^21\)  & \(\mathrm{G}.16\)      & \((1243)\)    & \(0\)      & \(3^{17}\)         \\
  \(4\) & \(8,11\) &                                     & 2  & 3  & \(1\)     & \(4\)   & \(0\)   & \(3^2\)     & \(32^21\)  & \(\mathrm{G}.16\)      & \((1243)\)    & \(0\)      & \(3^{16}\)         \\
\hline
  \(1\) & \(8,11\) & \(0;0,0,0,1\)                       & 0  & 2  & \(1^2\)   & \(5\)   & \(1\)   & \(43\)      & \(3^221\)  & \(\mathrm{d}.23^\ast\) & \((0243)\)    & \(1\)      & \(2\cdot 3^{18}\)  \\
  \(2\) & \(8,11\) & \(0;0,0,\pm 1,1\)                   & 1  & 2  & \(1^2\)   & \(4\)   & \(0\)   & \(43\)      & \(3^221\)  & \(\mathrm{F}.11\)      & \((2243)\)    & \(1^\ast\) & \(2\cdot 3^{18}\)  \\
  \(2\) & \(8,11\) & \(0;1,0,\pm 1,1\)                   & 1  & 2  & \(1^2\)   & \(4\)   & \(0\)   & \(43\)      & \(3^221\)  & \(\mathrm{F}.12\)      & \((3243)\)    & \(1^\ast\) & \(2\cdot 3^{18}\)  \\
  \(2\) & \(8,11\) & \(0;-1,0,\pm 1,1\)                  & 1  & 2  & \(1^2\)   & \(4\)   & \(0\)   & \(43\)      & \(3^221\)  & \(\mathrm{F}.12\)      & \((4243)\)    & \(1^\ast\) & \(2\cdot 3^{18}\)  \\
  \(2\) & \(8,11\) & \(0;\pm 1,0,0,1\)                   & 1  & 2  & \(1^2\)   & \(5\)   & \(1\)   & \(43\)      & \(3^221\)  & \(\mathrm{G}.16\)      & \((1243)\)    & \(1\)      & \(2\cdot 3^{18}\)  \\
 \(12\) & \(8,11\) &                                     & 1  & 3  & \(1^2\)   & \(4\)   & \(0\)   & \(3^2\)     & \(32^21\)  & \(\mathrm{d}.23\)      & \((0243)\)    & \(1^\ast\) & \(2\cdot 3^{17}\)  \\
  \(4\) & \(8,11\) &                                     & 1  & 3  & \(1^2\)   & \(4\)   & \(0\)   & \(3^2\)     & \(32^21\)  & \(\mathrm{d}.23\)      & \((0243)\)    & \(1^\ast\) & \(2\cdot 3^{16}\)  \\
  \(2\) & \(9,12\) & \(\pm 1;\alpha,0,\gamma,1\)         & 2  & 2  & \(1\)     & \(4\)   & \(0\)   & \(43\)      & \(4321\)   & \(\mathrm{G}.16\)      & \((1243)\)    & \(1^\ast\) & \(2\cdot 3^{20}\)  \\
  \(8\) & \(9,12\) & \(\pm 1;\alpha,0,\gamma,1\)         & 2  & 2  & \(1\)     & \(4\)   & \(0\)   & \(43\)      & \(4321\)   & \(\mathrm{G}.16\)      & \((1243)\)    & \(0\)      & \(3^{20}\)         \\
 \(12\) & \(9,12\) &                                     & 2  & 3  & \(1\)     & \(4\)   & \(0\)   & \(43\)      & \(3^221\)  & \(\mathrm{G}.16\)      & \((1243)\)    & \(0\)      & \(3^{19}\)         \\
  \(4\) & \(9,12\) &                                     & 2  & 3  & \(1\)     & \(4\)   & \(0\)   & \(43\)      & \(3^221\)  & \(\mathrm{G}.16\)      & \((1243)\)    & \(0\)      & \(3^{18}\)         \\
\hline
\end{tabular}
\end{center}
\end{table}

%--------------------------------------------------------------------------------

\begin{proposition}
\label{prp:TKTd23TreeCc4}
\textbf{(Periodicity and descendant numbers.)} \\
The branches \(\mathcal{B}(i)\), \(i\ge n_\ast=9\), of
the unique coclass-\(4\) tree \(\mathcal{T}^4(\langle 2187,64\rangle-\#2;54)\)
with mainline vertices of transfer kernel type \(\mathrm{d}.23^\ast\), \(\varkappa\sim (0243)\),
are purely periodic with primitive length \(\ell=2\) and without pre-period, \(\ell_\ast=0\), that is,
\(\mathcal{B}(i+2)\simeq\mathcal{B}(i)\) are isomorphic as digraphs, for all \(i\ge p_\ast=n_\ast+\ell_\ast=9\).

The graph theoretic structure of the tree is determined uniquely by the numbers
\(N_1\) of immediate descendants and \(C_1\) of capable immediate descendants
of the mainline vertices \(m_n\) with logarithmic order \(n=\mathrm{lo}(m_n)\ge n_\ast=9\)
and of capable vertices \(v\) with depth \(1\) and \(\mathrm{lo}(v)\ge n_\ast+1=10\): \\
\((N_1,C_1)=(13,2)\) for mainline vertices \(m_n\) with odd logarithmic order \(n\ge 9\), \\
\((N_1,C_1)=(25,3)\) for mainline vertices \(m_n\) with even logarithmic order \(n\ge 10\), \\
\((N_1,C_1)=(25,0)\) for the capable vertex \(v\) of depth \(1\) and even logarithmic order \(\mathrm{lo}(v)\ge 10\), \\
\((N_1,C_1)=(13,0)\) for two capable vertices \(v\) of depth \(1\) and odd logarithmic order \(\mathrm{lo}(v)\ge 11\).
\end{proposition}

\begin{proof}
This is a consequence of Proposition
\ref{prp:TKTd19and23TreeCc4}
together with Proposition
\ref{prp:TKTd19Tree1Cc4}.
The tree \(\mathcal{T}^4{R_4^4}\) corresponds to
the infinite metabelian pro-\(3\) group \(S_{3,6}\) in
\cite[Cnj. 15 (b), p. 116]{Ek}.
\end{proof}

%\newpage

%--------------------------------------------------------------------------------

\begin{figure}[ht]
\caption{The unique coclass-\(4\) tree \(\mathcal{T}^4(P_7-\#2;54)\) with mainline of type \(\mathrm{d}.23^\ast\)}
\label{fig:TKTd23TreeCc4}

% Unique Coclass-4 Tree with Mainline of Type d.23 in the Coclass Forest F(4)
% Structured Algebraically by Derived Length, Central Rank, and GI-Action

{\tiny

\setlength{\unitlength}{0.9cm}
\begin{picture}(18,20)(-10,-16)

% scale of orders
\put(-10,0.5){\makebox(0,0)[cb]{order \(3^n\)}}

\put(-10,0){\line(0,-1){12}}
\multiput(-10.1,0)(0,-2){7}{\line(1,0){0.2}}

\put(-10.2,0){\makebox(0,0)[rc]{\(19\,683\)}}
\put(-9.8,0){\makebox(0,0)[lc]{\(3^9\)}}
\put(-10.2,-2){\makebox(0,0)[rc]{\(59\,049\)}}
\put(-9.8,-2){\makebox(0,0)[lc]{\(3^{10}\)}}
\put(-10.2,-4){\makebox(0,0)[rc]{\(177\,147\)}}
\put(-9.8,-4){\makebox(0,0)[lc]{\(3^{11}\)}}
\put(-10.2,-6){\makebox(0,0)[rc]{\(531\,441\)}}
\put(-9.8,-6){\makebox(0,0)[lc]{\(3^{12}\)}}
\put(-10.2,-8){\makebox(0,0)[rc]{\(1\,594\,323\)}}
\put(-9.8,-8){\makebox(0,0)[lc]{\(3^{13}\)}}
\put(-10.2,-10){\makebox(0,0)[rc]{\(4\,782\,969\)}}
\put(-9.8,-10){\makebox(0,0)[lc]{\(3^{14}\)}}
\put(-10.2,-12){\makebox(0,0)[rc]{\(14\,348\,907\)}}
\put(-9.8,-12){\makebox(0,0)[lc]{\(3^{15}\)}}

\put(-10,-12){\vector(0,-1){2}}

% polarized transfer target
\put(-8,0.5){\makebox(0,0)[cb]{\(\tau(1)=\)}}

\put(-8,0){\makebox(0,0)[cc]{\((32)\)}}
\put(-8,-2){\makebox(0,0)[cc]{\((3^2)\)}}
\put(-8,-4){\makebox(0,0)[cc]{\((43)\)}}
\put(-8,-6){\makebox(0,0)[cc]{\((4^2)\)}}
\put(-8,-8){\makebox(0,0)[cc]{\((54)\)}}
\put(-8,-10){\makebox(0,0)[cc]{\((5^2)\)}}
\put(-8,-12){\makebox(0,0)[cc]{\((65)\)}}

\put(-8,-13){\makebox(0,0)[cc]{\textbf{TTT}}}
\put(-8.5,-13.2){\framebox(1,14){}}

% legend
\put(2.5,3.5){\makebox(0,0)[lc]{Legend:}}

\put(1,3){\makebox(0,0)[lc]{\(\mathrm{dl}=2\):}}
%\put(1,3){\circle{0.2}}
%\put(1,3){\circle*{0.1}}
\put(1,2.5){\circle*{0.2}}
\put(1,2){\circle{0.2}}
\put(1,1.5){\circle*{0.1}}
\put(1,1){\circle{0.1}}
%\put(1.5,3){\makebox(0,0)[lc]{\(\zeta=1^2,\sigma=2\)}}
\put(1.5,2.5){\makebox(0,0)[lc]{\(\zeta=1^2,\sigma=1\)}}
\put(1.5,2){\makebox(0,0)[lc]{\(\zeta=1^2,\sigma=0\)}}
\put(1.5,1.5){\makebox(0,0)[lc]{\(\zeta=1,\sigma=1\)}}
\put(1.5,1){\makebox(0,0)[lc]{\(\zeta=1,\sigma=0\)}}

\put(4,3){\makebox(0,0)[lc]{\(\mathrm{dl}=3\):}}
%\put(3.9,2.9){\framebox(0.2,0.2){\vrule height 3pt width 3pt}}
\put(3.9,2.4){\framebox(0.2,0.2){}}
\put(3.95,2.45){\framebox(0.1,0.1){}}
\put(3.9,1.9){\framebox(0.2,0.2){}}
\put(3.95,1.45){\framebox(0.1,0.1){}}
%\put(4.5,3){\makebox(0,0)[lc]{\(\zeta=1^2,\sigma=2\)}}
\put(4.5,2.5){\makebox(0,0)[lc]{\(\zeta=1^2,\sigma=1\)}}
\put(4.5,2){\makebox(0,0)[lc]{\(\zeta=1^2,\sigma=0\)}}
\put(4.5,1.5){\makebox(0,0)[lc]{\(\zeta=1,\sigma=0\)}}

% depth of branches
\put(5,-2){\vector(0,1){2}}
\put(5.1,-2){\makebox(0,0)[lc]{depth \(2\)}}
\put(5,-2){\vector(0,-1){2}}

% periodicity of branches
\put(-5.1,-4){\vector(0,1){2}}
\put(-5.3,-3){\makebox(0,0)[rc]{period length \(2\)}}
\put(-5.1,-4){\vector(0,-1){2}}

% metabelian GI-vertices with central rank 2
\multiput(0,0)(0,-2){7}{\circle*{0.2}}
\multiput(2,-4)(0,-4){3}{\circle*{0.2}}
\multiput(-2,-4)(0,-4){3}{\circle*{0.2}}
\multiput(-4,-4)(0,-4){3}{\circle*{0.2}}
\multiput(-6,-4)(0,-4){3}{\circle*{0.2}}
% metabelian non-GI vertices with central rank 2
\multiput(-2,-2)(0,-4){3}{\circle{0.2}}
\multiput(-4,-2)(0,-4){3}{\circle{0.2}}
\multiput(-6,-2)(0,-4){3}{\circle{0.2}}
\multiput(2,-2)(0,-4){3}{\circle{0.2}}

% non-metabelian GI-vertices
\multiput(-1.1,-4.1)(0,-4){3}{\framebox(0.2,0.2){}}
\multiput(-1.05,-4.05)(0,-4){3}{\framebox(0.1,0.1){}}
% non-metabelian non-GI vertices
\multiput(-1.1,-2.1)(0,-4){3}{\framebox(0.2,0.2){}}

% metabelian GI-vertices with central rank 1
\multiput(3,-6)(0,-4){2}{\circle*{0.1}}
% metabelian non-GI vertices with central rank 1
\multiput(4,-4)(0,-2){5}{\circle{0.1}}

% non-metabelian non-GI vertices
\multiput(5.95,-4.05)(0,-2){5}{\framebox(0.1,0.1){}}

% directed edges to vertices of depth 1
\multiput(0,0)(0,-2){6}{\line(0,-1){2}}
\multiput(0,0)(0,-2){6}{\line(-1,-2){1}}
\multiput(0,0)(0,-2){6}{\line(-1,-1){2}}
\multiput(0,0)(0,-2){6}{\line(-2,-1){4}}
\multiput(0,0)(0,-2){6}{\line(-3,-1){6}}
\multiput(0,0)(0,-2){6}{\line(1,-1){2}}

% directed edges to vertices of depth 2
\multiput(2,-4)(0,-4){2}{\line(1,-2){1}}
\multiput(2,-2)(0,-2){5}{\line(1,-1){2}}
\multiput(2,-2)(0,-2){5}{\line(2,-1){4}}

% multiplicity counters
\multiput(-1.1,-2.1)(0,-4){3}{\makebox(0,0)[rt]{\(8\ast\)}}
\multiput(-6.1,-4.1)(0,-4){3}{\makebox(0,0)[rt]{\(2\ast\)}}
\multiput(-4.1,-4.1)(0,-4){3}{\makebox(0,0)[rt]{\(2\ast\)}}
\multiput(-2.1,-4.1)(0,-4){3}{\makebox(0,0)[rt]{\(2\ast\)}}
\multiput(-1.1,-4.1)(0,-4){3}{\makebox(0,0)[rt]{\(16\ast\)}}
\multiput(1.9,-4.1)(0,-4){3}{\makebox(0,0)[rt]{\(2\ast\)}}
\multiput(3.9,-4.1)(0,-4){3}{\makebox(0,0)[rt]{\(9\ast\)}}
\multiput(3.9,-6.1)(0,-4){2}{\makebox(0,0)[rt]{\(4\ast\)}}
\multiput(5.9,-4.1)(0,-4){3}{\makebox(0,0)[rt]{\(16\ast\)}}
\multiput(5.9,-6.1)(0,-4){2}{\makebox(0,0)[rt]{\(8\ast\)}}

% infinite mainline
\put(0,-12){\vector(0,-1){2}}
\put(0.2,-13.6){\makebox(0,0)[lc]{infinite}}
\put(0.2,-14){\makebox(0,0)[lc]{mainline}}
\put(0.2,-14.4){\makebox(0,0)[lc]{\(\mathcal{T}^4(\langle 2187,64\rangle-\#2;54)\)}}

% root
\put(-0.1,0.1){\makebox(0,0)[rb]{\(\langle 2187,64\rangle-\#2;54\)}}

% 13 vertices of branch B(9)
\put(-0.2,-1){\makebox(0,0)[rc]{branch}}
\put(0.5,-1){\makebox(0,0)[cc]{\(\mathcal{B}(9)\)}}
\put(-6.1,-1.9){\makebox(0,0)[rb]{\(-\#1;7\)}}
\put(-4.1,-1.9){\makebox(0,0)[rb]{\(6\)}}
\put(-2.1,-1.9){\makebox(0,0)[rb]{\(4\)}}
\put(-0.8,-1.6){\makebox(0,0)[rb]{\(1..3\vert\)}}
\put(-1.0,-1.8){\makebox(0,0)[rb]{\(9..13\)}}
\put(-0.1,-1.9){\makebox(0,0)[rb]{\(8\)}}
\put(2.1,-1.9){\makebox(0,0)[lb]{\(5\)}}

% 25 vertices of branch B(10)
\put(0.5,-3){\makebox(0,0)[cc]{\(\mathcal{B}(10)\)}}
\put(-6.1,-3.9){\makebox(0,0)[rb]{\(-\#1;4\vert 6\)}}
\put(-4.2,-3.9){\makebox(0,0)[rb]{\(2\vert 8\)}}
\put(-2.1,-3.9){\makebox(0,0)[rb]{\(1\vert 9\)}}
\put(-1.0,-3.8){\makebox(0,0)[rb]{\(10..25\)}}
\put(-0.1,-3.9){\makebox(0,0)[rb]{\(5\)}}
\put(2.1,-3.9){\makebox(0,0)[lb]{\(3\vert 7\)}}

% 25 vertices of depth 2 of branch B(9)
\put(4,-3.9){\makebox(0,0)[lb]{\(7..15\)}}
\put(6,-3.9){\makebox(0,0)[lb]{\(1..6\vert 16..25\)}}

% 13 vertices of branch B(11)
\put(0.5,-5){\makebox(0,0)[cc]{\(\mathcal{B}(11)\)}}
\put(-6.1,-5.9){\makebox(0,0)[rb]{\(-\#1;2\)}}
\put(-4.1,-5.9){\makebox(0,0)[rb]{\(3\)}}
\put(-2.1,-5.9){\makebox(0,0)[rb]{\(4\)}}
\put(-1.0,-5.8){\makebox(0,0)[rb]{\(6..13\)}}
\put(-0.1,-5.9){\makebox(0,0)[rb]{\(1\)}}
\put(2.1,-5.9){\makebox(0,0)[lb]{\(5\)}}

% 13 vertices of depth 2 of branch B(10)
\put(3.1,-5.9){\makebox(0,0)[lb]{\(7\)}}
\put(4.1,-5.9){\makebox(0,0)[lb]{\(4..6\vert 8\)}}
\put(6.1,-5.9){\makebox(0,0)[lb]{\(1..3\vert 9..13\)}}

% 25 vertices of branch B(12)
\put(0.5,-7){\makebox(0,0)[cc]{\(\mathcal{B}(12)\)}}
\put(-6.1,-7.9){\makebox(0,0)[rb]{\(-\#1;2\vert 3\)}}
\put(-4.1,-7.9){\makebox(0,0)[rb]{\(4\vert 7\)}}
\put(-2.1,-7.9){\makebox(0,0)[rb]{\(5\vert 9\)}}
\put(-1.0,-7.8){\makebox(0,0)[rb]{\(10..25\)}}
\put(-0.1,-7.9){\makebox(0,0)[rb]{\(1\)}}
\put(2.1,-7.9){\makebox(0,0)[lb]{\(6\vert 8\)}}

% 25 vertices of depth 2 of branch B(11)
\put(4,-7.9){\makebox(0,0)[lb]{\(7..15\)}}
\put(6,-7.9){\makebox(0,0)[lb]{\(1..6\vert 16..25\)}}

% 13 vertices of branch B(13)
\put(0.5,-9){\makebox(0,0)[cc]{\(\mathcal{B}(13)\)}}
\put(-6.1,-9.9){\makebox(0,0)[rb]{\(-\#1;2\)}}
\put(-4.1,-9.9){\makebox(0,0)[rb]{\(3\)}}
\put(-2.1,-9.9){\makebox(0,0)[rb]{\(4\)}}
\put(-1.0,-9.8){\makebox(0,0)[rb]{\(6..13\)}}
\put(-0.1,-9.9){\makebox(0,0)[rb]{\(1\)}}
\put(2.1,-9.9){\makebox(0,0)[lb]{\(5\)}}

% 13 vertices of depth 2 of branch B(12)
\put(3.1,-9.9){\makebox(0,0)[lb]{\(7\)}}
\put(4.1,-9.9){\makebox(0,0)[lb]{\(4..6\vert 8\)}}
\put(6.1,-9.9){\makebox(0,0)[lb]{\(1..3\vert 9..13\)}}

% 25 vertices of branch B(14)
\put(0.5,-11){\makebox(0,0)[cc]{\(\mathcal{B}(14)\)}}
\put(-6.1,-11.9){\makebox(0,0)[rb]{\(-\#1;2\vert 3\)}}
\put(-4.1,-11.9){\makebox(0,0)[rb]{\(4\vert 7\)}}
\put(-2.1,-11.9){\makebox(0,0)[rb]{\(5\vert 9\)}}
\put(-1.0,-11.8){\makebox(0,0)[rb]{\(10..25\)}}
\put(-0.1,-11.9){\makebox(0,0)[rb]{\(1\)}}
\put(2.1,-11.9){\makebox(0,0)[lb]{\(6\vert 8\)}}

% 25 vertices of depth 2 of branch B(13)
\put(4,-11.9){\makebox(0,0)[lb]{\(7..15\)}}
\put(6,-11.9){\makebox(0,0)[lb]{\(1..6\vert 16..25\)}}

% isomorphism classes of coclass families
\put(-0.1,-12.4){\makebox(0,0)[ct]{\(\stackbin[0]{n-3,n}{G}\begin{pmatrix}0&0\\ 0&1\end{pmatrix}\)}}
\put(-2.1,-12.4){\makebox(0,0)[ct]{\(\stackbin[0]{n-3,n}{G}\begin{pmatrix}0&0\\ \pm 1&1\end{pmatrix}\)}}
\put(-4.2,-12.4){\makebox(0,0)[ct]{\(\stackbin[0]{n-3,n}{G}\begin{pmatrix}\pm 1&0\\ 1&1\end{pmatrix}\)}}
\put(-6.3,-12.4){\makebox(0,0)[ct]{\(\stackbin[0]{n-3,n}{G}\begin{pmatrix}\pm 1&0\\ -1&1\end{pmatrix}\)}}
\put(2,-12.4){\makebox(0,0)[ct]{\(\stackbin[0]{n-3,n}{G}\begin{pmatrix}\pm 1&0\\ 0&1\end{pmatrix}\)}}
\put(4,-12.4){\makebox(0,0)[ct]{\(\stackbin[\pm 1]{n-3,n}{G}\begin{pmatrix}\alpha&0\\ \gamma&1\end{pmatrix}\)}}

% TKT of coclass families
\put(-8,-15){\makebox(0,0)[cc]{\textbf{TKT}}}
\put(0,-15){\makebox(0,0)[cc]{d.23\({}^\ast\)}}
\put(-1,-15){\makebox(0,0)[cc]{d.23}}
\put(-2,-15){\makebox(0,0)[cc]{F.11}}
\put(-4,-15){\makebox(0,0)[cc]{F.12}}
\put(-6,-15){\makebox(0,0)[cc]{F.12}}
\put(2,-15){\makebox(0,0)[cc]{G.16}}
\put(3,-15){\makebox(0,0)[cc]{G.16}}
\put(4,-15){\makebox(0,0)[cc]{G.16}}
\put(6,-15){\makebox(0,0)[cc]{G.16}}
\put(-8,-15.5){\makebox(0,0)[cc]{\(\varkappa=\)}}
\put(0,-15.5){\makebox(0,0)[cc]{\((0243)\)}}
\put(-1,-15.5){\makebox(0,0)[cc]{\((0243)\)}}
\put(-2,-15.5){\makebox(0,0)[cc]{\((2243)\)}}
\put(-4,-15.5){\makebox(0,0)[cc]{\((3243)\)}}
\put(-6,-15.5){\makebox(0,0)[cc]{\((4243)\)}}
\put(2,-15.5){\makebox(0,0)[cc]{\((1243)\)}}
\put(3,-15.5){\makebox(0,0)[cc]{\((1243)\)}}
\put(4,-15.5){\makebox(0,0)[cc]{\((1243)\)}}
\put(6,-15.5){\makebox(0,0)[cc]{\((1243)\)}}
\put(-8.7,-15.7){\framebox(15.4,1){}}

\end{picture}

}

\end{figure}
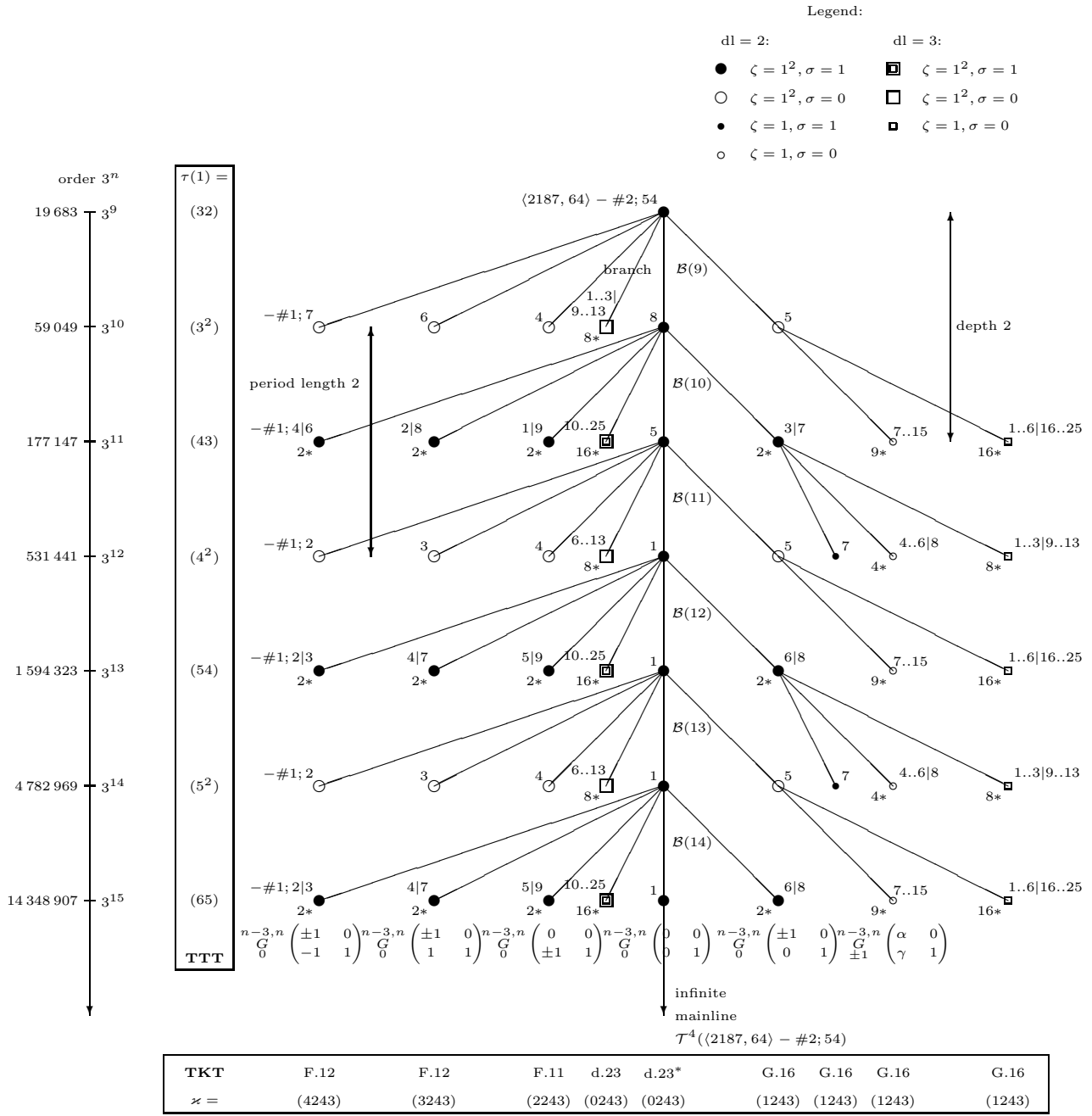

%\newpage

%--------------------------------------------------------------------------------

\begin{theorem}
\label{thm:TKTd23TreeCc4}
\textbf{(Graph theoretic and algebraic invariants.)} \\
The coclass-\(4\) tree \(\mathcal{T}:=\mathcal{T}^4{R_4^4}\) of \(3\)-groups \(G\) with coclass \(\mathrm{cc}(G)=4\)
which arises from the metabelian root \(R_4^4:=\langle 2187,64\rangle-\#2;54\)
has the following abstract graph theoretic properties.
\begin{enumerate}
\item
The branches \(\mathcal{B}(i)\), \(i\ge n_\ast=9\), are purely periodic with
primitive period \((\mathcal{B}(9),\mathcal{B}(10))\) of length \(\ell=2\).
\item
The cardinalities of the periodic branches are
\(\#\mathcal{B}(9)=38\) and \(\#\mathcal{B}(10)=51\).
\item
Depth, width, and information content of the tree are given by
\begin{equation}
\label{eqn:TKTd23TreeCc4}
\mathrm{dp}(\mathcal{T}^4{R_4^4})=2, \quad \mathrm{wd}(\mathcal{T}^4{R_4^4})=50, \quad \text{ and } \quad \mathrm{IC}(\mathcal{T}^4{R_4^4})=89.
\end{equation}
\end{enumerate}
The algebraic invariants of the vertices forming the primitive period \((\mathcal{B}(9),\mathcal{B}(10))\) of the tree are given in Table
\ref{tbl:TKTd23TreeCc4}.
The six leading branches \(\mathcal{B}(9),\ldots,\mathcal{B}(14)\) are drawn in Figure
\ref{fig:TKTd23TreeCc4}.
\end{theorem}

%--------------------------------------------------------------------------------

\begin{proof}
Since the tree \(\mathcal{T}^4(\langle 2187,64\rangle-\#2;54)\)
is isomorphic to the tree \(\mathcal{T}^4(\langle 2187,64\rangle-\#2;39)\)
as an abstract digraph,
the proof literally coincides with the proof of Theorem
\ref{thm:TKTd19Tree1Cc4}.
\end{proof}

%--------------------------------------------------------------------------------

\begin{corollary}
\label{cor:TKTd23TreeCc4}
\textbf{(Actions and relation ranks.)}
The algebraic invariants of the vertices of the structured coclass-\(4\) tree \(\mathcal{T}^4{R_4^4}\) are listed in Table
\ref{tbl:TKTd23TreeCc4}. In particular:
\begin{enumerate}
\item
There are no groups with \(V_4\)-action.
\item
Two distinguished terminal metabelian vertices of depth \(2\) with even class and type \(\mathrm{G}.16\),
all terminal vertices of depth \(1\) with odd class,
and the mainline vertices with even class,
possess an RI-action.
\item
The relation rank is given by
\(\mu=5\) for the mainline vertices \(\stackbin[0]{n-3,n}{G}\begin{pmatrix}0&0\\ 0&1\end{pmatrix}\) with \(n\ge 9\),
and the capable vertices \(\stackbin[0]{n-3,n}{G}\begin{pmatrix}\pm 1&0\\ 0&1\end{pmatrix}\) of depth \(1\) with \(n\ge 10\),
and
\(\mu=4\) otherwise.
\end{enumerate}
\end{corollary}

%\newpage

%--------------------------------------------------------------------------------

\renewcommand{\arraystretch}{1.2}

\begin{table}[ht]
%\caption{Root and primitive period \((\mathcal{B}(9),\mathcal{B}(10))\) of \(\mathcal{T}^4(P_7-\#2;57)\)}
\caption{Data for \(3\)-groups \(G\) with \(9\le n=\mathrm{lo}(G)\le 12\) of the coclass tree \(\mathcal{T}^4{R_5^4}\)}
\label{tbl:TKTd25TreeCc4}
\begin{center}
\begin{tabular}{|c|c|l||c|c|c||c|c||c|c|l|c||c|c|}
\hline
 \(\#\) & \(m,n\)  & \(\rho;\alpha,\beta,\gamma,\delta\) & dp & dl & \(\zeta\) & \(\mu\) & \(\nu\) & \(\tau(1)\) & \(\tau_2\) & Type                   & \(\varkappa\) & \(\sigma\) & \(\#\mathrm{Aut}\)  \\
\hline
  \(1\) & \(6,9\)  & \(0;0,1,0,0\)                       & 0  & 2  & \(1^2\)   & \(5\)   & \(1\)   & \(32\)      & \(2^31\)   & \(\mathrm{d}.25^\ast\) & \((0143)\)    & \(2\)      & \(2^2\cdot 3^{14}\) \\
\hline
  \(1\) & \(7,10\) & \(0;0,1,0,0\)                       & 0  & 2  & \(1^2\)   & \(5\)   & \(1\)   & \(3^2\)     & \(32^21\)  & \(\mathrm{d}.25^\ast\) & \((0143)\)    & \(2^\ast\) & \(2^2\cdot 3^{16}\) \\
  \(1\) & \(7,10\) & \(0;1,1,0,0\)                       & 1  & 2  & \(1^2\)   & \(4\)   & \(0\)   & \(3^2\)     & \(32^21\)  & \(\mathrm{F}.11\)      & \((1143)\)    & \(0\)      & \(3^{16}\)          \\
  \(1\) & \(7,10\) & \(0;1,1,1,0\)                       & 1  & 2  & \(1^2\)   & \(4\)   & \(0\)   & \(3^2\)     & \(32^21\)  & \(\mathrm{F}.13\)      & \((3143)\)    & \(0\)      & \(3^{16}\)          \\
  \(1\) & \(7,10\) & \(0;0,1,1,0\)                       & 1  & 2  & \(1^2\)   & \(5\)   & \(1\)   & \(3^2\)     & \(32^21\)  & \(\mathrm{G}.19\)      & \((2143)\)    & \(0\)      & \(3^{16}\)          \\
  \(3\) & \(7,10\) &                                     & 1  & 3  & \(1^2\)   & \(4\)   & \(0\)   & \(32\)      & \(2^31\)   & \(\mathrm{d}.25\)      & \((0143)\)    & \(0\)      & \(3^{15}\)          \\
  \(2\) & \(7,10\) &                                     & 1  & 3  & \(1^2\)   & \(4\)   & \(0\)   & \(32\)      & \(2^31\)   & \(\mathrm{d}.25\)      & \((0143)\)    & \(0\)      & \(3^{14}\)          \\
  \(6\) & \(8,11\) & \(\pm 1;\alpha,1,\gamma,0\)         & 2  & 2  & \(1\)     & \(4\)   & \(0\)   & \(3^2\)     & \(3^221\)  & \(\mathrm{G}.19\)      & \((2143)\)    & \(0\)      & \(3^{18}\)          \\
  \(7\) & \(8,11\) &                                     & 2  & 3  & \(1\)     & \(4\)   & \(0\)   & \(3^2\)     & \(32^21\)  & \(\mathrm{G}.19\)      & \((2143)\)    & \(0\)      & \(3^{17}\)          \\
  \(2\) & \(8,11\) &                                     & 2  & 3  & \(1\)     & \(4\)   & \(0\)   & \(3^2\)     & \(32^21\)  & \(\mathrm{G}.19\)      & \((2143)\)    & \(0\)      & \(3^{16}\)          \\
\hline
  \(1\) & \(8,11\) & \(0;0,1,0,0\)                       & 0  & 2  & \(1^2\)   & \(5\)   & \(1\)   & \(43\)      & \(3^221\)  & \(\mathrm{d}.25^\ast\) & \((0143)\)    & \(2\)      & \(2^2\cdot 3^{18}\) \\
  \(1\) & \(8,11\) & \(0;1,1,0,0\)                       & 1  & 2  & \(1^2\)   & \(4\)   & \(0\)   & \(43\)      & \(3^221\)  & \(\mathrm{F}.11\)      & \((1143)\)    & \(1^\ast\) & \(2\cdot 3^{18}\)   \\
  \(2\) & \(8,11\) & \(0;1,1,\pm 1,0\)                   & 1  & 2  & \(1^2\)   & \(4\)   & \(0\)   & \(43\)      & \(3^221\)  & \(\mathrm{F}.13\)      & \((3143)\)    & \(1^\ast\) & \(2\cdot 3^{18}\)   \\
  \(2\) & \(8,11\) & \(0;0,1,\pm 1,0\)                   & 1  & 2  & \(1^2\)   & \(5\)   & \(1\)   & \(43\)      & \(3^221\)  & \(\mathrm{G}.19\)      & \((2143)\)    & \(2\)      & \(2^2\cdot 3^{18}\) \\
  \(7\) & \(8,11\) &                                     & 1  & 3  & \(1^2\)   & \(4\)   & \(0\)   & \(3^2\)     & \(32^21\)  & \(\mathrm{d}.25\)      & \((0143)\)    & \(1^\ast\) & \(2\cdot 3^{17}\)   \\
  \(2\) & \(8,11\) &                                     & 1  & 3  & \(1^2\)   & \(4\)   & \(0\)   & \(3^2\)     & \(32^21\)  & \(\mathrm{d}.25\)      & \((0143)\)    & \(2^\ast\) & \(2^2\cdot 3^{16}\) \\
  \(2\) & \(9,12\) & \(\pm 1;0,1,\mp 1,0\)               & 2  & 2  & \(1\)     & \(4\)   & \(0\)   & \(43\)      & \(4321\)   & \(\mathrm{G}.19\)      & \((2143)\)    & \(2^\ast\) & \(2^2\cdot 3^{20}\) \\
  \(6\) & \(9,12\) & \(\pm 1;\alpha,1,\gamma,0\)         & 2  & 2  & \(1\)     & \(4\)   & \(0\)   & \(43\)      & \(4321\)   & \(\mathrm{G}.19\)      & \((2143)\)    & \(0\)      & \(3^{20}\)          \\
  \(8\) & \(9,12\) &                                     & 2  & 3  & \(1\)     & \(4\)   & \(0\)   & \(43\)      & \(3^221\)  & \(\mathrm{G}.19\)      & \((2143)\)    & \(0\)      & \(3^{19}\)          \\
  \(2\) & \(9,12\) &                                     & 2  & 3  & \(1\)     & \(4\)   & \(0\)   & \(43\)      & \(3^221\)  & \(\mathrm{G}.19\)      & \((2143)\)    & \(0\)      & \(3^{18}\)          \\
\hline
\end{tabular}
\end{center}
\end{table}

%\newpage

%--------------------------------------------------------------------------------

\subsection{Two mainlines of type \(\mathrm{d}.25^\ast\) for even coclass \(r\ge 4\)}
\label{ss:d25Cc4}

\begin{proposition}
\label{prp:TKTd25Tree1Cc4}
\textbf{(Periodicity and descendant numbers.)} \\
The branches \(\mathcal{B}(i)\), \(i\ge n_\ast=9\), of
the first coclass-\(4\) tree \(\mathcal{T}^4(\langle 2187,64\rangle-\#2;57)\)
with mainline vertices of transfer kernel type \(\mathrm{d}.25^\ast\), \(\varkappa\sim (0143)\),
are purely periodic with primitive length \(\ell=2\) and without pre-period, \(\ell_\ast=0\), that is,
\(\mathcal{B}(i+2)\simeq\mathcal{B}(i)\) are isomorphic as graphs, for all \(i\ge 9\).

The structure of the tree is determined uniquely by the numbers
\(N_1\) of immediate descendants and \(C_1\) of capable immediate descendants
for mainline vertices and for capable vertices of depth \(1\): \\
\((N_1,C_1)=(9,2)\) for mainline vertices \(m_n\) of odd logarithmic order \(n=\mathrm{lo}(m_n)\ge n_\ast=9\), \\
\((N_1,C_1)=(15,3)\) for mainline vertices \(m_n\) of even logarithmic order \(n=\mathrm{lo}(m_n)\ge 10\), \\
\((N_1,C_1)=(15,0)\) for a capable vertex \(v\) of depth \(1\) and even logarithmic order \(\mathrm{lo}(v)\ge n_\ast+1=10\), \\
\((N_1,C_1)=(9,0)\) for two capable vertices \(v\) of depth \(1\) and odd logarithmic order \(\mathrm{lo}(v)\ge 11\).
\end{proposition}

%\newpage

%--------------------------------------------------------------------------------

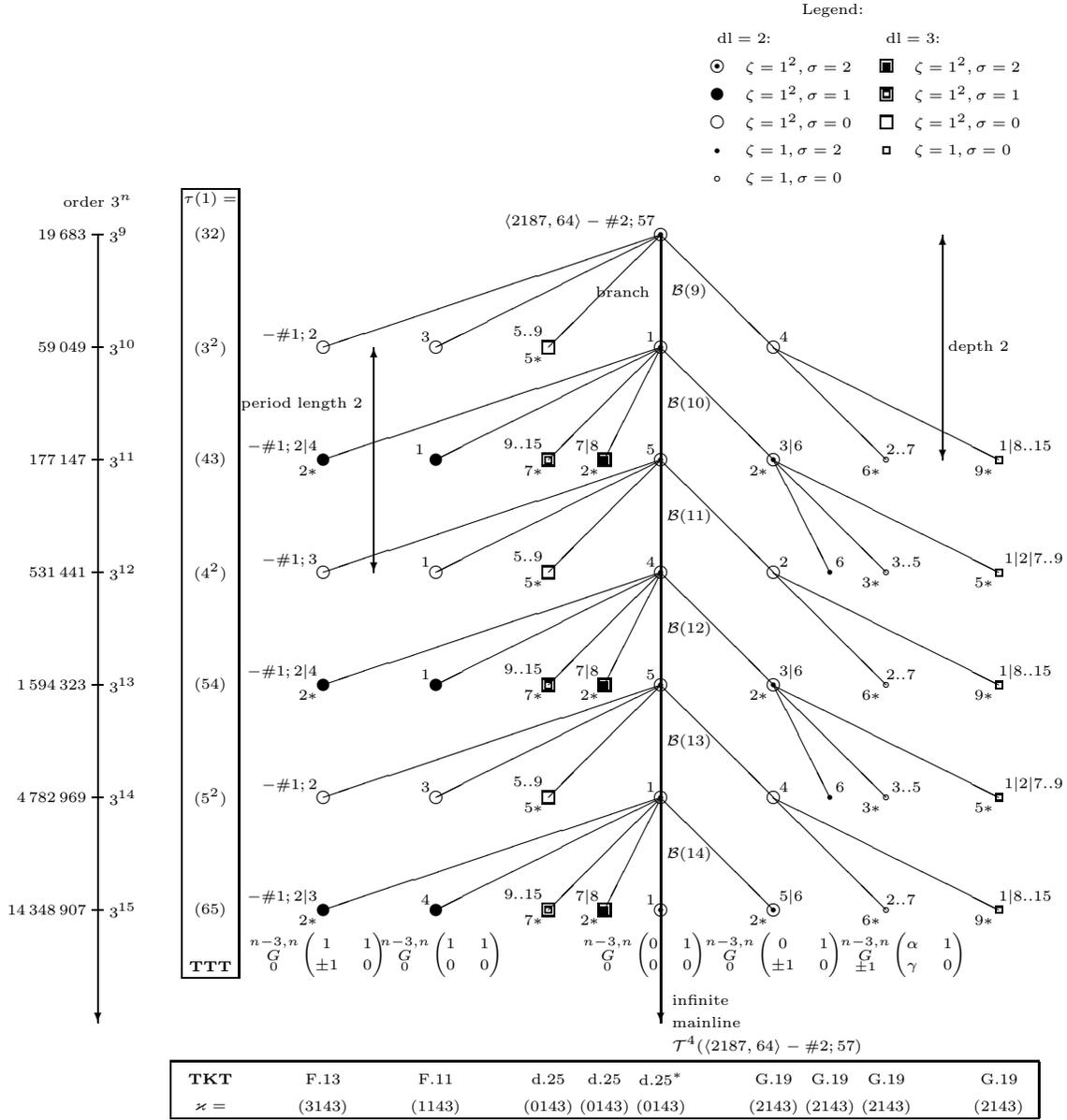
\begin{figure}[ht]
\caption{The first coclass-\(4\) tree \(\mathcal{T}^4(P_7-\#2;57)\) with mainline of type \(\mathrm{d}.25^\ast\)}
\label{fig:TKTd25Tree1Cc4}

% First Coclass-4 Tree with Mainline of Type d.25* in the Coclass Forest F(4)
% Structured Algebraically by Derived Length, Central Rank, and GI- or V4-Action

{\tiny

\setlength{\unitlength}{0.8cm}
\begin{picture}(18,20.5)(-10,-16)

% scale of orders
\put(-10,0.5){\makebox(0,0)[cb]{order \(3^n\)}}

\put(-10,0){\line(0,-1){12}}
\multiput(-10.1,0)(0,-2){7}{\line(1,0){0.2}}

\put(-10.2,0){\makebox(0,0)[rc]{\(19\,683\)}}
\put(-9.8,0){\makebox(0,0)[lc]{\(3^9\)}}
\put(-10.2,-2){\makebox(0,0)[rc]{\(59\,049\)}}
\put(-9.8,-2){\makebox(0,0)[lc]{\(3^{10}\)}}
\put(-10.2,-4){\makebox(0,0)[rc]{\(177\,147\)}}
\put(-9.8,-4){\makebox(0,0)[lc]{\(3^{11}\)}}
\put(-10.2,-6){\makebox(0,0)[rc]{\(531\,441\)}}
\put(-9.8,-6){\makebox(0,0)[lc]{\(3^{12}\)}}
\put(-10.2,-8){\makebox(0,0)[rc]{\(1\,594\,323\)}}
\put(-9.8,-8){\makebox(0,0)[lc]{\(3^{13}\)}}
\put(-10.2,-10){\makebox(0,0)[rc]{\(4\,782\,969\)}}
\put(-9.8,-10){\makebox(0,0)[lc]{\(3^{14}\)}}
\put(-10.2,-12){\makebox(0,0)[rc]{\(14\,348\,907\)}}
\put(-9.8,-12){\makebox(0,0)[lc]{\(3^{15}\)}}

\put(-10,-12){\vector(0,-1){2}}

% polarized transfer target
\put(-8,0.5){\makebox(0,0)[cb]{\(\tau(1)=\)}}

\put(-8,0){\makebox(0,0)[cc]{\((32)\)}}
\put(-8,-2){\makebox(0,0)[cc]{\((3^2)\)}}
\put(-8,-4){\makebox(0,0)[cc]{\((43)\)}}
\put(-8,-6){\makebox(0,0)[cc]{\((4^2)\)}}
\put(-8,-8){\makebox(0,0)[cc]{\((54)\)}}
\put(-8,-10){\makebox(0,0)[cc]{\((5^2)\)}}
\put(-8,-12){\makebox(0,0)[cc]{\((65)\)}}

\put(-8,-13){\makebox(0,0)[cc]{\textbf{TTT}}}
\put(-8.5,-13.2){\framebox(1,14){}}

% legend
\put(2.5,4){\makebox(0,0)[lc]{Legend:}}

\put(1,3.5){\makebox(0,0)[lc]{\(\mathrm{dl}=2\):}}
\put(1,3){\circle{0.2}}
\put(1,3){\circle*{0.1}}
\put(1,2.5){\circle*{0.2}}
\put(1,2){\circle{0.2}}
\put(1,1.5){\circle*{0.1}}
\put(1,1){\circle{0.1}}
\put(1.5,3){\makebox(0,0)[lc]{\(\zeta=1^2,\sigma=2\)}}
\put(1.5,2.5){\makebox(0,0)[lc]{\(\zeta=1^2,\sigma=1\)}}
\put(1.5,2){\makebox(0,0)[lc]{\(\zeta=1^2,\sigma=0\)}}
\put(1.5,1.5){\makebox(0,0)[lc]{\(\zeta=1,\sigma=2\)}}
\put(1.5,1){\makebox(0,0)[lc]{\(\zeta=1,\sigma=0\)}}

\put(4,3.5){\makebox(0,0)[lc]{\(\mathrm{dl}=3\):}}
\put(3.9,2.9){\framebox(0.2,0.2){\vrule height 3pt width 3pt}}
\put(3.9,2.4){\framebox(0.2,0.2){}}
\put(3.95,2.45){\framebox(0.1,0.1){}}
\put(3.9,1.9){\framebox(0.2,0.2){}}
\put(3.95,1.45){\framebox(0.1,0.1){}}
\put(4.5,3){\makebox(0,0)[lc]{\(\zeta=1^2,\sigma=2\)}}
\put(4.5,2.5){\makebox(0,0)[lc]{\(\zeta=1^2,\sigma=1\)}}
\put(4.5,2){\makebox(0,0)[lc]{\(\zeta=1^2,\sigma=0\)}}
\put(4.5,1.5){\makebox(0,0)[lc]{\(\zeta=1,\sigma=0\)}}

% depth of branches
\put(5,-2){\vector(0,1){2}}
\put(5.1,-2){\makebox(0,0)[lc]{depth \(2\)}}
\put(5,-2){\vector(0,-1){2}}

% periodicity of branches
\put(-5.1,-4){\vector(0,1){2}}
\put(-5.3,-3){\makebox(0,0)[rc]{period length \(2\)}}
\put(-5.1,-4){\vector(0,-1){2}}

% metabelian V4-vertices with central rank 2
\multiput(0,0)(0,-2){7}{\circle{0.2}}
\multiput(0,0)(0,-2){7}{\circle*{0.1}}
\multiput(2,-4)(0,-4){3}{\circle{0.2}}
\multiput(2,-4)(0,-4){3}{\circle*{0.1}}
% metabelian GI-vertices with central rank 2
\multiput(-4,-4)(0,-4){3}{\circle*{0.2}}
\multiput(-6,-4)(0,-4){3}{\circle*{0.2}}
% metabelian non-GI vertices with central rank 2
\multiput(-4,-2)(0,-4){3}{\circle{0.2}}
\multiput(-6,-2)(0,-4){3}{\circle{0.2}}
\multiput(2,-2)(0,-4){3}{\circle{0.2}}

% non-metabelian V4-vertices
\multiput(-1.1,-4.1)(0,-4){3}{\framebox(0.2,0.2){\vrule height 3pt width 3pt}}
% non-metabelian GI-vertices
\multiput(-2.1,-4.1)(0,-4){3}{\framebox(0.2,0.2){}}
\multiput(-2.05,-4.05)(0,-4){3}{\framebox(0.1,0.1){}}
% non-metabelian non-GI vertices
\multiput(-2.1,-2.1)(0,-4){3}{\framebox(0.2,0.2){}}

% metabelian V4-vertices with central rank 1
\multiput(3,-6)(0,-4){2}{\circle*{0.1}}

% metabelian non-GI vertices with central rank 1
\multiput(4,-4)(0,-2){5}{\circle{0.1}}

% non-metabelian non-GI vertices
\multiput(5.95,-4.05)(0,-2){5}{\framebox(0.1,0.1){}}

% directed edges to vertices of depth 1
\multiput(0,0)(0,-2){6}{\line(0,-1){2}}
\multiput(0,-2)(0,-4){3}{\line(-1,-2){1}}
\multiput(0,0)(0,-2){6}{\line(-1,-1){2}}
\multiput(0,0)(0,-2){6}{\line(-2,-1){4}}
\multiput(0,0)(0,-2){6}{\line(-3,-1){6}}
\multiput(0,0)(0,-2){6}{\line(1,-1){2}}

% directed edges to vertices of depth 2
\multiput(2,-4)(0,-4){2}{\line(1,-2){1}}
\multiput(2,-2)(0,-2){5}{\line(1,-1){2}}
\multiput(2,-2)(0,-2){5}{\line(2,-1){4}}

% multiplicity counters
\multiput(-2.1,-2.1)(0,-4){3}{\makebox(0,0)[rt]{\(5\ast\)}}
\multiput(-6.1,-4.1)(0,-4){3}{\makebox(0,0)[rt]{\(2\ast\)}}
\multiput(-2.1,-4.1)(0,-4){3}{\makebox(0,0)[rt]{\(7\ast\)}}
\multiput(-1.1,-4.1)(0,-4){3}{\makebox(0,0)[rt]{\(2\ast\)}}
\multiput(1.9,-4.1)(0,-4){3}{\makebox(0,0)[rt]{\(2\ast\)}}
\multiput(3.9,-4.1)(0,-4){3}{\makebox(0,0)[rt]{\(6\ast\)}}
\multiput(3.9,-6.1)(0,-4){2}{\makebox(0,0)[rt]{\(3\ast\)}}
\multiput(5.9,-4.1)(0,-4){3}{\makebox(0,0)[rt]{\(9\ast\)}}
\multiput(5.9,-6.1)(0,-4){2}{\makebox(0,0)[rt]{\(5\ast\)}}

% infinite mainline
\put(0,-12){\vector(0,-1){2}}
\put(0.2,-13.6){\makebox(0,0)[lc]{infinite}}
\put(0.2,-14){\makebox(0,0)[lc]{mainline}}
\put(0.2,-14.4){\makebox(0,0)[lc]{\(\mathcal{T}^4(\langle 2187,64\rangle-\#2;57)\)}}

% root
\put(-0.1,0.1){\makebox(0,0)[rb]{\(\langle 2187,64\rangle-\#2;57\)}}

% 9 vertices of branch B(9)
\put(-0.2,-1){\makebox(0,0)[rc]{branch}}
\put(0.5,-1){\makebox(0,0)[cc]{\(\mathcal{B}(9)\)}}
\put(-6.1,-1.9){\makebox(0,0)[rb]{\(-\#1;2\)}}
\put(-4.1,-1.9){\makebox(0,0)[rb]{\(3\)}}
\put(-2.1,-1.8){\makebox(0,0)[rb]{\(5..9\)}}
\put(-0.1,-1.9){\makebox(0,0)[rb]{\(1\)}}
\put(2.1,-1.9){\makebox(0,0)[lb]{\(4\)}}

% 15 vertices of branch B(10)
\put(0.5,-3){\makebox(0,0)[cc]{\(\mathcal{B}(10)\)}}
\put(-6.1,-3.9){\makebox(0,0)[rb]{\(-\#1;2\vert 4\)}}
\put(-4.2,-3.9){\makebox(0,0)[rb]{\(1\)}}
\put(-2.1,-3.8){\makebox(0,0)[rb]{\(9..15\)}}
\put(-1.1,-3.9){\makebox(0,0)[rb]{\(7\vert 8\)}}
\put(-0.1,-3.9){\makebox(0,0)[rb]{\(5\)}}
\put(2.1,-3.9){\makebox(0,0)[lb]{\(3\vert 6\)}}

% 15 vertices of branch B(9)
\put(4,-3.9){\makebox(0,0)[lb]{\(2..7\)}}
\put(6,-3.9){\makebox(0,0)[lb]{\(1\vert 8..15\)}}

% 9 vertices of branch B(11)
\put(0.5,-5){\makebox(0,0)[cc]{\(\mathcal{B}(11)\)}}
\put(-6.1,-5.9){\makebox(0,0)[rb]{\(-\#1;3\)}}
\put(-4.1,-5.9){\makebox(0,0)[rb]{\(1\)}}
\put(-2.1,-5.8){\makebox(0,0)[rb]{\(5..9\)}}
\put(-0.1,-5.9){\makebox(0,0)[rb]{\(4\)}}
\put(2.1,-5.9){\makebox(0,0)[lb]{\(2\)}}

% 9 vertices of branch B(10)
\put(3.1,-5.9){\makebox(0,0)[lb]{\(6\)}}
\put(4.1,-5.9){\makebox(0,0)[lb]{\(3..5\)}}
\put(6.1,-5.9){\makebox(0,0)[lb]{\(1\vert 2\vert 7..9\)}}

% 15 vertices of branch B(12)
\put(0.5,-7){\makebox(0,0)[cc]{\(\mathcal{B}(12)\)}}
\put(-6.1,-7.9){\makebox(0,0)[rb]{\(-\#1;2\vert 4\)}}
\put(-4.1,-7.9){\makebox(0,0)[rb]{\(1\)}}
\put(-2.1,-7.8){\makebox(0,0)[rb]{\(9..15\)}}
\put(-1.1,-7.9){\makebox(0,0)[rb]{\(7\vert 8\)}}
\put(-0.1,-7.9){\makebox(0,0)[rb]{\(5\)}}
\put(2.1,-7.9){\makebox(0,0)[lb]{\(3\vert 6\)}}

% 15 vertices of branch B(11)
\put(4,-7.9){\makebox(0,0)[lb]{\(2..7\)}}
\put(6,-7.9){\makebox(0,0)[lb]{\(1\vert 8..15\)}}

% 9 vertices of branch B(13)
\put(0.5,-9){\makebox(0,0)[cc]{\(\mathcal{B}(13)\)}}
\put(-6.1,-9.9){\makebox(0,0)[rb]{\(-\#1;2\)}}
\put(-4.1,-9.9){\makebox(0,0)[rb]{\(3\)}}
\put(-2.1,-9.8){\makebox(0,0)[rb]{\(5..9\)}}
\put(-0.1,-9.9){\makebox(0,0)[rb]{\(1\)}}
\put(2.1,-9.9){\makebox(0,0)[lb]{\(4\)}}

% 9 vertices of branch B(12)
\put(3.1,-9.9){\makebox(0,0)[lb]{\(6\)}}
\put(4.1,-9.9){\makebox(0,0)[lb]{\(3..5\)}}
\put(6.1,-9.9){\makebox(0,0)[lb]{\(1\vert 2\vert 7..9\)}}

% 15 vertices of branch B(14)
\put(0.5,-11){\makebox(0,0)[cc]{\(\mathcal{B}(14)\)}}
\put(-6.1,-11.9){\makebox(0,0)[rb]{\(-\#1;2\vert 3\)}}
\put(-4.1,-11.9){\makebox(0,0)[rb]{\(4\)}}
\put(-2.1,-11.8){\makebox(0,0)[rb]{\(9..15\)}}
\put(-1.1,-11.9){\makebox(0,0)[rb]{\(7\vert 8\)}}
\put(-0.1,-11.9){\makebox(0,0)[rb]{\(1\)}}
\put(2.1,-11.9){\makebox(0,0)[lb]{\(5\vert 6\)}}

% 15 vertices of branch B(13)
\put(4,-11.9){\makebox(0,0)[lb]{\(2..7\)}}
\put(6,-11.9){\makebox(0,0)[lb]{\(1\vert 8..15\)}}

% isomorphism classes of coclass families
\put(-0.3,-12.4){\makebox(0,0)[ct]{\(\stackbin[0]{n-3,n}{G}\begin{pmatrix}0&1\\ 0&0\end{pmatrix}\)}}
\put(-3.9,-12.4){\makebox(0,0)[ct]{\(\stackbin[0]{n-3,n}{G}\begin{pmatrix}1&1\\ 0&0\end{pmatrix}\)}}
\put(-6.1,-12.4){\makebox(0,0)[ct]{\(\stackbin[0]{n-3,n}{G}\begin{pmatrix}1&1\\ \pm 1&0\end{pmatrix}\)}}
\put(2,-12.4){\makebox(0,0)[ct]{\(\stackbin[0]{n-3,n}{G}\begin{pmatrix}0&1\\ \pm 1&0\end{pmatrix}\)}}
\put(4.3,-12.4){\makebox(0,0)[ct]{\(\stackbin[\pm 1]{n-3,n}{G}\begin{pmatrix}\alpha&1\\ \gamma&0\end{pmatrix}\)}}

% TKT of coclass families
\put(-8,-15){\makebox(0,0)[cc]{\textbf{TKT}}}
\put(0,-15){\makebox(0,0)[cc]{d.25\({}^\ast\)}}
\put(-1,-15){\makebox(0,0)[cc]{d.25}}
\put(-2,-15){\makebox(0,0)[cc]{d.25}}
\put(-4,-15){\makebox(0,0)[cc]{F.11}}
\put(-6,-15){\makebox(0,0)[cc]{F.13}}
\put(2,-15){\makebox(0,0)[cc]{G.19}}
\put(3,-15){\makebox(0,0)[cc]{G.19}}
\put(4,-15){\makebox(0,0)[cc]{G.19}}
\put(6,-15){\makebox(0,0)[cc]{G.19}}
\put(-8,-15.5){\makebox(0,0)[cc]{\(\varkappa=\)}}
\put(0,-15.5){\makebox(0,0)[cc]{\((0143)\)}}
\put(-1,-15.5){\makebox(0,0)[cc]{\((0143)\)}}
\put(-2,-15.5){\makebox(0,0)[cc]{\((0143)\)}}
\put(-4,-15.5){\makebox(0,0)[cc]{\((1143)\)}}
\put(-6,-15.5){\makebox(0,0)[cc]{\((3143)\)}}
\put(2,-15.5){\makebox(0,0)[cc]{\((2143)\)}}
\put(3,-15.5){\makebox(0,0)[cc]{\((2143)\)}}
\put(4,-15.5){\makebox(0,0)[cc]{\((2143)\)}}
\put(6,-15.5){\makebox(0,0)[cc]{\((2143)\)}}
\put(-8.7,-15.7){\framebox(15.4,1){}}

\end{picture}

}

\end{figure}

%\newpage

%--------------------------------------------------------------------------------

\begin{theorem}
\label{thm:TKTd25Tree1Cc4}
\textbf{(Graph theoretic and algebraic invariants.)} \\
The coclass-\(4\) tree \(\mathcal{T}:=\mathcal{T}^4{R_5^4}\) of \(3\)-groups \(G\) with coclass \(\mathrm{cc}(G)=4\)
which arises from the metabelian root \(R_5^4:=\langle 2187,64\rangle-\#2;57\)
has the following abstract graph theoretic properties.
\begin{enumerate}
\item
The branches \(\mathcal{B}(i)\), \(i\ge 9\), are purely periodic with
primitive period \((\mathcal{B}(9),\mathcal{B}(10))\) of length \(\ell=2\).
\item
The cardinalities of the periodic branches are
\(\#\mathcal{B}(9)=24\) and \(\#\mathcal{B}(10)=33\).
\item
Depth, width, and information content of the tree are given by
\begin{equation}
\label{eqn:}
\mathrm{dp}(\mathcal{T}^4{R_5^4})=2, \quad \mathrm{wd}(\mathcal{T}^4{R_5^4})=30, \quad \text{ and } \quad \mathrm{IC}(\mathcal{T}^4{R_5^4})=57.
\end{equation}
\end{enumerate}
The algebraic invariants of the vertices forming the primitive period \((\mathcal{B}(9),\mathcal{B}(10))\) of the tree
are presented in Table
\ref{tbl:TKTd25TreeCc4}.
The leading six branches \(\mathcal{B}(9),\ldots,\mathcal{B}(14)\) are drawn in Figure
\ref{fig:TKTd25Tree1Cc4}.
\end{theorem}

%\newpage

%--------------------------------------------------------------------------------

\begin{figure}[ht]
\caption{The second coclass tree \(\mathcal{T}^4(P_7-\#2;59)\) with mainline of type \(\mathrm{d}.25^\ast\)}
\label{fig:TKTd25Tree2Cc4}

% Second Coclass-4 Tree with Mainline of Type d.25* in the Coclass Forest F(4)
% Structured Algebraically by Derived Length, Central Rank, and GI- or V4-Action

{\tiny

\setlength{\unitlength}{0.8cm}
\begin{picture}(18,20.5)(-10,-16)

% scale of orders
\put(-10,0.5){\makebox(0,0)[cb]{order \(3^n\)}}

\put(-10,0){\line(0,-1){12}}
\multiput(-10.1,0)(0,-2){7}{\line(1,0){0.2}}

\put(-10.2,0){\makebox(0,0)[rc]{\(19\,683\)}}
\put(-9.8,0){\makebox(0,0)[lc]{\(3^9\)}}
\put(-10.2,-2){\makebox(0,0)[rc]{\(59\,049\)}}
\put(-9.8,-2){\makebox(0,0)[lc]{\(3^{10}\)}}
\put(-10.2,-4){\makebox(0,0)[rc]{\(177\,147\)}}
\put(-9.8,-4){\makebox(0,0)[lc]{\(3^{11}\)}}
\put(-10.2,-6){\makebox(0,0)[rc]{\(531\,441\)}}
\put(-9.8,-6){\makebox(0,0)[lc]{\(3^{12}\)}}
\put(-10.2,-8){\makebox(0,0)[rc]{\(1\,594\,323\)}}
\put(-9.8,-8){\makebox(0,0)[lc]{\(3^{13}\)}}
\put(-10.2,-10){\makebox(0,0)[rc]{\(4\,782\,969\)}}
\put(-9.8,-10){\makebox(0,0)[lc]{\(3^{14}\)}}
\put(-10.2,-12){\makebox(0,0)[rc]{\(14\,348\,907\)}}
\put(-9.8,-12){\makebox(0,0)[lc]{\(3^{15}\)}}

\put(-10,-12){\vector(0,-1){2}}

% polarized transfer target
\put(-8,0.5){\makebox(0,0)[cb]{\(\tau(1)=\)}}

\put(-8,0){\makebox(0,0)[cc]{\((32)\)}}
\put(-8,-2){\makebox(0,0)[cc]{\((3^2)\)}}
\put(-8,-4){\makebox(0,0)[cc]{\((43)\)}}
\put(-8,-6){\makebox(0,0)[cc]{\((4^2)\)}}
\put(-8,-8){\makebox(0,0)[cc]{\((54)\)}}
\put(-8,-10){\makebox(0,0)[cc]{\((5^2)\)}}
\put(-8,-12){\makebox(0,0)[cc]{\((65)\)}}

\put(-8,-13){\makebox(0,0)[cc]{\textbf{TTT}}}
\put(-8.5,-13.2){\framebox(1,14){}}

% legend
\put(2.5,4){\makebox(0,0)[lc]{Legend:}}

\put(1,3.5){\makebox(0,0)[lc]{\(\mathrm{dl}=2\):}}
\put(1,3){\circle{0.2}}
\put(1,3){\circle*{0.1}}
\put(1,2.5){\circle*{0.2}}
\put(1,2){\circle{0.2}}
\put(1,1.5){\circle*{0.1}}
\put(1,1){\circle{0.1}}
\put(1.5,3){\makebox(0,0)[lc]{\(\zeta=1^2,\sigma=2\)}}
\put(1.5,2.5){\makebox(0,0)[lc]{\(\zeta=1^2,\sigma=1\)}}
\put(1.5,2){\makebox(0,0)[lc]{\(\zeta=1^2,\sigma=0\)}}
\put(1.5,1.5){\makebox(0,0)[lc]{\(\zeta=1,\sigma=2\)}}
\put(1.5,1){\makebox(0,0)[lc]{\(\zeta=1,\sigma=0\)}}

\put(4,3.5){\makebox(0,0)[lc]{\(\mathrm{dl}=3\):}}
\put(3.9,2.9){\framebox(0.2,0.2){\vrule height 3pt width 3pt}}
\put(3.9,2.4){\framebox(0.2,0.2){}}
\put(3.95,2.45){\framebox(0.1,0.1){}}
\put(3.9,1.9){\framebox(0.2,0.2){}}
\put(3.95,1.45){\framebox(0.1,0.1){}}
\put(4.5,3){\makebox(0,0)[lc]{\(\zeta=1^2,\sigma=2\)}}
\put(4.5,2.5){\makebox(0,0)[lc]{\(\zeta=1^2,\sigma=1\)}}
\put(4.5,2){\makebox(0,0)[lc]{\(\zeta=1^2,\sigma=0\)}}
\put(4.5,1.5){\makebox(0,0)[lc]{\(\zeta=1,\sigma=0\)}}

% depth of branches
\put(5,-2){\vector(0,1){2}}
\put(5.1,-2){\makebox(0,0)[lc]{depth \(2\)}}
\put(5,-2){\vector(0,-1){2}}

% periodicity of branches
\put(-5.1,-4){\vector(0,1){2}}
\put(-5.3,-3){\makebox(0,0)[rc]{period length \(2\)}}
\put(-5.1,-4){\vector(0,-1){2}}

% metabelian V4-vertices with central rank 2
\multiput(0,0)(0,-2){7}{\circle{0.2}}
\multiput(0,0)(0,-2){7}{\circle*{0.1}}
\multiput(2,-4)(0,-4){3}{\circle{0.2}}
\multiput(2,-4)(0,-4){3}{\circle*{0.1}}
% metabelian GI-vertices with central rank 2
\multiput(-4,-4)(0,-4){3}{\circle*{0.2}}
\multiput(-6,-4)(0,-4){3}{\circle*{0.2}}
% metabelian non-GI vertices with central rank 2
\multiput(-4,-2)(0,-4){3}{\circle{0.2}}
\multiput(-6,-2)(0,-4){3}{\circle{0.2}}
\multiput(2,-2)(0,-4){3}{\circle{0.2}}

% non-metabelian V4-vertices
\multiput(-1.1,-4.1)(0,-4){3}{\framebox(0.2,0.2){\vrule height 3pt width 3pt}}
% non-metabelian GI-vertices
\multiput(-2.1,-4.1)(0,-4){3}{\framebox(0.2,0.2){}}
\multiput(-2.05,-4.05)(0,-4){3}{\framebox(0.1,0.1){}}
% non-metabelian non-GI vertices
\multiput(-2.1,-2.1)(0,-4){3}{\framebox(0.2,0.2){}}

% metabelian V4-vertices with central rank 1
\multiput(3,-6)(0,-4){2}{\circle*{0.1}}

% metabelian non-GI vertices with central rank 1
\multiput(4,-4)(0,-2){5}{\circle{0.1}}

% non-metabelian non-GI vertices
\multiput(5.95,-4.05)(0,-2){5}{\framebox(0.1,0.1){}}

% directed edges to vertices of depth 1
\multiput(0,0)(0,-2){6}{\line(0,-1){2}}
\multiput(0,-2)(0,-4){3}{\line(-1,-2){1}}
\multiput(0,0)(0,-2){6}{\line(-1,-1){2}}
\multiput(0,0)(0,-2){6}{\line(-2,-1){4}}
\multiput(0,0)(0,-2){6}{\line(-3,-1){6}}
\multiput(0,0)(0,-2){6}{\line(1,-1){2}}

% directed edges to vertices of depth 2
\multiput(2,-4)(0,-4){2}{\line(1,-2){1}}
\multiput(2,-2)(0,-2){5}{\line(1,-1){2}}
\multiput(2,-2)(0,-2){5}{\line(2,-1){4}}

% multiplicity counters
\multiput(-2.1,-2.1)(0,-4){3}{\makebox(0,0)[rt]{\(5\ast\)}}
\multiput(-6.1,-4.1)(0,-4){3}{\makebox(0,0)[rt]{\(2\ast\)}}
\multiput(-2.1,-4.1)(0,-4){3}{\makebox(0,0)[rt]{\(7\ast\)}}
\multiput(-1.1,-4.1)(0,-4){3}{\makebox(0,0)[rt]{\(2\ast\)}}
\multiput(1.9,-4.1)(0,-4){3}{\makebox(0,0)[rt]{\(2\ast\)}}
\multiput(3.9,-4.1)(0,-4){3}{\makebox(0,0)[rt]{\(6\ast\)}}
\multiput(3.9,-6.1)(0,-4){2}{\makebox(0,0)[rt]{\(3\ast\)}}
\multiput(5.9,-4.1)(0,-4){3}{\makebox(0,0)[rt]{\(9\ast\)}}
\multiput(5.9,-6.1)(0,-4){2}{\makebox(0,0)[rt]{\(5\ast\)}}

% infinite mainline
\put(0,-12){\vector(0,-1){2}}
\put(0.2,-13.6){\makebox(0,0)[lc]{infinite}}
\put(0.2,-14){\makebox(0,0)[lc]{mainline}}
\put(0.2,-14.4){\makebox(0,0)[lc]{\(\mathcal{T}^4(\langle 2187,64\rangle-\#2;59)\)}}

% root
\put(-0.1,0.1){\makebox(0,0)[rb]{\(\langle 2187,64\rangle-\#2;59\)}}

% 9 vertices of branch B(9)
\put(-0.2,-1){\makebox(0,0)[rc]{branch}}
\put(0.5,-1){\makebox(0,0)[cc]{\(\mathcal{B}(9)\)}}
\put(-6.1,-1.9){\makebox(0,0)[rb]{\(-\#1;4\)}}
\put(-4.1,-1.9){\makebox(0,0)[rb]{\(3\)}}
\put(-2.1,-1.9){\makebox(0,0)[rb]{\(1\vert 2\vert 7..9\)}}
\put(-0.1,-1.9){\makebox(0,0)[rb]{\(6\)}}
\put(2.1,-1.9){\makebox(0,0)[lb]{\(5\)}}

% 15 vertices of branch B(10)
\put(0.5,-3){\makebox(0,0)[cc]{\(\mathcal{B}(10)\)}}
\put(-6.1,-3.9){\makebox(0,0)[rb]{\(-\#1;3\vert 4\)}}
\put(-4.2,-3.9){\makebox(0,0)[rb]{\(1\)}}
\put(-2.1,-3.8){\makebox(0,0)[rb]{\(9..15\)}}
\put(-1.1,-3.9){\makebox(0,0)[rb]{\(7\vert 8\)}}
\put(-0.1,-3.9){\makebox(0,0)[rb]{\(5\)}}
\put(2.1,-3.9){\makebox(0,0)[lb]{\(2\vert 6\)}}

% 15 vertices of branch B(9)
\put(4,-3.9){\makebox(0,0)[lb]{\(2..7\)}}
\put(6,-3.9){\makebox(0,0)[lb]{\(1\vert 8..15\)}}

% 9 vertices of branch B(11)
\put(0.5,-5){\makebox(0,0)[cc]{\(\mathcal{B}(11)\)}}
\put(-6.1,-5.9){\makebox(0,0)[rb]{\(-\#1;1\)}}
\put(-4.1,-5.9){\makebox(0,0)[rb]{\(4\)}}
\put(-2.1,-5.8){\makebox(0,0)[rb]{\(5..9\)}}
\put(-0.1,-5.9){\makebox(0,0)[rb]{\(2\)}}
\put(2.1,-5.9){\makebox(0,0)[lb]{\(3\)}}

% 9 vertices of branch B(10)
\put(3.1,-5.9){\makebox(0,0)[lb]{\(6\)}}
\put(4.1,-5.9){\makebox(0,0)[lb]{\(3..5\)}}
\put(6.1,-5.9){\makebox(0,0)[lb]{\(1\vert 2\vert 7..9\)}}

% 15 vertices of branch B(12)
\put(0.5,-7){\makebox(0,0)[cc]{\(\mathcal{B}(12)\)}}
\put(-6.1,-7.9){\makebox(0,0)[rb]{\(-\#1;1\vert 3\)}}
\put(-4.1,-7.9){\makebox(0,0)[rb]{\(4\)}}
\put(-2.1,-7.8){\makebox(0,0)[rb]{\(9..15\)}}
\put(-1.1,-7.9){\makebox(0,0)[rb]{\(7\vert 8\)}}
\put(-0.1,-7.9){\makebox(0,0)[rb]{\(2\)}}
\put(2.1,-7.9){\makebox(0,0)[lb]{\(5\vert 6\)}}

% 15 vertices of branch B(11)
\put(4,-7.9){\makebox(0,0)[lb]{\(2..7\)}}
\put(6,-7.9){\makebox(0,0)[lb]{\(1\vert 8..15\)}}

% 9 vertices of branch B(13)
\put(0.5,-9){\makebox(0,0)[cc]{\(\mathcal{B}(13)\)}}
\put(-6.1,-9.9){\makebox(0,0)[rb]{\(-\#1;1\)}}
\put(-4.1,-9.9){\makebox(0,0)[rb]{\(4\)}}
\put(-2.1,-9.8){\makebox(0,0)[rb]{\(5..9\)}}
\put(-0.1,-9.9){\makebox(0,0)[rb]{\(2\)}}
\put(2.1,-9.9){\makebox(0,0)[lb]{\(3\)}}

% 9 vertices of branch B(12)
\put(3.1,-9.9){\makebox(0,0)[lb]{\(6\)}}
\put(4.1,-9.9){\makebox(0,0)[lb]{\(3..5\)}}
\put(6.1,-9.9){\makebox(0,0)[lb]{\(1\vert 2\vert 7..9\)}}

% 15 vertices of branch B(14)
\put(0.5,-11){\makebox(0,0)[cc]{\(\mathcal{B}(14)\)}}
\put(-6.1,-11.9){\makebox(0,0)[rb]{\(-\#1;1\vert 3\)}}
\put(-4.1,-11.9){\makebox(0,0)[rb]{\(4\)}}
\put(-2.1,-11.8){\makebox(0,0)[rb]{\(9..15\)}}
\put(-1.1,-11.9){\makebox(0,0)[rb]{\(7\vert 8\)}}
\put(-0.1,-11.9){\makebox(0,0)[rb]{\(2\)}}
\put(2.1,-11.9){\makebox(0,0)[lb]{\(5\vert 6\)}}

% 15 vertices of branch B(13)
\put(4,-11.9){\makebox(0,0)[lb]{\(2..7\)}}
\put(6,-11.9){\makebox(0,0)[lb]{\(1\vert 8..15\)}}

% isomorphism classes of coclass families
\put(-0.4,-12.4){\makebox(0,0)[ct]{\(\stackbin[0]{n-3,n}{G}\begin{pmatrix}0&-1\\ 0&0\end{pmatrix}\)}}
\put(-3.6,-12.4){\makebox(0,0)[ct]{\(\stackbin[0]{n-3,n}{G}\begin{pmatrix}1&-1\\ 0&0\end{pmatrix}\)}}
\put(-6.1,-12.4){\makebox(0,0)[ct]{\(\stackbin[0]{n-3,n}{G}\begin{pmatrix}1&-1\\ \pm 1&0\end{pmatrix}\)}}
\put(2.1,-12.4){\makebox(0,0)[ct]{\(\stackbin[0]{n-3,n}{G}\begin{pmatrix}0&-1\\ \pm 1&0\end{pmatrix}\)}}
\put(4.6,-12.4){\makebox(0,0)[ct]{\(\stackbin[\pm 1]{n-3,n}{G}\begin{pmatrix}\alpha&-1\\ \gamma&0\end{pmatrix}\)}}

% TKT of coclass families
\put(-8,-15){\makebox(0,0)[cc]{\textbf{TKT}}}
\put(0,-15){\makebox(0,0)[cc]{d.25\({}^\ast\)}}
\put(-1,-15){\makebox(0,0)[cc]{d.25}}
\put(-2,-15){\makebox(0,0)[cc]{d.25}}
\put(-4,-15){\makebox(0,0)[cc]{F.11}}
\put(-6,-15){\makebox(0,0)[cc]{F.13}}
\put(2,-15){\makebox(0,0)[cc]{G.19}}
\put(3,-15){\makebox(0,0)[cc]{G.19}}
\put(4,-15){\makebox(0,0)[cc]{G.19}}
\put(6,-15){\makebox(0,0)[cc]{G.19}}
\put(-8,-15.5){\makebox(0,0)[cc]{\(\varkappa=\)}}
\put(0,-15.5){\makebox(0,0)[cc]{\((0143)\)}}
\put(-1,-15.5){\makebox(0,0)[cc]{\((0143)\)}}
\put(-2,-15.5){\makebox(0,0)[cc]{\((0143)\)}}
\put(-4,-15.5){\makebox(0,0)[cc]{\((1143)\)}}
\put(-6,-15.5){\makebox(0,0)[cc]{\((3143)\)}}
\put(2,-15.5){\makebox(0,0)[cc]{\((2143)\)}}
\put(3,-15.5){\makebox(0,0)[cc]{\((2143)\)}}
\put(4,-15.5){\makebox(0,0)[cc]{\((2143)\)}}
\put(6,-15.5){\makebox(0,0)[cc]{\((2143)\)}}
\put(-8.7,-15.7){\framebox(15.4,1){}}

\end{picture}

}

\end{figure}
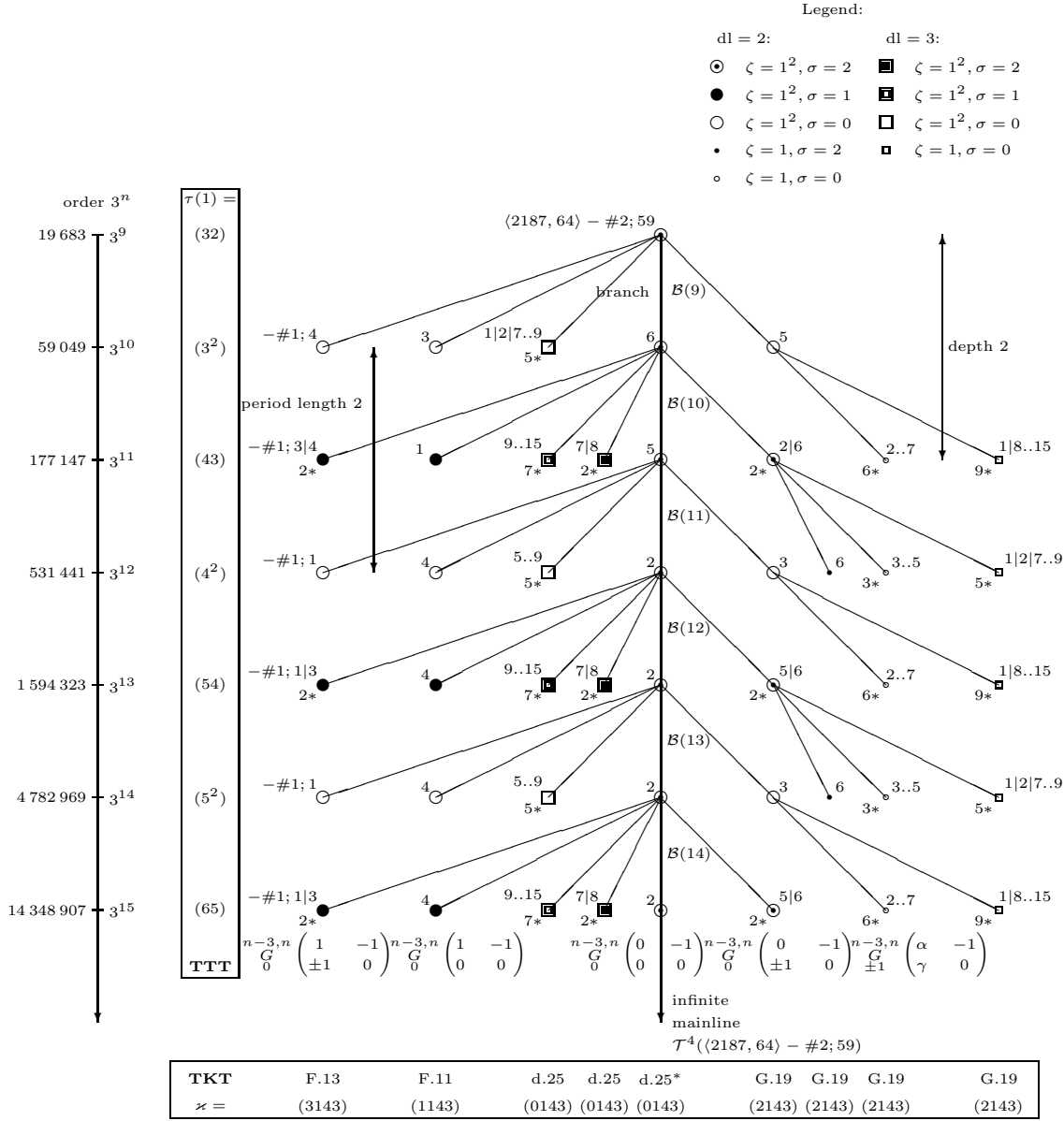

%\newpage

%--------------------------------------------------------------------------------

\begin{corollary}
\label{cor:TKTd25Tree1Cc4}
\textbf{(Actions and relation ranks.)} \\
The algebraic invariants of the vertices of the structured coclass-\(4\) tree \(\mathcal{T}^4{R_5^4}\) are listed in Table
\ref{tbl:TKTd25TreeCc4}. In particular:
\begin{enumerate}
\item
All mainline vertices,
two capable metabelian vertices of depth \(1\) with odd class and type \(\mathrm{G}.19\),
two distinguished terminal metabelian vertices of depth \(2\) with even class and type \(\mathrm{G}.19\), and
two distinguished terminal non-metabelian vertices of depth \(1\) with odd class and type \(\mathrm{d}.25\)
possess a \(V_4\)-action.
\item
Two distinguished terminal metabelian vertices of depth \(2\) with even class and type \(\mathrm{G}.19\),
all terminal vertices of depth \(1\) with odd class,
and the mainline vertices with even class,
possess an RI-action.
\item
The relation rank is given by
\(\mu=5\) for the mainline vertices \(\stackbin[0]{n-3,n}{G}\begin{pmatrix}0&1\\ 0&0\end{pmatrix}\) with \(n\ge 9\),
and the capable vertices \(\stackbin[0]{n-3,n}{G}\begin{pmatrix}0&1\\ \pm 1&0\end{pmatrix}\) of depth \(1\) with \(n\ge 10\),
and
\(\mu=4\) otherwise.
\end{enumerate}
\end{corollary}

%--------------------------------------------------------------------------------

\begin{proof}
(of Proposition
\ref{prp:TKTd25Tree1Cc4},
Theorem
\ref{thm:TKTd25Tree1Cc4},
and Corollary
\ref{cor:TKTd25Tree1Cc4}
)
The proofs are very similar to those of
Proposition
\ref{prp:TKTd19Tree1Cc4},
Theorem
\ref{thm:TKTd19Tree1Cc4},
and Corollary
\ref{cor:TKTd19Tree1Cc4}.
The differences are only the concrete numerical values of the invariants involved in the calculations: \\
\(\#\mathcal{B}(9)=9+15=24\), \(\#\mathcal{B}(10)=15+9+9=33\),
\(\mathrm{wd}(\mathcal{T})=\max(15+15,9+9+9)=\max(30,27)=30\), and
\(\mathrm{IC}(\mathcal{T})=\#\mathcal{B}(9)+\#\mathcal{B}(10)=24+33=57\). \\
In detail, we proved that there is no pre-period, \(\ell_\ast=0\), 
and the primitive period \((\mathcal{B}(9),\mathcal{B}(10))\) of length \(\ell=2\) consists of \\
\(4\), resp. \(6\), metabelian vertices with \(\zeta=1^2\), resp. \(\zeta=1\), and \\
\(5\), resp. \(9\), non-metabelian vertices with \(\zeta=1^2\), resp. \(\zeta=1\), \\
(\(9=4+5\) children of \(m_9\), and \(15=6+9\) children of \(v_2(m_9)\) with depth \(1\)) \\
together \(24\) vertices (\(10\) of them metabelian) in branch \(\mathcal{B}(9)\), and \\
\(6\), resp. \(8\), metabelian vertices with \(\zeta=1^2\), resp. \(\zeta=1\), and \\
\(9\), resp. \(10\), non-metabelian vertices with \(\zeta=1^2\), resp. \(\zeta=1\), \\
(\(15=6+9\) children of \(m_{10}\), and \(18=2\cdot (4+5)\) children of \(v_{2,3}(m_{10})\), both with depth \(1\)) \\
together \(33\) vertices (\(14\) of them metabelian) in branch \(\mathcal{B}(10)\). \\
The tree \(\mathcal{T}^4{R_5^4}\) corresponds to
the infinite metabelian pro-\(3\) group \(S_{3,2}\) in
\cite[Cnj. 15 (b), p. 116]{Ek}.
\end{proof}

%--------------------------------------------------------------------------------

\begin{theorem}
\label{thm:TKTd25Tree2Cc4}
\textbf{(Strict isomorphism of the two trees.)} \\
Viewed as an algebraically structured infinite digraph,
the second coclass-\(4\) tree \(\mathcal{T}^4(\langle 2187,64\rangle-\#2;59)\) with mainline of type \(\mathrm{d}.25^\ast\)
in Figure
\ref{fig:TKTd25Tree2Cc4}
is \textbf{strictly} isomorphic to
the first coclass-\(4\) tree \(\mathcal{T}^4(\langle 2187,64\rangle-\#2;57)\) with mainline  of type \(\mathrm{d}.25^\ast\)
in Figure
\ref{fig:TKTd25Tree1Cc4}.
Only the presentations of corresponding vertices are different,
but they share common algebraic invariants.
\end{theorem}

\begin{proof}
(Proof of Thm.
\ref{thm:TKTd25Tree1Cc4}
and Thm
\ref{thm:TKTd25Tree2Cc4}.)
The claims have been verified with the aid of MAGMA
\cite{MAGMA}
for all vertices \(V\) with logarithmic orders \(9\le\mathrm{lo}(V)\le 17\).
Pure periodicity of branches sets in with \(\mathcal{B}(9)\simeq\mathcal{B}(11)\).
Thus, the claims for all vertices \(V\) with logarithmic orders \(\mathrm{lo}(V)\ge 18\)
are a consequence of the periodicity theorems by du Sautoy in
\cite{dS}
and by Eick and Leedham-Green in
\cite{EkLg},
without the need of pruning the depth, which is bounded uniformly by \(2\).
\end{proof}

%--------------------------------------------------------------------------------

\renewcommand{\arraystretch}{1.2}

\begin{table}[hb]
%\caption{Invariants of \(P_9=\langle 2187,64\rangle-\#2;33\) and the sporadic part \(\mathcal{F}_0(5)\) of \(\mathcal{F}(5)\)}
\caption{Data for sporadic \(3\)-groups \(G\) with \(11\le n=\mathrm{lo}(G)\le 12\) of the forest \(\mathcal{F}(5)\)}
\label{tbl:SporCc5}
\begin{center}
\begin{tabular}{|c|c|l||c|c|c||c|c||c|c|l|c||c|c|}
\hline
 \(\#\) & \(m,n\)  & \(\rho;\alpha,\beta,\gamma,\delta\) & dp & dl & \(\zeta\) & \(\mu\) & \(\nu\) & \(\tau(1)\) & \(\tau_2\) & Type                   & \(\varkappa\) & \(\sigma\) & \(\#\mathrm{Aut}\)  \\
\hline
  \(1\) & \(6,9\)  & \(0;0,0,0,0\) (\(P_9\))             & 0  & 2  & \(1^2\)   & \(6\)   & \(2\)   & \(32\)      & \(2^31\)   & \(\mathrm{b}.10^\ast\) & \((0043)\)    & \(2\)      & \(2^3\cdot 3^{14}\) \\
\hline
  \(1\) & \(7,11\) & \(0;0,0,0,0\) (\(P_{11}\))          & 0  & 2  & \(1^2\)   & \(6\)   & \(2\)   & \(3^2\)     & \(32^3\)   & \(\mathrm{b}.10^\ast\) & \((0043)\)    & \(2\)      & \(2^3\cdot 3^{18}\) \\
  \(1\) & \(7,11\) & \(0;0,1,0,1\)                       & 0  & 2  & \(1^2\)   & \(5\)   & \(1\)   & \(3^2\)     & \(32^3\)   & \(\mathrm{d}.19^\ast\) & \((0343)\)    & \(0\)      & \(3^{18}\)          \\
  \(1\) & \(7,11\) & \(0;0,0,0,1\)                       & 0  & 2  & \(1^2\)   & \(5\)   & \(1\)   & \(3^2\)     & \(32^3\)   & \(\mathrm{d}.23^\ast\) & \((0243)\)    & \(0\)      & \(2\cdot 3^{18}\)   \\
  \(1\) & \(7,11\) & \(0;0,1,0,0\)                       & 0  & 2  & \(1^2\)   & \(5\)   & \(1\)   & \(3^2\)     & \(32^3\)   & \(\mathrm{d}.25^\ast\) & \((0143)\)    & \(0\)      & \(2\cdot 3^{18}\)   \\ 
  \(1\) & \(7,11\) & \(0;1,1,-1,1\)                      & 0  & 2  & \(1^2\)   & \(4\)   & \(0\)   & \(3^2\)     & \(32^3\)   & \(\mathrm{F}.7\)       & \((3443)\)    & \(0\)      & \(3^{18}\)          \\
  \(2\) & \(7,11\) & \(0;1,\pm 1,0,0\)                   & 0  & 2  & \(1^2\)   & \(4\)   & \(0\)   & \(3^2\)     & \(32^3\)   & \(\mathrm{F}.11\)      & \((1143)\)    & \(0\)      & \(2\cdot 3^{18}\)   \\
  \(2\) & \(7,11\) & \(0;\pm(1,0,1),1\)                  & 0  & 2  & \(1^2\)   & \(4\)   & \(0\)   & \(3^2\)     & \(32^3\)   & \(\mathrm{F}.12\)      & \((1343)\)    & \(0\)      & \(3^{18}\)          \\
  \(2\) & \(7,11\) & \(0;1,\pm(1,1),0\)                  & 0  & 2  & \(1^2\)   & \(4\)   & \(0\)   & \(3^2\)     & \(32^3\)   & \(\mathrm{F}.13\)      & \((3143)\)    & \(0\)      & \(3^{18}\)          \\
  \(1\) & \(7,11\) & \(0;1,0,0,1\)                       & 0  & 2  & \(1^2\)   & \(5\)   & \(1\)   & \(3^2\)     & \(32^3\)   & \(\mathrm{G}.16\)      & \((1243)\)    & \(0\)      & \(2\cdot 3^{18}\)   \\
  \(1\) & \(7,11\) & \(0;0,1,1,0\)                       & 0  & 2  & \(1^2\)   & \(5\)   & \(1\)   & \(3^2\)     & \(32^3\)   & \(\mathrm{G}.19\)      & \((2143)\)    & \(0\)      & \(2\cdot 3^{18}\)   \\
  \(2\) & \(7,11\) & \(0;1,\pm(1,1),1\)                  & 0  & 2  & \(1^2\)   & \(5\)   & \(1\)   & \(3^2\)     & \(32^3\)   & \(\mathrm{H}.4\)       & \((3343)\)    & \(0\)      & \(2\cdot 3^{18}\)   \\
\hline
 \(12\) & \(8,12\) &                                     & 1  & 2  & \(1\)     & \(4\)   & \(0\)   & \(3^2\)     & \(3^22^2\) &                        &               & \(0\)      & \(2\cdot 3^{20}\)   \\
 \(12\) & \(8,12\) &                                     & 1  & 2  & \(1\)     & \(4\)   & \(0\)   & \(3^2\)     & \(3^22^2\) &                        &               & \(0\)      & \(3^{20}\)          \\
\hline
  \(8\) & \(8,12\) &                                     & 1  & 3  & \(1\)     & \(4\)   & \(0\)   & \(3^2\)     & \(32^3\)   &                        &               & \(0\)      & \(2\cdot 3^{19}\)   \\
 \(20\) & \(8,12\) &                                     & 1  & 3  & \(1\)     & \(4\)   & \(0\)   & \(3^2\)     & \(32^3\)   &                        &               & \(0\)      & \(3^{19}\)          \\
  \(8\) & \(8,12\) &                                     & 1  & 3  & \(1\)     & \(4\)   & \(0\)   & \(3^2\)     & \(32^3\)   &                        &               & \(0\)      & \(3^{18}\)          \\
\hline
\end{tabular}
\end{center}
\end{table}

%\newpage

%--------------------------------------------------------------------------------

\section{Sporadic and periodic \(3\)-groups \(G\) of odd coclass \(\mathrm{cc}(G)\ge 5\)}
\label{s:PeriodicSporadic5}
\noindent
Although formulated for the particular coclass \(r=5\),
all results on sporadic and periodic groups in this section
are valid for any odd coclass \(r\ge 5\).
The exemplary (co-periodic) sporadic part \(\mathcal{F}_0(5)\) of the coclass forest \(\mathcal{F}(5)\)
is presented in the following Proposition
\ref{prp:SporCc5}.

\begin{proposition}
\label{prp:SporCc5}
The sporadic part \(\mathcal{F}_0(5)\) of the coclass-\(5\) forest \(\mathcal{F}(5)\) consists of
\begin{itemize}
\item
\(7\) \((1+2+2+2)\) isolated metabelian vertices with types \(\mathrm{F}.7\), \(\mathrm{F}.11\), \(\mathrm{F}.12\), \(\mathrm{F}.13\),
\item
\(4\) \((1+1+2)\) metabelian roots of finite trees with types \(\mathrm{G}.16\), \(\mathrm{G}.19\), \(\mathrm{H}.4\),
together with their \(24\) metabelian and \(36\) non-metabelian children, all with depth \(\mathrm{dp}=1\),
\item
\(34\) \((16+9+9)\) isolated vertices with \(\mathrm{dl}=3\) and types \(\mathrm{d}.19\), \(\mathrm{d}.23\), \(\mathrm{d}.25\),
\item
\(89\) isolated vertices with \(\mathrm{dl}=3\) and type \(\mathrm{b}.10\),
\item
\(13\) capable vertices with \(\mathrm{dl}=3\) and type \(\mathrm{b}.10\), \\
whose children do not belong to \(\mathcal{F}_0(5)\), by definition.
\end{itemize}
The action flag of all vertices is \(\sigma=0\), and consequently none of them has an RI- or \(V_4\)-action.

Together with the \(4\) metabelian roots of coclass-\(5\) trees,
the \(7+4+34+89+13=11+136\) vertices of depth \(\mathrm{dp}=0\)
are exactly the \(N_2=151\) children of step size \(s=2\) of \(P_9=\langle 2187,64\rangle-\#2;33\),
and the \(4+(4+13)\) capable vertices among them correspond to the invariant  \(C_2=21\) of \(P_9\).
\end{proposition}

%\newpage

%--------------------------------------------------------------------------------

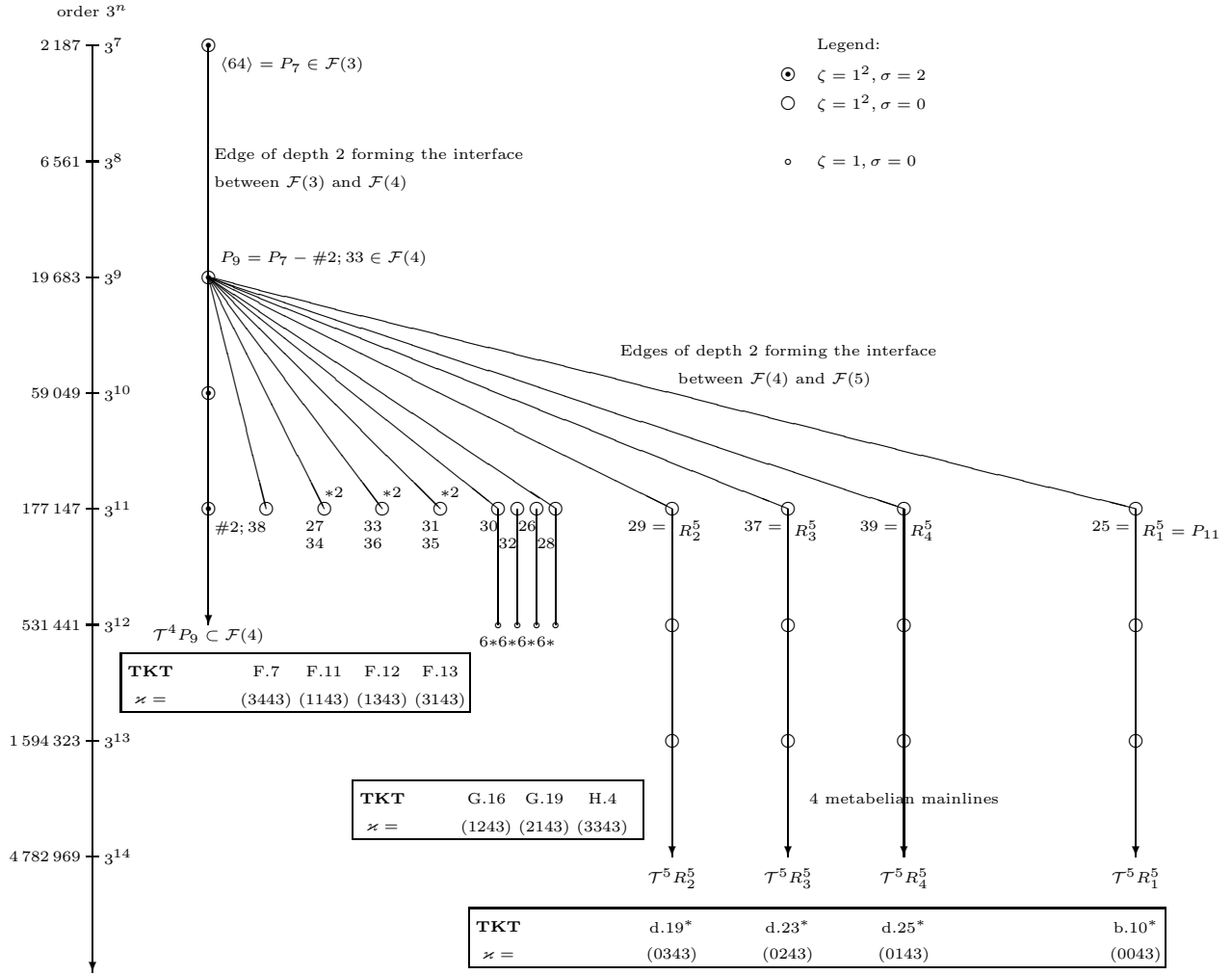
\begin{figure}[ht]
\caption{Metabelian interface between the coclass forests \(\mathcal{F}(4)\) and \(\mathcal{F}(5)\)}
\label{fig:SporCc5}

% Metabelian Interface Vertices of the Coclass Forest F(5)

{\tiny

\begin{center}

\setlength{\unitlength}{0.8cm}
\begin{picture}(19,17)(-3,-14)

% scale of orders
\put(-2,2.5){\makebox(0,0)[cb]{order \(3^n\)}}

\put(-2,2){\line(0,-1){14}}
\multiput(-2.1,2)(0,-2){8}{\line(1,0){0.2}}

\put(-2.2,2){\makebox(0,0)[rc]{\(2\,187\)}}
\put(-1.8,2){\makebox(0,0)[lc]{\(3^7\)}}
\put(-2.2,0){\makebox(0,0)[rc]{\(6\,561\)}}
\put(-1.8,0){\makebox(0,0)[lc]{\(3^8\)}}
\put(-2.2,-2){\makebox(0,0)[rc]{\(19\,683\)}}
\put(-1.8,-2){\makebox(0,0)[lc]{\(3^9\)}}
\put(-2.2,-4){\makebox(0,0)[rc]{\(59\,049\)}}
\put(-1.8,-4){\makebox(0,0)[lc]{\(3^{10}\)}}
\put(-2.2,-6){\makebox(0,0)[rc]{\(177\,147\)}}
\put(-1.8,-6){\makebox(0,0)[lc]{\(3^{11}\)}}
\put(-2.2,-8){\makebox(0,0)[rc]{\(531\,441\)}}
\put(-1.8,-8){\makebox(0,0)[lc]{\(3^{12}\)}}
\put(-2.2,-10){\makebox(0,0)[rc]{\(1\,594\,323\)}}
\put(-1.8,-10){\makebox(0,0)[lc]{\(3^{13}\)}}
\put(-2.2,-12){\makebox(0,0)[rc]{\(4\,782\,969\)}}
\put(-1.8,-12){\makebox(0,0)[lc]{\(3^{14}\)}}

\put(-2,-12){\vector(0,-1){2}}

% legend
\put(10.5,2){\makebox(0,0)[lc]{Legend:}}
\put(10,1.5){\circle{0.2}}
\put(10,1.5){\circle*{0.1}}
%\put(10,1){\circle*{0.2}}
\put(10,1){\circle{0.2}}
%\put(10,0){\circle*{0.1}}
\put(10,0){\circle{0.1}}
\put(10.5,1.5){\makebox(0,0)[lc]{\(\zeta=1^2,\sigma=2\)}}
%\put(10.5,1){\makebox(0,0)[lc]{\(\zeta=1^2,\sigma=1\)}}
\put(10.5,1){\makebox(0,0)[lc]{\(\zeta=1^2,\sigma=0\)}}
%\put(10.5,0){\makebox(0,0)[lc]{\(\zeta=1,\sigma=2\)}}
\put(10.5,0){\makebox(0,0)[lc]{\(\zeta=1,\sigma=0\)}}

% vertex of F(3)
\put(0,2){\circle{0.2}}
\put(0,2){\circle*{0.1}}

% directed edge of depth 2 from F(3) to F(4)
\put(0,2){\line(0,-1){4}}
\put(0.1,0){\makebox(0,0)[lb]{Edge of depth \(2\) forming the interface}}
\put(0.1,-0.5){\makebox(0,0)[lb]{between \(\mathcal{F}(3)\) and \(\mathcal{F}(4)\)}}

% vertex of F(4) main line
\multiput(0,-2)(0,-2){3}{\circle{0.2}}
\multiput(0,-2)(0,-2){3}{\circle*{0.1}}

% single F(4) main line
\put(0,-2){\vector(0,-1){6}}
\put(0,-8.0){\makebox(0,0)[ct]{\(\mathcal{T}^4P_9\subset\mathcal{F}(4)\)}}

% 10(+4) directed edges of depth 2 from F(4) to F(5)
\put(0,-2){\line(1,-4){1}}
\put(0,-2){\line(1,-2){2}}
\put(0,-2){\line(3,-4){3}}
\put(0,-2){\line(1,-1){4}}
\put(0,-2){\line(5,-4){5}}
\put(0,-2){\line(3,-2){6}}
\put(0,-2){\line(2,-1){8}}
\put(0,-2){\line(5,-2){10}}
\put(0,-2){\line(3,-1){12}}
\put(0,-2){\line(4,-1){16}}
\put(7.1,-3.4){\makebox(0,0)[lb]{Edges of depth \(2\) forming the interface}}
\put(8.1,-3.9){\makebox(0,0)[lb]{between \(\mathcal{F}(4)\) and \(\mathcal{F}(5)\)}}

% 7 top vertices of F(5)
\multiput(1,-6)(1,0){4}{\circle{0.2}}

% together with children
% regular
\multiput(5,-6)(0.33,0){4}{\circle{0.2}}
%\multiput(5,-6)(0.66,0){3}{\circle*{0.1}}

% irregular
%\put(5.33,-6){\circle{0.2}}
%\put(5.99,-6){\circle{0.2}}
%\put(5.99,-6){\circle*{0.1}}
%\put(6.65,-6){\circle{0.2}}

% together with further mainline vertices
\multiput(8,-6)(2,0){2}{\circle{0.2}}
\multiput(8,-8)(2,0){2}{\circle{0.2}}
\multiput(8,-10)(2,0){2}{\circle{0.2}}

\multiput(12,-6)(0,-2){3}{\circle{0.2}}
%\multiput(12,-6)(0,-2){3}{\circle*{0.1}}

\multiput(16,-6)(0,-2){3}{\circle{0.2}}
%\multiput(16,-6)(0,-2){3}{\circle*{0.1}}

% four F(5) mainlines
\put(8,-6){\vector(0,-1){6}}
\put(10,-6){\vector(0,-1){6}}
\put(12,-6){\vector(0,-1){6}}
\put(16,-6){\vector(0,-1){6}}

\put(12,-11){\makebox(0,0)[cc]{\(4\) metabelian mainlines}}

\put(8,-12.2){\makebox(0,0)[ct]{\(\mathcal{T}^5R_2^5\)}}
\put(10,-12.2){\makebox(0,0)[ct]{\(\mathcal{T}^5R_3^5\)}}
\put(12,-12.2){\makebox(0,0)[ct]{\(\mathcal{T}^5R_4^5\)}}
\put(16,-12.2){\makebox(0,0)[ct]{\(\mathcal{T}^5R_1^5\)}}

% directed edges to depth 1
\multiput(5,-6)(0.33,0){4}{\line(0,-1){2}}
% sporadic vertices of depth 1 of F(5)
% regular
\multiput(5,-8)(0.33,0){4}{\circle{0.1}}
% irregular
%\put(5.33,-8){\circle{0.1}}
%\put(5.99,-8){\circle{0.1}}
%\put(6.65,-8){\circle{0.1}}

% isomorphism classes and multiplicity counters
\put(0.2,1.8){\makebox(0,0)[lt]{\(\langle 64\rangle=P_7\in\mathcal{F}(3)\)}}
\put(0.2,-1.8){\makebox(0,0)[lb]{\(P_9=P_7-\#2;33\in\mathcal{F}(4)\)}}

%\put(1,-5.8){\makebox(0,0)[lb]{\(\ast 3\)}}
\put(1,-6.2){\makebox(0,0)[rt]{\(\#2;38\)}}
%\put(1,-6.5){\makebox(0,0)[rt]{\(56\)}}
%\put(1,-6.8){\makebox(0,0)[rt]{\(58\)}}
\put(2,-5.8){\makebox(0,0)[lb]{\(\ast 2\)}}
\put(2,-6.2){\makebox(0,0)[rt]{\(27\)}}
\put(2,-6.5){\makebox(0,0)[rt]{\(34\)}}
\put(3,-5.8){\makebox(0,0)[lb]{\(\ast 2\)}}
\put(3,-6.2){\makebox(0,0)[rt]{\(33\)}}
\put(3,-6.5){\makebox(0,0)[rt]{\(36\)}}
%\put(3,-6.8){\makebox(0,0)[rt]{\(51\)}}
%\put(3,-7.1){\makebox(0,0)[rt]{\(53\)}}
\put(4,-5.8){\makebox(0,0)[lb]{\(\ast 2\)}}
\put(4,-6.2){\makebox(0,0)[rt]{\(31\)}}
\put(4,-6.5){\makebox(0,0)[rt]{\(35\)}}
%\put(4,-6.8){\makebox(0,0)[rt]{\(50\)}}
%\put(4,-7.1){\makebox(0,0)[rt]{\(52\)}}

% regular
%\put(5.66,-5.5){\makebox(0,0)[cb]{\(\ast 2\)}}
%\put(6.32,-5.5){\makebox(0,0)[cb]{\(\ast 2\)}}
%\multiput(4.9,-7.8)(0.66,0){3}{\makebox(0,0)[rb]{r}}
\put(5.00,-6.2){\makebox(0,0)[rt]{\(30\)}}
\put(5.00,-8.2){\makebox(0,0)[rt]{\(6\ast\)}}
%\put(5.00,-8.2){\makebox(0,0)[rt]{\(\#1;6\)}}
\put(5.66,-6.2){\makebox(0,0)[rt]{\(26\)}}
%\put(5.66,-6.5){\makebox(0,0)[rt]{\(45\)}}
\put(5.66,-8.2){\makebox(0,0)[rt]{\(6\ast\)}}
%\put(5.66,-8.2){\makebox(0,0)[rt]{\(4\)}}
%\put(5.66,-8.5){\makebox(0,0)[rt]{\(4\)}}
%\put(6.32,-6.2){\makebox(0,0)[rt]{\(35\)}}
%\put(6.32,-6.5){\makebox(0,0)[rt]{\(37\)}}
%\put(6.32,-8.2){\makebox(0,0)[rt]{\(4\)}}
%\put(6.32,-8.5){\makebox(0,0)[rt]{\(4\)}}

% irregular
%\multiput(5.23,-7.8)(0.66,0){3}{\makebox(0,0)[rb]{i}}
\put(5.33,-6.5){\makebox(0,0)[rt]{\(32\)}}
\put(5.33,-8.2){\makebox(0,0)[rt]{\(6\ast\)}}
%\put(5.33,-8.2){\makebox(0,0)[rt]{\(2\)}}
\put(5.99,-6.5){\makebox(0,0)[rt]{\(28\)}}
\put(5.99,-8.2){\makebox(0,0)[rt]{\(6\ast\)}}
%\put(5.99,-8.2){\makebox(0,0)[rt]{\(4\)}}
%\put(6.65,-6.8){\makebox(0,0)[rt]{\(34\)}}
%\put(6.65,-8.2){\makebox(0,0)[rt]{\(7\)}}

%\put(8.1,-5.8){\makebox(0,0)[lb]{\(\ast 2\)}}
\put(7.9,-6.2){\makebox(0,0)[rt]{\(29=\)}}
\put(8.1,-6.2){\makebox(0,0)[lt]{\(R_2^5\)}}
%\put(7.9,-6.6){\makebox(0,0)[rt]{\(44=\)}}
%\put(8.1,-6.6){\makebox(0,0)[lt]{\(R_3^5\)}}
\put(9.9,-6.2){\makebox(0,0)[rt]{\(37=\)}}
\put(10.1,-6.2){\makebox(0,0)[lt]{\(R_3^5\)}}
%\put(12.1,-5.8){\makebox(0,0)[lb]{\(\ast 2\)}}
\put(11.9,-6.2){\makebox(0,0)[rt]{\(39=\)}}
\put(12.1,-6.2){\makebox(0,0)[lt]{\(R_4^5\)}}
%\put(11.9,-6.6){\makebox(0,0)[rt]{\(59=\)}}
%\put(12.1,-6.6){\makebox(0,0)[lt]{\(R_6^5\)}}
\put(15.9,-6.2){\makebox(0,0)[rt]{\(25=\)}}
\put(16.1,-6.2){\makebox(0,0)[lt]{\(R_1^5=P_{11}\)}}

% transfer kernel types
\put(-1,-8.8){\makebox(0,0)[cc]{\textbf{TKT}}}
\put(1,-8.8){\makebox(0,0)[cc]{F.7}}
\put(2,-8.8){\makebox(0,0)[cc]{F.11}}
\put(3,-8.8){\makebox(0,0)[cc]{F.12}}
\put(4,-8.8){\makebox(0,0)[cc]{F.13}}
\put(-1,-9.3){\makebox(0,0)[cc]{\(\varkappa=\)}}
\put(1,-9.3){\makebox(0,0)[cc]{\((3443)\)}}
\put(2,-9.3){\makebox(0,0)[cc]{\((1143)\)}}
\put(3,-9.3){\makebox(0,0)[cc]{\((1343)\)}}
\put(4,-9.3){\makebox(0,0)[cc]{\((3143)\)}}
\put(-1.5,-9.5){\framebox(6,1){}}

\put(3,-11){\makebox(0,0)[cc]{\textbf{TKT}}}
\put(4.8,-11){\makebox(0,0)[cc]{G.16}}
\put(5.8,-11){\makebox(0,0)[cc]{G.19}}
\put(6.8,-11){\makebox(0,0)[cc]{H.4}}
\put(3,-11.5){\makebox(0,0)[cc]{\(\varkappa=\)}}
\put(4.8,-11.5){\makebox(0,0)[cc]{\((1243)\)}}
\put(5.8,-11.5){\makebox(0,0)[cc]{\((2143)\)}}
\put(6.8,-11.5){\makebox(0,0)[cc]{\((3343)\)}}
\put(2.5,-11.7){\framebox(5,1){}}

\put(5,-13.2){\makebox(0,0)[cc]{\textbf{TKT}}}
\put(8,-13.2){\makebox(0,0)[cc]{d.19\({}^\ast\)}}
\put(10,-13.2){\makebox(0,0)[cc]{d.23\({}^\ast\)}}
\put(12,-13.2){\makebox(0,0)[cc]{d.25\({}^\ast\)}}
\put(16,-13.2){\makebox(0,0)[cc]{b.10\({}^\ast\)}}
\put(5,-13.7){\makebox(0,0)[cc]{\(\varkappa=\)}}
\put(8,-13.7){\makebox(0,0)[cc]{\((0343)\)}}
\put(10,-13.7){\makebox(0,0)[cc]{\((0243)\)}}
\put(12,-13.7){\makebox(0,0)[cc]{\((0143)\)}}
\put(16,-13.7){\makebox(0,0)[cc]{\((0043)\)}}
\put(4.5,-13.9){\framebox(12,1){}}

\end{picture}

\end{center}

}

\end{figure}

%\newpage

%--------------------------------------------------------------------------------

\noindent
Figure
\ref{fig:SporCc5}
sketches an outline of the \textit{metabelian skeleton} of the coclass forest \(\mathcal{F}(5)\) in its top region.
The vertices \(P_7=\langle 2187,64\rangle\in\mathcal{F}(3)\)
and \(P_9=\langle 2187,64\rangle-\#2;33\in\mathcal{F}(4)\),
with the crucial bifurcation from \(\mathcal{F}(4)\) to \(\mathcal{F}(5)\),
belong to the infinite main trunk (\S\
\ref{s:MainTrunk}).

%--------------------------------------------------------------------------------

\begin{theorem}
\label{thm:SporCc5}
The coclass-\(r\) forest \(\mathcal{F}(r)\) with any odd \(r\ge 5\) is the disjoint union of
its finite sporadic part \(\mathcal{F}_0(r)\) with total information content
\begin{equation}
\label{eqn:SporadicIC5}
s=\#\mathcal{F}_0(r)=207
\end{equation}
and \(t=4\) infinite coclass-\(r\) trees \(\mathcal{T}^r(R_i^r)\)
with roots \(R_i^r:=P_{2r-1}-\#2;n_i\), where \((n_i)_{1\le i\le 4}=(25,29,37,39)\) for \(r=5\).
The algebraic invariants for groups with centre \(\zeta=1^2\),
and in cumulative form for \(\zeta=1\), are given for \(r=5\) in Table
\ref{tbl:SporCc5},
where the parent vertex \(P_{2r-1}=P_9\) on the maintrunk is also included,
but the \(136\) non-metabelian top vertices of depth \(\mathrm{dp}=0\) are excluded.
\end{theorem}

%\newpage

%--------------------------------------------------------------------------------

\begin{proof}
(of Proposition
\ref{prp:SporCc5}
and Theorem
\ref{thm:SporCc5})
We have computed the sporadic parts \(\mathcal{F}_0(r)\) of coclass forests \(\mathcal{F}(r)\)
with odd \(r\ge 5\) up to \(r\le 21\) by means of MAGMA
\cite{MAGMA}.
They all share a common graph theoretic structure with \(\mathcal{F}_0(5)\).
The forest \(\mathcal{F}(5)\) contains \(4\) roots of coclass trees with metabelian mainlines
(a unique root \(R_1^5\) of type \(\mathrm{b}\) and three roots \(R_2^5,\ldots,R_4^5\) of type \(\mathrm{d}\)), namely \\
\(R_1^5=P_{11}=G^{7,11}_0(0,0,0,0)\simeq P_9-\#2;25\), \quad \(R_2^5=G^{7,11}_0(0,1,0,1)\simeq P_7-\#2;29\), \\
\(R_3^5=G^{7,11}_0(0,0,0,1)\simeq P_9-\#2;37\), \quad \(R_4^5=G^{7,11}_0(0,1,0,0)\simeq P_9-\#2;39\), \\
which give rise to the periodic part of \(\mathcal{F}(5)\),
and \(35\) sporadic metabelian groups of type \(\mathrm{F}\), \(\mathrm{G}\) or \(\mathrm{H}\).
Among the groups of the sporadic part \(\mathcal{F}_0(5)\), there are
\(7\) isolated metabelian vertices with type \(\mathrm{F}\),
and \(4\) metabelian roots of finite trees with type \(\mathrm{G}\) or \(\mathrm{H}\) and tree depth \(1\).
Among the \(60=15+15+15+15\) children,
there are \(24=6+6+6+6\) metabelian,
and \(36=9+9+9+9\) have derived length \(3\).
The latter are \textit{omitted} in the forest diagram, Figure
\ref{fig:SporCc5}.
Additionally, \(\mathcal{F}_0(5)\) contains \(136\) non-metabelian top vertices,
which gives a total information content
\(s=\#\mathcal{F}_0(5)\) of \(207=(11+136)+60\) representatives.

The metabelian skeleton consists of \(35=11+24\) vertices.
The results for metabelian groups are in accordance with the third tree diagram \(e\ge 4\), \(e\equiv 0\pmod{2}\), in
\cite[third page between pp. 191--192]{Ne}.
The metabelian groups in Table
\ref{tbl:SporCc5}
correspond to the representatives of isomorphism classes in
\cite[pp. 34--35]{Ne2}.
\end{proof}

%\newpage

%--------------------------------------------------------------------------------

\subsection{The unique mainline of type \(\mathrm{b}.10^\ast\) for odd coclass \(r\ge 5\)}
\label{ss:b10Cc5}

\begin{proposition}
\label{prp:TKTb10TreeCc5}
\textbf{(Periodicity and descendant numbers.)} \\
The branches \(\mathcal{B}(i)\), \(i\ge n_\ast=11\),
of the coclass-\(5\) tree \(\mathcal{T}^5(\langle 2187,64\rangle-\#2;33-\#2;25)\)
with mainline vertices of transfer kernel type \(\mathrm{b}.10^\ast\), \(\varkappa\sim (0043)\),
are periodic with pre-period length \(\ell_\ast=1\) and with primitive period length \(\ell=2\),
that is, \(\mathcal{B}(i+2)\simeq\mathcal{B}(i)\) are isomorphic as digraphs,
for all \(i\ge p_\ast=n_\ast+\ell_\ast=12\).

The graph theoretic structure of the tree is determined uniquely by the numbers
\(N_1\) of immediate descendants and \(C_1\) of capable immediate descendants
of the mainline vertices \(m_n\) with logarithmic order \(n=\mathrm{lo}(m_n)\ge n_\ast=11\): \\
\((N_1,C_1)=(30,1)\) for the root \(m_{11}\) with \(n=11\), \\
\((N_1,C_1)=(24,1)\) for all mainline vertices \(m_n\) with even logarithmic order \(n\ge 12\), \\
\((N_1,C_1)=(40,1)\) for all mainline vertices \(m_n\) with odd logarithmic order \(n\ge 13\).
\end{proposition}

%--------------------------------------------------------------------------------

\begin{proof}
(of Proposition
\ref{prp:TKTb10TreeCc5})
The statements concerning the numbers \(N_1(m_n)\) of immediate descendants
of the mainline vertices \(m_n\) with \(n\ge n_\ast=11\)
have been obtained by direct computation with the \(p\)-group generation algorithm
\cite{Nm2,Ob,HEO}
in MAGMA
\cite{MAGMA}.
In detail, we proved that there are \\
\(6\), resp. \(6\), metabelian vertices with bicyclic centre \(\zeta=1^2\), resp. cyclic centre \(\zeta=1\), and  \\
\(9\), resp. \(9\), non-metabelian vertices with \(\zeta=1^2\), resp. \(\zeta=1\),  \\
together \(30\) vertices (\(12\) of them metabelian) in the pre-periodic branch \(\mathcal{B}(11)\), \\
and the primitive period \((\mathcal{B}(12),\mathcal{B}(13))\) of length \(\ell=2\) consists of \\
\(4\), resp. \(6\), metabelian vertices with \(\zeta=1^2\), resp. \(\zeta=1\), and \\
\(5\), resp. \(9\), non-metabelian vertices with \(\zeta=1^2\), resp. \(\zeta=1\), \\
together \(24\) vertices (\(10\) of them metabelian) in branch \(\mathcal{B}(12)\), and \\
\(6\), resp. \(9\), metabelian vertices with \(\zeta=1^2\), resp. \(\zeta=1\), and \\
\(9\), resp. \(16\), non-metabelian vertices with \(\zeta=1^2\), resp. \(\zeta=1\), \\
together \(40\) vertices (\(15\) of them metabelian) in branch \(\mathcal{B}(13)\).

The results concerning the metabelian skeleton confirm the corresponding statements in the dissertation of Nebelung
\cite[Thm. 5.1.16, pp. 178--179, and the third Figure, \(e\ge 4\), \(e\equiv 0\pmod{2}\), on the third page between pp. 191--192]{Ne}.
The tree \(\mathcal{T}^5{P_{11}}\) corresponds to
the infinite metabelian pro-\(3\) group \(S_{4,1}\) in
\cite[Cnj. 15 (a), p. 116]{Ek}.

The claim of the virtual periodicity of branches
has been proved generally for any coclass tree in
\cite{dS}
and
\cite{EkLg}.
Here, the strict periodicity was confirmed by computation up to branch \(\mathcal{B}(33)\)
and clearly sets in at \(p_\ast=12\).
\end{proof}

%\newpage

%--------------------------------------------------------------------------------

\begin{figure}[hb]
\caption{The unique coclass-\(5\) tree \(\mathcal{T}^5{P_{11}}\) with mainline of type \(\mathrm{b}.10^\ast\)}
\label{fig:TKTb10TreeCc5}

\input{Figure13}

\end{figure}

%--------------------------------------------------------------------------------

\begin{theorem}
\label{thm:TKTb10TreeCc5}
\textbf{(Graph theoretic and algebraic invariants.)} \\
The coclass-\(5\) tree \(\mathcal{T}:=\mathcal{T}^5{P_{11}}\) of finite \(3\)-groups \(G\) with coclass \(\mathrm{cc}(G)=5\)
which arises from the metabelian root \(P_{11}:=\langle 2187,64\rangle-\#2;33-\#2;25\)
has the following graph theoretic properties.
\begin{enumerate}
\item
The pre-period \((\mathcal{B}(11))\) of length \(\ell_\ast=1\) is irregular.
\item
The cardinality of the irregular branch is \(\#\mathcal{B}(11)=30\).
\item
The branches \(\mathcal{B}(i)\), \(i\ge p_\ast=12\), are periodic with
primitive period \((\mathcal{B}(12),\mathcal{B}(13))\) of length \(\ell=2\).
\item
The cardinalities of the regular branches are
\(\#\mathcal{B}(12)=24\) and \(\#\mathcal{B}(13)=40\).
\item
Depth, width, and information content of the tree are given by
\begin{equation}
\label{eqn:TKTb10TreeCc5}
\mathrm{dp}(\mathcal{T}^5{P_{11}})=1, \quad \mathrm{wd}(\mathcal{T}^5{P_{11}})=40, \quad \text{ and } \quad \mathrm{IC}(\mathcal{T}^5{P_{11}})=94.
\end{equation}
\end{enumerate}
The algebraic invariants of the groups represented by vertices forming
the pre-period \((\mathcal{B}(11))\) and the primitive period \((\mathcal{B}(12),\mathcal{B}(13))\) of the tree
are given in Table
\ref{tbl:TKTb10TreeCc5}.
The six leading branches \(\mathcal{B}(11),\ldots,\mathcal{B}(16)\) are drawn in Figure
\ref{fig:TKTb10TreeCc5}.
\end{theorem}

%\newpage

%--------------------------------------------------------------------------------

\renewcommand{\arraystretch}{1.2}

\begin{table}[ht]
\caption{Data for \(3\)-groups \(G\) with \(11\le n=\mathrm{lo}(G)\le 14\) of the coclass tree \(\mathcal{T}^5{P_{11}}\)}
\label{tbl:TKTb10TreeCc5}
\begin{center}
\begin{tabular}{|c|c|l||c|c|c||c|c||c|c|l|c||c|c|}
\hline
 \(\#\) & \(m,n\)   & \(\rho;\alpha,\beta,\gamma,\delta\) & dp & dl & \(\zeta\) & \(\mu\) & \(\nu\) & \(\tau(1)\) & \(\tau_2\) & Type                   & \(\varkappa\) & \(\sigma\) & \(\#\mathrm{Aut}\)  \\
\hline
  \(1\) & \(7,11\)  & \(0;0,0,0,0\)                       & 0  & 2  & \(1^2\)   & \(6\)   & \(2\)   & \(3^2\)     & \(32^3\)   & \(\mathrm{b}.10^\ast\) & \((0043)\)    & \(2^\ast\) & \(2^3\cdot 3^{18}\) \\
\hline
  \(1\) & \(8,12\)  & \(0;0,0,0,0\)                       & 0  & 2  & \(1^2\)   & \(6\)   & \(1\)   & \(43\)      & \(3^22^2\) & \(\mathrm{b}.10^\ast\) & \((0043)\)    & \(2\)      & \(2^2\cdot 3^{20}\) \\
  \(2\) & \(8,12\)  & \(0;1,0,\pm 1,0\)                   & 1  & 2  & \(1^2\)   & \(5\)   & \(0\)   & \(43\)      & \(3^22^2\) & \(\mathrm{d}.19\)      & \((3043)\)    & \(1^\ast\) & \(2\cdot 3^{20}\)   \\
  \(1\) & \(8,12\)  & \(0;1,0,0,0\)                       & 1  & 2  & \(1^2\)   & \(5\)   & \(0\)   & \(43\)      & \(3^22^2\) & \(\mathrm{d}.23\)      & \((1043)\)    & \(1^\ast\) & \(2\cdot 3^{20}\)   \\
  \(2\) & \(8,12\)  & \(0;0,0,\pm 1,0\)                   & 1  & 2  & \(1^2\)   & \(5\)   & \(0\)   & \(43\)      & \(3^22^2\) & \(\mathrm{d}.25\)      & \((2043)\)    & \(2^\ast\) & \(2^2\cdot 3^{20}\) \\
  \(3\) & \(8,12\)  & \(\pm 1;\alpha,0,\gamma,0\)         & 1  & 2  & \(1\)     & \(5\)   & \(0\)   & \(3^2\)     & \(3^22^2\) & \(\mathrm{b}.10\)      & \((0043)\)    & \(2^\ast\) & \(2^2\cdot 3^{20}\) \\
  \(3\) & \(8,12\)  & \(\pm 1;\alpha,0,\gamma,0\)         & 1  & 2  & \(1\)     & \(5\)   & \(0\)   & \(3^2\)     & \(3^22^2\) & \(\mathrm{b}.10\)      & \((0043)\)    & \(1^\ast\) & \(2\cdot 3^{20}\)   \\
  \(6\) & \(8,12\)  &                                     & 1  & 3  & \(1^2\)   & \(5\)   & \(0\)   & \(3^2\)     & \(32^3\)   & \(\mathrm{b}.10\)      & \((0043)\)    & \(1^\ast\) & \(2\cdot 3^{19}\)   \\
  \(2\) & \(8,12\)  &                                     & 1  & 3  & \(1^2\)   & \(5\)   & \(0\)   & \(3^2\)     & \(32^3\)   & \(\mathrm{b}.10\)      & \((0043)\)    & \(2^\ast\) & \(2^2\cdot 3^{18}\) \\
  \(1\) & \(8,12\)  &                                     & 1  & 3  & \(1^2\)   & \(5\)   & \(0\)   & \(3^2\)     & \(32^3\)   & \(\mathrm{b}.10\)      & \((0043)\)    & \(1^\ast\) & \(2\cdot 3^{18}\)   \\
  \(2\) & \(8,12\)  &                                     & 1  & 3  & \(1\)     & \(5\)   & \(0\)   & \(3^2\)     & \(32^3\)   & \(\mathrm{b}.10\)      & \((0043)\)    & \(2^\ast\) & \(2^2\cdot 3^{19}\) \\
  \(5\) & \(8,12\)  &                                     & 1  & 3  & \(1\)     & \(5\)   & \(0\)   & \(3^2\)     & \(32^3\)   & \(\mathrm{b}.10\)      & \((0043)\)    & \(1^\ast\) & \(2\cdot 3^{19}\)   \\
  \(2\) & \(8,12\)  &                                     & 1  & 3  & \(1\)     & \(5\)   & \(0\)   & \(3^2\)     & \(32^3\)   & \(\mathrm{b}.10\)      & \((0043)\)    & \(1^\ast\) & \(2\cdot 3^{18}\)   \\
\hline
  \(1\) & \(9,13\)  & \(0;0,0,0,0\)                       & 0  & 2  & \(1^2\)   & \(6\)   & \(1\)   & \(4^2\)     & \(432^2\)  & \(\mathrm{b}.10^\ast\) & \((0043)\)    & \(2^\ast\) & \(2^2\cdot 3^{22}\) \\
  \(1\) & \(9,13\)  & \(0;1,0,1,0\)                       & 1  & 2  & \(1^2\)   & \(5\)   & \(0\)   & \(4^2\)     & \(432^2\)  & \(\mathrm{d}.19\)      & \((3043)\)    & \(0\)      & \(3^{22}\)          \\
  \(1\) & \(9,13\)  & \(0;1,0,0,0\)                       & 1  & 2  & \(1^2\)   & \(5\)   & \(0\)   & \(4^2\)     & \(432^2\)  & \(\mathrm{d}.23\)      & \((1043)\)    & \(0\)      & \(2\cdot 3^{22}\)   \\
  \(1\) & \(9,13\)  & \(0;0,0,1,0\)                       & 1  & 2  & \(1^2\)   & \(5\)   & \(0\)   & \(4^2\)     & \(432^2\)  & \(\mathrm{d}.25\)      & \((2043)\)    & \(0\)      & \(2\cdot 3^{22}\)   \\
  \(3\) & \(9,13\)  & \(\pm 1;\alpha,0,\gamma,0\)         & 1  & 2  & \(1\)     & \(5\)   & \(0\)   & \(43\)      & \(432^2\)  & \(\mathrm{b}.10\)      & \((0043)\)    & \(0\)      & \(2\cdot 3^{22}\)   \\
  \(3\) & \(9,13\)  & \(\pm 1;\alpha,0,\gamma,0\)         & 1  & 2  & \(1\)     & \(5\)   & \(0\)   & \(43\)      & \(432^2\)  & \(\mathrm{b}.10\)      & \((0043)\)    & \(0\)      & \(3^{22}\)          \\
  \(3\) & \(9,13\)  &                                     & 1  & 3  & \(1^2\)   & \(5\)   & \(0\)   & \(43\)      & \(3^22^2\) & \(\mathrm{b}.10\)      & \((0043)\)    & \(0\)      & \(3^{21}\)   \\
  \(2\) & \(9,13\)  &                                     & 1  & 3  & \(1^2\)   & \(5\)   & \(0\)   & \(43\)      & \(3^22^2\) & \(\mathrm{b}.10\)      & \((0043)\)    & \(0\)      & \(2\cdot 3^{20}\)   \\
  \(6\) & \(9,13\)  &                                     & 1  & 3  & \(1\)     & \(5\)   & \(0\)   & \(43\)      & \(3^22^2\) & \(\mathrm{b}.10\)      & \((0043)\)    & \(0\)      & \(3^{21}\)          \\
  \(2\) & \(9,13\)  &                                     & 1  & 3  & \(1\)     & \(5\)   & \(0\)   & \(43\)      & \(3^22^2\) & \(\mathrm{b}.10\)      & \((0043)\)    & \(0\)      & \(2\cdot 3^{20}\)   \\
  \(1\) & \(9,13\)  &                                     & 1  & 3  & \(1\)     & \(5\)   & \(0\)   & \(43\)      & \(3^22^2\) & \(\mathrm{b}.10\)      & \((0043)\)    & \(0\)      & \(3^{20}\)          \\
\hline
  \(1\) & \(10,14\) & \(0;0,0,0,0\)                       & 0  & 2  & \(1^2\)   & \(6\)   & \(1\)   & \(54\)      & \(4^22^2\) & \(\mathrm{b}.10^\ast\) & \((0043)\)    & \(2\)      & \(2^2\cdot 3^{24}\) \\
  \(2\) & \(10,14\) & \(0;1,0,\pm 1,0\)                   & 1  & 2  & \(1^2\)   & \(5\)   & \(0\)   & \(54\)      & \(4^22^2\) & \(\mathrm{d}.19\)      & \((3043)\)    & \(1^\ast\) & \(2\cdot 3^{24}\)   \\
  \(1\) & \(10,14\) & \(0;1,0,0,0\)                       & 1  & 2  & \(1^2\)   & \(5\)   & \(0\)   & \(54\)      & \(4^22^2\) & \(\mathrm{d}.23\)      & \((1043)\)    & \(1^\ast\) & \(2\cdot 3^{24}\)   \\
  \(2\) & \(10,14\) & \(0;0,0,\pm 1,0\)                   & 1  & 2  & \(1^2\)   & \(5\)   & \(0\)   & \(54\)      & \(4^22^2\) & \(\mathrm{d}.25\)      & \((2043)\)    & \(2^\ast\) & \(2^2\cdot 3^{24}\) \\
  \(9\) & \(10,14\) & \(\pm 1;\alpha,0,\gamma,0\)         & 1  & 2  & \(1\)     & \(5\)   & \(0\)   & \(4^2\)     & \(4^22^2\) & \(\mathrm{b}.10\)      & \((0043)\)    & \(1^\ast\) & \(2\cdot 3^{24}\)   \\
  \(6\) & \(10,14\) &                                     & 1  & 3  & \(1^2\)   & \(5\)   & \(0\)   & \(4^2\)     & \(432^2\)  & \(\mathrm{b}.10\)      & \((0043)\)    & \(1^\ast\) & \(2\cdot 3^{23}\)   \\
  \(2\) & \(10,14\) &                                     & 1  & 3  & \(1^2\)   & \(5\)   & \(0\)   & \(4^2\)     & \(432^2\)  & \(\mathrm{b}.10\)      & \((0043)\)    & \(2^\ast\) & \(2^2\cdot 3^{22}\) \\
  \(1\) & \(10,14\) &                                     & 1  & 3  & \(1^2\)   & \(5\)   & \(0\)   & \(4^2\)     & \(432^2\)  & \(\mathrm{b}.10\)      & \((0043)\)    & \(1^\ast\) & \(2\cdot 3^{22}\)   \\
 \(12\) & \(10,14\) &                                     & 1  & 3  & \(1\)     & \(5\)   & \(0\)   & \(4^2\)     & \(432^2\)  & \(\mathrm{b}.10\)      & \((0043)\)    & \(1^\ast\) & \(2\cdot 3^{23}\)   \\
  \(4\) & \(10,14\) &                                     & 1  & 3  & \(1\)     & \(5\)   & \(0\)   & \(4^2\)     & \(432^2\)  & \(\mathrm{b}.10\)      & \((0043)\)    & \(1^\ast\) & \(2\cdot 3^{22}\)   \\
\hline
\end{tabular}
\end{center}
\end{table}

%\newpage

%--------------------------------------------------------------------------------

\begin{remark}
\label{rmk:TKTb10TreeCc5}
The algebraic information in Table
\ref{tbl:TKTb10TreeCc5}
is visualized in Figure
\ref{fig:TKTb10TreeCc5}.
By periodic continuation, the figure shows more branches than the table
but less details concerning the exact order \(\#\mathrm{Aut}\) of the automorphism group.
\end{remark}

%\newpage

%--------------------------------------------------------------------------------

\begin{proof}
(of Theorem
\ref{thm:TKTb10TreeCc5})
According to Proposition
\ref{prp:TKTb10TreeCc5},
the logarithmic order of the tree root, respectively of the periodic root, is
\(n_\ast=11\), respectively \(p_\ast=n_\ast+\ell_\ast=12\).

Since \(C_1(m_n)=1\) for all mainline vertices \(m_n\) with \(n\ge n_\ast\),
according to Proposition
\ref{prp:TKTb10TreeCc5},
the unique capable child of \(m_n\) is \(m_{n+1}\),
and each branch has depth \(\mathrm{dp}(\mathcal{B}(n))=1\), for \(n\ge n_\ast\).
Consequently, the tree is also of depth \(\mathrm{dp}(\mathcal{T})=1\).

With the aid of Formula
\eqref{eqn:BranchDepth1}
in Theorem
\ref{thm:DescendantNumbers},
the claims (2) and (4) are consequences of Proposition
\ref{prp:TKTb10TreeCc5}:\\
\(\#\mathcal{B}(11)=N_1(m_{11})=30\),
\(\#\mathcal{B}(12)=N_1(m_{12})=24\), and
\(\#\mathcal{B}(13)=N_1(m_{13})=40\).

According to Formula
\eqref{eqn:TreeWidth1}
in Corollary
\ref{cor:TreeWidth},
where \(n\) runs from \(n_\ast+1=12\) to \(n_\ast+\ell_\ast+\ell+0=14\),
the tree width is the maximum \(\mathrm{wd}(\mathcal{T})=40\)
of the expressions \(N_1(m_{11})=30\), \(N_1(m_{12})=24\), and \(N_1(m_{13})=40\).

The information content of the tree is given by Formula
\eqref{eqn:InfoCont}
in the Definition
\ref{dfn:InfoCont}: \\
\(\mathrm{IC}(\mathcal{T})=\#\mathcal{B}(11)+(\#\mathcal{B}(12)+\#\mathcal{B}(13))=30+(24+40)=94\).

The algebraic invariants in Table
\ref{tbl:TKTb10TreeCc5},
that is,
depth \(\mathrm{dp}\),
derived length \(\mathrm{dl}\),
abelian type invariants of the centre \(\zeta\),
relation rank \(\mu\),
nuclear rank \(\nu\),
abelian quotient invariants \(\tau(1)\) of the first maximal subgroup,
respectively \(\tau_2\) of the commutator subgroup,
transfer kernel type \(\varkappa\),
action flag \(\sigma\), and
the factorized order \(\#\mathrm{Aut}\) of the automorphism group
have been computed by means of program scripts written for MAGMA
\cite{MAGMA}.

Each group is characterized by the parameters
of the normalized representative \(G_\rho^{m,n}(\alpha,\beta,\gamma,\delta)\) of its isomorphism class,
according to Formula
\eqref{eqn:Presentation},
and by its identifier \(\langle 3^n,i\rangle\) in the SmallGroups Database
\cite{BEO}.

The column with header \(\#\) contains
the number of groups with identical invariants (except the presentation),
for each row.
\end{proof}

%--------------------------------------------------------------------------------

\begin{corollary}
\label{cor:TKTb10TreeCc5}
\textbf{(Actions and relation ranks.)}
The algebraic invariants of the vertices of the structured coclass-\(5\) tree \(\mathcal{T}^5{P_{11}}\) are listed in Table
\ref{tbl:TKTb10TreeCc5}.
In particular:
\begin{enumerate}
\item
The groups with \(V_4\)-action are
all mainline vertices \(\stackbin[0]{n-4,n}{G}\begin{pmatrix}0&0\\ 0&0\end{pmatrix}\), \(n\ge 11\),
the two terminal vertices \(\stackbin[0]{n-4,n}{G}\begin{pmatrix}0&0\\ \pm 1&0\end{pmatrix}\) with even \(n\ge 12\),
two terminal non-metabelian vertices with \(\zeta=1^2\) and even \(n\ge 12\),
three pre-periodic terminal vertices \(\stackbin[\pm 1]{8,12}{G}\begin{pmatrix}0&0\\ 0&0\end{pmatrix}\), and
two pre-periodic terminal non-metabelian vertices with \(\zeta=1\) and \(n=12\).
\item
With respect to the kernel types, 
all mainline groups of type \(\mathrm{b}.10^\ast\), \(\varkappa=(0043)\),
the two leaves of type \(\mathrm{d}.25\), \(\varkappa\sim (2043)\), with every even logarithmic order,
two distinguished non-metabelian leaves of type \(\mathrm{b}.10\) with every even logarithmic order,
and five pre-periodic leaves of type \(\mathrm{b}.10\) with \(n=12\) possess a \(V_4\)-action.
\item
All terminal vertices of depth \(1\) with odd class
and the mainline vertices with even class
possess an RI-action.
\item
The relation rank is given by
\(\mu=6\) for the mainline vertices \(\stackbin[0]{n-4,n}{G}\begin{pmatrix}0&0\\ 0&0\end{pmatrix}\), \(n\ge 11\), and
\(\mu=5\) otherwise.
\end{enumerate}
\end{corollary}

%--------------------------------------------------------------------------------

\begin{proof}
(of Corollary
\ref{cor:TKTb10TreeCc5})
The existence of an RI-action on \(G\) has been checked by means of
an algorithm involving the \(p\)-covering group of \(G\), written for MAGMA
\cite{MAGMA}.
The other claims follow immediately from Table
\ref{tbl:TKTb10TreeCc5},
continued indefinitely with the aid of the periodicity in Proposition
\ref{prp:TKTb10TreeCc5}. 
\end{proof}

%\newpage

%--------------------------------------------------------------------------------

\begin{figure}[hb]
\caption{The unique coclass-\(5\) tree \(\mathcal{T}^5(P_9-\#2;29)\) with mainline of type \(\mathrm{d}.19^\ast\)}
\label{fig:TKTd19TreeCc5}

\input{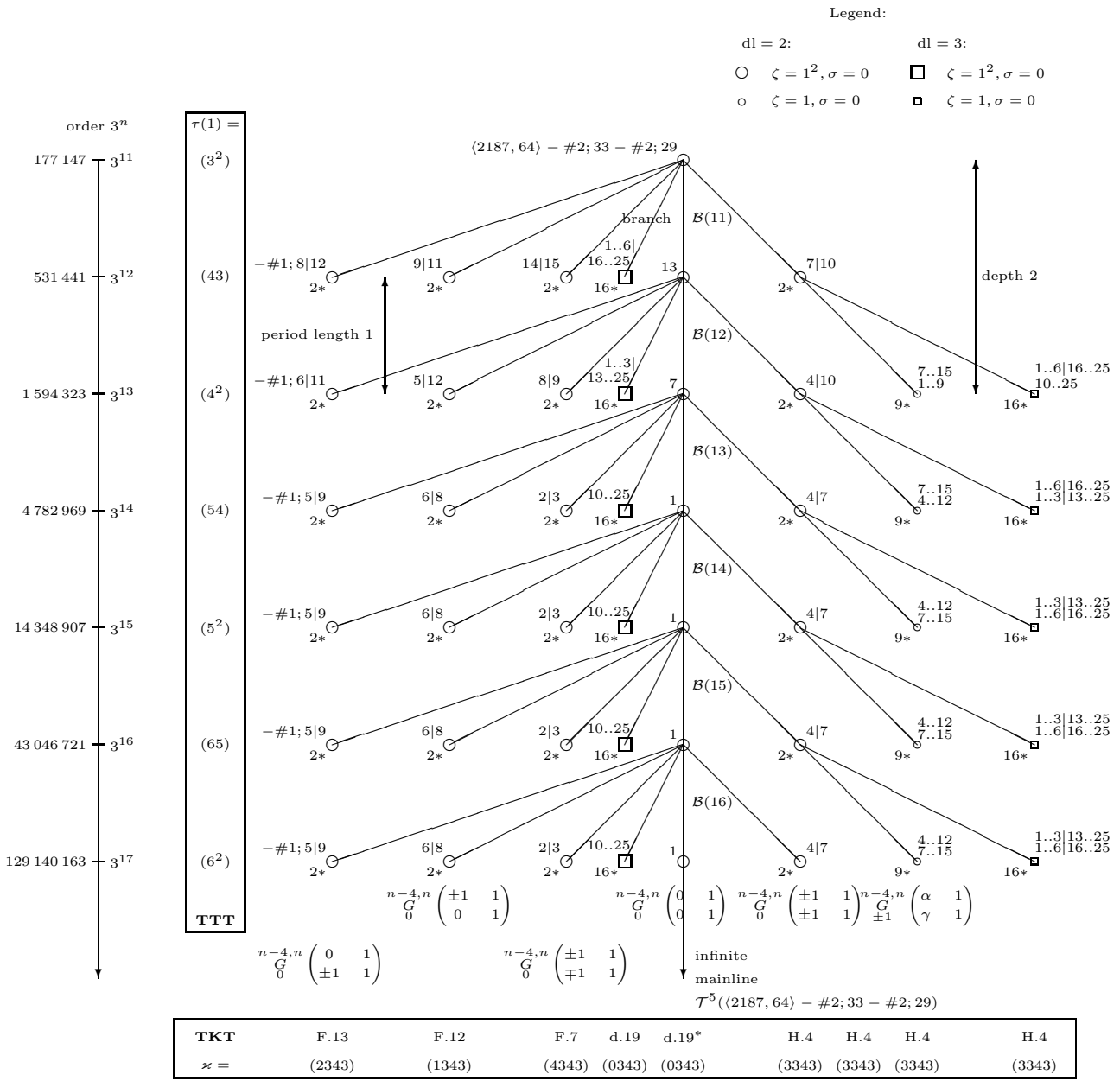}

\end{figure}

%\newpage

%--------------------------------------------------------------------------------

\subsection{The unique mainline of type \(\mathrm{d}.19^\ast\) for odd coclass \(r\ge 5\)}
\label{ss:d19Cc5}

\begin{proposition}
\label{prp:TKTd19TreeCc5}
\textbf{(Periodicity and descendant numbers.)} \\
The branches \(\mathcal{B}(i)\), \(i\ge n_\ast=11\), of
the unique coclass-\(5\) tree \(\mathcal{T}^5(P_9-\#2;29)\)
with mainline vertices of transfer kernel type \(\mathrm{d}.19^\ast\), \(\varkappa\sim (0343)\),
are purely periodic with primitive length \(\ell=1\) and without pre-period, \(\ell_\ast=0\), that is,
\(\mathcal{B}(i+1)\simeq\mathcal{B}(i)\) are isomorphic as structured digraphs, for all \(i\ge p_\ast=n_\ast+\ell_\ast=11\).

The graph theoretic structure of the tree is determined uniquely by the numbers
\(N_1\) of immediate descendants and \(C_1\) of capable immediate descendants
for mainline vertices \(m_n\) with logarithmic order \(n=\mathrm{lo}(m_n)\ge n_\ast=11\)
and for capable vertices \(v\) with depth \(1\) and \(\mathrm{lo}(v)\ge n_\ast+1=12\): \\
\((N_1,C_1)=(25,3)\) for all mainline vertices \(m_n\) of any logarithmic order \(n\ge 11\), \\
\((N_1,C_1)=(25,0)\) for two capable vertices \(v\) of depth \(1\) and any logarithmic order \(\mathrm{lo}(v)\ge 12\).
\end{proposition}

%--------------------------------------------------------------------------------

\begin{proof}
(of Proposition
\ref{prp:TKTd19TreeCc5})
The statements concerning the numbers \(N_1(m_n)\) of immediate descendants
of the mainline vertices \(m_n\) with \(n\ge n_\ast=11\),
and \(N_1(v)\) of vertices with depth \(\mathrm{dp}(v)=1\)
and logarithmic order \(\mathrm{lo}(v)\ge n_\ast+1=12\),
have been obtained by direct computation with the \(p\)-group generation algorithm
\cite{Nm2,Ob,HEO}
in MAGMA
\cite{MAGMA}.
In detail, we proved that there is no pre-period, \(\ell_\ast=0\), 
and the primitive period \((\mathcal{B}(11))\) of length \(\ell=1\) consists of \\
\(9\), resp. \(18\), metabelian vertices with \(\zeta=1^2\), resp. \(\zeta=1\), and \\
\(16\), resp. \(32\), non-metabelian vertices with \(\zeta=1^2\), resp. \(\zeta=1\), \\
(i.e. \(25=9+16\) children of \(m_{11}\), and \(50=2\cdot (9+16)\) children of \(v_{2,3}(m_{11})\), both with depth \(1\)) \\
together \(75\) vertices (\(27\) of them metabelian) in branch \(\mathcal{B}(11)\).

The results concerning the metabelian skeleton confirm the corresponding statements in the dissertation of Nebelung
\cite[Thm. 5.1.16, pp. 178--179, and the third Figure, \(e\ge 4\), \(e\equiv 0\pmod{2}\), on the third page between pp. 191--192]{Ne}.
The tree \(\mathcal{T}^5{R_2^5}\) corresponds to
the infinite metabelian pro-\(3\) group \(S_{4,4}\) in
\cite[Cnj. 15 (a), p. 116]{Ek}.

The claim of the virtual periodicity of branches
has been proved generally for any coclass tree in
\cite{dS}
and
\cite{EkLg}.
Here, the strict periodicity was confirmed by computation up to branch \(\mathcal{B}(30)\)
and certainly sets in at \(p_\ast=11\).
\end{proof}
%--------------------------------------------------------------------------------

\begin{theorem}
\label{thm:TKTd19TreeCc5}
\textbf{(Graph theoretic and algebraic invariants.)} \\
The coclass-\(5\) tree \(\mathcal{T}:=\mathcal{T}^5{R_2^5}\) of \(3\)-groups \(G\) with coclass \(\mathrm{cc}(G)=5\)
which arises from the metabelian root \(R_2^5:=P_9-\#2;29\)
has the following abstract graph theoretic properties.
\begin{enumerate}
\item
The branches \(\mathcal{B}(i)\), \(i\ge 11\), are purely periodic with
primitive period \((\mathcal{B}(11))\) of length \(\ell=1\).
\item
The cardinality of the periodic branch is
\(\#\mathcal{B}(11)=75\).
\item
Depth, width, and information content of the tree are given by
\begin{equation}
\label{eqn:TKTd19TreeCc5}
\mathrm{dp}(\mathcal{T}^5{R_2^5})=2, \quad \mathrm{wd}(\mathcal{T}^5{R_2^5})=75, \quad \text{ and } \quad \mathrm{IC}(\mathcal{T}^5{R_2^5})=75.
\end{equation}
\end{enumerate}
The algebraic invariants of the vertices forming the root and the primitive period \((\mathcal{B}(11))\) of the tree
are presented in Table
\ref{tbl:TKTd19TreeCc5}.
The leading six branches \(\mathcal{B}(11),\ldots,\mathcal{B}(16)\) are drawn in Figure
\ref{fig:TKTd19TreeCc5}.
\end{theorem}

%\newpage

%--------------------------------------------------------------------------------

\renewcommand{\arraystretch}{1.2}

\begin{table}[ht]
%\caption{Root and primitive period \((\mathcal{B}(11))\) of \(\mathcal{T}^4(P_9-\#2;29)\)}
\caption{Data for \(3\)-groups \(G\) with \(11\le n=\mathrm{lo}(G)\le 13\) of the coclass tree \(\mathcal{T}^5{R_2^5}\)}
\label{tbl:TKTd19TreeCc5}
\begin{center}
\begin{tabular}{|c|c|l||c|c|c||c|c||c|c|l|c||c|c|}
\hline
 \(\#\) & \(m,n\)  & \(\rho;\alpha,\beta,\gamma,\delta\) & dp & dl & \(\zeta\) & \(\mu\) & \(\nu\) & \(\tau(1)\) & \(\tau_2\) & Type                   & \(\varkappa\) & \(\sigma\) & \(\#\mathrm{Aut}\) \\
\hline
  \(1\) & \(7,11\) & \(0;0,1,0,1\)                       & 0  & 2  & \(1^2\)   & \(5\)   & \(1\)   & \(3^2\)     & \(32^3\)   & \(\mathrm{d}.19^\ast\) & \((0343)\) & \(0\)    & \(3^{18}\)         \\
\hline
  \(1\) & \(8,12\) & \(0;0,1,0,1\)                       & 0  & 2  & \(1^2\)   & \(5\)   & \(1\)   & \(43\)      & \(3^22^2\) & \(\mathrm{d}.19^\ast\) & \((0343)\) & \(0\)    & \(3^{20}\)         \\
  \(2\) & \(8,12\) & \(0;\pm 1,1,\mp 1,1\)               & 1  & 2  & \(1^2\)   & \(4\)   & \(0\)   & \(43\)      & \(32^21\)  & \(\mathrm{F}.7\)       & \((4343)\)    & \(0\)      & \(3^{20}\)         \\
  \(2\) & \(8,12\) & \(0;\pm 1,1,0,1\)                   & 1  & 2  & \(1^2\)   & \(4\)   & \(0\)   & \(43\)      & \(32^21\)  & \(\mathrm{F}.12\)      & \((1343)\)    & \(0\)      & \(3^{20}\)         \\
  \(2\) & \(8,12\) & \(0;0,1,\pm 1,1\)                   & 1  & 2  & \(1^2\)   & \(4\)   & \(0\)   & \(43\)      & \(32^21\)  & \(\mathrm{F}.13\)      & \((2343)\)    & \(0\)      & \(3^{20}\)         \\
  \(2\) & \(8,12\) & \(0;\pm 1,1,\pm 1,1\)               & 1  & 2  & \(1^2\)   & \(5\)   & \(1\)   & \(43\)      & \(32^21\)  & \(\mathrm{H}.4\)       & \((3343)\)    & \(0\)      & \(3^{20}\)         \\
 \(12\) & \(8,12\) &                                     & 1  & 3  & \(1^2\)   & \(4\)   & \(0\)   & \(3^2\)     & \(32^3\)   & \(\mathrm{d}.19\)      & \((0343)\)    & \(0\)      & \(3^{19}\)         \\
  \(4\) & \(8,12\) &                                     & 1  & 3  & \(1^2\)   & \(4\)   & \(0\)   & \(3^2\)     & \(32^3\)   & \(\mathrm{d}.19\)      & \((0343)\)    & \(0\)      & \(3^{18}\)         \\
  \(9\) & \(9,13\) & \(\pm 1;\alpha,1,\gamma,1\)         & 2  & 2  & \(1\)     & \(4\)   & \(0\)   & \(43\)     & \(432^2\)   & \(\mathrm{H}.4\)       & \((3343)\)    & \(0\)      & \(3^{22}\)         \\
 \(12\) & \(9,13\) &                                     & 2  & 3  & \(1\)     & \(4\)   & \(0\)   & \(43\)     & \(3^22^2\)  & \(\mathrm{H}.4\)       & \((3343)\)    & \(0\)      & \(3^{21}\)         \\
  \(4\) & \(9,13\) &                                     & 2  & 3  & \(1\)     & \(4\)   & \(0\)   & \(43\)     & \(3^22^2\)  & \(\mathrm{H}.4\)       & \((3343)\)    & \(0\)      & \(3^{20}\)         \\
\hline
\end{tabular}
\end{center}
\end{table}

%\newpage

%--------------------------------------------------------------------------------

\begin{proof}
(Proof of Theorem
\ref{thm:TKTd19TreeCc5})
Since every mainline vertex \(m_n\) of the tree \(\mathcal{T}\)
has three capable children, \(C_1(m_n)=3\),
but every capable vertex \(v\) of depth \(1\)
has only terminal children, \(C_1(v)=0\),
according to Proposition
\ref{prp:TKTd19TreeCc5},
the depth of the tree is \(\mathrm{dp}(\mathcal{T})=2\).
In this case,
the cardinality of a branch \(\mathcal{B}\) is the sum of
the number \(N_1(m_n)\) of immediate descendants of the branch root \(m_n\)
and the numbers \(N_1(v_i)\) of terminal children of capable vertices \(v_i\) of depth \(1\)
with \(2\le i\le C_1(m_n)\) (excluding the next mainline vertex \(v_1=m_{n+1}\)),
according to Formula
\eqref{eqn:BranchDepth2},
that is,
\[\#\mathcal{B}=N_1(m_n)+\sum_{i=2}^{C_1(m_n)}\,N_1(v_i).\]
Applied to the primitive period, this yields
\(\#\mathcal{B}(11)=25+25+25=75\). 
According to Formula
\eqref{eqn:TreeWidth2},
the width of the tree is the maximum of all sums of the shape
\[\#\mathrm{Lyr}_n{\mathcal{T}}=\#\lbrace v\in\mathcal{T}\mid\mathrm{lo}(v)=n\rbrace=N_1(m_{n-1})+\sum_{i=1}^{C_1(m_{n-2})}\,N_1(v_i(m_{n-2})),\]
taken over all branch roots \(m_n\) with logarithmic orders \(n_\ast+2\le n\le n_\ast+\ell_\ast+\ell+1\).
Applied to \(n_\ast=11\), \(\ell_\ast=0\), and \(\ell=1\), this yields
\(\mathrm{wd}(\mathcal{T})=\max(25+25+25)=\max(75)=75\). \\
Finally, we have \(\mathrm{IC}(\mathcal{T})=\#\mathcal{B}(11)=75\).
\end{proof}

%\newpage

%--------------------------------------------------------------------------------

\begin{corollary}
\label{cor:TKTd19TreeCc5}
\textbf{(Actions and relation ranks.)} \\
The algebraic invariants of the vertices of the structured coclass-\(5\) tree \(\mathcal{T}^5{R_2^5}\) are listed in Table
\ref{tbl:TKTd19TreeCc5}.
In particular:
\begin{enumerate}
\item
There are no groups with GI-action, let alone with RI- or \(V_4\)-action.
\item
The relation rank is given by
\(\mu=5\) for the mainline vertices \(\stackbin[0]{n-4,n}{G}\begin{pmatrix}0&1\\ 0&1\end{pmatrix}\) with \(n\ge 11\),
and the capable vertices \(\stackbin[0]{n-4,n}{G}\begin{pmatrix}\pm 1&1\\ \pm 1&1\end{pmatrix}\) of depth \(1\) with \(n\ge 12\), and
\(\mu=4\) otherwise.
\end{enumerate}
\end{corollary}

%--------------------------------------------------------------------------------

\begin{proof}
(of Corollary
\ref{cor:TKTd19TreeCc5})
The existence of an RI-action on \(G\) has been checked by means of
an algorithm involving the \(p\)-covering group of \(G\), written for MAGMA
\cite{MAGMA}.
The other claims follow immediately from Table
\ref{tbl:TKTd19TreeCc5},
continued indefinitely with the aid of the periodicity in Proposition
\ref{prp:TKTd19TreeCc5}. 
\end{proof}

%\newpage

%--------------------------------------------------------------------------------

\begin{figure}[hb]
\caption{The unique coclass-\(5\) tree \(\mathcal{T}^5(P_9-\#2;37)\) with mainline of type \(\mathrm{d}.23^\ast\)}
\label{fig:TKTd23TreeCc5}

\input{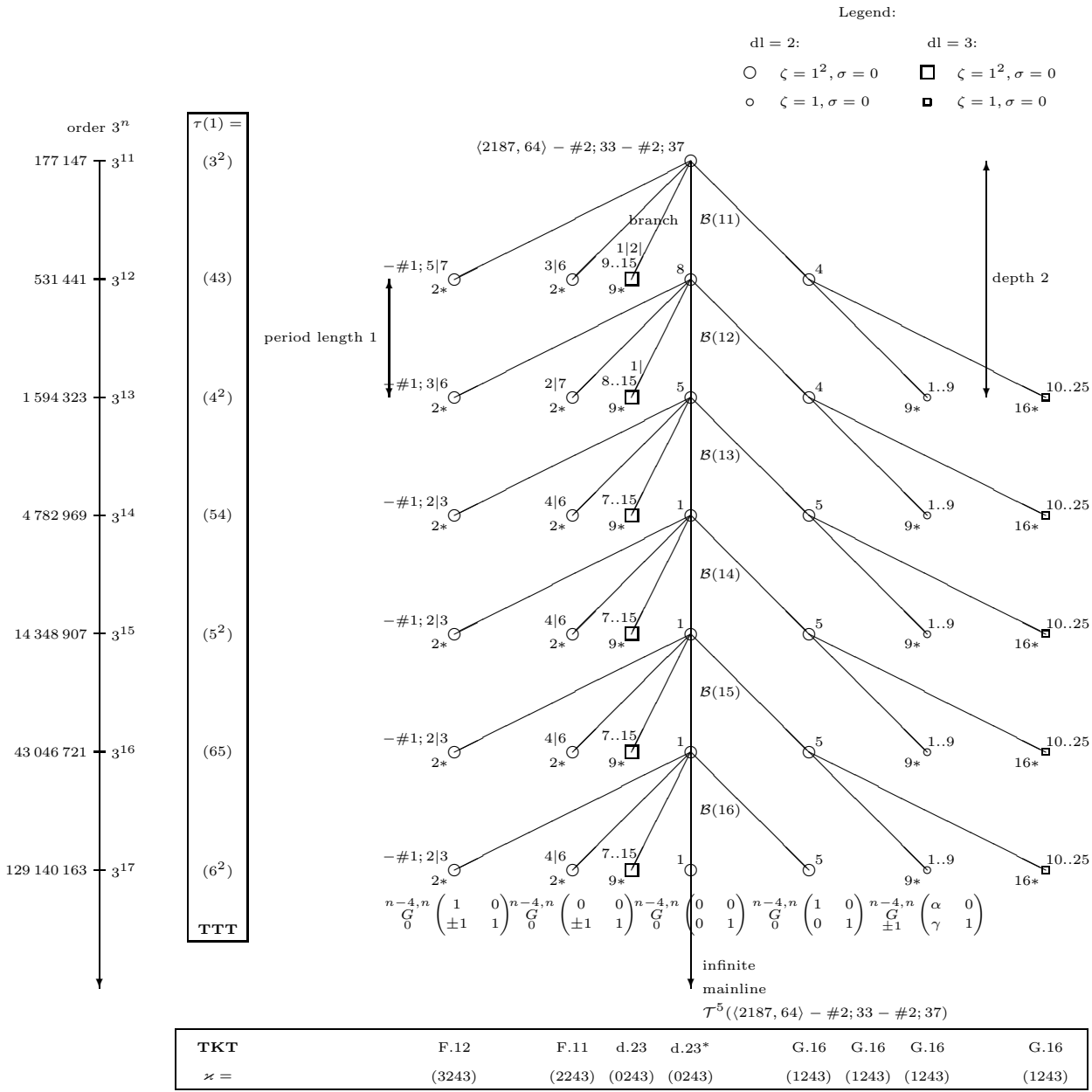}

\end{figure}

%\newpage

%--------------------------------------------------------------------------------

\subsection{The unique mainline of type \(\mathrm{d}.23^\ast\) for odd coclass \(r\ge 5\)}
\label{ss:d23Cc5}

\begin{proposition}
\label{prp:TKTd23TreeCc5}
\textbf{(Periodicity and descendant numbers.)} \\
The branches \(\mathcal{B}(i)\), \(i\ge n_\ast=11\), of
the unique coclass-\(5\) tree \(\mathcal{T}^5(P_9-\#2;37)\)
with mainline vertices of transfer kernel type \(\mathrm{d}.23^\ast\), \(\varkappa\sim (0243)\),
are purely periodic with primitive length \(\ell=1\) and without pre-period, \(\ell_\ast=0\), that is,
\(\mathcal{B}(i+1)\simeq\mathcal{B}(i)\) are isomorphic as structured digraphs, for all \(i\ge p_\ast=n_\ast+\ell_\ast=11\).

The graph theoretic structure of the tree is determined uniquely by the numbers
\(N_1\) of immediate descendants and \(C_1\) of capable immediate descendants
for mainline vertices \(m_n\) with logarithmic order \(n=\mathrm{lo}(m_n)\ge n_\ast=11\)
and for capable vertices \(v\) with depth \(1\) and \(\mathrm{lo}(v)\ge n_\ast+1=12\): \\
\((N_1,C_1)=(15,2)\) for all mainline vertices \(m_n\) of any logarithmic order \(n\ge 11\), \\
\((N_1,C_1)=(25,0)\) for the capable vertex \(v\) of depth \(1\) and any logarithmic order \(\mathrm{lo}(v)\ge 12\).
\end{proposition}

%--------------------------------------------------------------------------------

\begin{proof}
(of Proposition
\ref{prp:TKTd23TreeCc5})
The statements concerning the numbers \(N_1(m_n)\) of immediate descendants
of the mainline vertices \(m_n\) with \(n\ge n_\ast=11\),
and \(N_1(v)\) of vertices with depth \(\mathrm{dp}(v)=1\)
and logarithmic order \(\mathrm{lo}(v)\ge n_\ast+1=12\),
have been obtained by direct computation with the \(p\)-group generation algorithm
\cite{Nm2,Ob,HEO}
in MAGMA
\cite{MAGMA}.
In detail, we proved that there is no pre-period, \(\ell_\ast=0\), 
and the primitive period \((\mathcal{B}(11))\) of length \(\ell=1\) consists of \\
\(6\), resp. \(9\), metabelian vertices with \(\zeta=1^2\), resp. \(\zeta=1\), and \\
\(9\), resp. \(16\), non-metabelian vertices with \(\zeta=1^2\), resp. \(\zeta=1\), \\
(i.e. \(15=6+9\) children of \(m_{11}\), and \(25=9+16)\) children of \(v_{2}(m_{11})\) with depth \(1\)) \\
together \(40\) vertices (\(15\) of them metabelian) in branch \(\mathcal{B}(11)\).

The results concerning the metabelian skeleton confirm the corresponding statements in the dissertation of Nebelung
\cite[Thm. 5.1.16, pp. 178--179, and the third Figure, \(e\ge 4\), \(e\equiv 0\pmod{2}\), on the third page between pp. 191--192]{Ne}.
The tree \(\mathcal{T}^5{R_3^5}\) corresponds to
the infinite metabelian pro-\(3\) group \(S_{4,2}\) in
\cite[Cnj. 15 (a), p. 116]{Ek}.

The claim of the virtual periodicity of branches
has been proved generally for any coclass tree in
\cite{dS}
and
\cite{EkLg}.
Here, the strict periodicity was confirmed by computation up to branch \(\mathcal{B}(30)\)
and certainly sets in at \(p_\ast=11\).
\end{proof}
%--------------------------------------------------------------------------------

\begin{theorem}
\label{thm:TKTd23TreeCc5}
\textbf{(Graph theoretic and algebraic invariants.)} \\
The coclass-\(5\) tree \(\mathcal{T}:=\mathcal{T}^5{R_3^5}\) of \(3\)-groups \(G\) with coclass \(\mathrm{cc}(G)=5\)
arises from the metabelian root \(R_3^5:=P_9-\#2;37\)
and has the following abstract graph theoretic properties.
\begin{enumerate}
\item
The branches \(\mathcal{B}(i)\), \(i\ge 11\), are purely periodic with
primitive period \((\mathcal{B}(11))\) of length \(\ell=1\).
\item
The cardinality of the periodic branch is
\(\#\mathcal{B}(11)=40\).
\item
Depth, width, and information content of the tree are given by
\begin{equation}
\label{eqn:TKTd23TreeCc5}
\mathrm{dp}(\mathcal{T}^5{R_3^5})=2, \quad \mathrm{wd}(\mathcal{T}^5{R_3^5})=40, \quad \text{ and } \quad \mathrm{IC}(\mathcal{T}^5{R_3^5})=40.
\end{equation}
\end{enumerate}
The algebraic invariants of the vertices forming the root and the primitive period \((\mathcal{B}(11))\) of the tree
are presented in Table
\ref{tbl:TKTd23TreeCc5}.
The leading six branches \(\mathcal{B}(11),\ldots,\mathcal{B}(16)\) are drawn in Figure
\ref{fig:TKTd23TreeCc5}.
\end{theorem}

%\newpage

%--------------------------------------------------------------------------------

\renewcommand{\arraystretch}{1.2}

\begin{table}[ht]
%\caption{Root and primitive period \((\mathcal{B}(11))\) of \(\mathcal{T}^5(P_9-\#2;\ast)\)}
\caption{Data for \(3\)-groups \(G\) with \(11\le n=\mathrm{lo}(G)\le 13\) of the coclass tree \(\mathcal{T}^5{R_3^5}\)}
\label{tbl:TKTd23TreeCc5}
\begin{center}
\begin{tabular}{|c|c|l||c|c|c||c|c||c|c|l|c||c|c|}
\hline
 \(\#\) & \(m,n\)  & \(\rho;\alpha,\beta,\gamma,\delta\) & dp & dl & \(\zeta\) & \(\mu\) & \(\nu\) & \(\tau(1)\) & \(\tau_2\) & Type                   & \(\varkappa\) & \(\sigma\) & \(\#\mathrm{Aut}\) \\
\hline
  \(1\) & \(7,11\) & \(0;0,0,0,1\)                       & 0  & 2  & \(1^2\)   & \(5\)   & \(1\)   & \(3^2\)     & \(32^3\)   & \(\mathrm{d}.23^\ast\) & \((0243)\) & \(0\)    & \(3^{18}\)         \\
\hline
  \(1\) & \(8,12\) & \(0;0,0,0,1\)                       & 0  & 2  & \(1^2\)   & \(5\)   & \(1\)   & \(43\)      & \(3^22^2\) & \(\mathrm{d}.23^\ast\) & \((0243)\) & \(0\)    & \(3^{20}\)         \\
  \(2\) & \(8,12\) & \(0;0,0,\pm 1,1\)                   & 1  & 2  & \(1^2\)   & \(4\)   & \(0\)   & \(43\)      & \(3^22^2\) & \(\mathrm{F}.11\)      & \((2243)\)    & \(0\)      & \(2\cdot 3^{20}\)  \\
  \(2\) & \(8,12\) & \(0;1,0,\pm 1,1\)                   & 1  & 2  & \(1^2\)   & \(4\)   & \(0\)   & \(43\)      & \(3^22^2\) & \(\mathrm{F}.12\)      & \((3243)\)    & \(0\)      & \(3^{20}\)         \\
  \(1\) & \(8,12\) & \(0;1,0,0,1\)                       & 1  & 2  & \(1^2\)   & \(5\)   & \(1\)   & \(43\)      & \(3^22^2\) & \(\mathrm{G}.16\)      & \((1243)\)    & \(0\)      & \(3^{20}\)         \\
  \(6\) & \(8,12\) &                                     & 1  & 3  & \(1^2\)   & \(4\)   & \(0\)   & \(3^2\)     & \(32^3\)   & \(\mathrm{d}.23\)      & \((0243)\)    & \(0\)      & \(3^{19}\)         \\
  \(2\) & \(8,12\) &                                     & 1  & 3  & \(1^2\)   & \(4\)   & \(0\)   & \(3^2\)     & \(32^3\)   & \(\mathrm{d}.23\)      & \((0243)\)    & \(0\)      & \(2\cdot 3^{18}\)  \\
  \(1\) & \(8,12\) &                                     & 1  & 3  & \(1^2\)   & \(4\)   & \(0\)   & \(3^2\)     & \(32^3\)   & \(\mathrm{d}.23\)      & \((0243)\)    & \(0\)      & \(3^{18}\)         \\
  \(9\) & \(9,13\) & \(\pm 1;\alpha,0,\gamma,1\)         & 2  & 2  & \(1\)     & \(4\)   & \(0\)   & \(43\)     & \(432^2\)   & \(\mathrm{G}.16\)      & \((1243)\)    & \(0\)      & \(3^{22}\)         \\
 \(12\) & \(9,13\) &                                     & 2  & 3  & \(1\)     & \(4\)   & \(0\)   & \(43\)     & \(3^22^2\)  & \(\mathrm{G}.16\)      & \((1243)\)    & \(0\)      & \(3^{21}\)         \\
  \(4\) & \(9,13\) &                                     & 2  & 3  & \(1\)     & \(4\)   & \(0\)   & \(43\)     & \(3^22^2\)  & \(\mathrm{G}.16\)      & \((1243)\)    & \(0\)      & \(3^{20}\)         \\
\hline
\end{tabular}
\end{center}
\end{table}

%\newpage

%--------------------------------------------------------------------------------

\begin{proof}
(Proof of Theorem
\ref{thm:TKTd23TreeCc5})
Since every mainline vertex \(m_n\) of the tree \(\mathcal{T}\)
has two capable children, \(C_1(m_n)=2\),
but every capable vertex \(v\) of depth \(1\)
has only terminal children, \(C_1(v)=0\),
according to Proposition
\ref{prp:TKTd23TreeCc5},
the depth of the tree is \(\mathrm{dp}(\mathcal{T})=2\).
In this case,
the cardinality of a branch \(\mathcal{B}\) is the sum of the number \(N_1(m_n)\) of immediate descendants of the branch root \(m_n\)
and the numbers \(N_1(v_i)\) of terminal children of capable vertices \(v_i\) of depth \(1\),
\(2\le i\le C_1(m_n)\) (\(v_1\) the next mainline vertex must be omitted),
according to Formula
\eqref{eqn:BranchDepth2},
\[\#\mathcal{B}=N_1(m_n)+\sum_{i=2}^{C_1(m_n)}\,N_1(v_i).\]
Applied to the primitive period, this yields
\(\#\mathcal{B}(11)=15+25=40\). 
According to Formula
\eqref{eqn:TreeWidth2},
the width of the tree is the maximum of all sums of the shape
\[\#\mathrm{Lyr}_n{\mathcal{T}}=\#\lbrace v\in\mathcal{T}\mid\mathrm{lo}(v)=n\rbrace=N_1(m_{n-1})+\sum_{i=1}^{C_1(m_{n-2})}\,N_1(v_i(m_{n-2})),\]
taken over all branch roots \(m_n\) with logarithmic orders \(n_\ast+2\le n\le n_\ast+\ell_\ast+\ell+1\).
Applied to \(n_\ast=11\), \(\ell_\ast=0\), and \(\ell=1\), this yields
\(\mathrm{wd}(\mathcal{T})=\max(15+25)=\max(40)=40\).
Finally, we have \(\mathrm{IC}(\mathcal{T})=\#\mathcal{B}(11)=40\).
\end{proof}

%\newpage

%--------------------------------------------------------------------------------

\begin{corollary}
\label{cor:TKTd23TreeCc5}
\textbf{(Actions and relation ranks.)}
The algebraic invariants of the vertices of the structured coclass-\(5\) tree \(\mathcal{T}^5{R_3^5}\) are listed in Table
\ref{tbl:TKTd23TreeCc5}.
In particular:
\begin{enumerate}
\item
There are no groups with GI-action, let alone with RI- or \(V_4\)-action.
\item
The relation rank is given by
\(\mu=5\) for the mainline vertices \(\stackbin[0]{n-4,n}{G}\begin{pmatrix}0&0\\ 0&1\end{pmatrix}\) with \(n\ge 11\),
and the capable vertices \(\stackbin[0]{n-4,n}{G}\begin{pmatrix}1&0\\ 0&1\end{pmatrix}\) of depth \(1\) with \(n\ge 12\), and
\(\mu=4\) otherwise.
\end{enumerate}
\end{corollary}

%--------------------------------------------------------------------------------

\begin{proof}
(of Corollary
\ref{cor:TKTd23TreeCc5})
The existence of an RI-action on \(G\) has been checked by means of
an algorithm involving the \(p\)-covering group of \(G\), written for MAGMA
\cite{MAGMA}.
The other claims follow immediately from Table
\ref{tbl:TKTd23TreeCc5},
continued indefinitely with the aid of the periodicity in Prop.
\ref{prp:TKTd23TreeCc5}. 
\end{proof}

%\newpage

%--------------------------------------------------------------------------------

\begin{figure}[hb]
\caption{The unique coclass-\(5\) tree \(\mathcal{T}^5(P_9-\#2;39)\) with mainline of type \(\mathrm{d}.25^\ast\)}
\label{fig:TKTd25TreeCc5}

\input{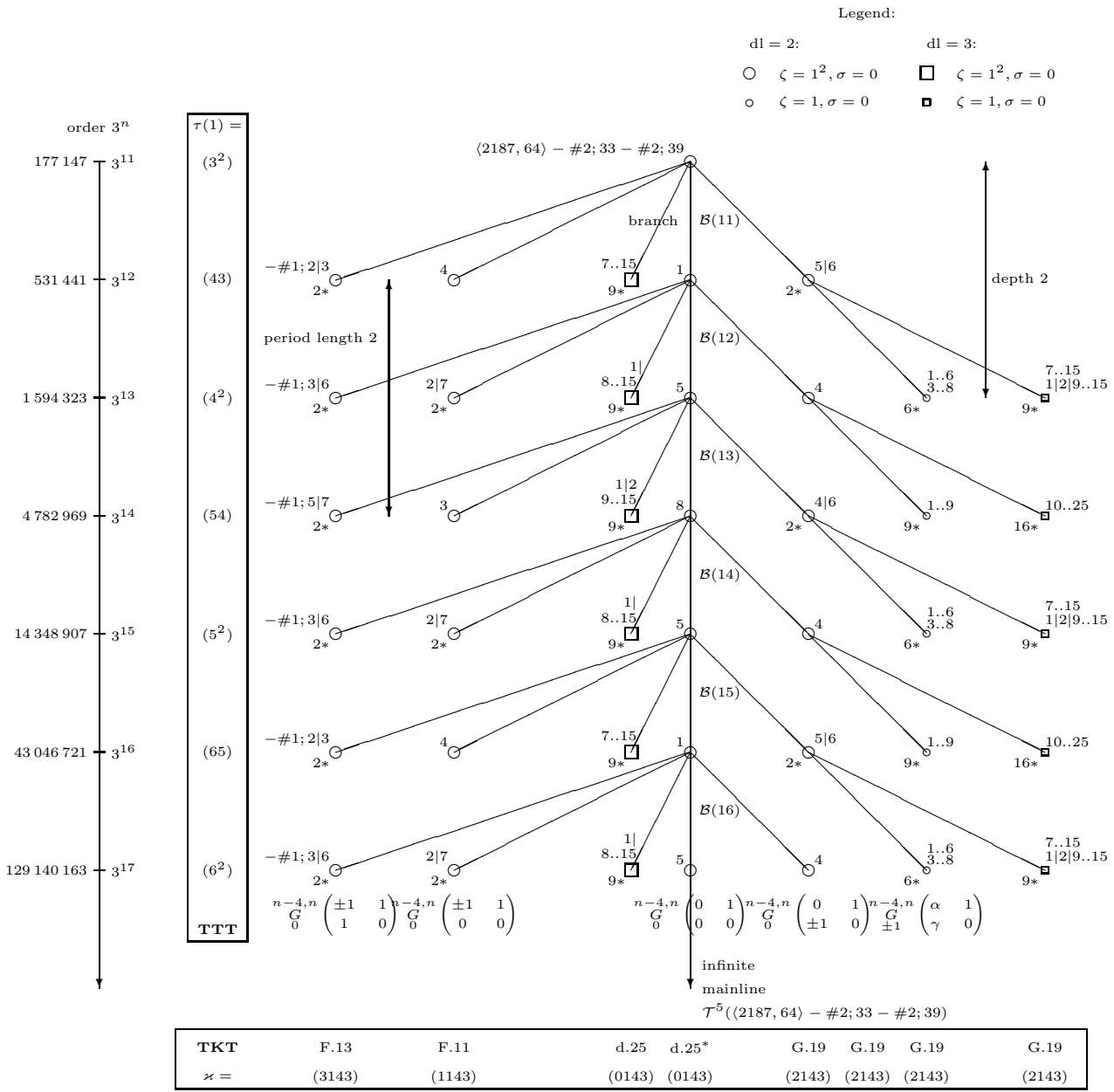}

\end{figure}

%\newpage

%--------------------------------------------------------------------------------

\subsection{The unique mainline of type \(\mathrm{d}.25^\ast\) for odd coclass \(r\ge 5\)}
\label{ss:d25Cc5}

\begin{proposition}
\label{prp:TKTd25TreeCc5}
\textbf{(Periodicity and descendant numbers.)} \\
The branches \(\mathcal{B}(i)\), \(i\ge n_\ast=11\), of
the unique coclass-\(5\) tree \(\mathcal{T}^5(P_9-\#2;39)\)
with mainline vertices of transfer kernel type \(\mathrm{d}.25^\ast\), \(\varkappa\sim (0143)\),
are purely periodic with primitive length \(\ell=2\) and without pre-period, \(\ell_\ast=0\), that is,
\(\mathcal{B}(i+2)\simeq\mathcal{B}(i)\) are isomorphic as structured digraphs, for all \(i\ge p_\ast=n_\ast+\ell_\ast=11\).

The graph theoretic structure of the tree is determined uniquely by the numbers
\(N_1\) of immediate descendants and \(C_1\) of capable immediate descendants
for mainline vertices \(m_n\) with logarithmic order \(n=\mathrm{lo}(m_n)\ge n_\ast=11\)
and for capable vertices \(v\) with depth \(1\) and \(\mathrm{lo}(v)\ge n_\ast+1=12\): \\
\((N_1,C_1)=(15,3)\) for all mainline vertices \(m_n\) of odd logarithmic order \(n\ge 11\), \\
\((N_1,C_1)=(15,2)\) for all mainline vertices \(m_n\) of even logarithmic order \(n\ge 12\), \\
\((N_1,C_1)=(15,0)\) for two capable vertices \(v\) of depth \(1\) and even logarithmic order \(\mathrm{lo}(v)\ge 12\), \\
\((N_1,C_1)=(25,0)\) for the capable vertex \(v\) of depth \(1\) and odd logarithmic order \(\mathrm{lo}(v)\ge 13\).
\end{proposition}

%--------------------------------------------------------------------------------

\begin{proof}
(of Proposition
\ref{prp:TKTd25TreeCc5})
The statements concerning the numbers \(N_1(m_n)\) of immediate descendants
of the mainline vertices \(m_n\) with \(n\ge n_\ast=11\),
and \(N_1(v)\) of vertices with depth \(\mathrm{dp}(v)=1\)
and logarithmic order \(\mathrm{lo}(v)\ge n_\ast+1=12\),
have been obtained by direct computation with the \(p\)-group generation algorithm
\cite{Nm2,Ob,HEO}
in MAGMA
\cite{MAGMA}.
In detail, we proved that there is no pre-period, \(\ell_\ast=0\), 
and the primitive period \((\mathcal{B}(11),\mathcal{B}(12))\) of length \(\ell=2\) consists of \\
\(6\), resp. \(12\), metabelian vertices with \(\zeta=1^2\), resp. \(\zeta=1\), and \\
\(9\), resp. \(18\), non-metabelian vertices with \(\zeta=1^2\), resp. \(\zeta=1\), \\
(i.e. \(15=6+9\) children of \(m_{11}\), and \(30=2\cdot (6+9)\) children of \(v_{2,3}(m_{11})\), both with depth \(1\)) \\
together \(45\) vertices (\(18\) of them metabelian) in branch \(\mathcal{B}(11)\), and \\
\(6\), resp. \(9\), metabelian vertices with \(\zeta=1^2\), resp. \(\zeta=1\), and \\
\(9\), resp. \(16\), non-metabelian vertices with \(\zeta=1^2\), resp. \(\zeta=1\), \\
(i.e. \(15=6+9\) children of \(m_{12}\), and \(25=9+16\) children of \(v_{2}(m_{12})\) with depth \(1\)) \\
together \(40\) vertices (\(15\) of them metabelian) in branch \(\mathcal{B}(12)\).

The results concerning the metabelian skeleton confirm the corresponding statements in the dissertation of Nebelung
\cite[Thm. 5.1.16, pp. 178--179, and the third Figure, \(e\ge 4\), \(e\equiv 0\pmod{2}\), on the third page between pp. 191--192]{Ne}.
The tree \(\mathcal{T}^5{R_4^5}\) corresponds to
the infinite metabelian pro-\(3\) group \(S_{4,3}\) in
\cite[Cnj. 15 (a), p. 116]{Ek}.

The claim of the virtual periodicity of branches
has been proved generally for any coclass tree in
\cite{dS}
and
\cite{EkLg}.
Here, the strict periodicity was confirmed by computation up to branch \(\mathcal{B}(30)\)
and certainly sets in at \(p_\ast=11\).
\end{proof}

%--------------------------------------------------------------------------------

\begin{theorem}
\label{thm:TKTd25TreeCc5}
\textbf{(Graph theoretic and algebraic invariants.)} \\
The coclass-\(5\) tree \(\mathcal{T}:=\mathcal{T}^5{R_4^5}\) of \(3\)-groups \(G\) with coclass \(\mathrm{cc}(G)=5\)
which arises from the metabelian root \(R_4^5:=P_9-\#2;39\)
has the following abstract graph theoretic properties.
\begin{enumerate}
\item
The branches \(\mathcal{B}(i)\), \(i\ge 11\), are purely periodic with
primitive period \((\mathcal{B}(11),\mathcal{B}(12))\) of length \(\ell=2\).
\item
The cardinalities of the periodic branches are
\(\#\mathcal{B}(11)=45\) and \(\#\mathcal{B}(12)=40\).
\item
Depth, width, and information content of the tree are given by
\begin{equation}
\label{eqn:TKTd25TreeCc5}
\mathrm{dp}(\mathcal{T}^5{R_4^5})=2, \quad \mathrm{wd}(\mathcal{T}^5{R_4^5})=45, \quad \text{ and } \quad \mathrm{IC}(\mathcal{T}^5{R_4^5})=85.
\end{equation}
\end{enumerate}
The algebraic invariants of the vertices forming the root and the primitive period \((\mathcal{B}(11),\mathcal{B}(12))\) of the tree
are presented in Table
\ref{tbl:TKTd25TreeCc5}.
The leading six branches \(\mathcal{B}(11),\ldots,\mathcal{B}(16)\) are drawn in Figure
\ref{fig:TKTd25TreeCc5}.
\end{theorem}

%\newpage

%--------------------------------------------------------------------------------

\renewcommand{\arraystretch}{1.2}

\begin{table}[ht]
%\caption{Root and primitive period \((\mathcal{B}(11))\) of \(\mathcal{T}^5(P_9-\#2;\bullet)\)}
\caption{Data for \(3\)-groups \(G\) with \(11\le n=\mathrm{lo}(G)\le 14\) of the coclass tree \(\mathcal{T}^5{R_4^5}\)}
\label{tbl:TKTd25TreeCc5}
\begin{center}
\begin{tabular}{|c|c|l||c|c|c||c|c||c|c|l|c||c|c|}
\hline
 \(\#\) & \(m,n\)  & \(\rho;\alpha,\beta,\gamma,\delta\) & dp & dl & \(\zeta\) & \(\mu\) & \(\nu\) & \(\tau(1)\) & \(\tau_2\) & Type                   & \(\varkappa\) & \(\sigma\) & \(\#\mathrm{Aut}\) \\
\hline
  \(1\) & \(7,11\) & \(0;0,1,0,0\)                       & 0  & 2  & \(1^2\)   & \(5\)   & \(1\)   & \(3^2\)     & \(32^3\)   & \(\mathrm{d}.25^\ast\) & \((0143)\) & \(0\)    & \(2\cdot 3^{18}\)  \\
\hline
  \(1\) & \(8,12\) & \(0;0,1,0,0\)                       & 0  & 2  & \(1^2\)   & \(5\)   & \(1\)   & \(43\)      & \(3^22^2\) & \(\mathrm{d}.25^\ast\) & \((0143)\) & \(0\)    & \(2\cdot 3^{20}\)  \\
  \(1\) & \(8,12\) & \(0;1,1,0,0\)                       & 1  & 2  & \(1^2\)   & \(4\)   & \(0\)   & \(43\)      & \(3^22^2\) & \(\mathrm{F}.11\)      & \((1143)\)    & \(0\)      & \(3^{20}\)         \\
  \(2\) & \(8,12\) & \(0;1,1,\pm 1,0\)                   & 1  & 2  & \(1^2\)   & \(4\)   & \(0\)   & \(43\)      & \(3^22^2\) & \(\mathrm{F}.13\)      & \((3143)\)    & \(0\)      & \(3^{20}\)         \\
  \(2\) & \(8,12\) & \(0;0,1,\pm 1,0\)                   & 1  & 2  & \(1^2\)   & \(5\)   & \(1\)   & \(43\)      & \(3^22^2\) & \(\mathrm{G}.19\)      & \((2143)\)    & \(0\)      & \(2\cdot 3^{20}\)  \\
  \(6\) & \(8,12\) &                                     & 1  & 3  & \(1^2\)   & \(4\)   & \(0\)   & \(3^2\)     & \(32^3\)   & \(\mathrm{d}.25\)      & \((0143)\)    & \(0\)      & \(3^{19}\)         \\
  \(2\) & \(8,12\) &                                     & 1  & 3  & \(1^2\)   & \(4\)   & \(0\)   & \(3^2\)     & \(32^3\)   & \(\mathrm{d}.25\)      & \((0143)\)    & \(0\)      & \(2\cdot 3^{18}\)  \\
  \(1\) & \(8,12\) &                                     & 1  & 3  & \(1^2\)   & \(4\)   & \(0\)   & \(3^2\)     & \(32^3\)   & \(\mathrm{d}.25\)      & \((0143)\)    & \(0\)      & \(3^{18}\)         \\
  \(3\) & \(9,13\) & \(\pm 1;\alpha,1,\gamma,0\)         & 2  & 2  & \(1\)     & \(4\)   & \(0\)   & \(43\)      & \(432^2\)  & \(\mathrm{G}.19\)      & \((2143)\)    & \(0\)      & \(2\cdot 3^{22}\)  \\
  \(3\) & \(9,13\) & \(\pm 1;\alpha,1,\gamma,0\)         & 2  & 2  & \(1\)     & \(4\)   & \(0\)   & \(43\)      & \(432^2\)  & \(\mathrm{G}.19\)      & \((2143)\)    & \(0\)      & \(3^{22}\)         \\
  \(6\) & \(9,13\) &                                     & 2  & 3  & \(1\)     & \(4\)   & \(0\)   & \(43\)      & \(3^22^2\) & \(\mathrm{G}.19\)      & \((2143)\)    & \(0\)      & \(3^{21}\)         \\
  \(2\) & \(9,13\) &                                     & 2  & 3  & \(1\)     & \(4\)   & \(0\)   & \(43\)      & \(3^22^2\) & \(\mathrm{G}.19\)      & \((2143)\)    & \(0\)      & \(2\cdot 3^{20}\)  \\
  \(1\) & \(9,13\) &                                     & 2  & 3  & \(1\)     & \(4\)   & \(0\)   & \(43\)      & \(3^22^2\) & \(\mathrm{G}.19\)      & \((2143)\)    & \(0\)      & \(3^{20}\)         \\
\hline
  \(1\) & \(9,13\) & \(0;0,1,0,0\)                       & 0  & 2  & \(1^2\)   & \(5\)   & \(1\)   & \(4^2\)     & \(432^2\)  & \(\mathrm{d}.25^\ast\) & \((0143)\) & \(0\)    & \(2\cdot 3^{22}\)  \\
  \(2\) & \(9,13\) & \(0;\pm 1,1,0,0\)                   & 1  & 2  & \(1^2\)   & \(4\)   & \(0\)   & \(4^2\)     & \(432^2\)  & \(\mathrm{F}.11\)      & \((1143)\)    & \(0\)      & \(3^{22}\)         \\
  \(2\) & \(9,13\) & \(0;\pm 1,1,1,0\)                   & 1  & 2  & \(1^2\)   & \(4\)   & \(0\)   & \(4^2\)     & \(432^2\)  & \(\mathrm{F}.13\)      & \((3143)\)    & \(0\)      & \(3^{22}\)         \\
  \(1\) & \(9,13\) & \(0;0,1,1,0\)                       & 1  & 2  & \(1^2\)   & \(5\)   & \(1\)   & \(4^2\)     & \(432^2\)  & \(\mathrm{G}.19\)      & \((2143)\)    & \(0\)      & \(2\cdot 3^{22}\)  \\
  \(6\) & \(9,13\) &                                     & 1  & 3  & \(1^2\)   & \(4\)   & \(0\)   & \(43\)      & \(3^22^2\) & \(\mathrm{d}.25\)      & \((0143)\)    & \(0\)      & \(3^{21}\)         \\
  \(2\) & \(9,13\) &                                     & 1  & 3  & \(1^2\)   & \(4\)   & \(0\)   & \(43\)      & \(3^22^2\) & \(\mathrm{d}.25\)      & \((0143)\)    & \(0\)      & \(2\cdot 3^{20}\)  \\
  \(1\) & \(9,13\) &                                     & 1  & 3  & \(1^2\)   & \(4\)   & \(0\)   & \(43\)      & \(3^22^2\) & \(\mathrm{d}.25\)      & \((0143)\)    & \(0\)      & \(3^{20}\)         \\
  \(9\) & \(10,14\)& \(\pm 1;\alpha,1,\gamma,0\)         & 2  & 2  & \(1\)     & \(4\)   & \(0\)   & \(4^2\)     & \(4^22^2\) & \(\mathrm{G}.19\)      & \((2143)\)    & \(0\)      & \(3^{24}\)         \\
 \(12\) & \(10,14\)&                                     & 2  & 3  & \(1\)     & \(4\)   & \(0\)   & \(4^2\)     & \(432^2\)  & \(\mathrm{G}.19\)      & \((2143)\)    & \(0\)      & \(3^{23}\)         \\
  \(4\) & \(10,14\)&                                     & 2  & 3  & \(1\)     & \(4\)   & \(0\)   & \(4^2\)     & \(432^2\)  & \(\mathrm{G}.19\)      & \((2143)\)    & \(0\)      & \(3^{22}\)         \\
\hline
\end{tabular}
\end{center}
\end{table}

%\newpage

%--------------------------------------------------------------------------------

\begin{proof}
(Proof of Theorem
\ref{thm:TKTd25TreeCc5})
Since every mainline vertex \(m_n\) of the tree \(\mathcal{T}\)
has several capable children, \(C_1(m_n)\ge 2\),
but every capable vertex \(v\) of depth \(1\)
has only terminal children, \(C_1(v)=0\),
according to Proposition
\ref{prp:TKTd25TreeCc5},
the depth of the tree is \(\mathrm{dp}(\mathcal{T})=2\).
In this case,
the cardinality of a branch \(\mathcal{B}\) is the sum of the number \(N_1(m_n)\) of immediate descendants of the branch root \(m_n\)
and the numbers \(N_1(v_i)\) of terminal children of capable vertices \(v_i\) of depth \(1\),
\(2\le i\le C_1(m_n)\) (where \(v_1\) is the next mainline vertex and must be discouraged),
according to Formula
\eqref{eqn:BranchDepth2},
\[\#\mathcal{B}=N_1(m_n)+\sum_{i=2}^{C_1(m_n)}\,N_1(v_i).\]
Applied to the primitive period, this yields
\(\#\mathcal{B}(11)=15+(15+15)=45\) and \(\#\mathcal{B}(12)=15+25=40\). 
According to Formula
\eqref{eqn:TreeWidth2},
the width of the tree is the maximum of all sums of the shape
\[\#\mathrm{Lyr}_n{\mathcal{T}}=\#\lbrace v\in\mathcal{T}\mid\mathrm{lo}(v)=n\rbrace=N_1(m_{n-1})+\sum_{i=1}^{C_1(m_{n-2})}\,N_1(v_i(m_{n-2})),\]
taken over all branch roots \(m_n\) with logarithmic orders \(n_\ast+2\le n\le n_\ast+\ell_\ast+\ell+1\).
Applied to \(n_\ast=11\), \(\ell_\ast=0\), and \(\ell=2\), this yields
\(\mathrm{wd}(\mathcal{T})=\max(15+(15+15),15+25)=\max(45,40)=45\).
Finally, we have \(\mathrm{IC}(\mathcal{T})=45+40=85\).
\end{proof}

%\newpage

%--------------------------------------------------------------------------------

\begin{corollary}
\label{cor:TKTd25TreeCc5}
\textbf{(Actions and relation ranks.)} \\
The algebraic invariants of the vertices of the structured coclass-\(5\) tree \(\mathcal{T}^5{R_4^5}\) are listed in Table
\ref{tbl:TKTd25TreeCc5}.
In particular:
\begin{enumerate}
\item
There are no groups with GI-action, let alone with RI- or \(V_4\)-action.
\item
The relation rank is given by
\(\mu=5\) for the mainline vertices \(\stackbin[0]{n-4,n}{G}\begin{pmatrix}0&1\\ 0&0\end{pmatrix}\) with \(n\ge 11\),
and the capable vertices \(\stackbin[0]{n-4,n}{G}\begin{pmatrix}0&1\\ \pm 1&0\end{pmatrix}\) of depth \(1\) with \(n\ge 12\), and
\(\mu=4\) otherwise.
\end{enumerate}
\end{corollary}

%--------------------------------------------------------------------------------

\begin{proof}
(of Corollary
\ref{cor:TKTd25TreeCc5})
The existence of an RI-action on \(G\) has been checked by means of
an algorithm involving the \(p\)-covering group of \(G\), written for MAGMA
\cite{MAGMA}.
The other claims follow immediately from Table
\ref{tbl:TKTd25TreeCc5},
continued indefinitely with the aid of the periodicity in Proposition
\ref{prp:TKTd25TreeCc5}. 
\end{proof}

%\newpage

%--------------------------------------------------------------------------------

\section{The forest of \(3\)-groups with coclass \(1\)}
\label{s:Coclass1}
\noindent
The coclass forests \(\mathcal{F}(1)\) and \(\mathcal{F}(2)\),
and even their metabelian skeletons, sporadic parts, and individual coclass trees,
do not reveal any isomorphism to higher coclass forests \(\mathcal{F}(r)\), with \(r\ge 3\),
or parts of them.
The metabelian skeleton of the coclass forest \(\mathcal{F}(3)\)
is isomorphic to the metabelian skeleton of any coclass forest \(\mathcal{F}(r)\), with odd \(r\ge 5\),
according to Nebelung
\cite{Ne},
but neither its sporadic part nor its coclass trees
are isomorphic to the corresponding components of other coclass forests.

Whereas the complexity of the pre-periodic forests \(\mathcal{F}(2)\) and \(\mathcal{F}(3)\) is very high,
the simplest forest \(\mathcal{F}(1)\) can be described easily.
In particular,
the coclass-\(1\) forest \(\mathcal{F}(1)=\mathcal{T}^1{R_1^1}\)
coincides with its unique coclass tree
arising from the abelian root \(R_1^1:=\langle 9,2\rangle\simeq C_3\times C_3\).
Its sporadic part \(\mathcal{F}_0(1)=\emptyset\) is void.

%\newpage

%--------------------------------------------------------------------------------

\begin{figure}[hb]
\caption{The unique coclass-\(1\) tree \(\mathcal{T}^1\langle 9,2\rangle\) with mainline of type \(\mathrm{a}.1^\ast\)}
\label{fig:TKTa1TreeCc1}

\input{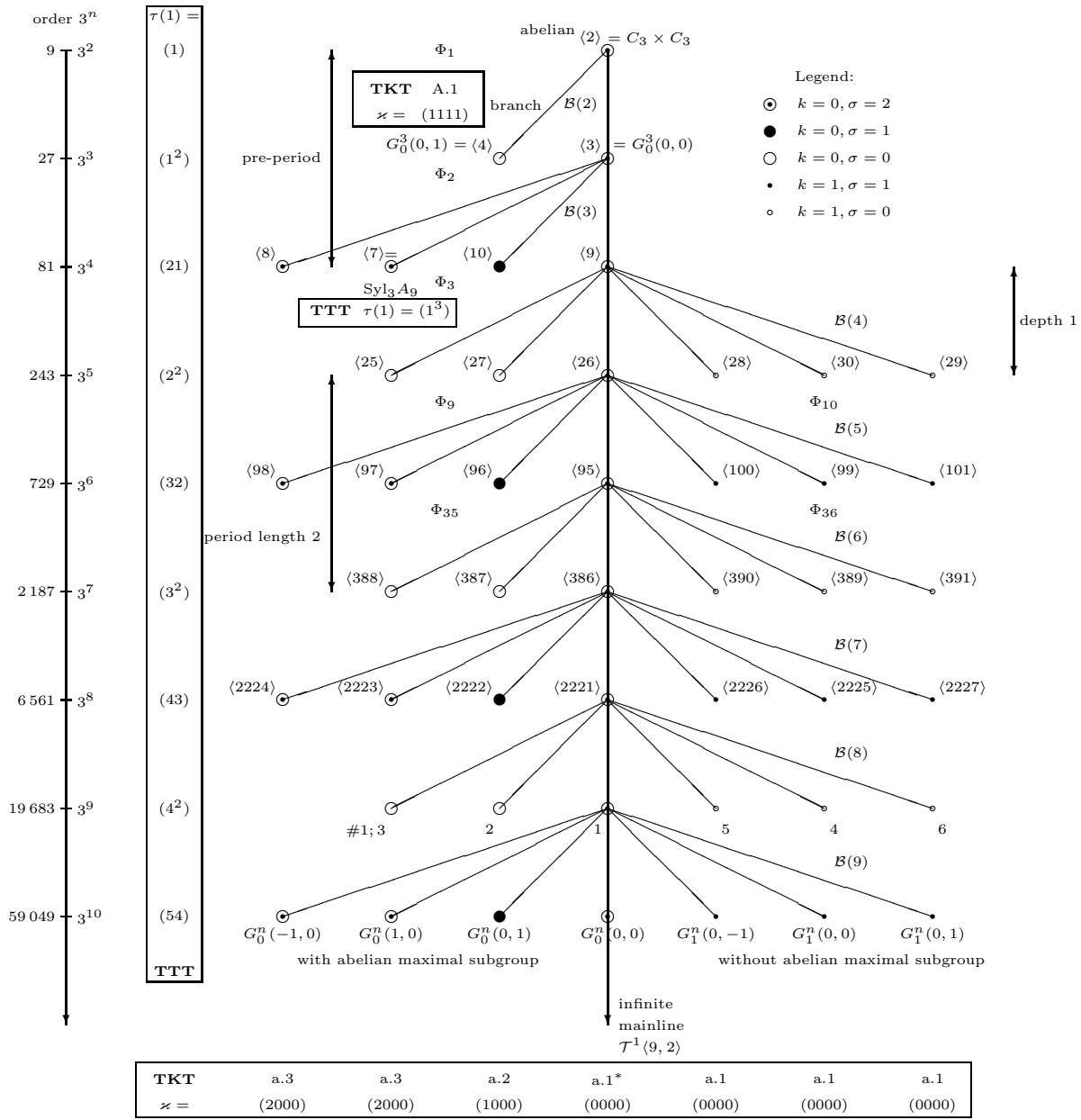}

\end{figure}

%\newpage

%--------------------------------------------------------------------------------

\subsection{The unique mainline of type \(\mathrm{a}.1^\ast\) for coclass \(r=1\)}
\label{ss:Coclass1}

\begin{proposition}
\label{prp:TKTa1TreeCc1}
\textbf{(Periodicity and descendant numbers.)} \\
The branches \(\mathcal{B}(i)\), \(i\ge n_\ast=2\),
of the coclass-\(1\) tree \(\mathcal{T}^1\langle 9,2\rangle\)
with mainline vertices of transfer kernel type \(\mathrm{a}.1^\ast\), \(\varkappa=(0000)\),
are periodic with pre-period length \(\ell_\ast=2\)
and with primitive period length \(\ell=2\), that is,
\(\mathcal{B}(i+2)\simeq\mathcal{B}(i)\) are isomorphic as digraphs,
for all \(i\ge p_\ast=n_\ast+\ell_\ast=4\).

The graph theoretic structure of the tree is determined uniquely by the numbers
\(N_1\) of immediate descendants and \(C_1\) of capable immediate descendants
of the mainline vertices \(m_n\) with logarithmic order \(n=\mathrm{lo}(m_n)\ge n_\ast=2\): \\
\((N_1,C_1)=(2,1)\) for the root \(m_2\) with \(n=2\), \\
\((N_1,C_1)=(4,1)\) for the mainline vertex \(m_3\) with \(n=3\), \\
\((N_1,C_1)=(6,1)\) for mainline vertices \(m_n\) with even logarithmic order \(n\ge 4\), \\
\((N_1,C_1)=(7,1)\) for mainline vertices \(m_n\) with odd logarithmic order \(n\ge 5\).
\end{proposition}

%--------------------------------------------------------------------------------

\begin{proof}
(of Proposition
\ref{prp:TKTa1TreeCc1})
The statements concerning the numbers \(N_1(m_n)\) of immediate descendants
of the mainline vertices \(m_n\) with \(n\ge 3\)
are due to Blackburn
\cite[Thm. 4.2 and Thm 4.3, p. 88]{Bl2},
who distinguishes the groups according to their defect of commutativity \(k=k(G)\in\lbrace 0,1\rbrace\),
which is defined by
\(\lbrack\chi_2{G},\gamma_2{G}\rbrack=\gamma_{c+1-k}{G}\)
in terms of the lower central series \((\gamma_j{G})_{j\ge 1}\),
nilpotency class \(c\ge 3\), and the two-step centralizer \(\chi_2{G}\) of \(G\).

In detail, Blackburn proved that there are \\
\(4\) vertices \(v\) with defect \(k(v)=0\) in the pre-periodic branch \(\mathcal{B}(3)\), \\
and the primitive period \((\mathcal{B}(4),\mathcal{B}(5))\) of length \(\ell=2\) consists of \\
\(3\) vertices \(v\) with defect \(k(v)=0\), and
\(3\) vertices \(v\) with defect \(k(v)=1\), \\
together \(6\) vertices in branch \(\mathcal{B}(4)\), and \\
\(4\) vertices \(v\) with defect \(k(v)=0\), and
\(3\) vertices \(v\) with defect \(k(v)=1\), \\
together \(7\) vertices in branch \(\mathcal{B}(5)\). All vertices of the tree are metabelian.

The results were reproduced and supplemented with \(N_1(m_2)=2\)
by Nebelung
\cite[Thm. 5.1.17, pp. 179--180]{Ne},
and have been verified by ourselves independently by direct computation with MAGMA
\cite{MAGMA},
where the \(p\)-group generation algorithm by Newman and O'Brien
\cite{Nm2,Ob,HEO}
is implemented.

Accordingly,
the pre-period \((\mathcal{B}(2),\mathcal{B}(3))\) of length \(\ell_\ast=2\) consists of \\
\(2\) vertices \(v\) with defect \(k(v)=0\) in branch \(\mathcal{B}(2)\), and \\
\(4\) vertices \(v\) with defect \(k(v)=0\) in branch \(\mathcal{B}(3)\).

The claim of the virtual periodicity of branches
has been proved generally for any coclass tree by du Sautoy
\cite{dS},
and independently by Eick and Leedham-Green
\cite{EkLg}.
Here, the strict periodicity is also a consequence of Blackburn's results,
and has been tested up to \(\mathcal{B}(25)\) computationally.
\end{proof}

%\newpage

%--------------------------------------------------------------------------------

\renewcommand{\arraystretch}{1.2}

\begin{table}[hb]
%\caption{Root, pre-period \((\mathcal{B}(2),\mathcal{B}(3))\), and primitive period \((\mathcal{B}(4),\mathcal{B}(5))\) of \(\mathcal{T}^1\langle 9,2\rangle\)}
\caption{Data for \(3\)-groups \(G\) with \(2\le n=\mathrm{lo}(G)\le 6\) of the coclass tree \(\mathcal{T}^1\langle 9,2\rangle\)}
\label{tbl:TKTa1TreeCc1}
\begin{center}
\begin{tabular}{|c|c|l|l||c|c|c|c||c|c||c|c|l|c||c|c|}
\hline
 \(\#\) & \(n\) & \(a;z,w\)     & \(\langle 3^n,i\rangle\)          & \(k\) & dp & dl & \(\zeta\) & \(\mu\) & \(\nu\) & \(\tau(1)\) & \(\tau_2\) & Type                  & \(\varkappa\) & \(\sigma\) & \(\#\mathrm{Aut}\)  \\
\hline
  \(1\) & \(2\) & \(0;0,0\)     & \(\langle 3^2,2\rangle\)          & \(0\) & 0  & 1  & \(1^2\)   & \(3\)   & \(3\)   & \(1\)       & \(0\)      & \(\mathrm{a}.1^\ast\) & \((0000)\)    & \(2\)      & \(2^4\cdot 3\)      \\
\hline
  \(1\) & \(3\) & \(0;0,0\)     & \(\langle 3^3,3\rangle\)          & \(0\) & 0  & 2  & \(1\)     & \(4\)   & \(2\)   & \(1^2\)     & \(1\)      & \(\mathrm{a}.1^\ast\) & \((0000)\)    & \(2^\ast\) & \(2^4\cdot 3^{3}\)  \\
  \(1\) & \(3\) & \(0;0,1\)     & \(\langle 3^3,4\rangle\)          & \(0\) & 1  & 2  & \(1\)     & \(2\)   & \(0\)   & \(1^2\)     & \(1\)      & \(\mathrm{A}.1\)      & \((1111)\)    & \(0\)      & \(2\cdot 3^{3}\)    \\
\hline
  \(1\) & \(4\) & \(0;0,0\)     & \(\langle 3^4,9\rangle\)          & \(0\) & 0  & 2  & \(1\)     & \(4\)   & \(1\)   & \(21\)      & \(1^2\)    & \(\mathrm{a}.1^\ast\) & \((0000)\)    & \(2\)      & \(2^2\cdot 3^{5}\)  \\
  \(1\) & \(4\) & \(0;0,1\)     & \(\langle 3^4,10\rangle\)         & \(0\) & 1  & 2  & \(1\)     & \(3\)   & \(0\)   & \(21\)      & \(1^2\)    & \(\mathrm{a}.2\)      & \((1000)\)    & \(1^\ast\) & \(2\cdot 3^{5}\)    \\
  \(1\) & \(4\) & \(0;1,0\)     & \(\langle 3^4,7\rangle\)          & \(0\) & 1  & 2  & \(1\)     & \(3\)   & \(0\)   & \(1^3\)     & \(1^2\)    & \(\mathrm{a}.3\)      & \((2000)\)    & \(2^\ast\) & \(2^2\cdot 3^{4}\)  \\
  \(1\) & \(4\) & \(0;-1,0\)    & \(\langle 3^4,8\rangle\)          & \(0\) & 1  & 2  & \(1\)     & \(3\)   & \(0\)   & \(21\)      & \(1^2\)    & \(\mathrm{a}.3\)      & \((2000)\)    & \(2^\ast\) & \(2^2\cdot 3^{4}\)  \\
\hline
  \(1\) & \(5\) & \(0;0,0\)     & \(\langle 3^5,26\rangle\)         & \(0\) & 0  & 2  & \(1\)     & \(4\)   & \(1\)   & \(2^2\)     & \(21\)     & \(\mathrm{a}.1^\ast\) & \((0000)\)    & \(2^\ast\) & \(2^2\cdot 3^{7}\)  \\
  \(1\) & \(5\) & \(0;0,1\)     & \(\langle 3^5,27\rangle\)         & \(0\) & 1  & 2  & \(1\)     & \(3\)   & \(0\)   & \(2^2\)     & \(21\)     & \(\mathrm{a}.2\)      & \((1000)\)    & \(0\)      & \(2\cdot 3^{7}\)    \\
  \(1\) & \(5\) & \(0;1,0\)     & \(\langle 3^5,25\rangle\)         & \(0\) & 1  & 2  & \(1\)     & \(3\)   & \(0\)   & \(2^2\)     & \(21\)     & \(\mathrm{a}.3\)      & \((2000)\)    & \(0\)      & \(2\cdot 3^{6}\)    \\
  \(3\) & \(5\) & \(1;0,w\)     & \(\langle 3^5,28..30\rangle\)     & \(1\) & 1  & 2  & \(1\)     & \(3\)   & \(0\)   & \(21\)      & \(21\)     & \(\mathrm{a}.1\)      & \((0000)\)    & \(0\)      & \(2\cdot 3^{6}\)    \\
\hline
  \(1\) & \(6\) & \(0;0,0\)     & \(\langle 3^6,95\rangle\)         & \(0\) & 0  & 2  & \(1\)     & \(4\)   & \(1\)   & \(32\)      & \(2^2\)    & \(\mathrm{a}.1^\ast\) & \((0000)\)    & \(2\)      & \(2^2\cdot 3^{9}\)  \\
  \(1\) & \(6\) & \(0;0,1\)     & \(\langle 3^6,96\rangle\)         & \(0\) & 1  & 2  & \(1\)     & \(3\)   & \(0\)   & \(32\)      & \(2^2\)    & \(\mathrm{a}.2\)      & \((1000)\)    & \(1^\ast\) & \(2\cdot 3^{9}\)    \\
  \(2\) & \(6\) & \(0;\pm 1,0\) & \(\langle 3^6,97\vert 98\rangle\) & \(0\) & 1  & 2  & \(1\)     & \(3\)   & \(0\)   & \(32\)      & \(2^2\)    & \(\mathrm{a}.3\)      & \((2000)\)    & \(2^\ast\) & \(2^2\cdot 3^{8}\)  \\
  \(3\) & \(6\) & \(1;0,w\)     & \(\langle 3^6,99..101\rangle\)    & \(1\) & 1  & 2  & \(1\)     & \(3\)   & \(0\)   & \(2^2\)     & \(2^2\)    & \(\mathrm{a}.1\)      & \((0000)\)    & \(1^\ast\) & \(2\cdot 3^{8}\)    \\
\hline
\end{tabular}
\end{center}
\end{table}

%\newpage

%--------------------------------------------------------------------------------

\begin{theorem}
\label{thm:TKTa1TreeCc1}
\textbf{(Graph theoretic and algebraic invariants.)} \\
The coclass-\(1\) tree \(\mathcal{T}:=\mathcal{T}^1{R_1^1}\)
of all finite \(3\)-groups \(G\not\simeq C_9\) with coclass \(\mathrm{cc}(G)=1\)
arises from the abelian root \(R_1^1:=\langle 9,2\rangle\simeq C_3\times C_3\)
and has the following graph theoretic properties.
\begin{enumerate}
\item
The pre-period \((\mathcal{B}(2),\mathcal{B}(3))\) of length \(\ell_\ast=2\) is irregular.
\item
The cardinalities of the irregular branches are
\(\#\mathcal{B}(2)=2\) and \(\#\mathcal{B}(3)=4\).
\item
The branches \(\mathcal{B}(i)\), \(i\ge 4\), are periodic
with primitive period \((\mathcal{B}(4),\mathcal{B}(5))\) of length \(\ell=2\).
\item
The cardinalities of the regular branches are
\(\#\mathcal{B}(4)=6\) and \(\#\mathcal{B}(5)=7\).
\item
Depth, width, and information content of the tree are given by
\begin{equation}
\mathrm{dp}(\mathcal{T}^1{R_1^1})=1, \quad \mathrm{wd}(\mathcal{T}^1{R_1^1})=7, \quad \text{ and } \quad \mathrm{IC}(\mathcal{T}^1{R_1^1})=19.
\end{equation}
\end{enumerate}
The algebraic invariants of the groups represented by vertices forming
the pre-period \((\mathcal{B}(2),\mathcal{B}(3))\)
and the primitive period \((\mathcal{B}(4),\mathcal{B}(5))\) of the tree
are given in Table
\ref{tbl:TKTa1TreeCc1}.
The leading eight branches \(\mathcal{B}(2),\ldots,\mathcal{B}(9)\) are drawn in Figure
\ref{fig:TKTa1TreeCc1}.
All vertices of the tree are metabelian.
\end{theorem}

%\newpage

%--------------------------------------------------------------------------------

\begin{proof}
(of Theorem
\ref{thm:TKTa1TreeCc1})
According to Proposition
\ref{prp:TKTa1TreeCc1},
the logarithmic order of the tree root, respectively of the periodic root, is
\(n_\ast=2\), respectively \(p_\ast=n_\ast+\ell_\ast=4\).

Since \(C_1(m_n)=1\) for all mainline vertices \(m_n\) with \(n\ge n_\ast\),
according to Proposition
\ref{prp:TKTa1TreeCc1},
the unique capable child of \(m_n\) is \(m_{n+1}\),
and each branch has depth \(\mathrm{dp}(\mathcal{B}(n))=1\), for \(n\ge n_\ast\).
Consequently, the tree is also of depth \(\mathrm{dp}(\mathcal{T})=1\).

With the aid of Formula
\eqref{eqn:BranchDepth1}
in Theorem
\ref{thm:DescendantNumbers},
the claims (2) and (4) are consequences of Proposition
\ref{prp:TKTa1TreeCc1}:\\
\(\#\mathcal{B}(2)=N_1(m_2)=2\),
\(\#\mathcal{B}(3)=N_1(m_3)=4\),
\(\#\mathcal{B}(4)=N_1(m_4)=6\), and
\(\#\mathcal{B}(5)=N_1(m_5)=7\).

According to Formula
\eqref{eqn:TreeWidth1}
in Corollary
\ref{cor:TreeWidth},
where \(n\) runs from \(n_\ast+1=3\) to \(n_\ast+\ell_\ast+\ell+0=6\),
the tree width is the maximum \(\mathrm{wd}(\mathcal{T})=7\)
of the expressions \(N_1(m_2)=2\), \(N_1(m_3)=4\), \(N_1(m_4)=6\), and \(N_1(m_5)=7\).

The information content of the tree is given by Formula
\eqref{eqn:InfoCont}
in the Definition
\ref{dfn:InfoCont}: \\
\(\mathrm{IC}(\mathcal{T})=(\#\mathcal{B}(2)+\#\mathcal{B}(3))+(\#\mathcal{B}(4)+\#\mathcal{B}(5))=(2+4)+(6+7)=19\).

The algebraic invariants in Table
\ref{tbl:TKTa1TreeCc1},
that is,
defect of commutativity \(k\),
depth \(\mathrm{dp}\),
derived length \(\mathrm{dl}\),
abelian type invariants of the centre \(\zeta\),
relation rank \(\mu\),
nuclear rank \(\nu\),
abelian quotient invariants \(\tau(1)\) of the first maximal subgroup,
respectively \(\tau_2\) of the commutator subgroup,
transfer kernel type \(\varkappa\),
action flag \(\sigma\), and
the factorized order \(\#\mathrm{Aut}\) of the automorphism group
have been computed by means of program scripts written for MAGMA
\cite{MAGMA}.

Each group is characterized by the parameters
of the normalized representative \(G_a^n(z,w)\) of its isomorphism class,
according to Formula
\eqref{eqn:PresentationCc1},
and by its identifier \(\langle 3^n,i\rangle\) in the SmallGroups Database
\cite{BEO}.

The column with header \(\#\) contains
the number of groups with identical invariants (except the presentation and identifier),
for each row.
\end{proof}

%--------------------------------------------------------------------------------

\begin{corollary}
\label{cor:TKTa1TreeCc1}
\textbf{(Actions and relation ranks.)} \\
The algebraic invariants of the vertices of the structured coclass-\(1\) tree \(\mathcal{T}^1{R}\)
with abelian root \(R=\langle 9,2\rangle\simeq C_3\times C_3\),
which is drawn in Figure
\ref{fig:TKTa1TreeCc1},
are listed in Table
\ref{tbl:TKTa1TreeCc1}.
In particular:
\begin{enumerate}
\item
The groups with \(V_4\)-action are
the root \(R\),
all mainline vertices \(G_0^n(0,0)\), \(n\ge 3\),
and the terminal vertices \(G_0^n(\pm 1,0)\) with even logarithmic order \(n\ge 4\).
\item
With respect to the transfer kernel types, 
all mainline groups of type \(\mathrm{a}.1^\ast\), \(\varkappa=(0000)\),
and the leaves of type \(\mathrm{a}.3\), \(\varkappa\sim (2000)\), with odd class \(c=n-1\ge 3\),
possess a \(V_4\)-action.
\item
The relation rank is given by
\(\mu=4\) for the mainline vertices \(G_0^n(0,0)\), \(n\ge 3\),
\(\mu=2\) for the terminal extraspecial group \(G_0^3(0,1)\), and
\(\mu=3\) otherwise.
\item
All terminal vertices with odd class,
and the mainline vertices with even class,
possess an RI-action.
The terminal vertices with odd class are Schur\(+1\) \(\sigma\)-groups
\cite{BBH,BBH2}.
\end{enumerate}
\end{corollary}

%--------------------------------------------------------------------------------

\begin{proof}
(of Corollary
\ref{cor:TKTa1TreeCc1})
The existence of an RI-action on \(G\) has been checked by means of
an algorithm involving the \(p\)-covering group of \(G\), written for MAGMA
\cite{MAGMA}.
The other claims follow immediately from Table
\ref{tbl:TKTa1TreeCc1},
continued indefinitely with the aid of the periodicity in Proposition
\ref{prp:TKTa1TreeCc1}.
A Schur\(+1\) \(\sigma\)-group has an RI-action and relation rank \(\mu\le 3\)
\cite{BBH,BBH2}.
\end{proof}

%\newpage

%--------------------------------------------------------------------------------

\section{Conclusion}
\label{s:Conclusion}
\noindent
In the core sections
\ref{s:PeriodicSporadic4}
and
\ref{s:PeriodicSporadic5}
of this paper,
we have elaborated our long desired proof that
the pruned tree of all finite \(3\)-groups with elementary bicyclic commutator quotient,
which do not arise as descendants of non-metabelian groups,
can be described with a finite amount of data.

\begin{theorem}
\label{thm:CoclassForestsIC}
\textbf{(Main Theorem on the Finite Information Content.)} \\
The total information content of the coclass forest \(\mathcal{F}(r)\) is given by
\begin{equation}
\label{eqn:CoclassForestsIC}
\mathrm{IC}(\mathcal{F}(r))=
\begin{cases}
19 & \text{ for } r=1, \\
739=515+224 & \text{ for } r=4, \\
501=207+294 & \text{ for each odd } r\ge 5, \\
581=357+224 & \text{ for each even } r\ge 6.
\end{cases}
\end{equation}
\end{theorem}

\begin{proof}
The total information content of a coclass forest
is the sum of
the cardinality of its sporadic part \(\mathcal{F}_0(r)\) and
the information contents of its \textit{pairwise non-isomorphic} coclass trees \(\mathcal{T}^rR_i^r\)
\[\mathrm{IC}(\mathcal{F}(r))=\#\mathcal{F}_0(r)+\sum_i\,\mathrm{IC}(\mathcal{T}^rR_i^r).\qedhere\]
\end{proof}

\noindent
Due to the exceptional complexity of the pre-periodic forests \(\mathcal{F}(2)\) and \(\mathcal{F}(3)\),
their information contents are unknown up to now.
However, since they are certainly finite,
this does not obfuscate our clear and beautiful results
concerning the infinitely many co-periodic forests \(\mathcal{F}(r)\) with \(r\ge 4\),
which can be reduced to the finite information content of the primitive co-period \((\mathcal{F}(5),\mathcal{F}(6))\).

%\newpage

%--------------------------------------------------------------------------------

\section{Acknowledgement}
\label{s:Thanks}
\noindent
We gratefully acknowledge that our research was supported by
the Austrian Science Fund (FWF): Project P 26008-N25.
Indebtedness is expressed to the anonymous referees
for valuable suggestions concerning the readability
and, in particular, for drawing our attention to the paper
\cite{Ek}.

%\newpage

%--------------------------------------------------------------------------------

%--------------------------------------------------------------------------------


\begin{thebibliography}{XX}
%
\bibitem{dS}
M. du Sautoy,
\textit{Counting \(p\)-groups and nilpotent groups},
Inst. Hautes \'Etudes Sci. Publ. Math.
\textbf{92}
(2001)
63--112.
%
\bibitem{EkLg}
B. Eick and C. Leedham-Green,
\textit{On the classification of prime-power groups by coclass},
Bull. London Math. Soc.
\textbf{40} (2)
(2008),
274--288,
DOI 10.1112/blms/bdn007.
%
\bibitem{Lg}
C. R. Leedham-Green,
\textit{The structure of finite \(p\)-groups},
J. London Math. Soc.
\textbf{50}
(1994),
49--67,
DOI 10.1112/jlms/50.1.49.
%
\bibitem{Sv}
A. Shalev,
\textit{The structure of finite \(p\)-groups: effective proof of the coclass conjectures},
Invent. Math.
\textbf{115}
(1994),
315--345.
%
\bibitem{dSSg}
M. du Sautoy and D. Segal,
\textit{Zeta functions of groups},
pp. 249--286,
in: New horizons in pro-\(p\) groups,
Progress in Mathematics, Vol.
\textbf{184},
Birkh\"auser, Basel,
2000.
%
\bibitem{Ek}
B. Eick,
\textit{Metabelian \(p\)-groups and coclass theory},
J. Algebra
\textbf{421}
(2015),
102--118.
%
\bibitem{Ne}
B. Nebelung,
\textit{Klassifikation metabelscher \(3\)-Gruppen
mit Faktorkommutatorgruppe vom Typ \((3,3)\)
und Anwendung auf das Kapitulationsproblem},
Inauguraldissertation
(W. Jehne),
Band 1,
Universit\"at zu K\"oln,
1989.
%
\bibitem{LgNm}
C. R. Leedham-Green and M. F. Newman,
\textit{Space groups and groups of prime power order I},
Arch. Math.
\textbf{35}
(1980),
193--203,
DOI 10.1007/bf01235338.
%
\bibitem{Bl2}
N. Blackburn,
\textit{On a special class of \(p\)-groups},
Acta Math.
\textbf{100}
(1958),
45--92.
%
\bibitem{Ma9}
D. C. Mayer,
\textit{Artin transfer patterns on descendant trees of finite \(p\)-groups},
Adv. Pure Math.
\textbf{6}
(2016),
no. 2,
66--104,
DOI 10.4236/apm.2016.62008,
Special Issue on Group Theory Research,
January 2016.
%
\bibitem{Ma2}
D. C. Mayer,
\textit{Transfers of metabelian \(p\)-groups},
Monatsh. Math.
\textbf{166}
(2012),
no. 3--4,
467--495,
DOI 10.1007/s00605-010-0277-x.
%
\bibitem{Nm2}
M. F. Newman,
\textit{Determination of groups of prime-power order},
pp. 73--84,
in: Group Theory, Canberra, 1975,
Lecture Notes in Math.,
vol. 573,
Springer,
Berlin,
1977.
%
\bibitem{Ob}
E. A. O'Brien, 
\textit{The \(p\)-group generation algorithm}, 
J. Symbolic Comput.
\textbf{9}
(1990),
677--698,
DOI 10.1016/s0747-7171(08)80082-x.
%
\bibitem{HEO}
D. F. Holt, B. Eick and E. A. O'Brien,
\textit{Handbook of computational group theory},
Discrete mathematics and its applications,
Chapman and Hall/CRC Press,
2005.
%
\bibitem{BCP}
W. Bosma, J. Cannon, and C. Playoust,
\textit{The Magma algebra system. I. The user language},
J. Symbolic Comput.
\textbf{24}
(1997),
235--265.
%
\bibitem{BCFS}
W. Bosma, J. J. Cannon, C. Fieker, and A. Steels (eds.),
\textit{Handbook of Magma functions}
(Edition 2.23,
Sydney,
2017).
%
\bibitem{MAGMA}
MAGMA Developer Group,
MAGMA \textit{Computational Algebra System},
Version 2.23-6,
Sydney,
2017, \\
\verb+(http://magma.maths.usyd.edu.au)+.
%
\bibitem{BEO1}
H. U. Besche, B. Eick and E. A. O'Brien,
\textit{A millennium project: constructing small groups},
Int. J. Algebra Comput.
\textbf{12}
(2002),
623-644,
DOI 10.1142/s0218196702001115.
%
\bibitem{BEO}
H. U. Besche, B. Eick and E. A. O'Brien,
\textit{The SmallGroups Library --- a Library of Groups of Small Order},
2005,
an accepted and refereed GAP package, available also in MAGMA.
%
\bibitem{GNO}
G. Gamble, W. Nickel and E. A. O'Brien,
ANUPQ --- \textit{p-Quotient and p-Group Generation Algorithms},
2006,
an accepted GAP package, available also in MAGMA.
%
\bibitem{As1}
J. A. Ascione,
\textit{On \(3\)-groups of second maximal class},
Ph.D. Thesis
(M. F. Newman),
Australian National University,
Canberra,
1979.
%
\bibitem{As2}
J. A. Ascione,
\textit{On \(3\)-groups of second maximal class},
Bull. Austral. Math. Soc.
\textbf{21}
(1980),
473--474.
%
\bibitem{AHL}
J. A. Ascione, G. Havas and C. R. Leedham-Green,
\textit{A computer aided classification of certain groups of prime power order},
Bull. Austral. Math. Soc.
\textbf{17}
(1977),
257--274, Corrigendum 317--319, Microfiche Supplement p. 320,
DOI 10.1017/s0004972700010467.
%
\bibitem{Ma6}
D. C. Mayer,
\textit{Periodic bifurcations in descendant trees of finite \(p\)-groups},
Adv. Pure Math.
\textbf{5}
(2015),
no. 4,
162--195,
DOI 10.4236/apm.2015.54020,
Special Issue on Group Theory,
March 2015.
%
\bibitem{Ma16}
D. C. Mayer,
\textit{Modeling rooted in-trees by finite \(p\)-groups},
to appear in the open access book \textit{Graph Theory}, Ed. B. Sirmacek, InTech, January 2018,
\verb+http://www.algebra.at/ModelingInTrees.pdf+.
%
\bibitem{Nm}
M. F. Newman,
\textit{Groups of prime-power order},
Groups --- Canberra 1989,
Lecture Notes in Math.,
vol. 1456,
Springer,
1990,
pp. 49--62,
DOI 10.1007/bfb0100730.
%
\bibitem{Bl1}
N. Blackburn,
\textit{On prime-power groups in which the derived group has two generators},
Proc. Camb. Phil. Soc.
\textbf{53}
(1957),
19--27.
%
\bibitem{Ne2}
B. Nebelung,
\textit{Anhang zu Klassifikation metabelscher \(3\)-Gruppen
mit Faktorkommutatorgruppe vom Typ \((3,3)\)
und Anwendung auf das Kapitulationsproblem},
Inauguraldissertation,
Band 2,
Universit\"at zu K\"oln,
1989.
%
\bibitem{BBH}
N. Boston, M. R. Bush and F. Hajir,
\textit{Heuristics for \(p\)-class towers of imaginary quadratic fields},
Math. Ann.
\textbf{368}
(2017),
no. 1,
633--669,
DOI 10.1007/s00208-016-1449-3.
%
\bibitem{BBH2}
N. Boston, M. R. Bush and F. Hajir,
\textit{Heuristics for \(p\)-class towers of real quadratic fields},
preprint,
2017.
%
\end{thebibliography}
\end{document}